\documentclass[12pt]{amsart}

\usepackage{amssymb}
\usepackage{amsmath, amsfonts, amsthm, amssymb, cite}
\usepackage[usenames,dvipsnames]{color}
\usepackage[all]{xy}
\usepackage{nomencl}
\makeindex
\makenomenclature

\usepackage{fancyhdr}
\DeclareMathOperator{\res}{res}

\begin{document}



\newtheorem{theorem}{Theorem}[section]
\newcommand{\mar}[1]{{\marginpar{\textsf{#1}}}}
\numberwithin{equation}{section}
\newtheorem*{theorem*}{Theorem}
\newtheorem{prop}[theorem]{Proposition}
\newtheorem*{prop*}{Proposition}
\newtheorem{lemma}[theorem]{Lemma}
\newtheorem{corollary}[theorem]{Corollary}
\newtheorem*{conj*}{Conjecture}
\newtheorem*{corollary*}{Corollary}
\newtheorem{definition}[theorem]{Definition}
\newtheorem*{definition*}{Definition}
\newtheorem{remarks}[theorem]{Remarks}
\newtheorem*{remarks*}{remarks}

\newtheorem*{not*}{Notation}
\newcommand\pa{\partial}
\newcommand\cohom{\operatorname{H}}
\newcommand\Td{\operatorname{Td}}
\newcommand\Trig{\operatorname{Trig}}
\newcommand\Hom{\operatorname{Hom}}
\newcommand\End{\operatorname{End}}
\newcommand\Ker{\operatorname{Ker}}
\newcommand\Ind{\operatorname{Ind}}
\newcommand\cker{\operatorname{coker}}
\newcommand\oH{\operatorname{H}}
\newcommand\oK{\operatorname{K}}
\newcommand\codim{\operatorname{codim}}
\newcommand\Exp{\operatorname{Exp}}
\newcommand\CAP{{\mathcal AP}}
\newcommand\T{\mathbb T}
\newcommand{\M}{\mathcal{M}}
\newcommand\ep{\epsilon}
\newcommand\te{\tilde e}
\newcommand\Dd{{\mathcal D}}

\newcommand\what{\widehat}
\newcommand\wtit{\widetilde}
\newcommand\mfS{{\mathfrak S}}
\newcommand\cA{{\mathcal A}}
\newcommand\maA{{\mathcal A}}
\newcommand\F{{\mathcal F}}
\newcommand\maN{{\mathcal N}}
\newcommand\cM{{\mathcal M}}
\newcommand\maE{{\mathcal E}}
\newcommand\cF{{\mathcal F}}
\newcommand\maG{{\mathcal G}}
\newcommand\cG{{\mathcal G}}
\newcommand\cH{{\mathcal H}}
\newcommand\maH{{\mathcal H}}
\renewcommand\H{{\mathcal H}}
\newcommand\cO{{\mathcal O}}
\newcommand\cR{{\mathcal R}}
\newcommand\cS{{\mathcal S}}
\newcommand\cU{{\mathcal U}}
\newcommand\cV{{\mathcal V}}
\newcommand\cX{{\mathcal X}}
\newcommand\cD{{\mathcal D}}
\newcommand\cnn{{\mathcal N}}
\newcommand\wD{\widetilde{D}}
\newcommand\wL{\widetilde{L}}
\newcommand\wM{\widetilde{M}}
\newcommand\wV{\widetilde{V}}
\newcommand\Ee{{\mathcal E}}
\newcommand{\npartial}{\slash\!\!\!\partial}
\newcommand{\Heis}{\operatorname{Heis}}
\newcommand{\Solv}{\operatorname{Solv}}
\newcommand{\Spin}{\operatorname{Spin}}
\newcommand{\SO}{\operatorname{SO}}
\newcommand{\ind}{\operatorname{ind}}
\newcommand{\Index}{\operatorname{Index}}
\newcommand{\ch}{\operatorname{ch}}
\newcommand{\rank}{\operatorname{rank}}
\newcommand{\G}{\Gamma}
\newcommand{\HK}{\operatorname{HK}}
\newcommand{\Dix}{\operatorname{Dix}}

\newcommand{\tM}{\tilde{M}}  
\newcommand{\tS}{\tilde{S}}
\newcommand{\tH}{\tilde{\mathcal H}}
\newcommand{\tg}{\tilde{g}}
\newcommand{\tx}{\tilde{x}}
\newcommand{\ty}{\tilde{y}}
\newcommand{\ox}{\otimes}


\newcommand{\abs}[1]{\lvert#1\rvert}
 \newcommand{\A}{{\mathcal A}}
        \newcommand{\D}{{\mathcal D}}
\newcommand{\HH}{{\mathcal H}}
        \newcommand{\LL}{{\mathcal L}}
        \newcommand{\B}{{\mathcal B}}
        \newcommand{\K}{{\mathcal K}}
\newcommand{\oo}{{\mathcal O}}
         \newcommand{\PP}{{\mathcal P}}
         \newcommand{\Q}{{\mathcal Q}}
        \newcommand{\s}{\sigma}
        \newcommand{\coker}{{\mbox coker}}
        \newcommand{\dd}{|\D|}
        \newcommand{\n}{\parallel}
\newcommand{\bma}{\left(\begin{array}{cc}}
\newcommand{\ema}{\end{array}\right)}

\newcommand{\bca}{\left(\begin{array}{c}}
\newcommand{\eca}{\end{array}\right)}
\newcommand{\sr}{\stackrel}
\newcommand{\da}{\downarrow}
\newcommand{\tD}{\tilde{\D}}
        \newcommand{\R}{\mathbb R}
        \newcommand{\C}{\mathbb C}
        \newcommand{\h}{\mathbb H}
\newcommand{\Z}{\mathcal Z}
\newcommand{\N}{\mathbb N}
\newcommand{\tto}{\longrightarrow}
\newcommand{\ZZ}{{\mathcal Z}}
\newcommand{\ben}{\begin{displaymath}}
        \newcommand{\een}{\end{displaymath}}
\newcommand{\be}{\begin{equation}}
\newcommand{\ee}{\end{equation}}
        \newcommand{\bean}{\begin{eqnarray*}}
        \newcommand{\eean}{\end{eqnarray*}}
\newcommand{\nno}{\nonumber\\}
\newcommand{\bea}{\begin{eqnarray}}
        \newcommand{\eea}{\end{eqnarray}}
\newcommand{\x}{\times}

\newcommand{\Ga}{\Gamma}
\newcommand{\e}{\epsilon}
\renewcommand{\L}{\mathcal{L}}
\newcommand{\supp}[1]{\operatorname{#1}}
\newcommand{\norm}[1]{\parallel\, #1\, \parallel}
\newcommand{\ip}[2]{\langle #1,#2\rangle}
\newcommand{\nc}{\newcommand}
\nc{\gf}[2]{\genfrac{}{}{0pt}{}{#1}{#2}}
\nc{\mb}[1]{{\mbox{$ #1 $}}}
\nc{\real}{{\mathbb R}}
\nc{\comp}{{\mathbb C}}
\nc{\ints}{{\mathbb Z}}
\nc{\Ltoo}{\mb{L^2({\mathbf H})}}
\nc{\rtoo}{\mb{{\mathbf R}^2}}
\nc{\slr}{{\mathbf {SL}}(2,\real)}
\nc{\slz}{{\mathbf {SL}}(2,\ints)}
\nc{\su}{{\mathbf {SU}}(1,1)}
\nc{\so}{{\mathbf {SO}}}
\nc{\hyp}{{\mathbb H}}
\nc{\disc}{{\mathbf D}}
\nc{\torus}{{\mathbb T}}
\newcommand{\tk}{\widetilde{K}}
\newcommand{\boe}{{\bf e}}\newcommand{\bt}{{\bf t}}
\newcommand{\vth}{\vartheta}
\newcommand{\CGh}{\widetilde{\CG}}
\newcommand{\db}{\overline{\partial}}
\newcommand{\tE}{\widetilde{E}}
\newcommand{\tr}{{\rm tr}}
\newcommand{\ta}{\widetilde{\alpha}}
\newcommand{\tb}{\widetilde{\beta}}
\newcommand{\txi}{\widetilde{\xi}}
\newcommand{\hV}{\hat{V}}
\newcommand{\IC}{\mathbf{C}}
\newcommand{\IZ}{\mathbf{Z}}
\newcommand{\IP}{\mathbf{P}}
\newcommand{\IR}{\mathbf{R}}
\newcommand{\IH}{\mathbf{H}}
\newcommand{\IG}{\mathbf{G}}
\newcommand{\IS}{\mathbf{S}}
\newcommand{\CC}{{\mathcal C}}
\newcommand{\CD}{{\mathcal D}}
\newcommand{\CS}{{\mathcal S}}
\newcommand{\CG}{{\mathcal G}}
\newcommand{\CL}{{\mathcal L}}
\newcommand{\CO}{{\mathcal O}}
\nc{\ca}{{\mathcal A}}
\nc{\cag}{{{\mathcal A}^\Gamma}}
\nc{\cg}{{\mathcal G}}
\nc{\chh}{{\mathcal H}}
\nc{\ck}{{\mathcal B}}
\nc{\cd}{{\mathcal D}}
\nc{\cl}{{\mathcal L}}
\nc{\cm}{{\mathcal M}}
\nc{\cn}{{\mathcal N}}
\nc{\cs}{{\mathcal S}}
\nc{\cz}{{\mathcal Z}}
\nc{\sind}{\sigma{\rm -ind}}

\newcommand\clFN{{\mathcal F_\tau(\mathcal N)}}       
\newcommand\clKN{{\mathcal K_\tau(\mathcal N)}}       
\newcommand\clQN{{\mathcal Q_\tau(\mathcal N)}}       %
\newcommand\tF{\tilde F}
\newcommand\clA{\mathcal A}
\newcommand\clH{\mathcal H}
\newcommand\clN{\mathcal N}
\newcommand\Del{\Delta}
\newcommand\g{\gamma}
\newcommand\eps{\varepsilon}
\newcommand\vf{\varphi}
\newcommand\E{\mathcal E}

\newcommand{\CDA}{\mathcal{C_D(A)}} 
\newcommand{\dslash}{{\pa\mkern-10mu/\,}}

\newcommand{\sepword}[1]{\quad\mbox{#1}\quad} 
\newcommand{\comment}[1]{\textsf{#1}}   

\newcommand{\anti}{\blacksquare}

\parindent=0.0in

 \title{Index Theory for Locally Compact Noncommutative Geometries}

\author{A. L. CAREY}
\address{Mathematical Sciences Institute,
 Australian National University\\
 Canberra ACT, 0200 AUSTRALIA\\
 e-mail: alan.carey@anu.edu.au\\}
\author{V. GAYRAL}
\address{Laboratoire de Math\'ematiques\\
Universit\'e Reims Champagne-Ardenne\\
Moulin de la Housse-BP 1039, 51687 Reims FRANCE and   
Laboratoire de Math\'ematiques et Applications de Metz
UMR 7122,
Universit\'e de Metz et CNRS,
Bat. A, Ile du Saulcy 
F-57045 METZ Cedex 1 
FRANCE \\
e-mail: victor.gayral@univ-reims.fr}
\author{A. RENNIE}
\address{School of Mathematics and Applied Statistics\\
University of Wollongong\\
Wollongong NSW, 2522, AUSTRALIA\\
e-mail: renniea@uow.edu.au}
\author{F. A. SUKOCHEV}
\address{School of Mathematics and Statistics,
University of New South Wales\\
Kensington NSW, 2052 AUSTRALIA\\
e-mail: f.sukochev@unsw.edu.au}

\begin{abstract}
Spectral triples for nonunital algebras model locally compact spaces in noncommutative 
geometry. In the present text, we prove the local index formula for  spectral triples over nonunital algebras,
without the assumption of local units in our algebra. 
This formula has been successfully used to calculate index pairings in 
numerous noncommutative examples.  The absence of any other effective method of
investigating index problems in geometries that are genuinely noncommutative,
particularly in the nonunital situation,
was a primary motivation for this study and we illustrate this point with two examples in 
 the text.

In order to understand what is new in our approach in the commutative setting we prove an analogue 
of the Gromov-Lawson relative index formula (for Dirac type operators) 
for even dimensional manifolds with bounded geometry, without
invoking compact supports. For odd dimensional manifolds our index formula appears to be
completely new.
As we prove our local index formula in the framework of semifinite noncommutative 
geometry we are also able to prove, for manifolds of bounded geometry, a version 
of Atiyah's $L^2$-index Theorem for covering spaces.  We also explain how to interpret
the McKean-Singer formula in the nonunital case.

In order to prove the local index formula, we develop an integration theory compatible with
a refinement of the existing pseudodifferential calculus for spectral triples. We also clarify
some aspects of index theory for nonunital algebras.
\end{abstract}

\maketitle


\tableofcontents
\parskip=4pt

\section{Introduction}

Our objective in writing this memoir is to establish
 a unified framework to deal with index theory on 
locally compact spaces, both commutative and noncommutative. 
In the commutative situation this entails index theory on noncompact manifolds 
where  Dirac-type operators, for example, typically have noncompact resolvent, are not Fredholm, 
and so do not have a well-defined index. 
In initiating this study we were 
also interested to understand previous approaches to this problem such as 
those of Gromov-Lawson \cite{GL} and Roe \cite{Ro} from a new viewpoint: 
that of noncommutative geometry. In this latter setting the main
tool, the Connes-Moscovici local index formula, is not adapted to nonunital
examples. Thus our primary objective here is 
to extend that theorem to this broader context.

Index theory provided one of the main motivations for noncommutative geometry.
In \cite{Co1, Co4} it is explained how to express index pairings between the 
between the K-theory and K-homology of noncommutative algebras
using Connes' Chern character formula.
In examples this formula can be difficult to compute.  A more tractable analytic formula was
established by Connes and Moscovici in \cite{CoM} using a representative of the Chern character
that arises from unbounded Kasparov modules or `spectral triples' as they have come to be known.
Their resulting `local index formula' is an analytic cohomological expression for index pairings
that has been exploited by many authors in calculations in  fully noncommutative settings.

In previous work \cite{CPRS2, CPRS3, CPRS4} some of the present authors found a new proof of the formula
that applied for unital spectral triples in semifinite von Neumann algebras. However for some
time the understanding of the Connes-Moscovici formula in nonunital situations has remained unsatisfactory.
The main result of this article is a residue 
formula of Connes-Moscovici type for calculating the index pairing between the 
$K$-homology of nonunital algebras and their $K$-theory. This latter view of index 
theory, as generalised by Kasparov's bivariant $KK$ functor,
is central to our approach and we follow the general philosophy 
enunciated by Higson and Roe, \cite{HR}. 
One of our main advances is to avoid ad hoc 
assumptions on our algebras (such as the existence of local units). 

To illustrate our main result in practice we present two examples in Section 6.
Elsewhere we will explain how a version of the example of nonunital Toeplitz theory in \cite{PR}
can be derived from our local index formula.

To understand what is new about our theorem in the commutative case we apply our residue formula to manifolds of bounded geometry, obtaining a cohomological
formula of Atiyah-Singer type for the index pairing. We also prove an $L^2$-index theorem for
coverings of such manifolds.

We now explain in some detail these and our other results.

\subsubsection*{The noncommutative results}
The index theorems we prove rely on a general nonunital noncommutative integration theory and the
index theory
developed in detail in Sections 2 and 3. 

Section 2 presents an integration theory for weights which is compatible with
Connes and Moscovici's approach to the pseudodifferential calculus for spectral triples. This integration
theory is the key technical innovation, and allows us to treat the unital and nonunital cases
on the same footing.

An important feature of our approach is that we can eliminate the need to assume the existence of 
`local units' which mimic the notion of compact support, \cite{GGISV,Re,Re2}. 
The difficulty with the local unit approach
is that there are no general results guaranteeing their existence. Instead we
identify subalgebras of integrable and square integrable elements of our algebra, without the need to 
control `supports'.

In Section 3 we introduce a  triple $(\A, \HH, \D)$ 
where $\HH$ is a Hilbert space, 
$\A$ is a (nonunital) $*$-algebra of operators represented in a semifinite 
von Neumann subalgebra of $\B(\HH)$, 
and $\D$ is a self-adjoint unbounded operator on $\HH$ whose resolvent need 
not be compact, not even in the sense of semifinite von Neumann algebras. Instead
we ask that the product $a(1+\D^2)^{-1/2}$ is compact, and it is the need to control
this product that produces much of the technical difficulty.

We remark that there are good cohomological reasons for taking 
the effort to prove our results in the setting of semifinite noncommutative 
geometry, and that these arguments are explained in \cite{CCu}. 
In particular, \cite[Th\'eor\`eme 15]{CCu}  identifies a class of cyclic cocycles on a
given algebra  which have a natural representation as Chern characters, provided
one allows semifinite Fredholm modules.

We 
refer to the case when $\D$ does not have compact resolvent as the `nonunital case', 
and justify this terminology in Lemma \ref{lem:nonunital-spec-trip}. Instead of requiring that $\D$ be Fredholm 
we show that a  spectral triple $(\A, \HH, \D)$, in the sense of Section 3, 
defines an associated semifinite Fredholm module and a $KK$-class for $\A$.

This is an important point. 
It is essential in the nonunital version of the theory to have an appropriate definition of the index which we are computing.
Since the operator $\D$ of a general spectral triple need not be Fredholm, this is accomplished by following \cite{KNR}
to produce a $KK$-class. Then the index pairing can be defined via the Kasparov product. 

The role of the additional
smoothness and summability assumptions on the spectral triple is to produce the local index formula for computing the 
index pairing. Our smoothness and summability conditions are defined using the smooth version
of the integration theory in Section 2. This approach is justified by Propositions \ref{one-way} and \ref{lem:necessary}, which compare our definition with a more standard definition of finite summability.

Having identified workable definitions of smoothness and summability, the main technical  
obstacle we have to overcome in Section 3 is to find a suitable Fr\'echet completion 
of $\A$ stable under the holomorphic functional calculus. The integration theory of Section 2 provides
such an algebra, and in the unital case it reduces to previous solutions of this problem, \cite[Lemma 16]{Re}\footnote{Despite being about nonunital spectral triples, \cite[Lemma 16]{Re} produces a Fr\'echet completion which only takes smoothness, not integrability, into account.}.

In Section 4 we establish our local index formula in the sense of Connes-Moscovici. 
The underlying idea here is that Connes' Chern character, which defines an element 
of the cyclic cohomology of $\A$, computes the index pairing defined by a Fredholm 
module. Any cocycle in the same cohomology class as the Chern character will 
therefore also compute the index pairing. In this memoir we define several cocycles 
that represent the Chern character and which are expressed in terms of the 
unbounded operator $\D$. These cocycles generalise those found in 
\cite{CPRS2, CPRS3, CPRS4} (where semifinite versions of the local 
index formula were first proved) to the nonunital case. We have to prove 
that these additional cocycles, including the residue cocycle, are in the 
class of the Chern character in the $(b,B)$-complex.

Our main result  (stated in Theorem \ref{localindex} of Section 4) is then an expression 
for the index pairing using a nonunital version of the semifinite local index 
formula of \cite{CPRS2,CPRS3}, which is in turn a generalisation to the 
setting of semifinite von Neumann algebras of the original Connes-Moscovici 
\cite{CM} formula. Our noncommutative index formula is given by a sum of residues of zeta functions 
and is easily recognisable as a direct generalisation of the unital formulas of 
\cite {CM,CPRS2, CPRS3}. We emphasise that even for the standard $\B(\HH)$ case our local index formula is new. 

One of the main difficulties that we have to overcome is that while there is a well 
understood theory of Fredholm (or Kasparov) modules for nonunital algebras, the 
`right framework' for working with unbounded representatives of these $K$-homology 
classes has proved elusive. We believe that we have found the appropriate 
formalism and the resulting residue index formula  provides 
evidence that the approach to  spectral triples over nonunital algebras initiated in \cite{CGRS1} 
is fundamentally sound and leads to interesting applications. Related ideas on 
the $K$-homology point of view for relative index theorems are to be found in 
\cite{Ro2}, \cite{Bu} and \cite{Ca}, and further references in these texts.

We also discuss some fully noncommutative applications 
in Section 6, including the type I spectral triple of the Moyal plane
constructed in \cite{GGISV} and  semifinite spectral triples arising from 
torus actions on $C^\ast$-algebras, but leave other applications, 
such as those to the results in \cite{Perr}, \cite{PR} and \cite{Wa}, to elsewhere. 

To explain how we arrived at the technical framework described here, 
consider the simplest possible classical case, where $\HH=L^2({\mathbb R})$, 
$\D=\frac{d}{idx}$ and $\A$ is a certain $*$-subalgebra of the algebra of 
smooth functions on $\mathbb R$. Let $P=\chi_{[0,\infty)}(\D)$ be the 
projection defined using the functional calculus and the characteristic 
function of the half-line and let $u$ be a unitary in $\A$ such that $u-1$ 
converges to zero at $\pm\infty$ `sufficiently rapidly'. Then the classical 
Gohberg-Krein theory gives a formula for the index of the Fredholm 
operator $PM_uP$ where $M_u$ is the operator of multiplication by 
$u$ on $L^2({\mathbb R})$. In proving this theorem for general symbols 
$u$, one confronts the classical question (studied in depth in \cite{Simon}) 
of when an operator of the form $(M_u-1)(1+\D^2)^{-s/2}$, $s>0$, is trace 
class. In the general noncommutative setting of this article, this question 
and generalisations must still be confronted and this is done in Section 2.

\subsubsection*{The results for manifolds}
 In the case of closed manifolds, the local index formula in noncommutative geometry (due to Connes-Moscovici \cite{CM}) can serve as a starting point to derive
the Atiyah-Singer index theorem for Dirac type operators. This 
proceeds by a Getzler type argument 
 enunciated in this setting by Ponge, \cite{Ponge}, though similar arguments have been used previously with
 the JLO cocycle as a starting point in \cite{BF,CoM}. 
 While there is already a version of this Connes-Moscovici formula that 
 applies in the noncompact case \cite{Re2}, it
 relies heavily on the use of compact support assumptions.

For the application to noncompact manifolds $M$,  
we find that our noncommutative index theorem dictates that the 
appropriate algebra $\A$ consists of smooth functions which, together 
with all their derivatives, lie in $L^1(M)$. We show how to construct 
$K$-homology classes for this algebra from the Dirac operator on the 
spinor bundle over $M$. This $K$-homology viewpoint is 
related to Roe's approach \cite{Ro2} and to the relative index theory of \cite{GL}.

Then the results, for Dirac operators coupled to 
connections on sections of bundles over noncompact manifolds of 
bounded geometry, essentially follow as corollaries of the work of Ponge \cite{Ponge}. 
The theorems we obtain for even dimensional manifolds are not comparable with those in \cite{Ro}, but are 
closely related to the viewpoint of Gromov-Lawson \cite{GL}. 
For odd dimensional manifolds we obtain an index theorem for 
generalised Toeplitz operators that appears to be new, 
although one can see an analogy with the results of H\"ormander \cite[section 19.3]{Ho}. 

We now digress to give more detail on how, for noncompact even 
dimensional spin manifolds $M$, our local index formula implies a 
result analogous to the Gromov-Lawson relative index theorem \cite{GL}. 
What we compute is an index pairing of $K$-homology classes for the  algebra 
$\A$ of smooth functions which, along with their derivatives, all lie in $L^1(M)$,
with differences of classes $[E]-[E']$ 
in the $K$-theory of $\A$. We verify that the Dirac operator on a spin manifold of 
bounded geometry satisfies the hypotheses needed to use our residue 
cocycle formula so that we obtain a local index formula of the form
\begin{equation} 
\label{AaSs}
\langle [E]-[E'],[\D]\rangle=(const)\int\widehat{A}(M)({\rm Ch}(E)-{\rm Ch}(E')).
\end{equation}
where ${\rm Ch}(E)$ and ${\rm Ch}(E')$ are the Chern classes of vector bundles $E$ and $E'$ over 
$M$. 
We emphasise that in our approach, the connections that lead to the curvature terms in 
${\rm Ch}(E)$ and ${\rm Ch}(E')$, do not have to coincide outside a compact set as in \cite{GL}. 
Instead they satisfy constraints that make the difference of curvature terms integrable over $M$.

We reiterate that, for our notion  of spectral triple, the 
operator $\D$ need not be Fredholm and that the choice of the algebra $\A$ is dictated 
by the noncommutative theory developed in Section 3. In that section we explain  
the minimal assumptions on the pair $(\D, \A)$ such that we can define a Kasparov module and so a
$KK$-class. The further assumptions required 
for the local index formula are specified, almost uniquely, 
by the noncommutative integration theory developed in
Section 2. We 
verify  (in Section 5) what these assumptions mean for the commutative algebra $\A$ of
functions on a manifold and 
Dirac-type operator $\D$, in the case of a  noncompact 
manifold of bounded geometry,  and  prove that in this case we do indeed 
obtain a  spectral triple in the sense of our general definition.

In the odd dimensional case, for manifolds of bounded geometry, we obtain an 
index formula that is apparently new, although it is of APS-type. The residues in the noncommutative 
formula are again calculable by the techniques employed by  \cite{Ponge} in the compact case. 
This results in a formula for the pairing of the Chern character of a unitary $u$ in a 
matrix algebra over $\A$, representing an odd $K$-theory class, with the 
$K$-homology class of  a Dirac-type operator $\D$ of the form
\begin{equation} 
\label{AaPpSs}
\langle [u],[\D]\rangle=(const)\int\hat A(M) {\rm Ch}(u).
\end{equation}
We emphasise that the assumptions on the algebra $\A$ of functions on $M$ are such that this integral 
exists but they do not require compact support conditions.

We were also motivated to consider Atiyah's $L^2$-index Theorem in this setting. 
Because we prove our index formula in the general framework of operators 
affiliated to semifinite von Neumann algebras we are able, with some additional effort, 
to obtain at the same time a version of the $L^2$-index Theorem of Atiyah for Dirac 
type operators on the universal cover of $M$ (whether $M$ is closed or not). 
We are able to reduce our proof in this $L^2$-setting to known results about the 
local asymptotics at small time of heat kernels on covering spaces. 
The key point here
is that our residue cocycle formula gives a uniform approach to all of these `classical' index theorems.

\subsubsection*{Summary of the exposition}
Section 2 begins by introducing the integration theory we employ, which is a refinement of
the ideas introduced in \cite{CGRS1}. Then we examine the interaction of our
integration theory with various notions of smoothness for spectral triples. 
In particular, we follow Higson, \cite{Hi}, and \cite{CPRS2} in extending the 
Connes-Moscovici pseudo-differential calculus to the nonunital setting. 
Finally we prove some trace 
estimates that play a key role in the subsequent technical parts of the discussion.
All these  generalisations are required for the proof of our main result in Section 4. 

Section 3 explains how our definition of   semifinite spectral triple 
results in an index pairing from Kasparov's point of view. In other words, 
while our  spectral triple does not {\it a priori} involve (possibly unbounded) 
Fredholm operators, there is an associated index problem for bounded 
Fredholm operators in the setting of Kasparov's $KK$-theory. We then show that by modifying our original 
spectral triple we may obtain an index problem for unbounded 
Fredholm operators without changing the Kasparov class in the bounded 
picture. This modification of our unbounded spectral triple proves to be 
essential, in two ways,  for us to obtain our residue formula in Section 4. 

The method we use in Section 4 to prove the existence of a formula of 
Connes-Moscovici type for the index pairing of our $K$-homology class 
with the $K$-theory of the nonunital algebra $\A$ is a modification of the 
argument in \cite{CPRS4}. This argument is in turn closely related to 
the approach of Higson \cite{Hi} to the Connes-Moscovici formula. 

The idea is to start with the 
resolvent cocycle of \cite{CPRS2,CPRS3, CPRS4} and show that 
it is well defined in the nonunital setting. We then show that there 
is an extension of the results in \cite{CPRS4} that gives a homotopy 
of the resolvent cocycle to the Chern character for the Fredholm module 
associated to the  spectral triple. The residue cocycle can then be 
derived from the resolvent cocycle in the nonunital case by much the 
same argument as in \cite{CPRS2, CPRS3}. 

In order to avoid cluttering 
our exposition with proofs of nonunital modifications of the estimates 
of these earlier papers, we relegate much detail to the Appendix. Modulo 
these technicalities we are able to show, essentially as in \cite{CPRS4}, 
that the residue cocycle and the resolvent cocycle are index cocycles 
in the class of the Chern character. Then Theorem \ref{localindex} in Section 
4 is the main result of this memoir. It gives a residue formula for the 
numerical index defined in Section 3 for  spectral triples.

We conclude Section 4 with a nonunital McKean-Singer formula and an
example showing that the integrability hypotheses can be weakened still further,
though we do not pursue the issue 
of finding the weakest conditions for our local index formula to hold in this text.

The applications to the index theory for Dirac-type operators on manifolds 
of bounded geometry are contained in Section 5. Also in Section 5 is a 
version of the Atiyah $L^2$-index Theorem that applies to covering spaces 
of noncompact manifolds of bounded geometry. In Section 6 we make a 
start on noncommutative examples, looking at torus actions on $C^*$-algebras
and at the Moyal plane. Any further treatment of noncommutative examples would add 
considerably to the length of this article, and is best left for another place.

{\bf Acknowledgements}. This research was supported by the Australian Research Council, 
the Max Planck Institute for Mathematics (Bonn) and the Banff International 
Research Station. A. Carey also thanks the Alexander von Humboldt Stiftung 
and colleagues in the University of M\"unster and V. Gayral also thanks the 
CNRS and the University of Metz.
We would like to thank our colleagues John Phillips and Magda Georgescu for discussions 
on nonunital spectral flow. Special thanks are given to Dima Zanin and Roger Senior 
for  careful readings  
of this manuscript at various stages.
We 
also thank Emmanuel Pedon for discussions on the Kato inequality, Raimar 
Wulkenhaar for discussions on index computations for the Moyal plane, and Gilles Carron, Thierry Coulhon, 
Batu G\"uneysu and Yuri Korduykov for discussions related to heat-kernels on noncompact manifolds.
Finally, it is a pleasure to thank the referees: their exceptional efforts have greatly improved this work.

\section{Pseudodifferential calculus and summability}\label{sec:psido-and-sum}

In this section we introduce our chief technical innovation on which most of our results rely.
It consists  of an $L^1$-type summability theory for weights adapted to both the 
nonunital and noncommutative settings.

It has become apparent to us while writing, that the integration theory presented here is closely
related to Haagerup's noncommutative $L^p$-spaces for weights, at least for $p=1,\,2$. Despite this, it is sufficiently
different to require a self-contained discussion.

It is an essential and important feature in all that follows that our approach comes essentially
from an $L^2$-theory:  we are forced to employ weights,  and 
a direct  $L^1$-approach is technically unsatisfactory for weights. This is because given
a weight $\varphi$ on a von Neumann algebra, the map $T\mapsto \varphi(|T|)$ is not 
subadditive  in
general.

Throughout this section, $\HH$ denotes a separable Hilbert space, $\cn\subset \B(\HH)$ is
a semifinite von Neumann algebra, $\D:{\rm dom}\,\D\to\HH$ is a self-adjoint operator
affiliated to $\cn$, and $\tau$ is a faithful, normal, semifinite trace on $\cn$. 
Our integration theory will also be parameterised by a
 real number $p\geq 1$, which will play the role of a dimension.

Different parts of the integration and pseudodifferential theory which we
introduce rely on different parts of the above data.
The pseudodifferential calculus can be formulated for any unbounded self-adjoint
operator $\D$ on a Hilbert space $\HH$. This point of view 
is implicit in Higson's abstract differential algebras, \cite{Hi}, and was made more 
explicit in \cite{CPRS2}.

The definition of summability we employ
depends on all the data above, namely $\D$, the pair $(\cn, \tau)$ and the number $p\geq 1$. 
We show in subsection \ref{subsec:sum-weights} how the pseudodifferential 
calculus is compatible with our definition of summability for spectral triples, and this will dictate 
our generalisation of finitely summable spectral triple to the nonunital case in  Section \ref{sec:Chern}.

The proof of the local index formula that we use in the nonunital setting requires 
some estimates on trace norms that are different from those used in the unital case. 
These are found in subsection \ref{subsec:trace-ests}. To prepare for these estimates, we also need some refinements 
of the pseudodifferential calculus introduced by Connes and Moscovici for unital 
spectral triples in \cite{Co5,CM}.

\subsection{Square-summability from  weight domains}

\label{subsec:sum-weights}

In this subsection we  show how an unbounded 
self-adjoint operator affiliated to a semifinite von Neumann algebra provides the foundation 
of an integration theory 
suitable for discussing finite summability for spectral triples.

Throughout this subsection, we let $\D$\index{$\D$, a self-adjoint operator} be a self-adjoint operator affiliated to a semifinite von 
Neumann algebra $\cn$\index{$\cn$, a semifinite von Neumann algebra} with faithful normal semifinite trace \index{$\cn$, a semifinite von Neumann algebra}
$\tau$, and let $p\geq 1$ be a real number. \index{$\tau$, a faithful, semifinite, normal trace on a von Neumann algebra}

\begin{definition}\label{def:d-does-int}
For any positive number $s>0$, we define the weight $\vf_s$ on $\cn$ by 
\index{$\vf_s$, weights defined by $\D$}
$$
T\in\cn_+\mapsto\vf_s(T):=\tau\big((1+\D^2)^{-s/4}T(1+\D^2)^{-s/4}\big)\in[0,+\infty].
$$
As usual, we set
$$
{\rm dom}(\vf_s):={\rm span}\{{\rm dom}(\vf_s)_+\}={\rm span}\big\{\big({\rm dom}(\vf_s)^{1/2}\big)^*{\rm dom}(\vf_s)^{1/2}\}\subset\cn,
$$
where \index{${\rm dom}(\vf_s)$, domain of the weight $\vf_s$}\index{${\rm dom}(\vf_s)^{1/2}$, `$L^2$-domain' of the weight $\vf_s$}
$${\rm dom}(\vf_s)_+
:=\left\{T\in\cn_+:\vf_s(T)<\infty\right\}\quad {\rm and} \quad{\rm dom}(\vf_s)^{1/2}:=\{T\in\cn:T^*T\in {\rm dom}(\vf_s)_+\}.
$$
\end{definition}

In the following, ${\rm dom}(\vf_s)_+$ is called  the positive domain and ${\rm dom}(\vf_s)^{1/2}$
the half domain.

\begin{lemma} 
The weights $\vf_s$, $s>0$, are faithful normal and semifinite,  with modular group given by
$$
\cn\ni T\mapsto (1+\D^2)^{-is/2}T(1+\D^2)^{is/2}.
$$
\end{lemma}

\begin{proof} Normality of $\vf_s$ follows directly from the normality of $\tau$. 
To prove faithfulness of $\vf_s$, using faithfulness of $\tau$, we also need the 
fact that the bounded operator $(1+\D^2)^{-s/4}$ is injective. Indeed, let 
$S\in\mbox{dom}(\vf_s)^{1/2}$ and $T:=S^*S\in\mbox{dom}(\vf_s)_+$ with 
$\vf_s(T)=0$. From the trace property, we obtain $\vf_s(T)=\tau(S(1+\D^2)^{-s/2}S^*)$, 
so by the faithfulness of $\tau$, we obtain $0=S(1+\D^2)^{-s/2}S^*=|(1+\D^2)^{-s/4}S^*|^2$, 
so $(1+\D^2)^{-s/4}S^*=0$, which by injectivity implies $S^*=0$ and thus $T=0$. 
Regarding semifiniteness of $\vf_s$, one uses semifiniteness of $\tau$ to obtain 
that for any $T\in\cn_+$, there exists $S\in\cn_+$  of finite trace, with 
$S\leq (1+\D^2)^{-s/4}T(1+\D^2)^{-s/4}$. Thus $S':=(1+\D^2)^{s/4}S(1+\D^2)^{s/4}\leq T$ 
is non-negative, bounded and belongs to  $\mbox{dom}(\vf_s)_+$, as needed. The 
form of the modular group follows from the definition of the modular group of a weight.
\end{proof}

Domains of weights, and, a fortiori, intersections of domains of weights, are 
$*$-subalgebras of $\cn$. However,  $\mbox{dom}(\vf_s)^{1/2}$ is not a 
$*$-algebra but only a left ideal in $\cn$. To obtain a $*$-algebra structure 
from the latter, we need to force the $*$-invariance.  Since $\vf_s$ is faithful for each $s>0$, 
the inclusion of
$\mbox{dom}(\vf_s)^{1/2}\bigcap (\mbox{dom}(\vf_s)^{1/2})^*$ in its Hilbert space
completion (for the inner product coming from $\vf_s$) is injective. Hence by \cite[Theorem 2.6]{Tak2}, 
$\mbox{dom}(\vf_s)^{1/2}\bigcap (\mbox{dom}(\vf_s)^{1/2})^*$ is a full left
Hilbert algebra. Thus we may define a $*$-subalgebra of $\cn$ for each $p\geq 1$.

\begin{definition} Let  $\D$ be a self-adjoint operator affiliated to a 
semifinite von Neumann algebra $\cn$ with faithful normal semifinite trace $\tau$. 
Then for each $p\geq 1$ we define\index{$\B_2(\D,p)$, algebra of bounded square integrable elements of $\cn$}
$$
\B_2(\D,p):=\bigcap_{s>p} \Big({\rm dom}(\vf_s)^{1/2}\bigcap ({\rm dom}(\vf_s)^{1/2})^*\Big).
$$
The norms 
\begin{equation}\label{pn}
\B_2(\D,p)\ni T\mapsto \Q_n(T)
:=\left(\|T\|^2+\varphi_{p+1/n}(|T|^2)+\varphi_{p+1/n}(|T^*|^2)\right)^{1/2},\quad n\in\N,
\end{equation}
take finite values on $\B_2(\D,p)$ and provide a topology on $\B_2(\D,p)$ \index{$\Q_n$, seminorms on $\B_2(\D,p)$}\index{seminorms}\index{seminorms!$\Q_n$, seminorms on $\B_2(\D,p)$}
stronger than the norm topology. 
Unless mentioned otherwise we will always suppose that $\B_2(\D,p)$ has the topology defined by these norms.
\end{definition}

{\bf Notation.} Given a 
semifinite von Neumann algebra $\cn$ with faithful normal semifinite trace $\tau$,
we let $\tilde{\L}^p(\cn,\tau)$, $1\leq p<\infty$, denote the set of 
$\tau$-measurable operators $T$ affiliated to
$\cn$ with $\tau(|T|^p)<\infty$. We do not often use this notion of $p$-integrable elements, preferring
to use the bounded analogue, $\L^p(\cn,\tau):=\tilde{\L}^p(\cn,\tau)\cap\cn$, normed with 
$T\mapsto\tau(|T|^p)^{1/p}+\|T\|$.
\index{Schatten ideals}\index{$\L^p$, $\tilde{\L}^p$, Schatten ideals}

{\bf Remarks.} (i) If $(1+\D^2)^{-s/2}\in\L^1(\cn,\tau)$ for all $\Re(s)>p\geq 1$, then
$\B_2(\D,p)=\cn$, since then the weights $\vf_s$, $s>p$, are bounded and the norms $\Q_n$ are
all equivalent to the operator norm.

(ii) The triangle inequality for $\Q_n$ follows from the Cauchy-Schwarz 
inequality applied to the  inner product 
$\langle T,S\rangle_n=\vf_{p+1/n}(T^*S)$, and 
$\Q_n(T)^2=\Vert T\Vert^2 +\langle T,T\rangle_n+\langle T^*,T^*\rangle_n$. 
In concrete terms, an element $T\in\cn$ belongs to $\B_2(\D,p)$ if 
and only if for all  $s>p$, both $T(1+\D^2)^{-s/4}$ and $T^* (1+\D^2)^{-s/4}$ belong to 
$\cl^2(\cn,\tau)$, the ideal of $\tau$-Hilbert-Schmidt operators.

(iii) The norms $\Q_n$ are increasing, in the sense that for 
$n\leq m$ we have $\Q_n\leq \Q_m$. We
leave this as an exercise, but observe that
this requires the cyclicity of the trace. The following result of Brown and 
Kosaki gives the strongest statement on this
cyclicity. By the preceding Remark (ii), we do not need the full power of this result here, but record it
for future use.

\begin{prop} \cite[Theorem 17]{BrK}
\label{prop:everybody-knows}
Let $\tau$ be a faithful normal semifinite trace on a von Neumann algebra $\cn$, and
let $A$, $B$ be $\tau$-measurable operators affiliated to $\cn$. If 
$AB,\,BA\in\tilde{\cl}^1(\cn,\tau)$ then
$\tau(AB)=\tau(BA)$.
\end{prop}

Another important result that we will frequently use comes from Bikchentaev's work.

\begin{prop}\cite[Theorem 3]{Bik}
\label{prop:bikky}
Let $\cn$ be a semifinite von Neumann algebra with faithful normal semfinite trace. If
$A,\,B\in\cn$ satisfy $A\geq0$, $B\geq 0$, and are such that $AB$ is trace class, then
$B^{1/2}AB^{1/2}$ and $A^{1/2}BA^{1/2}$ are also trace class, with
$\tau(AB)=\tau(B^{1/2}AB^{1/2})=\tau(A^{1/2}BA^{1/2})$. 
\end{prop}

Next we show that the topological algebra  $\B_2(\D,p)$ is complete  and thus is a Fr\'echet algebra. 
The completeness argument relies on the Fatou property for the trace $\tau$, \cite{FK}.

\begin{prop}\label{Frechet1} 
The $*$-algebra $\B_2(\D,p)\subset \cn$ is a Fr\'echet algebra.
\end{prop}

\begin{proof} 
Showing that $\B_2(\D,p)$ is a $*$-algebra is routine with the aid of the following argument.
For $T,\,S\in \B_2(\D,p)$, the operator inequality $S^*T^*TS\leq \Vert T^*T\Vert\,S^*S$ shows that
$$
\varphi_{p+1/n}(|TS|^2)=\varphi_{p+1/n}(S^*T^*TS)\leq\|T\|^2\vf_{p+1/n}(|S|^2).
$$
and, therefore, $\Q_n(TS)\leq \Q_n(T)\,\Q_n(S)$. 

For the completeness, let $(T_k)_{k\geq 1}$ be a Cauchy sequence in 
$\B_2(\D,p).$ Then $(T_k)_{k\geq 1}$ converges in norm, 
and so there exists $T\in\cn$ such that $T_k\to T$ in $\mathcal{N}$. 
For each norm $\Q_n$ we have $|\,\Q_n(T_k)-\Q_n(T_l)\,|\leq \Q_n(T_k-T_l)$, 
so we see that the numerical sequence $(\Q_n(T_k))_{k\geq 1}$ possesses  a limit. 
Now since
$$
(1+\D^2)^{-p/4-1/4n}T_k^*T_k(1+\D^2)^{-p/4-1/4n}\to (1+\D^2)^{-p/4-1/4n}T^*T(1+\D^2)^{-p/4-1/4n},
$$ 
in norm, it also converges in measure, and so  we may apply the 
Fatou Lemma, \cite[Theorem 3.5 (i)]{FK}, to deduce that
$$
\tau\big((1+\D^2)^{-p/4-1/4n}T^*T(1+\D^2)^{-p/4-1/4n}\big)\leq 
\liminf_{k\to\infty}\tau\big((1+\D^2)^{-p/4-1/4n}T_k^*T_k(1+\D^2)^{-p/4-1/4n}\big).
$$
Since the same conclusion holds for $TT^*$ in place of $T^*T$, we see that 
$$
\Q_n(T)\leq\liminf_{k\to\infty}\Q_n(T_k)=\lim_{k\to\infty}\Q_n(T_k)<\infty,
$$
and so $T\in\B_2(\D,p)$. 
Finally, fix $\eps>0$ and $n\geq 1$.  Now choose $N$ large enough so that 
$\Q_n(T_k-T_l)\leq \eps$ for all $k,\,l >N$.  
Applying the Fatou Lemma \index{Fatou Lemma}to the sequence $(T_k)_{k\geq 1}$, we have 
$\Q_n(T-T_l)\leq \liminf _{k\to \infty} \Q_n(T_k-T_l)\leq \eps$. 
Hence $T_k\to T$ in the topology of $\B_2(\D,p)$.
 \end{proof}

We now give  some easy but useful stability properties of the algebras $\B_2(\D,p)$.

\begin{lemma}\label{Bp-grew} 
Let $T\in\mathcal{B}_2(\D,p)$, 
$S\in\mathcal{N}$ and let $f\in L^{\infty}(\mathbb{R}).$
\begin{enumerate}
\item The operators $Tf(\mathcal{D}),\ f(\mathcal{D})T$ are in $\B_2(\D,p)$. 
If $T^*=T,$ then $Tf(T)\in\B_2(\D,p)$. In all these cases, $\Q_n(Tf(\D)),\,\Q_n(f(\D)T),\,\Q_n(Tf(T))\leq\Vert f\Vert_\infty \Q_n(T)$.
\item If $S^*S\leq T^*T$ and $SS^*\leq TT^*$, then $S\in\B_2(\D,p)$ with $\Q_n(S)\leq\Q_n(T)$. 
\item We have $S\in\B_2(\D,p)$ if and only if $|S|,|S^*|\in\B_2(\D,p)$.
\item The real and imaginary parts $\Re(T),\,\Im(T)$ belong to $\B_2(\D,p)$. 
\item If $T=T^*$,
let $T=T_+-T_-$ be  the Jordan decomposition of $T$ into positive and negative parts. \index{Jordan decomposition}
Then $T_+,\,T_-\in \B_2(\D,p)$. Consequently $\B_2(\D,p)={\rm span}\{\B_2(\D,p)_+\}$.
\end{enumerate}
\end{lemma}

\begin{proof} (1) Since $T(1+\D^2)^{-s/4}$, $T^*(1+\D^2)^{-s/4}\in\cl^2(\cn,\tau)$, 
we immediately see that 
$$
Tf(\D)(1+\D^2)^{-s/4}=T(1+\D^2)^{-s/4}f(\D),\ \bar f(\D)T^*(1+\D^2)^{-s/4}\ \in\cl^2(\cn,\tau),
$$ 
and when 
$T$ is self-adjoint, we also have
$$
Tf(T)(1+\D^2)^{-s/4}=f(T)T(1+\D^2)^{-s/4},\ \bar f(T)T(1+\D^2)^{-s/4}\ \in\cl^2(\cn,\tau).
$$ 
To prove the inequality we use the trace property to see that 
\begin{align*}
\tau((1+\D^2)^{-s/4}\bar{f}(\D)T^*Tf(\D)(1+\D^2)^{-s/4})&
=\tau(T(1+\D^2)^{-s/4}|f|^2(\D)(1+\D^2)^{-s/4}T^*)\\
&\leq \Vert f\Vert_\infty^2\tau((1+\D^2)^{-s/4}T^*T(1+\D^2)^{-s/4}),
\end{align*}
and similarly for $Tf(\D)$ and $Tf(T)$ when $T^*=T$.\\
(2) Clearly, $\varphi_s(S^*S)\leq\varphi_s(T^*T)$ and $\varphi_s(SS^*)\leq\varphi_s(TT^*)$. 
The assertion follows immediately.\\
(3) This follows from $\Q_n(T)=(\Q_n(|T|)+\Q_n(|T^*|))/2.$ Item (4) follows since $\B_2(\D,p)$ is
a $*$-algebra, 
and then item (5) follows from (2), since for a self-adjoint element $T\in \B_2(\D,p)$:
$$
T^*T=|T|^2=(T_++T_-)^2=T_+^2+T_-^2\geq\ T_+^2,\,T_-^2.
$$
This completes the proof.
\end{proof}

The algebras $\B_2(\D,p)$ are stable under the holomorphic functional calculus. We remind the reader
that when $\B$ is a  nonunital algebra, this means that for all $T\in \B$ and functions
$f$ holomorphic in a neighbourhood of the spectrum of $T$ {\em with $f(0)=0$} we have $f(T)\in \B$.

\begin{lemma}
\label{Hol-Calc} 
 For any $n\in\mathbb{N}$ the $*$-algebra
 $M_n(\B_2(\D,p))$ is stable under the holomorphic functional calculus
 in its $C^*$-completion.
\end{lemma}

\begin{proof} 
We begin with the $n=1$ case.
If $T\in\B_2(\D,p)$ is such that $1+T$ is invertible in $\mathcal{N}$, 
then by (a minor extension of)
Lemma \ref{Bp-grew} (1), we see that
\begin{align}
(1+T)^{-1}-1=-T(1+T)^{-1}\in\B_2(\D,p).
\label{eq:fracs}
\end{align}
Equation \eqref{eq:fracs} and Lemma \ref{Bp-grew} part (1) gives, for $z$  in the
resolvent set  of $T$,
$$
\Q_n\big((z-T)^{-1}-z^{-1}\big)=\Q_n\big(z^{-1}T(z-T)^{-1}\big)\leq 
\Vert (1+T)(z-T)^{-1}\Vert\,\Q_n\big({z}^{-1}T(1+T)^{-1}\big).
$$
Set $C_z=\Vert (1+T)(z-T)^{-1}\Vert$ and
 let $\Gamma$ be a positively oriented 
contour surrounding the spectrum of $T$ with $0\not\in\Gamma$, and
$f$  holomorphic in a neighborhood of 
the spectrum of $T$ containing $\Gamma$. Then 
$$
\Q_n\left(\frac{1}{2\pi i}\int_\Gamma f(z)\left[(z-T)^{-1}-z^{-1}\right]dz\right)\leq 
\frac{C}{2\pi}\,\Q_n(T(1+T)^{-1})\int_\Gamma \left|\frac{f(z)dz}{z}\right|<\infty,
$$
where $C=\sup_{z\in\Gamma}C_z$. Thus
we have (when $\B_2(\D,p)\subset \cn$ is nonunital)
$$
\int_\Gamma{f(z)}{(z-T)^{-1}}\,dz\in\,\B_2(\D,p)\oplus\mathbb{C}\,{\rm Id}_\cn,
$$
with the scalar component equal to $f(0){\rm Id}_\cn$. 

The general case follows from 
the $n=1$ case by 
main theorem of \cite{LBS}.
\end{proof}

\subsection{Summability from  weight domains}
\label{symdom}

As in the last subsection, we let $\D$ be a self-adjoint operator affiliated to a semifinite von 
Neumann algebra $\cn$ with faithful normal semifinite trace 
$\tau$ and $p\geq 1$. 

In the previous subsection, we have seen that 
the algebra $\B_2(\D,p)$ plays the role of a $*$-invariant 
$L^2$-space in the setting of weights.
To construct a $*$-invariant  $L^1$-type space associated 
with the data $(\cn,\tau,\D,p)$, there are two obvious
strategies. 

One strategy is to define seminorms on $\B_2(\D,p)^2$ 
(the finite span of products) and  to then complete this space. The other approach
is to take the projective tensor product completion of 
$\B_2(\D,p)\otimes \B_2(\D,p)$ and then consider its image
in $\cn$ under the multiplication map. In fact both approaches yield the 
same answer, and complementary benefits.
\newcommand{\q}{a}
\newcommand{\qq}{A}

We begin by recalling the projective tensor product\index{projective tensor product} topology in our setting. It is defined to be the
strongest locally convex topology on the algebraic tensor product such that the natural bilinear map 
$$
\B_2(\D,p)\times\B_2(\D,p)\mapsto \B_2(\D,p)\otimes \B_2(\D,p),
$$
is continuous, \cite[Definition 43.2]{Trev}. The projective tensor 
product topology can be described in terms of seminorms
$\tilde{\PP}_{n,m}$ defined by
\begin{equation}
\tilde{\PP}_{n,m}(T)
:=\inf\Big\{\sum_{{\rm finite}}\Q_n(T_{i,1})\,\Q_m(T_{i,2}):\ T=\sum_{{\rm finite}}T_{i,1}\otimes T_{i,2}\Big\},\quad n,m\in\N.
\label{eq:tilde-pp}
\end{equation}
(In fact, since the $\Q_n$ are norms, so too are the $\tilde{\PP}_{n,m}$).
Using the fact that the norms $\Q_n$ are increasing 
and from the arguments 
of Corollary \ref{onecor}, we see  that for $k\leq n$ and $l\leq m$ we
have $\tilde{\PP}_{k,l}\leq \tilde{\PP}_{n,m}$. This allows us to show that the projective tensor product
topology is in fact determined by the subfamily of seminorms $\tilde{\PP}_{n}:=\tilde{\PP}_{n,n}$, and 
accordingly we restrict to this
family for the rest of this discussion.

Then we let $\B_2(\D,p)\otimes_\pi\B_2(\D,p)$ denote the completion of 
$\B_2(\D,p)\otimes \B_2(\D,p)$ in the projective
tensor product topology. The projective tensor product 
topology is the unique topology on $\B_2(\D,p)\otimes\B_2(\D,p)$
such that, \cite[Proposition 43.4]{Trev}, 
for any locally convex topological vector space $G$, the canonical isomorphism 
$$
\big\{{\rm bilinear\ maps}\ \B_2(\D,p)\times\B_2(\D,p)\to G\big\}
\longrightarrow\big\{{\rm linear\ maps}\ \B_2(\D,p)\otimes\B_2(\D,p)\to G\big\},
$$
gives an (algebraic) isomorphism
\begin{align*}
&\big\{{\rm continuous\ bilinear\ maps}\ \B_2(\D,p)\times\B_2(\D,p)\to G\big\}\longrightarrow\\
&\hspace{6,3cm} \big\{{\rm continuous\ linear\ maps}\ \B_2(\D,p)\otimes\B_2(\D,p)\to G\big\}.
\end{align*}
Since the multiplication map is a continuous bilinear map 
$m:\B_2(\D,p)\times\B_2(\D,p)\to \B_2(\D,p)$, we obtain
a continuous (with respect to the projective tensor product topology) 
linear map $\tilde{m}:\B_2(\D,p)\otimes\B_2(\D,p)\to \B_2(\D,p)$. 
We extend $\tilde{m}$ to the completion
$\B_2(\D,p)\otimes_\pi\B_2(\D,p)$ and denote by $\tilde{\B}_1(\D,p)\subset \B_2(\D,p)$ 
the image of $\tilde{m}$. Since $\tilde{m}$ is continuous, 
$\tilde{m}$ has closed kernel, and there is an isomorphism of
topological vector spaces between
$\tilde{\B}_1(\D,p)$ with the quotient topology (defined below)
and $\B_2(\D,p)\otimes_\pi\B_2(\D,p)/\ker\tilde{m}$.

Now by \cite[Theorem 45.1]{Trev}, any $\Theta\in \B_2(\D,p)\otimes_\pi\B_2(\D,p)$ 
admits a representation as an absolutely convergent sum (i.e. convergent for all $\tilde{\PP}_n$)
$$
\Theta=\sum_{i=0}^\infty \lambda_i R_i \otimes S_i,\qquad  R_i,\,S_i\in \B_2(\D,p),
$$
such that
\begin{align}
  \sum_{i=0}^\infty\lambda_i<\infty\ \  \mbox{and}\ \ 
\Q_n(R_i),\, \Q_n(S_i)\to 0,\;i\to\infty\ \ \mbox{ for all }n\in\N.
\label{eq:nice-rep}
\end{align}
By defining $\tilde{R}_i=\lambda_i^{1/2}R_i$ and $\tilde{S}_i=\lambda_i^{1/2}S_i$, we see that we can
represent $\Theta$ as an absolutely convergent sum in each of the norms $\tilde{\PP}_n$
\begin{equation}
\Theta=\sum_{i=0}^\infty \tilde{R}_i\otimes\tilde{S}_i, \quad {\rm such\ that\  for\ all\ }n\geq 1\ \ 
\big(\Q_n(\tilde{R}_i)\big)_{i\geq 0},\,\big(\Q_n(\tilde{S}_i)\big)_{i\geq 0}\in \ell^2(\N_0).
\label{eq:l2-rep-sum}
\end{equation}

Having considered the basic features of the projective tensor product approach, we
now consider the approach based on products of elements of $\B_2(\D,p)$.
So we let $\B_2(\D,p)^2$ be the finite linear span of  products from $\B_2(\D,p)$,
and define a family of  norms, $\{\PP_{n,m}:\,n,m\in\N\}$, on $\B_2(\D,p)^2$, by setting
\begin{equation}
\label{new-norm}
\PP_{n,m}(T)
:=\inf\Big\{\sum_{i=1}^k\Q_n( T_{1,i})\,\Q_m(T_{2,i})\ :\  
T=\sum_{i=1}^kT_{1,i}T_{2,i},\ T_{1,i},\,T_{2,i}\in \B_2(\D,p)\Big\}.
\end{equation}
Here the sums are finite and the infimum runs over all possible such representations of $T$. 
Just as we did for the norms $\tilde{\PP}$ after Equation \eqref{eq:tilde-pp}, 
we may use the fact that the $\Q_n$ are increasing to show that the topology
determined by the norms $\PP_{n,m}$ is the same as that determined by the smaller set of norms
$\PP_n:=\PP_{n,n}$. Thus we may restrict attention to the norms $\PP_n$.\index{$\PP_n$, seminorms on $\B_1(\D,p)$}\index{seminorms!$\PP_n$, seminorms on $\B_1(\D,p)$}

Now $\B_2(\D,p)^2\subset \tilde{\B}_1(\D,p)$ and, regarding $\tilde{\B}_1(\D,p)$
as a quotient as above, we claim that the norms $\PP_n$ are the natural seminorms (restricted to
$\B_2(\D,p)^2$) defining the Fr\'echet topology
on the quotient, \cite[Proposition 7.9]{Trev}. 

To see this, recall that the quotient seminorms $\tilde{\PP}_{n,q}$ on $\tilde{\B}_1(\D,P)$ are defined, for 
$T\in\tilde{\B}_1(\D,p)\cong \B_2(\D,p)\otimes_\pi\B_2(\D,p)/\ker\tilde{m}$, by
$$
\tilde{\PP}_{n,q}(T):=\inf_{T=\tilde{m}(\Theta)}\tilde{\PP}_n(\Theta).
$$
Then for $T\in \B(\D,p)^2$ we have the elementary equalities
\begin{align*}
\PP_n(T)&=\inf\Big\{\sum_{{\rm finite}}\Q_n(T_{i,1})\Q_n(T_{i,2}):\ T=\sum_{{\rm finite}} T_{i,1}T_{i,2}\Big\}\\
&= \inf\Big\{\sum_{{\rm finite}}\Q_n(T_{i,1})\Q_n(T_{i,2}):\ \Theta=\sum_{{\rm finite}} T_{i,1}\otimes T_{i,2}\ \&\ \ \tilde{m}(\Theta)=T\Big\}
=\inf_{\tilde{m}(\Theta)=T}\tilde{\PP}_{n}(\Theta).
\end{align*}

Thus the $\PP_n$ are norms on $\B_2(\D,p)^2$.

\begin{definition}
\label{def:sumpin-like-this}
Let  $\B_1(\D,p)$ be the completion of $\B_2(\D,p)^2$ \index{$\B_1(\D,p)$, algebra of integrable elements of $\cn$}
with respect to the topology determined by the 
family of  
norms $\{\PP_n:\,n\in\N\}$. 
\end{definition}

\begin{theorem} 
We have an equality of Fr\'echet spaces $\B_1(\D,p)=\tilde{\B}_1(\D,p)$.
\end{theorem}

\begin{proof}
For $T\in \tilde{\B}_1(\D,p)$, there exists 
$\Theta=\sum_{i=0}^\infty R_i\otimes S_i\in \B_2(\D,p)\otimes_\pi\B_2(\D,p)$
with $\tilde{m}(\Theta)=T$ and such that the sequences 
$(\Q_n(R_i))_{i\geq 0},\,(\Q_n(S_i))_{i\geq 0}$ are in $\ell^2(\N_0)$ for
each $n$. Now
$$
\Theta=\lim_{N\to\infty}\sum_{i=0}^NR_i\otimes S_i\quad{\rm and}\quad 
\tilde{m}\Big(\sum_{i=0}^NR_i\otimes S_i\Big)=\sum_{i=0}^NR_iS_i,
$$
so by the continuity of $\tilde{m}$
$$
T=\tilde{m}(\Theta)=\lim_{N\to\infty}\sum_{i=0}^NR_iS_i.
$$
Here the limit defining $T$ is with respect to the family of norms 
$\tilde{\PP}_{n,q}=\PP_n$ on $\B_2(\D,p)^2$. 
Hence, by definition,  
$T\in \B_1(\D,p)$, and so $\tilde{\B}_1(\D,p)\subset \B_1(\D,p)$.

Now observe that we have the containments
$$
\B_2(\D,p)^2\subset \tilde{\B}_1(\D,p)\subset \B_1(\D,p),
$$
and as $\B_2(\D,p)^2$ is dense in $\B_1(\D,p)$ by definition, 
$\B_2(\D,p)^2$ is dense in $\tilde{\B}_1(\D,p)$.
As $\tilde{\PP}_{n,q}=\PP_n$ on $\B_2(\D,p)^2$, we see that 
$\tilde{\B}_1(\D,p)$ is a dense and closed subset of 
$\B_1(\D,p)$. Hence $\tilde{\B}_1(\D,p)=\B_1(\D,p)$.
\end{proof}
Therefore,  we will employ the single notation $\B_1(\D,p)$ from now on.

{\bf Remark.} For $R,\,S\in\B_2(\D,p)$ we have 
$RS\in\B_1(\D,p)$ with $\PP_n(RS)\leq\Q_n(R)\Q_n(S)$.
By applying $\tilde{m}$ to a representation of $\Theta\in \B_2(\D,p)\otimes_\pi\B_2(\D,p)$ as in
Equation \eqref{eq:l2-rep-sum}, this allows us to see that every $T\in \B_1(\D,p)$ can be
represented as a sum, convergent for every $\PP_n$,
$$
T=\sum_{i=0}^\infty R_iS_i,\quad {\rm such\ that\  for\ all\ }n\geq 1\ \ 
\left(\Q_n({R}_i)\right)_{i\geq 0},\,\left(\Q_n({R}_i)\right)_{i\geq 0}\in \ell^2(\N_0).
$$

We now show that $\B_1(\D,p)$ is  a $*$-algebra,  and that the norms $\PP_n$
are  submultiplicative. The first step is to show that $\B_1(\D,p)$ is naturally included in $\B_2(\D,p)$.

\begin{lemma}
\label{B1-implies-B2}
The algebra $\B_1(\D,p)$ is continuously embedded in $\B_2(\D,p)$. In particular,
for all $T\in \B_1(\D,p)$ and all $n\in\N$, $\Q_n(T)\leq \PP_n(T)$.
\end{lemma}

\begin{proof}
Let $T\in \B_1(\D,p)$. That $T$ belongs to $\B_2(\D,p)$ follows
from the submultiplicativity of the norms $\Q_n$. To see this, fix $n\in\N$. Then,  
for any  representation  $T=\sum_{i=0}^\infty R_i S_i$,
the submultiplicativity of the norms $\Q_n$ gives us
$$
\Q_n(T)=\Q_n\Big(\sum_{i=0}^\infty R_i S_i\Big)\leq \sum_{i=0}^\infty \Q_n(R_i S_i)\leq
\sum_{i=0}^\infty \Q_n(R_i)\, \Q_n(S_i).
$$
Since this is true for any representation $T=\sum_{i=0}^\infty R_i S_i$, this implies that $\Q_n(T)\leq\PP_n(T)$,
proving that $\B_1(\D,p)$ embeds continuously  in $\B_2(\D,p)$. 
\end{proof}

\begin{corollary}
\label{onecor}
The Fr\'echet space $\B_1(\D,p)$ is a $*$-subalgebra of $\cn$. Moreover, the norms
$\PP_n$ are $*$-invariant, submultiplicative, and for $n\leq m$ satisfy $\PP_n\leq \PP_m$.\index{$\PP_n$, seminorms on $\B_1(\D,p)$}\index{seminorms!$\PP_n$, seminorms on $\B_1(\D,p)$}
\end{corollary}

\begin{proof}
We begin by showing that each $\PP_n$ is 
a $*$-invariant norm. Using the $*$-invariance of $\Q_n(\cdot)$, we have
for any $T\in \B_2(\D,p)^2$
\begin{align*}
\PP_n( T^*)&=\inf\Big\{\sum_i\Q_n( S_{1,i})\,\Q_n( S_{2,i})\ :\ T^*=\sum_iS_{1,i}S_{2,i}\Big\}\\
&\leq \inf\Big\{\sum_i\Q_n( T^*_{2,i})\,\Q_n( T^*_{1,i})\ :\ T=\sum_iT_{1,i}T_{2,i}\Big\}\\
&= \inf\Big\{\sum_i\Q_n( T_{2,i})\,\Q_n( T_{1,i})\ :\ T=\sum_iT_{1,i}T_{2,i}\Big\}
=\PP_n( T).
\end{align*}
Hence $\PP_n( T^*)\leq \PP_n( T)$, and by replacing $T^*$ with $T$ we find that
$\PP_n( T^*)= \PP_n( T)$. It now follows that each $\PP_n$ is $*$-invariant on all of $\B_1(\D,p)$.

That   $\B_1(\D,p)$ is an algebra, follows from the embedding $\B_1(\D,p)\subset \B_2(\D,p)$ 
proven in Lemma \ref{B1-implies-B2}:
$$
\B_1(\D,p)\cdot \B_1(\D,p)\subset \B_2(\D,p)\cdot \B_2(\D,p)\subset \B_1(\D,p).
$$
For the submultiplicativity of the norms $\PP_n$, we observe for  $T,\,S\in\B_1(\D,p)$
$$
\PP_n(TS)\leq\Q_n(T)\Q_n(S)\leq\PP_n(T)\PP_n(S),
$$
where the first inequality follows from the definition of $\PP_n$ and the second from the norm
estimate of Lemma \ref{B1-implies-B2}. 

To prove that $\PP_n(\cdot)\leq\PP_m(\cdot)$ for $n\leq m$, take $T\in \B_2(\D,p)^2$ and consider any representation
$T=\sum_{i=1}^k T_{i,1}\,T_{i,2}$. Then, since $\Q_n(\cdot)\leq\Q_m(\cdot)$ for $n\leq m$, we have
$$
\sum_{i=1}^k \Q_n(T_{i,1})\,\Q_n(T_{i,2})\leq \sum_{i=1}^k \Q_m(T_{i,1})\,\Q_m(T_{i,2}),
$$
and thus
\begin{equation}
\PP_n(T)\leq \sum_{i=1}^k \Q_m(T_{i,1})\,\Q_m(T_{i,2}).
\label{eq:pp-est}
\end{equation}
Since the inequality \eqref{eq:pp-est} is true for any such representation, we  have  $\PP_n(T)\leq\PP_m(T)$. 
Now let $T\in \B_1(\D,p)$ be the limit of the sequence $(T_N)_{N\geq 1}\subset \B_2(\D,p)^2$.
Then $\PP_n(T)=\lim_{N\to\infty}\PP_n(T_N)\leq \lim_{N\to\infty}\PP_m(T_N)=\PP_m(T)$.
\end{proof}

Next we show the compatibility of the  norms $\PP_n$ with  positivity.

\begin{lemma}
\label{lem:b1-pos}
Let $0\leq A\in\cn$. Then $A\in\B_1(\D,p)$ if and only if $A^{1/2}\in\B_2(\D,p)$ with 
$$
\PP_n(A)=\Q_n(A^{1/2})^2,\quad\forall n\in\N.
$$
Moreover if  $0\leq A\leq B\in\cn$ and $B\in\B_1(\D,p)$, then $A\in\B_1(\D,p)$, 
with $\PP_n(A)\leq \PP_n(B)$ for all $n\in\N$.
\label{posi}
\end{lemma}

\begin{proof}
Given $0\leq A\in \cn$ with $A^{1/2}\in\B_2(\D,p)$,  it follows 
from the definitions that $A\in \B_1(\D,p)$ and
$\PP_n(A)\leq\Q_n(A^{1/2})^2$.
So suppose $0\leq A\in\B_1(\D,p)$ and choose any  representation
$$
A=\sum_{i=0}^\infty R_iS_i,\qquad 
\sum_{i=0}^\infty \Q_n(R_i)\Q_n(S_i)<\infty,\ \mbox{for all }n\in\N.
$$
Then using the self-adjointness of $A$, the definitions, and the Cauchy-Schwarz inequality 
yields

\begin{align*}
&\Q_n(A^{1/2})^2
=\Q_n\Big(\big(\sum_{i=0}^\infty R_iS_i\big)^{1/2}\Big)^2=
\big\|\sum_{i=0}^\infty R_iS_i\big\|+\vf_{p+1/n}\big(\sum_{i=0}^\infty R_iS_i\big)
+\vf_{p+1/n}\big(\sum_{i=0}^\infty S_iR_i\big)
\\
&\leq\sum_{i=0}^\infty \big\| R_i\big\|\,\big\|S_i\big\|+\big|\vf_{p+1/n}\big(R_iS_i\big)\big|
+\big|\vf_{p+1/n}\big(S_iR_i\big)\big|\\
&\leq \sum_{i=0}^\infty \| R_i\|\,\|S_i\|+\vf_{p+1/n}\big(R_iR_i^*\big)^{1/2}\vf_{p+1/n}\big(S_i^*S_i\big)^{1/2}
+\vf_{p+1/n}\big(S_iS_i^*\big)^{1/2}\vf_{p+1/n}\big(R_i^*R_i\big)^{1/2}\\
&\leq  \sum_{i=0}^\infty \Q_n(R_i)\Q_n(S_i).
 \end{align*}
The last inequality follows from applying the Cauchy-Schwarz inequality,
$$
(r_1s_1+r_2s_2+r_3s_3)^2\leq(r_1^2+r_2^2+r_3^2)(s_1^2+s_2^2+s_3^2),
$$ 
to each term in
the sum.

 Thus for any representation of $A$ we have
 $\Q_n(A^{1/2})^2\leq \sum_{i=0}^\infty \Q_n(R_i)\Q_n(S_i)$,
 which entails  $\Q_n(A^{1/2})^2\leq \PP_n(A)$ as needed.
 For the last statement, let $0\leq B\in \B_1(\D,p)$ and suppose that $0\leq A\in \cn$ satisfies $B\geq A$. 
 Then  $B^{1/2}\geq A^{1/2}$ and $B^{1/2}\in \B_2(\D,p)$, so  Lemma \ref{Bp-grew} (2) completes the proof. 
 \end{proof}
 
Since $\B_1(\D,p)$ is a $*$-algebra, we have $T\in \B_1(\D,p)$ if and only if $T^*\in \B_1(\D,p)$. 
Thus given $T=T^*\in\B_1(\D,p)$, it is natural to ask whether the 
positive and negative parts $T_+,\,T_-$ of the Jordan decomposition of $T$
are in $\B_1(\D,p)$. We can not answer this question, but can nevertheless prove that $\B_1(\D,p)$ is the 
(finite) span of its
positive cone.

\begin{prop} 
\label{cor:polar-decomp}
For $T\in\B_1(\D,p)$, 
there exist four positive operators $T_0,\dots,T_3\in \B_1(\D,p)$ such that
$$
T=\big(T_0-T_2\big)+i\big(T_1-T_3\big).
$$
Here $\Re(T)=T_0-T_2$ and $\Im(T)=T_1-T_3$, but this need not be the
Jordan decomposition since it may not be that $T_0T_2=T_1T_3=0$.\index{Jordan decomposition}
Nevertheless, the space $\B_1(\D,p)$ is the linear span of its positive cone.
\end{prop}

\begin{proof}
Let $T\in \B_1(\D,p)$ have the  representation $T=\sum_j R_jS_j$. By
Equation \eqref{eq:l2-rep-sum}, this means that for each $n$ the sequences
$(\Q_n(R_j))_{j=0}^\infty$ and $(\Q_n(S_j))_{j=0}^\infty$ belong to $\ell^2(\N_0)$.
Now, from the polarization identity
$$
4R^*S=\sum_{k=0}^3 i^k(S+i^kR)^*(S+i^kR),
$$
we can decompose $T=\sum_{k=0}^3 i^k T_k$, with
$$
T_k=\frac14\sum_{j=0}^\infty(S_j+i^kR_j^*)^*(S_j+i^kR_j^*)\geq 0.
$$
Since both  $(\Q_n(R_j))_{j=0}^\infty$ and $(\Q_n(S_j))_{j=0}^\infty$ 
belong to $\ell^2(\N_0)$ and using the $*$-invariance of the norms $\Q_n$, we see that the four elements
$T_k$, $k=0,1,2,3$, all belong to $\B_1(\D,p)$.
Now it is straightforward to check that $\Re(T)=T_0-T_2$ and $\Im(T)=T_1-T_3$, however, these need 
not give 
the canonical decomposition into positive and 
negative parts since we may not have $T_0T_2=0$ and
$T_1T_3=0$.
\end{proof}

{\bf Remark.} The previous proposition shows that we can represent elements of $\B_1(\D,p)$
as finite sums of products of elements of $\B_2(\D,p)$, and so have a correspondingly
simpler description of the norms. We will not pursue this further here.

The next lemma is analogous to Lemma \ref{Bp-grew} (1). It shows that $\B_1(\D,p)$ is a bimodule
for the natural actions of the commutative von Neumann algebra generated by the spectral family of 
the operator $\D$.

\begin{lemma}
\label{B1-grew}
Let $T\in \B_1(\D,p)$ and $f\in L^\infty(\R)$. Then  $Tf(\D)$ and $f(\D)T$ belong to $\B_1(\D,p)$ with
$\PP_n\big(Tf(\D)\big),\PP_n\big(f(\D)T\big)\leq \|f\|_\infty \PP_n(T)$ for all  $n\in\N$.
\end{lemma}

\begin{proof}
Fix $T\in \B_1(\D,p)$,  $f\in L^\infty(\R)$ and $n\in\N$. Consider an arbitrary 
representation $T=\sum_{i=0}^\infty R_i\,S_i$. Then we claim that $\sum_{i=0}^\infty R_i\,\big(S_if(\D)\big)$
is a representation of $Tf(D)$. Indeed, it follows by Lemma \ref{Bp-grew} (1) that
$$
\sum_{i=0}^\infty \Q_n(R_i)\,\Q_n\big(S_if(\D)\big)\leq\|f\|_\infty \sum_{i=0}^\infty \Q_n(R_i)\,\Q_n(S_i)<\infty,
$$
showing that $Tf(\D)\in \B_1(\D,p)$. Moreover, the preceding inequality entails that
\begin{align*}
\PP_n\big(Tf(\D)\big)&\leq \inf\Big\{ \sum_{i=0}^\infty \Q_n(R_i)\,\Q_n\big(S_if(\D)\big)\ :\ T=\sum_{i=0}^\infty R_i\,S_i\Big\}\\
&\leq \|f\|_\infty \inf\Big\{ \sum_{i=0}^\infty \Q_n(R_i)\,\Q_n(S_i)\ :\ T=\sum_{i=0}^\infty R_i\,S_i\Big\}=\|f\|_\infty\,\PP_n(T).
\end{align*}
The case of $f(\D)T$ is  similar.
\end{proof}

Our next aim is to prove 
that $\B_1(\D,p)$ is stable under the holomorphic functional calculus in its $C^*$-completion. 
This will be a corollary of the following two
lemmas. 

\begin{lemma} 
\label{R-side}
Let $T,\,R$ be elements of $\B_2(\D,p)$ with $1+R$ invertible in $\cn$. Then
$T(1+R)^{-1}\in \B_2(\D,p)$, and for all $n\in\N$ we have
$$
\Q_n\big(T{(1+R)}^{-1}\big)\leq C_n(R)\,\Q_n(T),
$$
where the constant $C_n(R)$ is given by
$$
C_n(R):=4\sqrt{2}\,{\rm max}\{1,\Vert (1+R)^{-1}\Vert\}\,{\rm max}\{1,\Q_n(R)\}.
$$
\end{lemma}
\begin{proof}
For any $n\in\N$ we have
\begin{align}
\Q_n(T{(1+R)^{-1}})^2
&=\|T{(1+R)^{-1}}\|^2+\vf_{p+1/n}({(1+R^*)^{-1}}|T|^2{(1+R)^{-1}})+
\vf_{p+1/n}(T{|1+R|^{-2}}T^*)\nonumber\\
&\leq \|(1+R)^{-1}\|^2\,\big(\|T\|^2+\vf_{p+1/n}(TT^*)\big)+\vf_{p+1/n}({(1+R^*)^{-1}}|T|^2{(1+R)^{-1}})\nonumber\\
&\leq \|(1+R)^{-1}\|^2\,\Q_n(T)^2+\vf_{p+1/n}({(1+R^*)^{-1}}|T|^2{(1+R)^{-1}}),
\label{eq:3-lines}
\end{align}

where the first inequality follows
by an application of the operator inequality $A^*B^*BA\leq\|B\|^2A^*A$, while the second follows 
from the definition of the norm $\Q_n$.
Writing
\begin{align*}
&{(1+R^*)^{-1}}|T|^2{(1+R)^{-1}}\\
&\qquad\qquad=|T|^2- {R^*}{(1+R^*)^{-1}}|T|^2-|T|^2R{(1+R)^{-1}}+ {R^*}{(1+R^*)^{-1}}|T|^2 R{(1+R)^{-1}},
\end{align*}
the Cauchy-Schwarz inequality for the weight $\vf_{p+1/n}$ gives 
\begin{align*}
&\vf_{p+1/n}({(1+R^*)^{-1}}|T|^2{(1+R)^{-1}})
\leq \vf_{p+1/n}(|T|^2)+\vf_{p+1/n}( {R^*}{(1+R^*)^{-1}}|T|^2 R{(1+R)^{-1}})\\
&\hspace{3cm}+
 \vf_{p+1/n}(|T|^4)^{1/2}\left(\vf_{p+1/n}( {|R|^2}{|1+R|^{-2}})^{1/2}
 +\vf_{p+1/n}({|R^*|^2}{|1+R^*|^{-2}})^{1/2}\right).
 \end{align*}
 Using the operator inequality  $A^*B^*BA\leq\|B\|^2A^*A$ as above, we deduce that
 \begin{align*}
&\vf_{p+1/n}({(1+R^*)^{-1}}|T|^2{(1+R)^{-1}})\leq \vf_{p+1/n}(|T|^2)+\|T\|^2\,\|(1+R)^{-1}\|^2\,\vf_{p+1/n}(|R|^2)\\
&\hspace{3cm}+
\|T\|\,\|(1+R)^{-1}\|\,\vf_{p+1/n}(|T|^2)^{1/2}\left(\vf_{p+1/n}(|R|^2)^{1/2}+\vf_{p+1/n}(|R^*|^2)^{1/2}\right),
\end{align*}

Simplifying this last expression, using $\Vert T\Vert,\,\vf(|T|^2)^{1/2}\leq \Q_n(T)$ and similarly for $R$, we find
$$
\vf_{p+1/n}({(1+R^*)^{-1}}|T|^2{(1+R)^{-1}})\leq \Q_n(T)^2\,\left(1+\Vert (1+R)^{-1}\Vert\,\Q_n(R)\right)^2.
$$
This yields
$$
\Q_n(T(1+R)^{-1})\leq \sqrt{\Vert (1+R)^{-1}\Vert^2+(1+\Vert (1+R)^{-1}\Vert\,\Q_n(R))^2}\,\Q_n(T).
$$
Finally we employ, for $a,b>0$, the numerical inequalities 
\begin{align*}
\sqrt{a^2+(1+ab)^2}&\leq \sqrt{(ac)^2+(1+ac)^2},\quad c:={\rm max}\{1,b\}\\
&\leq \sqrt{2}(1+ac)\leq \sqrt{2}(1+a)(1+c)\\
&\leq 4\sqrt{2}\,{\rm max}\{1,a\}\,{\rm max}\{1,c\}\leq 4\sqrt{2}\,{\rm max}\{1,a\}\,{\rm max}\{1,b\},
\end{align*}
to arrive at the inequality of the statement of the Lemma.
\end{proof}

\begin{lemma} 
\label{R-siside}
Let $T\in \B_1(\D,p)$ and $R\in \B_2(\D,p)$, with $1+R$ invertible in $\cn$. Then
the operator $T(1+R)^{-1}$ belongs to  $\B_1(\D,p)$, with
$$
\PP_n\big(T{(1+R)^{-1}}\big)\leq C_n(R)\PP_n(T),\quad\mbox{for all }n\in\N,
$$
for the finite constant $C_n(R)$ of Lemma \ref{R-side}.
\end{lemma}

\begin{proof}
To see this, fix $n\in\N$ and consider any  representation of $T$
$$
T=\sum_{i=0}^\infty T_{1,i}T_{2,i}\quad\mbox{with}
\quad T_{1,i},T_{2,i}\in\B_2(\D,p)\quad\mbox{and}\quad
\sum_{i=0}^\infty \Q_n(T_{1,i})\Q_n(T_{2,i})<\infty.
$$
Then
\begin{align*}
&\PP_n(T{(1+R)^{-1}})\leq \sum_{i=0}^\infty \Q_n(T_{1,i})\,\Q_n(T_{2,i} {(1+R)^{-1}})
\leq C_n(R) \sum_{i=0}^\infty \Q_n(T_{1,i})\,\Q_n(T_{1,i}),
\end{align*}
where we used   Lemma \ref{R-side} to obtain the second estimate. 
Since the constant does not depend on the 
representation chosen, we have the inequality
$$
\PP_n\big(T{(1+R)^{-1}}\big)\leq C_n(R) \,\PP_n(T),
 $$
 which completes the proof.
\end{proof}

\begin{prop}
\label{HolB1}
For any $n\in\mathbb{N}$ and $p\geq 1$, the $*$-algebra
 $M_n(\B_1(\D,p))$ is stable under the holomorphic functional calculus.
\end{prop}

\begin{proof}
We begin with the case $n=1$.
Let $T\in\B_1(\D,p)$ and let $f$ be a function holomorphic in a neighborhood of 
the spectrum of $T$. Let $\Gamma$ be a positively oriented 
contour surrounding the spectrum of $T$, taking care that $0$ 
does not lie on $\Gamma$. We want to show that
(when $\B_1(\D,p)$ is a   nonunital subalgebra of $\cn$) 
$$
\int_\Gamma{f(z)}{(z-T)^{-1}}\,dz\in\,\B_1(\D,p)\oplus\mathbb{C}\,{\rm Id}_\cn,
$$
with the scalar component equal to $f(0){\rm Id}_\cn$.
Since
\begin{align*}
\int_\Gamma{f(z)}{(z-T)^{-1}}\,dz-f(0)\,{\rm Id}_\cn= \int_\Gamma f(z){T}{z^{-1}(z-T)^{-1}}\,dz,
\end{align*}
we get for all $n\in\N$
$$
\PP_n\left(\int_\Gamma{f(z)}{(z-T)^{-1}}\,dz-f(0)\,{\rm Id}_\cn\right)\leq \int_\Gamma \left|\frac{f(z)}{z^2}\right|\PP_n(T)C_n(-T/z)\,dz,
$$
where $C_n$ is the constant from Lemmas \ref{R-side} and \ref{R-siside}, 
and we have used Lemma \ref{B1-implies-B2} 
to see that $T/z\in \B_2(\D,p)$. 
Then the inequality 
$$
C_n(-T/z)\leq 4\sqrt{2}\,{\rm max}\{1,\Vert(1-T/z)^{-1}\Vert\}\,{\rm max}\{1,\Q_n(T)/|z|\},
$$ 
allows us to conclude. Again, the general case follows from  \cite{LBS}.
\end{proof}

We conclude this section by showing that  when the weights $\vf_s$, $s>0$,  are tracial, 
then our space
of integrable element $\B_1(\D,p)$, coincides with an intersection of trace-ideals. 
This fact will be of relevance
in two of our applications (Section \ref{Mfd} and subsection \ref{MP}), where the restriction 
of the faithful normal semifinite weights  $\vf_s$ to an appropriate   sub-von Neumann algebra
are faithful normal semifinite traces. 

\begin{prop} 
\label{tracial-case}
Assume that there exists a von Neumann
subalgebra $\cm\subset\cn$ such that for all $n\in\N$, the restriction of the
faithful normal semifinite weight $\tau_n:=\vf_{p+1/n}|_{\cm}$ is a faithful normal semifinite trace.
Then 
$$
\B_1(\cn,\tau)\bigcap\cm=\bigcap_{n\geq 1} \cl^1(\cm,\tau_n).
$$ 
Here  $\cl^1(\cm,\tau_n)$ denotes the trace ideal of $\cm$ associated 
with the faithful normal semifinite trace $\tau_n$.
Moreover, for any $n\in\N$,  $\PP_n(\cdot)=\Vert\cdot\Vert+2\Vert\cdot\Vert_{\tau_n}$, 
where $\Vert\cdot\Vert_{\tau_n}$ is  the 
trace-norm on $\cl^1(\cm,\tau_n)$.
\end{prop}

\begin{proof}
Note first that the tracial property of the faithful normal semifinite trace $\tau_n:=\vf_{p+1/n}|_{\cm}$, immediately implies 
that 
$$
\B_2(\cn,\tau)\bigcap\cm=\bigcap_{n\geq 1}\L^2(\cm,\tau_n),
$$
that is, the half-domain of $\tau_n$ on $\cm$ is already $*$-invariant and moreover
$$
\Q_n(T)=\big(\|T\|^2+2\||T|^2\|_{\tau_n}\big)^{1/2}.
$$
Now, take $T\in \B_1(\D,p)\bigcap\cm$, and any representation $T=\sum_{i=1}^\infty R_iS_i$. 
Observe then that the H\"{o}lder inequality
gives
$$
\Vert T\Vert+2\Vert T\Vert_{\tau_n}\leq 
\sum_{i=1}^\infty
 \Vert R_iS_i\Vert+2\Vert R_iS_i\Vert_{\tau_n}\leq
\sum_{i=1}^\infty\Q_n(R_i)\,\Q_n(S_i).
$$
Since this inequality is valid for any such  representation, 
it gives $\Vert T\Vert+2\Vert T\Vert_{\tau_n}\leq \PP_n(T)$ and hence
 $\B_1(\cn,\tau)\bigcap \cm\subset \bigcap_{n\geq 1}\L^1(\cm,\tau_n)$. 

Conversely, let $T\in\bigcap_n\L^1(\cm,\tau_n)$. If $T\geq 0$ then 
$T=\sqrt{T}\,\sqrt{T}$ and $\sqrt{T}\in \B_2(\D,p)\cap\cm$, by the first part of the proof and the fact that 
$\sqrt{T}\in \bigcap_n\L^2(\cm,\tau_n)$. Thus $T\in \B_1(\D,p)\cap\cm$ and, by Lemma 
\ref{lem:b1-pos},
$\PP_n(T)=\Q_n(\sqrt{T})^2= \Vert T\Vert+2\Vert T\Vert_{\tau_n}$. 
If  $T$ is now arbitrary in $\bigcap_n\L^1(\cm,\tau_n)$, we may write it as a linear combination of four
positive elements, $T= c_1T_1+c_2T_2+c_3T_3+c_4T_4$, with: $|c_j|=1$ for each $j=1,2,3,4$;   
$0\leq T_j\in\L^1(\cm,\tau_n)$ for each $n$; and $\|T_j\|+2\|T_j\|_{\tau_n}\leq \|T\|+2\|T\|_{\tau_n}$.
Hence
 $ \bigcap_{n\geq 1}\L^1(\cm,\tau_n) \subset\B_1(\cn,\tau)\bigcap \cm$. 
 
 Regarding the equality of norms, for $T\in  \bigcap_{n\geq 1}\L^1(\cm,\tau_n) =
 \B_1(\cn,\tau)\bigcap \cm$, write $T=S|T|$ for the  polar decomposition.
 Then by construction of the norms $\PP_n$ and the value of the norms $\Q_n$ we see that
 $$
 \PP_n(T)\leq \Q_n(S|T|^{1/2})\Q_n(|T|^{1/2})\leq \||T|^{1/2}\|^2+2\||T|\|_{\tau_n}
=\|T\|+2\|T\|_{\tau_n},
$$
and we conclude using the converse inequality already proven.
\end{proof}

\subsection{Smoothness and summability}

Anticipating the pseudodifferential calculus, we introduce subalgebras 
of $\B_1(\D,p)$ which `see' smoothness as well as summability. There
are several operators naturally associated to our notions of smoothness.

We recall that  $\D$ is a self-adjoint operator affiliated to a semifinite von Neumann algebra 
$\cn$ with
faithful normal semifinite trace $\tau$, and $p\geq 1$. For a few definitions, like the next, we do not require all
of this information.

\begin{definition}
\label{parup}
Let $\mathcal{D}$ be a self-adjoint 
operator affiliated to a semifinite von Neumann algebra $\mathcal{N}\subset \B(\H)$,
where $\H$ is a Hilbert space.\index{$\H$, a Hilbert space}
Set $\HH_\infty=\bigcap_{k\geq 0} {\rm dom}\,\D^k$. For an 
operator $T\in\cn$ such that $T:\HH_\infty\to\HH_\infty$
we set
\begin{equation}
\delta(T):=[|\D|,T],\quad \delta'(T):=[(1+\D^2)^{1/2},T],\quad T\in \cn.
\label{eq:deltas}
\end{equation}
In addition, we recursively set   \index{$\delta$, $\delta'$, derivations associated with $\D$}
\begin{equation}
T^{(n)}:=[\D^2,T^{(n-1)}],\; n\in\N\quad \mbox{and}\quad T^{(0)}:=T.
\label{eq:iterated-comms}
\end{equation} 
Finally, let \index{$T^{(n)}$} \index{$L$, $R$, operators associated with $\D^2$}
\begin{align}
\label{LR}
L(T):=(1+\mathcal{D}^2)^{-1/2}[\mathcal{D}^2,T],\quad R(T):=[\mathcal{D}^2,T](1+\mathcal{D}^2)^{-1/2}.
\end{align}
\end{definition}

We have defined $\delta,\,\delta',\,L,\,R$ for operators in $\cn$ preserving $\H_\infty$, and
so consider the domains of $\delta,\,\delta',\,L,\,R$ to be subsets of $\cn$. If $T\in{\rm dom}\,\delta$, say, so that
$\delta(T)$ is bounded, then it is straightforward to check that $\delta(T)$ commutes with
every operator in the commutant of $\cn$, and hence $\delta(T)\in\cn$. Similar comments
apply to $\delta',\,L,\,R$.

It follows from the proof of \cite[Proposition 6.5]{CPRS2} and $R(T)^*=-L(T^*)$ that
\begin{equation}
\bigcap_{n\geq 0}{\rm dom}\,{L}^n
=\bigcap_{n\geq 0}{\rm dom}\,{R}^n\,=\bigcap_{k,\,l\geq 0}{\rm dom}\,L^k\circ R^l.
\label{eq:LR}
\end{equation}
Similarly, using the fact that $|x|-(1+x^2)^{1/2}$ is a bounded function, it is proved after
the Definition 2.2 of 
\cite{CPRS2} that
\begin{equation}
\bigcap_{n\in\N}{\rm dom}\,\delta^n
=\bigcap_{n\in\N}{\rm dom}\,{\delta'}^n.
\label{eq:delta}
\end{equation}
Finally, it is proven in 
\cite{CM,Co5} and \cite[Proposition 6.5]{CPRS2} that we have equalities of all the smooth domains
in Equations \eqref{eq:LR}, \eqref{eq:delta}.

\begin{definition} 
\label{def-Bp}
 Let  $\D$ be a self-adjoint operator affiliated to a semifinite von Neumann algebra 
$\cn$ with
faithful normal semifinite trace $\tau$, and $p\geq 1$.
 Then
 define for $k\in\N_0=\N\cup\{0\}$
\begin{align*}
\B_1^k(\D,p)
&:=\big\{T\in\cn\,:\,\mbox{for all } l=0,\dots,k,
\, \delta^l(T)\in \B_1(\D,p)\big\},
\end{align*}
where $\delta=[|\D|,\cdot]$ as in Equation \eqref{eq:deltas}.
Also set
$$
\B_1^\infty(\D,p):=\bigcap_{k=0}^\infty \B_1^k(\D,p).
$$

We equip $\B_1^k(\D,p)$, $k\in\N_0\cup\{\infty\}$, 
with the topology determined by the seminorms 
\begin{equation}
\cn\ni T\mapsto \PP_{n,l}(T)
:=\sum_{j=0}^l\PP_n(\delta^j(T)),
\quad n\in\N,\quad l\in\N_0.
\label{eq:pp-n-l}
\end{equation}
\end{definition}
\index{$\B_2^k(\D,p)$, smooth version of $\B_2(\D,p)$}\index{$\B_1^k(\D,p)$, smooth version of $\B_1(\D,p)$}
\index{$\PP_{n,l}$, seminorms on $\B_1^k(\D,p)$ and $\B_1^\infty(\D,p)$}\index{seminorms!$\PP_{n,l}$, seminorms on $\B_1^k(\D,p)$, $\B_1^\infty(\D,p)$}

The triangle inequality for the seminorms $\PP_{n,l}$ follows from the linearity of $\delta^l$ 
and the triangle inequality for the norm $\PP_n$. Submultiplicativity then follows from the Leibniz rule
as well as the triangle inequality and submultiplicativity for $\PP_n$.
For $k$ finite, it is sufficient to 
consider the subfamily of norms $\{\PP_{n,k}\}_{n\in\N}$.

{\bf Remarks.}  (i) Defining
 $\B_2^k(\D,p)
:=\big\{T\in\cn\,:\,\mbox{for all } l=0,\dots,k,
\, \delta^l(T)\in \B_2(\D,p)\big\}$, an application of the Leibniz rule shows that
 $\B_2^k(\D,p)^2\subset\B_1^k(\D,p)$. 
 
It is important to observe that $\B_2^\infty(\D,p)$ is non-empty, 
and so $\B_1^\infty(\D,p)$ is non-empty. Note first that $\B_2(\D,p)$ is non empty
as it contains $\L^2(\cn,\tau)$. Then, for 
$T\in \B_2(\D,p)$, and $f\in C_c(\R)$ and $k,l\in\N_0$ 
arbitrary, $|\D|^kf(\D)Tf(\D)|\D|^l$ is well defined and is in 
$\B_2(\D,p)$ by Lemma \ref{Bp-grew}. This
implies that $\delta^k\big(f(\D)Tf(\D)\big)\in \B_2(\D,p)$ 
for any $k\in\N_0$ and thus $f(\D)Tf(\D)$ is in $\B_2^\infty(\D,p)$.\\

(ii) Using Lemma \ref{B1-grew}, we see that
the topology on the algebras $\B_1^k(\D,p)$ could have been equivalently defined with 
$\delta'=[(1+\D^2)^{1/2},\cdot]$ instead of
$\delta$. This follows since $f(\D)=|\D|-(1+\D^2)^{1/2}$ is bounded.
 Indeed,
Lemma \ref{B1-grew} shows that
$$
\PP_n(\delta(T))=\PP_n(\delta'(T)+[f(\D),T])\leq \PP_n(\delta'(T))+2\Vert f\Vert_\infty\,\PP_n(T),
$$
and similarly that $\PP_n(\delta'(T))\leq \PP_n(\delta(T))+2\Vert f\Vert_\infty\,\PP_n(T)$. Hence
convergence in the topology defined using $\delta$ implies convergence in the topology defined
by $\delta'$, and conversely. Similar comments apply for $\B_2^k(\D,p)$.

(iii) In Lemma \ref{delta-LR}, we will show that we could also use the seminorms $\PP_n(L^k(\cdot))$
(and similarly for $R^k$ and $L^k\circ R^j$) to define the topologies of $\B_1^\infty(\D,p)$
and $\B_2^\infty(\D,p)$.

We begin by proving that the algebra $\B_1^k(\D,p)$ is a Fr\'echet $*$-subalgebra of $\mathcal{N}$.

\begin{prop}
\label{Frechet2} 
For any $n\in\N$, $l=\N_0\cup\{\infty\}$ and 
$p\geq 1$, the $*$-algebra $M_n(\B_1^l(\D,p))$
is  Fr\'echet and stable under the holomorphic functional calculus.
\end{prop}

\begin{proof}
We first regard the question of completeness and  treat the case $l=1$ 
 and $n=1$ only, since
the general case is similar.

Let $(T_k)_{k\geq 0}$ be a Cauchy sequence in  $\B_1^1(\D,p)$. Since
$$
\PP_{n,1}(T_k-T_l)
=\PP_{n}(T_k-T_l)+\PP_n\big(\delta(T_k)-\delta(T_l)\big)
\geq \PP_n\big(\delta(T_k)-\delta(T_l)\big) \ ,\ \PP_{n}(T_k-T_l),
$$
we see that both $(S_k)_{k\geq 0}:=(\delta(T_k))_{k\geq 0}$ and $(T_k)_{k\geq 0}$ 
are Cauchy sequences in  
$\B_1(\D,p)$. Since $\B_1(\D,p)$ is complete, 
both $(S_k)_{k\geq 0}$ and $(T_k)_{k\geq 0}$ 
converge, say  to $S\in \B_1(\D,p)$ and $T\in \B_1(\D,p)$ respectively. 

Next observe that $\delta:{\rm dom}\,\delta\subset \cn\to \cn$ is bounded, 
where we give on ${\rm dom}\,\delta$ the topology determined by the norm 
$\Vert\cdot\Vert+\Vert\delta(\cdot)\Vert$. Hence $\delta$ has closed graph, and since 
$T_k\to T$ in norm and $\delta(T_k)$ converges in norm also, we have $S=\delta(T)$.
Finally, since $(\delta(T_k))_{k\geq 0}$ is Cauchy in $\B_1(\D,p)$, we have 
$S=\delta(T)\in \B_1(\D,p)$.

Next we pass to the question of stability under holomorphic functional calculus.
As before, the proof for $M_n(\B_1^k(\D,p))$, will follow from the proof for
$\B_1^k(\D,p)$.
By completeness of $\B_1^k(\D,p)$, it is enough to show that for $T\in \B_1^k(\D,p)$,
$T(1+T)^{-1}\in\B_1^k(\D,p)$ (see the proof of Proposition \ref{HolB1}).
But this follows from an iterative use of the relation
$$
\delta\Big( T{(1+T)^{-1}}\Big)=\delta(T){(1+T)^{-1}}-{T}{(1+T)^{-1}}\delta(T){(1+T)^{-1}},
$$
together with Lemma \ref{R-siside} and the fact that $\B_1(\D,p)$ is an algebra.
\end{proof}




\subsection{The pseudodifferential calculus} 
The pseudodifferential calculus of Connes-Moscovici, 
\cite{Co5,CM},
depends only on an unbounded self-adjoint operator $\D$. In its original form, this 
calculus characterises
those operators which are smooth `as far as $\D$ is concerned'. In subsection \ref{symdom} 
we saw that we 
could also talk about operators which are `integrable as far as $\D$ is concerned'. 
This latter notion also requires the trace $\tau$ and the dimension $p$.
We combine all these 
ideas in the following definition, to obtain a notion of pseudodifferential
operator adapted to the nonunital setting. 

Once again, throughout this subsection
we let  $\D$ be a self-adjoint operator affiliated to a semifinite von Neumann algebra 
$\cn$ with
faithful normal semifinite trace $\tau$ and $p\geq 1$.

 \begin{definition}
 \label{op0}
The set of
{\bf  order-$r$ tame pseudodifferential operators} associated with \index{pseudodifferential operators}\index{pseudodifferential operators!tame}\index{${\rm OP}^r_0$, tame pseudodifferential operators of order $r$}
$(\HH,\D)$, $(\cn,\tau)$ and $p\geq 1$ is given by
$$
{\rm OP}^r_0:=(1+\D^2)^{r/2}\B_1^\infty(\D,p) ,\quad r\in\mathbb R,\qquad {\rm OP}^*_0
:=\bigcup_{r\in\R}{\rm OP}^r_0.
$$

We topologise  ${\rm OP}^r_0$  with the family of norms
\begin{equation}
\PP_{n,l}^r(T):=\PP_{n,l}\big((1+\D^2)^{-r/2}T\big), \quad  n\in\N,\quad
l\in \N_0.
\label{eq:pp-n-l-r}
\end{equation}\index{seminorms!$\PP_{n,l}^r$, seminorms on ${\rm OP}^r_0$}
\end{definition}

{\bf Remark.} To lighten the notation, we do not make explicit the 
important dependence on the real number $p\geq 1$ and the operator $\D$
in the definition of the tame pseudodifferential operators.

With this definition, 
${\rm OP}^r_0$ is a Fr\'echet space and ${\rm OP}^0_0$ is a Fr\'echet $*$-algebra. 
In Corollary \ref{consistency} we will see that 
$\bigcup_{r<-p}{\rm OP}^r_0\subset \L^1(\cn,\tau)$, which is the basic justification
for the introduction  of tame pseudodifferential operators.

However, since
$\B_1^\infty(\D,p)$ is {\it a priori} a nonunital algebra, functions of $\D$ alone do not
belong to ${\rm OP}^*_0$. In particular, not all `differential operators',
such as powers of $\D$, are  tame pseudodifferential operators.

\begin{definition}
The set of {\bf regular order-$r$ pseudodifferential operators} is\index{pseudodifferential operators!regular}\index{${\rm OP}^r$, regular pseudodifferential operators of order $r$}
$$
{{\rm OP}^r}:=
(1+\D^2)^{r/2}\Big(\bigcap_{n\in\mathbb N}{\rm dom}\,\delta^n\Big),\quad r\in\mathbb R, 
\qquad {\rm OP}^*:=\bigcup_{r\in\R}{\rm OP}^r.
$$

The natural topology of ${\rm OP}^r$  is associated with the family of norms
$$
\sum_{k=0}^l\Vert \delta^k((1+\D^2)^{-r/2}T)\Vert, \quad  l\in\N_0.
$$
\end{definition}

By a slight adaptation of Lemma \ref{B1-implies-B2},  we see 
that $ \B_1^\infty(\D,p)\subset \B_2^\infty(\D,p)$ with $\Q_{n,k}(\cdot)\leq \PP_{n,k}(\cdot)$ for all $n\geq 1$ and $k\geq 0$. 
Moreover, we have from the definition that  $ \B_2^\infty(\D,p)\subset \bigcap_{n\in\mathbb N}{\rm dom}\,\delta^n$, 
with $\Vert\delta^k(\cdot)\Vert\leq \Q_{n,k}(\cdot)$. Thus  $ \B_1^\infty(\D,p)\subset  \bigcap_{n\in\mathbb N}{\rm dom}\,\delta^n$, 
with $\Vert\delta^k(\cdot)\Vert\leq \PP_{n,k}(\cdot)$. Hence,  we  
have a continuous inclusion  
${\rm OP}^r_0\subset{{\rm OP}^r}$. For $r>0$, ${{\rm OP}^r}$  contains all polynomials 
in $\D$ of order smaller than $r$. In particular, ${\rm Id}_\cn\in {\rm OP^0}$.

To prove that our definition of tame pseudodifferential operators is symmetric, namely that 
\begin{equation}
\label{symmetric}
{\rm OP}^r_0=(1+\D^2)^{r/2-\theta}\B_1^\infty(\D,p) (1+\D^2)^{\theta},
\quad \mbox{for all }\theta\in[0,r/2],
\end{equation}

we introduce the  complex one-parameter group $\sigma$ of automorphisms of ${\rm OP}^*$
 defined by
\begin{equation}
\label{sigma}
\sigma^z(T):=(1+\D^2)^{z/2}\,T\,(1+\D^2)^{-z/2},\quad z\in\C,\ T\in {\rm OP}^*.\index{$\sigma^z$, the one parameter group associated to $\D^2$}
\end{equation}
It is then clear that if we know that $\sigma$ preserves each ${\rm OP}^r_0$, then
Equation \eqref{symmetric} will follow immediately. The next few results show that
$\sigma$ restricts to a group of automorphisms of each ${\rm OP}^r$ and each ${\rm OP}^r_0$, $r\in\R$.

\begin{lemma}\label{root commutator estimate} 
There exists  $C>0$  such that for every   $T\in\mathcal{B}_1^{\infty}(\D,p)$ and
$\varepsilon\in[0,1/3],$ we have 
$\PP_n\big([(1+\mathcal{D}^2)^{\varepsilon/2},T]\big)\leq C\,\PP_n\big(\delta(T)\big)$.
\end{lemma}

\begin{proof} Let $g$ be a function on $\R$ such that the Fourier transform of 
$g'$ is integrable. The elementary equality
$$
[g(|\D|),T]
=-2i\pi\int_{\mathbb{R}}\widehat{g}(\xi)\xi\int_0^1\,
e^{-2i\pi\xi s|\D|}\,[|\D|,T]\,e^{-2i\pi\xi (1-s)|\D|}\,ds\,d\xi,
$$
implies by Lemma \ref{B1-grew}  that
$$
\PP_n\big([g(|\mathcal{D}|),T]\big)\leq \,\|\widehat{g'}\|_1\,\PP_n\big(\delta(T)\big).
$$

The estimate
$\|\widehat{g'}\|_1\leq\sqrt{2}(\|g'\|_2+\|g''\|_2)$ is well known. Setting $g_\eps(t)=(1+t^2)^{\varepsilon/2}$, 
an explicit computation of the associated $2$-norms proves that for $\eps\in[0,\tfrac12)$ we have
\begin{equation}
\label{RTT}
\|\widehat{g'_\eps}\|_1\leq \eps\,\pi^{1/4}\Big(\frac{\Gamma(\tfrac12-\eps)^{1/2}}{\Gamma(2-\eps)^{1/2}}
+\frac{\sqrt6 (2-\eps)\Gamma(\tfrac32-\eps)^{1/2}}{2\Gamma(4-\eps)^{1/2}}\Big).
\end{equation}
Since this estimate is uniform in $\eps$ on compact subintervals of 
$[0,\tfrac12)$, in particular on $[0,\tfrac13]$ and is independent of $T\in\mathcal{B}_1^{\infty}(\D,p)$, the assertion follows
immediately.
\end{proof}

\begin{lemma}\label{cont-grp} 
 Then there is a constant $C\geq 1$ such that
for all $T\in\B_1^\infty(\D,p)$ and $z\in\C$ 
$$
\PP_{n,l}\big(\sigma^z(T)\big)\leq \sum_{k=l}^{\lfloor 3\Re(z)\rfloor+l+1}C^{k}\PP_{n,k}(T).
$$
Thus $\sigma_z$ preserves $\B_1^\infty(\D,p)$.
\end{lemma}

\begin{proof} It is clear that
\begin{equation}
\label{trucus}
\sigma^z(T)
=T+[(1+\mathcal{D}^2)^{z/2},T](1+\mathcal{D}^2)^{-z/2}
=T+(1+\mathcal{D}^2)^{z/2}[(1+\mathcal{D}^2)^{-z/2},T].
\end{equation}
It follows from Lemma \ref{B1-grew} and  Lemma \ref{root commutator estimate} that 
for
$z\in[-1/3,1/3]$ we have
$$
\PP_n\big(\sigma^z(T)\big)\leq \PP_n(T)+C\,\PP_n\big(\delta(T)\big)\leq C\,\PP_{n,1}(T),
$$ 
with the same constant as in Lemma \ref{root commutator estimate} 
(which is thus independent of $T \in\B_1^\infty(\D,p)$ and $z\in\C$). By the group 
property, we have
$$
\PP_n\big(\sigma^z(T)\big)\leq \sum_{k=0}^{\lfloor 3\Re(z)\rfloor+1}C^{k}\PP_{n,k}(T),
$$ 
for $z\in \R$, and
as $\sigma^z$ commutes with  $\delta,$ we have
$\PP_{n,l}\big(\sigma^z(T)\big)\leq \sum_{k=l}^{\lfloor 3\Re(z)\rfloor+l+1}C^{k}\PP_{n,k}(T)$ for every
$z\in\mathbb{R}$. Finally, as $\sigma^z=\sigma^{i\Im (z)}\sigma^{\Re (z)}$ and 
$\sigma^{i\Im (z)}$ is  isometric for each $\PP_{n,l}$ (by Lemma \ref{B1-grew} again), the assertion follows.
\end{proof}

\begin{prop} 
\label{pr:cts}
The maps $\sigma^z:\mathcal{B}_1^{\infty}(\D,p)\to\mathcal{B}_1^{\infty}(\D,p)$, $z\in\C$, 
form a strongly continuous group of automorphisms which is uniformly continuous on 
vertical strips.
\end{prop}

\begin{proof}
Fix $T\in\B_1^\infty(\D,p)$. We need to prove that  the map 
$ z\mapsto \s^z(T)$ is continuous  from $\mathbb C$ to $\B_1^\infty(\D,p)$, for the 
topology determined by the norms $\PP_{n,l}$. By Lemma \ref{cont-grp} 
we know that $\sigma^z$ preserves $\mathcal{B}_1^{\infty}(\D,p)$ and  
since $\{\sigma^z\}_{z\in\C}$  is a group of automorphisms, continuity 
everywhere will follow from continuity at $z=0$. So, let $z\in\mathbb C$ 
with $|z|\leq\tfrac13$. From Equation \eqref{trucus}, 
it is enough to treat the case $\Re(z)\geq 0$. Moreover, 
Lemma \ref{B1-grew}  gives us
$$
\PP_{n,l}\big(\s^z(T)-T\big)\leq \PP_{n,l}\big([(1+\D^2)^{z/2},T]\big),
$$
and from the same  reasoning as that leading to the estimate \eqref{RTT}, we obtain
\begin{align*}
&\PP_{n,l}\big([(1+\D^2)^{z/2},T]\big)\\
&\qquad\leq|z|\,\pi^{1/4}\Big(\frac{\Gamma(\tfrac12-|\Re(z)|)^{1/2}}{\Gamma(2-|\Re(z)|)^{1/2}}
+\frac{\sqrt6 (2-|\Re(z)|)\Gamma(\tfrac32-|\Re(z)|)^{1/2}}{2\Gamma(4-|\Re(z)|)^{1/2}}\Big)\
\,\PP_{n,l+1}(T)=:|z|\,C(z).
\end{align*}
Since $C(z)$ is uniformly bounded on the vertical strip  $0\leq \Re(z)\leq \tfrac13$, we obtain the result. 
\end{proof}

{\bf Remark.} Using Lemma \ref{Bp-grew} in place of Lemma  \ref{B1-grew}, we see that Lemmas
\ref{root commutator estimate}, \ref{cont-grp} and Proposition \ref{pr:cts} hold also with 
$\B_2^\infty(\D,p)$ 
instead of $\B_1^\infty(\D,p)$.

We now deduce that these continuity results also hold for both tame and regular pseudodifferential operators.

\begin{prop}
\label{cont-op}
The group $\s$ is  strongly continuous  on
${\rm OP}^r_0$ for its natural topology, and similarly for ${\rm OP}^r$.
\end{prop}

\begin{proof}
Since $T\in {\rm OP}^r_0$ if and only if $(1+\D^2)^{-r/2}T\in \B_1^\infty(\D,p)$ and
since $\s^z$ commutes with the left multiplication by $(1+\D^2)^{-r/2}$, the proof is a
direct corollary of Proposition \ref{pr:cts}. The proof for ${\rm OP}^r$ is simpler since it uses
only the operator norm and not the norms $\PP_n^r$; we refer to \cite{Co5,CM, CPRS2} for a proof.
\end{proof}

We can now show that $\B_1^\infty(\D,p)$ has an equivalent definition in terms of 
the $L$ and/or $R$ operators, defined in Equation \eqref{LR}.
Unlike the equivalent definition in terms of $\delta'$ mentioned in the remark after
Definition \ref{def-Bp}, this does not work for $\B_1^k(\D,p)$, $k\neq \infty$.

\begin{lemma}
\label{delta-LR}
We have the equality
\begin{align*}
\B_1^\infty(\D,p)
&=\big\{T\in\cn\,:\,\forall  l\in\N_0,
\,\,\, L^l(T)\in\B_1(\D,p)\big\},
\end{align*}
where $L(\cdot)=(1+\D^2)^{-1/2}[\D^2,\cdot]$ is as in Definition \ref{parup}. The analogous
statement  with $R$ replacing $L$ is also  true.
\end{lemma}

\begin{proof}
We have the simple identity $L=(1+\sigma^{-1})\circ \delta'$, which 
with Proposition \ref{pr:cts} yields one of the inclusions.

For the other direction, it suffices to show that for every $m,n\in\N$ we have
$$
\PP_m(\delta'^n(A))\leq \max_{n\leq k\leq 2n}\PP_m(L^k(A)).
$$
Using the integral formula for fractional powers we have
$$
\delta'(T)=[(1+\D^2)(1+\D^2)^{-1/2},T]
=\frac1{\pi}\int_0^{\infty}\lambda^{-1/2}[({1+\mathcal{D}^2}){(1+\lambda+\mathcal{D}^2)^{-1}},T]
d\lambda.
$$

However, a little algebra gives
$$
\Big[\frac{1+\mathcal{D}^2}{1+\lambda+\mathcal{D}^2},T\Big]
=\Big(\frac{(1+\mathcal{D}^2)^{1/2}}{1+\lambda+\mathcal{D}^2}
-\frac{(1+\mathcal{D}^2)^{3/2}}{(1+\lambda+\mathcal{D}^2)^2}\Big)L(T)
+\lambda\frac{1+\mathcal{D}^2}{(1+\lambda+\mathcal{D}^2)^2}L^2(T)\frac{1}{1+\lambda+\mathcal{D}^2}.
$$
The following formula can be proved in the scalar case, and by an appeal to the spectral representation
proved in general:
$$
\int_0^{\infty}\lambda^{-1/2}
\Big(\frac{(1+\mathcal{D}^2)^{1/2}}{1+\lambda+\mathcal{D}^2}
-\frac{(1+\mathcal{D}^2)^{3/2}}{(1+\lambda+\mathcal{D}^2)^2}\Big)d\lambda=\frac{\pi}{2}.
$$
Therefore,
$$
\delta'(T)=
\tfrac12L(T)+\frac1{\pi}\int_0^{\infty}\lambda^{1/2}
\frac{1+\mathcal{D}^2}{(1+\lambda+\mathcal{D}^2)^2}L^2(T)\frac{1}{1+\lambda+\mathcal{D}^2}d\lambda.
$$
An induction now shows that
$$
\delta'^n(T)
=2^{-n}\sum_{k=0}^n{n\choose k}
\Big(\frac{2}{\pi}\Big)^k\int_{\mathbb{R}^k_+}
\prod_{l=1}^k\frac{\lambda_l^{1/2}(1+\mathcal{D}^2)}{(1+\lambda_l+\mathcal{D}^2)^2}
L^{n+k}(T)\prod_{l=1}^k\frac{d\lambda_l}{1+\lambda_l+\mathcal{D}^2}.
$$
The functional calculus then gives
$$
({1+\lambda+\mathcal{D}^2})^{-1}\leq({1+\lambda})^{-1},
\quad {\lambda^{1/2}(1+\mathcal{D}^2)}{(1+\lambda+\mathcal{D}^2)^{-2}}\leq{\lambda^{-1/2}/4},
$$
and so by Lemma \ref{B1-grew}  we have 
$$
\PP_m\big(\delta'^n(T)\big)
\leq 2^{-n}\left(1+\sum_{k=1}^n{n\choose k}\Big(\frac{2}{\pi}\Big)^k
\prod_{l=1}^k\int_0^{\infty}\frac{d\lambda_l}{4\lambda_l^{1/2}(1+\lambda_l)}\right) 
\max_{n\leq k\leq 2n}\PP_m\big(L^k(T)\big).
$$
The assertion now follows by the second remark following Definition \ref{def-Bp} that
we may equivalently use $\delta'$ to define $\B_1^k(\D,p)$ for $k\in\N\cup\{\infty\}$.
\end{proof}

We now begin to prove the important properties of this pseudodifferential calculus,
such as trace-class properties and the pseudodifferential expansion.
First, by combining Proposition \ref{cont-op} with the Definition \ref{op0}, 
we obtain our first trace class property.

\begin{corollary}
\label{consistency}
For $r>p$, we have ${\rm OP}^{-r}_0\subset \cl^1(\cn,\tau)$.
\end{corollary}

\begin{proof}
 Let $T_r\in {\rm OP}^{-r}_0$. By Definition \ref{op0} and Proposition \ref{cont-op},
we see that the symmetric definition of ${\rm OP}^r_0$ in Equation \eqref{symmetric} is
equivalent to the original definition. Thus, there exists $A\in \B_1^\infty(\D,p)\subset
\B_1(\D,p)$ such that 
$$
T_r=(1+\D^2)^{-r/4}A(1+\D^2)^{-r/4}.
$$
Define $n:=\lfloor(r-p)^{-1}\rfloor$ and write $A=\sum_{k=0}^3i^kA_k$ with 
$A_k\in\B_1(\D,p)$ positive, as in
Proposition \ref{cor:polar-decomp}.
The H\"older inequality then entails that
\begin{align*}
\|T_r\|_1&=\|(1+\D^2)^{-r/4}A(1+\D^2)^{-r/4}\|_1
\leq
\|(1+\D^2)^{-p/4-1/4n}A(1+\D^2)^{-p/4-1/4n}\|_1\\
&\leq \sum_{k=0}^3\big\|(1+\D^2)^{-p/4-1/4n}\sqrt{A_k}\big\|_2^2
\leq  \sum_{k=0}^3 \Q_n\big(\sqrt{A_k}\big)\Q_n\big(\sqrt{A_k}\big)
=\sum_{k=0}^3\PP_n(A_k)<\infty,
\end{align*}
which is enough to conclude.
\end{proof}

As expected, the product of a tame pseudodifferential operator by a regular
pseudodifferential operator is a tame pseudodifferential operator.

\begin{lemma}
\label{usefull}
For all $r,\,t\in\mathbb R$ we have 
$\big({\rm OP}^r_0\,{\rm OP}^t\cup {\rm OP}^t\,{\rm OP}^r_0\big)\subset{\rm OP}^{r+t}_0$.
\end{lemma}

\begin{proof}
Since $\s$ preserves both ${\rm OP}^r_0$ and ${\rm OP}^r$,  it suffices to prove the claim for
$r=t=0$. Indeed,  for $T_r\in {\rm OP}^r_0$ and $T_s\in{\rm OP}^s$, there exist
$A\in {\rm OP}^0_0$ and $B\in{\rm OP}^0$ such that $T_r=(1+\D^2)^{r/2}A$ and $ T_s=(1+\D^2)^{s/2}B$.
Thus, the general case will follow from the case $t=s=0$  by writing
$$
T_rT_s=(1+\D^2)^{(r+s)/2}\sigma^{-s}(A)B.
$$
So let $T\in {\rm OP}^0_0$ and $S\in{\rm OP}^0$. We need to show that
$TS\in {\rm OP}^0_0=\B_1^\infty(\D,p)$. For this, let $T=\sum_{i=10}^\infty T_{1,i}T_{2,i}$ any 
representation. We will prove that 
$$
\sum_{i=0}^\infty T_{1,i}\,(T_{2,i}S),
$$
is a representation of the product $TS$. Indeed, we have
\begin{align*}
&\Q_n(T_{2,i}S)^2=\|T_{2,i}S\|^2+\|T_{2,i}S(1+\D^2)^{-p/4-1/4n}\|_2^2+\|S^*T_{2,i}^*(1+\D^2)^{-p/4-1/4n}\|_2^2\\
&\leq \|S\|^2\|T_{2,i}\|^2+\|\sigma^{p/4+1/4n}(S)\|^2\|T_{2,i}(1+\D^2)^{-p/4-1/4n}\|_2^2+\|S\|^2\|T_{2,i}^*(1+\D^2)^{-p/4-1/4n}\|_2^2\\
&\leq \big( \|S\|+\|\sigma^{p/4+1/4n}(S)\|\big)^2\Q_n(T_{2,i})^2,
\end{align*}
which is finite because ${\rm OP}^0=\bigcap_{n\in\N}{\rm dom}\,\delta^n$ 
is invariant under $\sigma$ by Proposition  \ref{cont-op}. This immediately shows that $TS\in \B_1(\D,p)$
since
$$
\PP_n(TS)\leq \sum_{i=0}^\infty \Q_n(T_{1,i})\,\Q_n(T_{2,i}S)\leq \big( \|S\|+\|\sigma^{p/4+1/4n}(S)\|\big)
\sum_{i=0}^\infty \Q_n(T_{1,i})\,\Q_n(T_{2,i})<\infty.
$$
In particular, one finds $\PP_n(TS)\leq \big( \|S\|+\|\sigma^{p/4+1/4n}(S)\|\big)\PP_n(T)$.
Now the formula $\delta^k(TS)=\sum_{j=0}^k{{k}\choose{j}}\delta^{j}(T)\delta^{k-j}(S)$ and the last estimate
shows that $\PP_{n,k}(TS)=\PP_{n}(\delta^{k}(TS))$ is finite and so $TS\in \B_1^\infty(\D,p)$.
That ${\rm OP}^t\,{\rm OP}^r_0\subset{\rm OP}^{r+t}_0$ can be proven in the same way.
\end{proof}

{\bf Remark.} Lemma \ref{usefull} shows that $\B_1^\infty(\D,p)$ is a two-sided ideal in 
$\bigcap{\rm dom}\,\delta^k$.

The following is a Taylor-expansion type theorem for ${\rm OP}^r_0$ 
just as in \cite{Co5,CM}, and adapted to our setting.

\begin{prop}
\label{TE}
Let $T\in{\rm OP}^r_0$ and $z=n+1-\alpha$ with $n\in\N_0$ and $\Re(\alpha)\in(0,1)$. Then we have
$$
\s^{2z}(T)-\sum_{k=0}^nC_k(z)\,(\s^2-{\rm Id})^k(T)\in {\rm OP}^{r-n-1}_0\quad{\rm with}
\quad C_k(z):=\frac{z(z-1)\cdots(z-k+1)}{k!}.
$$
\end{prop}
\index{pseudodifferential operators!Taylor expansion}

\begin{proof}
The proof is exactly the same as that in \cite{CM,Co5} once we realise that
if $T\in {\rm OP}^r_0$ then $(\s^2-{\rm Id})^k(T)\in{\rm OP}^{r-k}_0$. This follows from
$$
(\s^2-{\rm Id})^k(T)=(1+\D^2)^{-k/2}\s^k\big(\delta'^k(T)),
$$
and the invariance of each ${\rm OP}^r_0$ under $\delta'=[(1+\D^2)^{1/2},\cdot]$ and $\s$. For
$\delta'$ this follows from the second remark following Definition \ref{def-Bp}.
\end{proof}

\begin{lemma} 
\label{lem:derivs}
If $A\in {\rm OP}^r_0$ and $n\in\N_0$, then $A^{(n)}\in {\rm OP}^{r+n}_0$, 
where $A^{(n)}$ is as in Definition \ref{parup}.
\end{lemma}

\begin{proof} For $n=1$, by assumption there is an operator $T\in {\rm OP}^0_0$ such that
$A=(1+\D^2)^{r/2}T$. Then $A^{(1)}=(1+\D^2)^{r/2}T^{(1)}=(1+\D^2)^{(r+1)/2}L(T)$.
So the proof follows from the relation $L=(1+\sigma^{-1})\circ\delta'$ and the fact that 
both $\sigma^{-1}$ and $\delta'$ preserve ${\rm OP}^0_0$, by Lemma \ref{cont-grp}. 
The general case follows by
induction.
\end{proof}

\begin{prop} 
The derivation $LD$ defined by $LD(T):=[\log(1+\D^2),T]$, 
preserves ${\rm OP}^r_0$, for all $r\in\mathbb{R}$.
\end{prop}

\begin{proof} Set $g(t)=\log(1+t^2).$ We have $\|\widehat{g'}\|_1<\infty$ and
$$
LD(T)=[g(|\D|),T]=
-2i\pi\int_\R \widehat {g}(\xi)\xi\int_0^1\,e^{-2i\pi\xi s|\D|}\,\delta(T)\,e^{-2i\pi\xi (1-s)|\D|}\,ds\,d\xi.
$$
The assertion follows as in Lemma \ref{root commutator estimate}.
\end{proof}
We next improve Proposition \ref{cont-op}.
\begin{prop}
\label{LD}
The map $\sigma : \C\times{\rm OP}^r_0\to{\rm OP}^{r}_0$, is strongly holomorphic (entire),  with
$$
\frac d{dz} \sigma^z=\tfrac12\s^z\circ LD.
$$
\end{prop}
\begin{proof}
If $z-z_0=u,$ then we have
$$
\Big(\frac{\sigma^z-\sigma^{z_0}}{z-z_0}-\tfrac12\sigma^{z_0}\circ LD\Big)
=\sigma^{z_0}\circ\Big(\frac{\sigma^u-1}{u}-\tfrac12 LD\Big).
$$
Since $\sigma^{z_0}$ is strongly continuous, it is sufficient to prove  holomorphy at $z_0=0.$
Then for $T\in {\rm OP}_0^r$ we see that
\begin{align}
\label{vache}
\frac{\sigma^z(T)-T}{z}-\tfrac12LD(T)
=[g_z(\mathcal{D}),T]+z^{-1}[(1+\mathcal{D}^2)^{z/2},T]\big((1+\mathcal{D}^2)^{-z/2}-1\big),
\end{align}
with $g_z(s)=z^{-1}\big((1+s^2)^{z/2}-1\big)-\frac12\log(1+s^2).$ 
An explicit computation shows that $\|g_z'\|_2+\|g_z''\|_2=O(|z|)$. 
Since $\sqrt{2}(\Vert g_z'\Vert_2+\Vert g_z''\Vert_2)\geq \Vert \widehat{g_z'}\Vert_1$, we see
that  $\|\widehat{g_z'}\|_1\to 0$ as $z\to 0$.  It follows, as in Lemma \ref{root commutator estimate}, 
that the first term tends 
to $0$ in the $\PP_{n,l}^{r}$-norms, as $z\to0$.

It remains to treat the second commutator in 
Equation \eqref{vache}. We let $z\in\C$ with $0<\Re(z)< 1$. Employing 
the integral formula for complex  powers of a positive operator $A\in\cn$
\begin{equation}
\label{RER}
A^z=\pi^{-1}{\sin(\pi z)}\int_0^\infty \lambda^{-z}A({1+\lambda A})^{-1} d\lambda,\quad 0<\Re(z)<1,
\end{equation} 
gives
\begin{align*}
(1+\D^2)^{-z/2}=\left((1+\D^2)^{-1/2}\right)^z&=\pi^{-1}{\sin(\pi z)}\int_0^\infty \lambda^{-z}
(1+\D^2)^{-1/2}(1+\lambda(1+\D^2)^{-1/2})^{-1}d\lambda\\
&=\pi^{-1}{\sin(\pi z)}\int_0^\infty \lambda^{-z}
((1+\D^2)^{1/2}+\lambda)^{-1}d\lambda.
\end{align*}
We apply this formula by choosing $0<\eps<(1-\Re(z))$ and writing
\begin{align*}
&\frac{1}{z}[(1+\D^2)^{z/2},T]\big((1+\mathcal{D}^2)^{-z/2}-1\big)\\
&=-\frac{1}{z}(1+\D^2)^{z/2}[(1+\D^2)^{-z/2},T](1+\D^2)^{z/2}\big((1+\mathcal{D}^2)^{-z/2}-1\big)\\
&=\frac{\sin(\pi z)}{\pi z}\int_0^\infty \lambda^{-z}
(1+\D^2)^{z/2}((1+\D^2)^{1/2}+\lambda)^{-1}\delta'(T)((1+\D^2)^{1/2}+\lambda)^{-1}
(1+\D^2)^{(z+\eps)/2}\\
&\qquad\qquad\qquad\qquad\qquad
\times(1+\D^2)^{-\eps/2}\big((1+\mathcal{D}^2)^{-z/2}-1\big)d\lambda.
\end{align*}
Using the elementary estimate 
$$
\Vert ((1+\D^2)^{1/2}+\lambda)^{-1}(1+\D^2)^{z/2}\Vert_\infty\leq (1+\lambda)^{\Re(z)-1},
$$
we have 
\begin{align*}
&\PP^r_{n,l}\left(\frac{1}{z}[(1+\D^2)^{z/2},T]\big((1+\mathcal{D}^2)^{-z/2}-1\big)\right)\\
&\leq \frac{|\sin(\pi z)|}{\pi }\,\PP^r_{n,l}(\delta'(T))\,
\Big\Vert \frac{1}{z}(1+\D^2)^{-\eps/2}\big((1+\mathcal{D}^2)^{-z/2}-1\big)\Big\Vert_\infty \,
\int_0^\infty 
\lambda^{-\Re(z)}(1+\lambda)^{2\Re(z)-2+\eps}d\lambda.
\end{align*}

This concludes the proof since, as $0<\Re(z)< 1-\eps$, the last norm is 
bounded in a neighborhood of $z=0$, while the integral over $\lambda$ is 
bounded (provided $\eps$
is small enough) and $|\sin(\pi z)|$ goes to zero with $z$.
\end{proof}

Last, we prove that the derivation $LD(\cdot)=[\log(1+\D^2),\cdot]$ `almost' lowers the order of a tame pseudodifferential operator by one.

\begin{prop}
\label{RTY}
For all $r\in \R$ and for any $\eps\in(0,1)$, $LD$ continuously maps  
${\rm OP}^r_0$ to ${\rm OP}^{r-1+\eps}_0$.
\end{prop}

\begin{proof}
Since the proof for a generic $r\in\R$ will follows from those of a fixed 
$r_0\in\R$, we may assume that $r=0$. Let $T\in {\rm OP}^0_0$. 
We need to show that $LD(T)\in{\rm OP}^{-1+\eps}_0$ for any $\eps>0$, 
or equivalently, that $LD(T)(1+\D^2)^{1/2-\eps/2}\in{\rm OP}^{0}_0$ for any $\eps>0$. 

We use the integral representation
$$
\log(1+\D^2)=\D^2\,\int_0^1({1+w\D^2})^{-1}\,dw,
$$
which follows from $\log(1+x)=\int_0^x \frac{1}{1+\lambda}\,d\lambda$ via the change of variables
$\lambda=xw$. Then
\begin{align*}
[\log(1+\D^2),T](1+\D^2)^{1/2-\eps/2}
&=[\D^2,T](1+\D^2)^{-1/2}\,\int_0^1\frac{(1+\D^2)^{1-\eps/2}}{1+w\D^2}\,dw\\
&-\D^2\,\int_0^1\frac{w}{1+w\D^2}[\D^2,T](1+\D^2)^{-1/2}\frac{(1+\D^2)^{1-\eps/2}}{1+w\D^2}\,dw.
\end{align*}
Now elementary calculus shows that for $1>\alpha>0$ and $1\geq x\geq 0$ we have
$$
\frac{(1+x)^\alpha}{(1+xw)}\leq \left(\frac{\alpha}{w}\right)^\alpha \left(\frac{1-\alpha}{1-w}\right)^{1-\alpha}
\quad\mbox{and}\quad
\int_0^1 w^{-\alpha}(1-w)^{\alpha-1}dw=\Gamma(1-\alpha)\,\Gamma(\alpha),
$$
and so we obtain the integral estimate
$$
\int_0^1 \frac{(1+x)^\alpha}{(1+xw)}dw\leq\alpha^\alpha\,(1-\alpha)^{1-\alpha}\,\Gamma(1-\alpha)\,\Gamma(\alpha).
$$
Then using $R(T)=[\D^2,T](1+\D^2)^{-1/2}$ and elementary spectral theory gives
$$
\PP_{n,k}\left([\log(1+\D^2),T](1+\D^2)^{1/2-\eps/2}\right)\leq 2\PP_{n,k}(R(T))
\,(1-\eps/2)^{1-\eps/2}\,(\eps/2)^{\eps/2}\,\Gamma(\eps/2)\,\Gamma((1-\eps)/2),
$$
which gives the bound for all $0<\eps<1$.
\end{proof}




\subsection{Schatten norm estimates for tame pseudodifferential operators}
\label{subsec:trace-ests}

In this subsection we prove the  Schatten norm estimates \index{Schatten norm estimates}
we will require in our proof
of the local index formula. As before, we 
let  $\D$ be a self-adjoint operator affiliated to a semifinite von Neumann algebra 
$\cn$ with
faithful normal semifinite trace $\tau$ and $p\geq 1$.

\begin{lemma}
\label{interpolation}
Let $A\in {\rm OP}^0_0$ and $\alpha,\beta\geq 0$ with $\alpha+\beta>0$. 
Then $(1+\D^2)^{-\beta/2}A(1+\D^2)^{-\alpha/2}$ belongs to $\cl^q(\cn,\tau)$ 
for all $q>p/(\alpha+\beta)$, provided  $q\geq 1$. 
\end{lemma}

\begin{proof}
Since $(1+\D^2)^{-\beta/2}A(1+\D^2)^{-\alpha/2}=\sigma^{-\beta}(A)(1+\D^2)^{-\alpha/2-\beta/2}$ 
and because $\sigma$ is continuous, Proposition \ref{pr:cts}, on ${\rm OP}^0_0=\B_1^\infty(\D,p)$  we 
can assume $\beta=0$.

So let  $A\in {\rm OP}^0_0$.
Note first that for $y\in\mathbb R$ we have $A(1+\D^2)^{iy/2}\in\cn$ and by  Corollary \ref {consistency}
$A(1+\D^2)^{-\alpha q/2+iy/2}\in\cl^1(\cn,\tau)$, since $\alpha q>p$. 
Consider then, on the strip $0\leq\Re(z)\leq1$ the 
holomorphic operator-valued function given by
$F(z):=A(1+\D^2)^{-\alpha qz/2}$. The previous observation gives $F(iy)\in\cn$ 
and $F(1+iy)\in\cl^1(\cn,\tau)$. Then, a standard complex interpolation 
argument gives $F(1/q+iy)\in\cl^q(\cn,\tau)$, for $q\geq 1$, which was all we  needed.
\end{proof}

\begin{lemma}
\label{truc}
For $\alpha\in[0,1]$, $\beta,\gamma\in\R$ with $\alpha+\beta+\gamma>0$ and 
$A\in{\rm OP}^0_0$ we let
\begin{align*}
B_{\alpha,\beta,\gamma}&:=(1+\D^2)^{-\beta/2}\big[(1+\D^2)^{(1-\alpha)/2},A\big](1+\D^2)^{-\gamma/2},\\
C_{\alpha,\beta,\gamma}&:=
(1+\D^2)^{-\beta/2}\big[(1+\D^2)^{(1-\alpha)/2},A\big](1+\D^2)^{-\gamma/2}\log(1+\D^2),\\
D_{\alpha,\beta,\gamma}&:=
(1+\D^2)^{-\beta/2}\big[(1+\D^2)^{(1-\alpha)/2}\log(1+\D^2),A\big](1+\D^2)^{-\gamma/2}.
\end{align*}
Then $B_{\alpha,\beta,\gamma},C_{\alpha,\beta,\gamma},D_{\alpha,\beta,\gamma}\in\LL^{q}(\cn,\tau)$ 
for all $q>p/(\alpha+\beta+\gamma)$, provided  $q\geq 1$.  Moreover, the same conclusion holds with 
$|\D|$ instead of $(1+\D^2)^{1/2}$ in the commutator.
\end{lemma}

\begin{proof}
There exists $\eps>0$ such $\alpha+\beta+\gamma-\eps>0$. Since moreover  
$(1+\D^2)^{-\eps/2}\log(1+\D^2)$ is bounded for all $\eps>0$, 
we see that the assertion for $B_{\alpha,\beta,\gamma-\eps/2}$ implies the assertion for
$C_{\alpha,\beta,\gamma}$. Note also that the Leibniz rule implies
$$
D_{\alpha,\beta,\gamma}= 
C_{\alpha,\beta,\gamma}+(1+\D^2)^{1/2-(\alpha+\beta)/2}LD(A)(1+\D^2)^{-\gamma/2},
$$
so the third case follows from the second case  using Proposition 
\ref{RTY} and Lemma \ref{interpolation}.

Thus it suffices to treat the case of $B_{\alpha,\beta,\gamma}$. Moreover, we can 
further assume that $\alpha\in(0,1)$ (for $\alpha=1$ there is nothing to prove and 
for $\alpha=0$, the statement  follows from Lemma \ref{interpolation}) and, as in the 
proof of the preceding lemma,  we can assume $\beta=0$. 
Using the integral formula for fractional powers, Equation \eqref{RER}, 
for $0<\alpha<1$, we see that
\begin{align*}
B_{\alpha,0,\gamma}
&=-(1+\D^2)^{(1-\alpha)/2}[(1+\D^2)^{(\alpha-1)/2},A](1+\D^2)^{(1-\alpha)/2}(1+\D^2)^{-\gamma/2}\\
&=\pi^{-1}{\sin\pi(1-\alpha)/2}\int_0^\infty 
\lambda^{(1-\alpha)/2}(1+\D^2)^{(1-\alpha)/2}(1+\D^2+\lambda)^{-1}\\
&\qquad\qquad\qquad\qquad\qquad\qquad
\times[\D^2,A](1+\D^2+\lambda)^{-1}(1+\D^2)^{(1-\alpha-\gamma)/2}d\lambda\\
&=\pi^{-1}{\sin\pi(1-\alpha)/2}\int_0^\infty 
\lambda^{(1-\alpha)/2}(1+\D^2)^{1-\alpha/2}(1+\D^2+\lambda)^{-1}\\
&\qquad\qquad\qquad\qquad\qquad\qquad
\times L(A)(1+\D^2)^{(\eps-\alpha-\gamma)/2}(1+\D^2+\lambda)^{-1}(1+\D^2)^{(1-\eps)/2}d\lambda.
\end{align*}
By Lemma \ref{interpolation} we see that for  $\eps>0$ sufficiently small,
$L(A)(1+\D^2)^{(\eps-\alpha-\gamma)/2}\in \L^{q}(\cn,\tau)$ for all $q>p/(\alpha+\gamma-\eps)$ provided $q\geq 1$. 
So estimating in
the $q$ norm with $q:=p/(\alpha+\gamma-2\eps)>p/(\alpha+\gamma-\eps)$ gives
\begin{align*}
\Vert B_{\alpha,0,\gamma}\Vert_{q}
&\leq \Vert L(A)(1+\D^2)^{(\eps-\alpha-\gamma)/2} \Vert_{q}\,
\int_0^\infty \lambda^{-(1-\alpha)/2}(1+\lambda)^{-\alpha/2}(1+\lambda)^{-1/2-\eps/2}\,d\lambda,
\end{align*}
which is finite.
Finally, the same conclusion holds with 
$|\D|$ instead of $(1+\D^2)^{1/2}$ in the commutator,
and this follows from the same estimates and the fact that
$|\D|^{1-\alpha}-(1+\D^2)^{(1-\alpha)/2}$ extends to a bounded operator for $\alpha\in[0,1]$.
\end{proof}

In the course of our proof of the local index formula, we will require additional parameters.
In the following lemma we use the same notation as later in the paper for ease of reference.

\begin{lemma}\label{lem:was-5.3}
Assume that there exists $\mu>0$
such that $\D^2\geq \mu^2$.
Let $A\in {\rm OP}^0_0$,  $\lambda=a+iv$, $0<a<\mu^2/2$, $v\in\mathbb R$,  
$s\in\mathbb R$ and $t\in[0,1]$, and set 
$$
R_{s,t}(\lambda)=(\lambda-(t+s^2+\D^2))^{-1}.
$$   
Let also $q\in[1,\infty)$ and
$N_1,N_2\in\tfrac12\mathbb{N}\cup\{0\}$, with $N_1+N_2>p/2q$. 
Then  for each $\eps>0$, there exists a finite constant $C$ such that
\begin{align*}
&\Vert R_{s,t}(\lambda)^{N_1} AR_{s,t}(\lambda)^{N_2}\Vert_q\leq C
((t+\mu^2/2+s^2-a)^2+v^2)^{-(N_1+N_2)/2+p/4q+\eps}.
\end{align*}
(For half integers, we use the principal branch of the square root function).\index{resolvent function}
\end{lemma}

{\bf Remark.} Here is the point where we require $0<a<\mu^2/2$ in the definition of our 
contour of integration $\ell$. It is clear from the proof below, where this condition is
used, that there is some flexibility to reformulate this condition.

\begin{proof}
By the functional calculus  (see the proof of \cite[Lemmas 5.2 \& 5.3]{CPRS2} for more details) 
and the fact that $a<\mu^2/2$, we have the operator inequalities for any 
$N\in\tfrac12\N\cup\{0\}$ and $Q<N$
\begin{align*}
|R_{s,t}(\lambda)^N|&\leq (\D^2-\mu^2/2)^{-Q}\,((t+\mu^2/2+s^2-a)^2+v^2)^{-N/2+Q/2},
\end{align*}
which gives the following estimate
\begin{align*}
&\Vert R_{s,t}(\lambda)^{N_1} AR_{s,t}(\lambda)^{N_2}\Vert_q\\
&\leq \|R_{s,t}(\lambda)^{N_1} (\D^2-\mu^2/2)^{Q_1}\| \|R_{s,t}(\lambda)^{N_2} 
(\D^2-\mu/2)^{Q_2}\| \|(\D^2-\mu^2/2)^{-Q_1}A(\D^2-\mu^2/2)^{-Q_2}\|_q\\
&\leq ((t+\mu^2/2+s^2-a)^2+v^2)^{-(N_1+N_2)/2+(Q_1+Q_2)/2} \|(\D^2-\mu^2/2)^{-Q_1}A(\D^2-\mu^2/2)^{-Q_2}\|_q.
\end{align*}
One concludes the proof using Lemma \ref{interpolation} by choosing $Q_1\leq N_1,\,Q_2\leq N_2$ 
such that 
$Q_1+Q_2=p/2q+\eps$.
\end{proof}

{\bf Remark.}
For $\lambda=0$ and with the same constraints on $q$ and $N$ as above, 
the same operator inequalities as those of  \cite[Lemma 5.10]{CPRS4}, gives
\begin{align}
\Vert A(t+s^2+\D^2)^{-N}\Vert_q\leq\Vert A(\D^2-\mu^2/2)^{-(p/q+\eps)/2}\Vert_q
(\mu^2/2+s^2)^{-N+(p/2q+\eps)}.
\label{eq:q-trace-est}
\end{align}




\section{Index pairings for  semifinite spectral triples}
\label{sec:Chern}
\vspace{-4pt}
In this section we define the notion of a
{\it smoothly summable semifinite spectral triple} 
$(\A,\H,\D)$ relative to  a semifinite von Neumann algebra
with faithful normal semifinite trace $(\cn,\tau)$, and 
show that such a spectral triple produces, via Kasparov theory, a well-defined numerical index pairing
with $K_*(\A)$, the $K$-theory of $\A$.

The `standard case' of spectral triples with $(\cn,\tau)=(\B(\HH),\mbox{Tr})$ for 
some separable Hilbert space $\HH$, is presented in \cite{Co1}. In this case
there is an associated
Fredholm module, and hence $K$-homology class. Then there is a 
pairing between $K$-theory and $K$-homology, integer valued in this case,
that is well-defined and explained in detail in \cite[Chapter 8]{HR}. The discussion in \cite{HR} 
applies to both the unital and nonunital situations.
The extension of \cite[Chapter 8]{HR}
to deal with both the semifinite situation and nonunitality
require some refinements that are not difficult, 
but are worth making explicit to the reader for the purpose of
explaining the  basis of our approach.

Recall also that when the spectral triple is semifinite {\em and} has $(1+\D^2)^{-s/2}\in\L^1(\cn,\tau)$
for all $s>p\geq 1$, for some $p$, then there is an analytic formula for the index
pairing, given in terms of the $\R$-valued index of 
suitable $\tau$-Fredholm operators, \cite{BeF, CP1,CP2,CPRS3}.

However, for a  semifinite spectral triple with $(1+\D^2)^{-1/2}$ not 
$\tau$-compact, we need a different approach, and so
we follow the route indicated in \cite{KNR}. There it is shown that we can associate a
Kasparov module, and so a $KK$-class, to a semifinite spectral triple. This gives us a
well-defined pairing with $K_*(\A)$ via the Kasparov product,  with
and modulo some technicalities, this pairing takes values in $K_0(\K_\cn)$,
the $K$-theory of the $\tau$-compact operators $\K_\cn$ in $\cn$.
Composing this pairing with the map on $K_0(\K_\cn)$ induced by the trace 
$\tau$ gives us a
numerical index  which computes the usual index when the triple is `unital'. 
When we specialise to particular representatives of our Kasparov class, we will see 
that we  are
also computing the $\R$-valued  indices of suitable $\tau$-Fredholm operators.

\subsection{Basic definitions for spectral triples}
\label{subsec:spec-trips}
\vspace{-4pt}
In this subsection, we give the minimal definition for a semifinite spectral triple, in order to have 
a Kasparov (and also Fredholm) module.
Recall that we denote by  $\K(\cn,\tau)$, or $\K_\cn$ when $\tau$ is understood,  
the ideal of $\tau$-compact operators in $\cn$. This is the norm closed ideal in $\cn$ 
generated by projections with finite $\tau$-trace.\index{$\K(\cn,\tau)$, $\K_\cn$, the 
$\tau$-compact operators in $\cn$}

\begin{definition} 
\label{def:ST}
A  semifinite
spectral triple $(\A,\HH,\D)$, relative to $(\cn,\tau)$, is given by a Hilbert space $\HH$, a
$*$-subalgebra  $\A\subset\, \cn$\index{$\A$, a $*$-algebra}
acting on
$\HH$, and a densely defined unbounded self-adjoint operator $\D$ affiliated
to $\cn$ such that:

1. $a\cdot{\rm dom}\,\D\subset {\rm dom}\,\D$ for all $a\in\A$, so that
$da:=[\D,a]$ is densely defined.  Moreover, $da$ extends to a bounded operator in
$\cn$ for all $a\in\A$;

2. $a(1+\D^2)^{-1/2}\in\K(\cn,\tau)$ for all $a\in\A$.

We say that $(\A,\HH,\D)$ is even if in addition there is a $\mathbb{Z}_2$-grading
such that $\A$ is even and $\D$ is odd. This means there is an operator $\gamma$ such that
$\gamma=\gamma^*$, $\gamma^2={\rm Id}_\cn  
$, $\gamma a=a\gamma$ for all $a\in\A$ and
$\D\gamma+\gamma\D=0$. Otherwise we say that $(\A,\HH,\D)$ is odd.
\end{definition}
\index{semifinite spectral triple}\index{$\mathbb{Z}_2$-grading}

{\bf Remark.} 1) We will write $\gamma$ in all our formulae, with the
understanding
that, if $(\A,\HH,\D)$ is odd, $\gamma= {\rm Id}_\cn  $ and of course, 
we drop the assumption that
$\D\gamma+\gamma\D=0$.\\
2) By density, we immediately see that the second condition in the definition
of a semifinite spectral triple, also holds for all  elements in  the $C^*$-completion 
of $\A$.\\
3) The condition $a(1+\D^2)^{-1/2}\in\K(\cn,\tau)$ is equivalent to $a(i+\D)^{-1}\in\K(\cn,\tau)$.
This follows since $(1+\D^2)^{1/2}(i+\D)^{-1}$ is unitary.

Our first task is to justify the terminology `nonunital' for the situation where $\D$ does not have 
$\tau$-compact resolvent. What we show is that if $\A$ is unital, then we obtain a spectral triple
on the Hilbert space $\overline{1_\A\H}$ for which $1_\A\D\,1_\A$ has compact resolvent. On the other hand, 
one can have a spectral triple with nonunital algebra whose `Dirac' operator has compact resolvent, as in
\cite{GL,GW,W}.

\begin{lemma} 
\label{lem:nonunital-spec-trip}
Let $(\A,\HH,\D)$ be a  semifinite spectral triple relative to $(\cn,\tau)$,
and suppose that $\A$ possesses  a unit $P\neq {\rm Id}_\cn$. Then  
$(P+(P\D P)^2)^{-1/2}\in \K(P\cn P,\tau|_{P\cn P})$. Hence, $(\A,P\HH,P\D P)$ is a unital spectral triple 
relative to $(P\cn P,\tau|_{P\cn P})$.
\end{lemma}

\begin{proof}
It is a short exercise to show that $\tau|_{P\cn P}$ is a faithful normal semifinite trace on $P\cn P$.

We just need to show that $(Pi+P\D P)^{-1}$ is compact in $P\cn P$. 
To do this we show that we can approximate $(Pi+P\D P)^{-1}$ by $P(i+\D)^{-1}P$
up to compacts. This follows from
$$
(Pi+P\D P)P(i+\D)^{-1}P=P(i+\D)P(i+\D)^{-1}P=P[\D,P](i+\D)^{-1}P+P,
$$
 the compactness of  $(i+\D)^{-1}P$ and the boundness of $P[\D,P]$
 and of $(Pi+P\D P)^{-1}$. 
 \end{proof}

Thus, we may without loss of generality assume that a  spectral triple $(\A,\HH,\D )$ whose operator
 $\D$ does not have compact resolvent, must have  a nonunital algebra
$\A$. Adapting this proof shows that similar results hold
for  spectral triples with additional hypotheses such as summability or smoothness, introduced
below.

\subsection{The Kasparov class and Fredholm module of a  spectral triple}
\label{subsec:Kas-mod}

In this subsection, we use Kasparov modules for trivially graded $C^*$-algebras, 
\cite{KasparovTech}. {\em Nonunital algebras are assumed to be separable},
with the exception of $\K(\cn,\tau)$ which typically is not separable nor even $\sigma$-unital.
By separable, we always mean separable for the norm topology and not necessarily 
for other topologies like the $\delta$-$\vf$-topology introduced in Definition \ref{delta-phi}.
Information about $C^*$-modules and their endomorphisms can be found in \cite{RW}.
Given a $C^*$-algebra $B$ and a 
right $B$-$C^*$-module $X$, we let ${\rm End}_B(X)$ denote the $C^*$-algebra
of $B$-linear adjointable endomorphisms of $X$, and let 
${\rm End}^0_B(X)$ be the ideal of $B$-compact 
adjointable endomorphisms.

We briefly recall the definition of Kasparov modules, and the equivalence relation on 
them used to construct the $KK$-groups. 

\begin{definition}\label{Kasparov} 
Let $A$ and $B$ be $C^*$-algebras, with $A$ separable.
An  odd Kasparov $A$-$B$-module consists of a
countably generated ungraded right $B$-$C^*$-module $X$, with 
$\pi:A\to {\rm End}_B(X)$ a $*$-homomorphism, together with $F\in {\rm End}_B(X)$  such that
$\pi(a)(F-F^*),\ \pi(a)(F^2-1),\ [F,\pi(a)]$ are  compact adjointable
endomorphisms of $X$, for each $a\in A$. 

An  even Kasparov $A$-$B$-module is an odd Kasparov $A$-$B$-module,
together with a grading by a self-adjoint adjointable
endomorphism
$\gamma$ with $\gamma^2=1$ and
$\pi(a)\gamma=\gamma\pi(a)$, $F\gamma+\gamma F=0$.
\end{definition}
\index{Kasparov module}

We will use the notation $({}_AX_B,F)$ or $({}_AX_B,F,\gamma)$ for Kasparov modules,
generally omitting the representation $\pi$. A Kasparov module $({}_AX_B,F)$ 
with $\pi(a)(F-F^*)=\pi(a)(F^2-1)=[F,\pi(a)]=0$, for all $a\in A$, is called
{\em degenerate}.

We now describe the equivalence relation on 
Kasparov $A$-$B$-modules which defines classes in the 
abelian group $KK(A,B)=KK^0(A,B)$ (even case) or $KK^1(A,B)$ (odd case).
The relation consists of three separate equivalence relations: unitary equivalence, stable equivalence
and operator homotopy. More details can be found in \cite{KasparovTech}.

Two Kasparov $A$-$B$-modules $({}_A(X_1)_B,F_1)$ and $({}_A(X_2)_B,F_2)$
are {\em unitarily equivalent} if there is an adjointable unitary $B$-module map $U:X_1\to X_2$ such 
that $\pi_2(a)=  U\pi_1(a)U^*$, $\mbox{for all } a\in A$ and $F_2=U\,F_1\,U^*$.

Two Kasparov $A$-$B$-modules $({}_A(X_1)_B,F_1)$ and $({}_A(X_2)_B,F_2)$
are {\em stably equivalent} if there is a degenerate 
Kasparov $A$-$B$-module $({}_A(X_3)_B,F_3)$ with 
$({}_A(X_1)_B,F_1)=({}_A(X_2\oplus X_3)_B,F_2\oplus F_3)$ and $\pi_1=\pi_2\oplus\pi_3$.

Two Kasparov $A$-$B$-modules $({}_A(X)_B,G)$ and $({}_A(X)_B,H)$
(with the same representation $\pi$ of $A$) are called {\em operator homotopic} if there is
a norm continuous family $(F_t)_{t\in[0,1]}\subset {\rm End}_B(X)$ such that for each $t\in[0,1]$
$({}_A(X_1)_B,F_t)$ is a Kasparov module and $F_0=G$, $F_1=H$.

Two Kasparov modules $({}_A(X)_B,G)$ and $({}_A(X)_B,G)$ are equivalent if after the addition 
of degenerate modules, they are operator homotopic to unitarily equivalent Kasparov modules. 
The equivalence classes of even (resp. odd) Kasparov $A$-$B$ modules
form an abelian group denoted $KK^0(A,B)$ (resp. $KK^1(A,B)$). The zero element 
is represented by any degenerate
Kasparov module, and the inverse of a class $[({}_A(X)_B,F)]$ is the class of  $({}_A(X)_B,-F)$, with
grading $-\gamma$ in the even case.

This equivalence relation, in conjunction with the Kasparov product, 
implies further equivalences between Kasparov modules, such as Morita equivalence.
This is discussed in \cite{Bl,KasparovTech}, where more information on the Kasparov product can
also be found.
With these definitions in hand, we can state our first result linking   semifinite spectral triples
and Kasparov theory.

\begin{lemma}[see \cite{KNR}]
\label{lem:spec-trip-to-Kas-mod}
Let $(\A,\HH,\D)$  be a  semifinite  spectral triple relative to $(\cn,\tau)$  with $\A$ separable.
For $\eps>0$ (resp $\eps\geq0$ when $\D$ is invertible), set
$F_\eps:=\D(\eps+\D^2)^{-1/2}$ and let $A$ be the $C^*$-completion of $\A$. 
Then, $[F_\eps,a]\in\K_\cn$ for
all $a\in A$.
In particular,  provided that $\K_\cn$ is $\sigma$-unital, and letting  $X:=\K_\cn$ as a right $\K_\cn$-$C^*$-module,
 the data $({}_AX_{\K_\cn},F_\eps)$ defines a Kasparov module with class
$[({}_AX_{\K_\cn},F_\eps)]\in KK^\bullet(A,\K_\cn)$,  where
$\bullet=0$ if the spectral triple $(\A,\HH,\D)$ is $\mathbb{Z}_2$-graded and 
$\bullet=1$ otherwise. The class
$[({}_AX_{\K_\cn},F_\eps)]$ is independent of $\eps>0$ 
(or even $\eps\geq 0$ if $\D$ is invertible).
\end{lemma}

\begin{proof}
Regarding $X=\K_\cn$ as a right $\K_\cn$-$C^*$-module via $(T_1|T_2):=T_1^*T_2$,
we see immediately
that left multiplication by $F_\eps$ on $\K_\cn$ gives $F_\eps\in {\rm End}_{\K_\cn}(\K_\cn)$, the
adjointable endomorphisms, see \cite{RW}, and
left multiplication by $a\in A$, the $C^*$-completion of $\A$, gives a representation of
$A$ as adjointable endomorphisms of $X$ also.

Since the algebra of compact endomorphisms of $X$ is just $\K_\cn$, and we have assumed $\K_\cn$
is $\sigma$-unital, we see that $X$ is countably generated, by \cite[Proposition 5.50]{RW}.

That $F_\eps^*=F_\eps$  as an 
endomorphism follows from the functional calculus. 
Now let $a,b\in\mathcal{A}$. The integral formula for fractional powers gives
$$
(\eps+\mathcal{D}^2)^{-1/2}
={\pi}^{-1}\int_0^{\infty}\lambda^{-1/2}(\eps+\lambda+\mathcal{D}^2)^{-1}d\lambda,
$$
and with a nod to \cite[Lemma 3.3]{CP1} we obtain 
\begin{align*}
\mathcal{D}\big[(\eps+\mathcal{D}^2)^{-1/2},a\big]b
&={\pi}^{-1}\int_0^\infty \lambda^{-1/2}
\Big(\mathcal{D}^2(\eps+\lambda+\mathcal{D}^2)^{-1}[\mathcal{D},a](\eps+\lambda+\mathcal{D}^2)^{-1}b\\
&\qquad\qquad\qquad\qquad
+\mathcal{D}(\eps+\lambda+\mathcal{D}^2)^{-1}[\mathcal{D},a]\mathcal{D}
(\eps+\lambda+\mathcal{D}^2)^{-1}b\Big)d\lambda.
\end{align*}
By the definition of a spectral triple, the integrand is $\tau$-compact, and so is in the compact
endomorphisms of our module. The functional calculus yields the norm estimates
$$
\|\mathcal{D}^2(\eps+\lambda+\mathcal{D}^2)^{-1}[\mathcal{D},a](\eps+\lambda+\mathcal{D}^2)^{-1}b\|
\leq
\|[\D,a]\|\|b\|(\eps+\lambda)^{-1},
$$
and
$$
\|\mathcal{D}(\eps+\lambda+\mathcal{D}^2)^{-1}[\mathcal{D},a]\mathcal{D}
(\eps+\lambda+\mathcal{D}^2)^{-1}b\|\leq\|[\D,a]\|\|b\|(\eps+\lambda)^{-1}.
$$
Therefore, the integral above is norm-convergent. 
Thus, $\mathcal{D}[(\eps+\mathcal{D}^2)^{-1/2},a]b$ is $\tau$-compact and
$$
[F_\eps,a]b
=\mathcal{D}[(\eps+\mathcal{D}^2)^{-1/2},a]b+[\mathcal{D},a](\eps+\mathcal{D}^2)^{-1/2}b,
$$
is $\tau$-compact too. Similarly, $a[F_\eps,b]$ is $\tau$-compact. Finally,
$[F_\eps,ab]=a[F_\eps,b]+[F_\eps,a]b$
is $\tau$-compact, and so a compact endomorphism.
Taking norm limits now shows that $[F_\eps,ab]$ is $\tau$-compact 
for all $a,b\in A.$ By the norm density of products in $A$, one concludes that 
$[F_\eps,a]$ is compact for all $a\in A$. 
Finally for $a\in A$ we have
$
a(1-F_\eps^2)=a\eps(\eps+\D^2)^{-1},
$
and this is $\tau$-compact since $(\A,\H,\D)$ is a  spectral triple. Thus $({}_AX_{\K_\cn},F_\eps)$
is a Kasparov module.

To show that the associated $KK$-class is independent of $\eps$, it suffices to show that
$\eps\mapsto F_\eps$ is continuous in operator norm, \cite{KasparovTech}. This follows from
the integral formula for fractional powers which shows that
$$
F_{\eps_1}-F_{\eps_2}=\frac{\eps_2-\eps_1}{\pi}\int_0^\infty \lambda^{-1/2}
\D(\eps_1+\lambda+\D^2)^{-1}(\eps_2+\lambda+\D^2)^{-1}\,d\lambda,
$$
since the integral converges in norm independent of $\eps_1,\,\eps_2>0$. If $\D$ is invertible we
can also take $\eps_i=0$.
This completes the proof.
\end{proof}

The assumption that $\K_\cn$ is $\sigma$-unital is never satisfied in the type II setting, and so we do not obtain a countably
generated $C^*$-module. 
In order to go beyond this assumption, we adopt the method of \cite{KNR}. 

\begin{definition} 
\label{defn:not-kn}
Given $(\A,\H,\D)$ relative to $(\cn,\tau)$,
we let ${\mathcal C}\subset \K_\cn$ be the algebra generated by the operators
$$
F_\eps[F_\eps,a],\ \  b[F_\eps,a],\ \ [F_\eps,a],\ \ 
F_\eps b[F_\eps,a],\ \  a\varphi(\D),\quad a,\,b\in\A, \ \ \varphi\in C_0(\R).
$$
\end{definition}
If $\A$ is separable, so too is ${\mathcal C}$. This allows us to repeat the construction of 
Lemma \ref{lem:spec-trip-to-Kas-mod} using ${\mathcal C}$ instead of $\K_\cn$. The result is a Kasparov module
$({}_AX_{C},F_\eps)$ with class in $KK^\bullet(A,C)$, where $C$ is the norm closure of ${\mathcal C}$.

\begin{corollary}
\label{cor:spec-trip-to-Kas-mod}
Let $(\A,\HH,\D)$  be a  semifinite  spectral triple relative to $(\cn,\tau)$
 with $\A$ separable.
For $\eps>0$ (resp $\eps\geq0$ when $\D$ is invertible), set
$F_\eps:=\D(\eps+\D^2)^{-1/2}$ and let $A$ be the $C^*$-completion of $\A$. 
Then, $[F_\eps,a]\in C\subset\K_\cn$ for
all $a\in A$.
In particular, letting  $X:=C$ as a right $C$-$C^*$-module,
 the data $({}_AX_{C},F_\eps)$ defines a Kasparov module with class
$[({}_AX_{C},F_\eps)]\in KK^\bullet(A,C)$,  where
$\bullet=0$ if the spectral triple $(\A,\HH,\D)$ is $\mathbb{Z}_2$-graded and 
$\bullet=1$ otherwise. The class
$[({}_AX_{C},F_\eps)]$ is independent of $\eps>0$ 
(or even $\eps\geq 0$ if $\D$ is invertible).
\end{corollary}

Using the Kasparov product we now have a well-defined map
\begin{equation}
\label{K-index-paring}
\cdot\otimes_A[(\K_\cn,F_\eps)]\,:\,K_\bullet(A)=KK^\bullet(\C,A)\to \ K_0(C).
\end{equation}
For this pairing to make sense it is required that $A$ be separable, \cite[Theorem 18.4.4]{Bl},
and we remind the reader that we always suppose this to be the case. We refer to the map 
given in Equation \eqref{K-index-paring} as the {\em $K$-theoretical index pairing}.
\index{index pairing}
\index{index pairing!$K$-theoretical }

Let $\F_\cn$ denote  the ideal of `finite rank' operators in $\K_\cn$; that is, $\F_\cn$ is the ideal of $\cn$ generated
by projections of finite trace, without taking the norm completion.
In \cite[Section 6]{KNR}, it shown that for all $n\geq 1$, $M_n(\F_\cn)$ is stable under the holomorphic functional
calculus inside $M_n(\K_\cn)$, and so $K_0(\F_\cn)\cong K_0(\K_\cn)$. 

One may now deduce that $M_n(C\cap \F_\cn)$ is stable under the holomorphic functional calculus inside 
$M_n(C\cap\K_\cn)=M_n(C)$. Thus every class in $K_0(C)$ may be represented as $[e]-[f]$ where $e,\,f$ are
projections in a matrix algebra over the unitisation of $C\cap \F_\cn$. As in \cite{KNR}, the map $\tau_*:K_0(C)\to\R$
is then well-defined.

\begin{definition}
Let $\A$ be a $*$-algebra (continuously) represented in $\cn$, 
 a semifinite von Neumann
algebra  with faithful semifinite normal trace $\tau$.
A {semifinite  pre-Fredholm module} for $\A$ relative to $(\cn,\tau)$,  
is a pair $(\HH, F)$, where $ \H$ is
a separable Hilbert space carrying a faithful representation of 
 $\mathcal N$ 
and $F$ is an 
operator in $\mathcal N$ satisfying:

$1.\: a(1-F^2),\ a(F-F^*) \in {\mathcal {K_N}},\:and$

$2.\: [F,a] \in {\mathcal {K_N}}\: for\: a \in { \A}.$

\noindent If $1-F^2=0=F-F^*$ we drop the prefix ``pre-".
 If our (pre-)Fredholm module satisfies
 $[F,a]\in {\mathcal L}^{p+1}(\cn,\tau)$ and $a(1-F^2)\in\L^{(p+1)/2}(\cn,\tau)$
for all $a\in \mathcal A$,
we say that $({\HH, F})$ is $(p+1)$-summable.

We say that $(\HH,F)$ is even if in addition there is a $\mathbb{Z}_2$-grading
such that $\A$ is even and $F$ is odd. This means there is an operator $\gamma$ such that
$\gamma=\gamma^*$, $\gamma^2={\rm Id}_\cn  
$, $\gamma a=a\gamma$ for all $a\in\A$ and
$F\gamma+\gamma F=0$. Otherwise we say that $(\HH,F)$ is odd.
\end{definition}
\index{semifinite Fredholm module}\index{semifinite Fredholm module!pre-Fredholm module}\index{semifinite Fredholm module!finitely summable}
A semifinite pre-Fredholm module for a $*$-algebra $\A$ extends to a
semifinite pre-Fredholm module for the norm completion of $\A$ in $\cn$, by essentially the
same proof as Lemma \ref{lem:spec-trip-to-Kas-mod}. For completeness we state
this as a  lemma.

\begin{lemma}\label{z simplification lemma1} 
Let $(\mathcal{A},\mathcal{H},\mathcal{D})$ be a 
semifinite spectral triple relative to $(\mathcal{N},\tau)$. 
Let $A$ be the  $C^*$-completion of $\mathcal{A}$. 
If $F_\eps=\mathcal{D}(\eps+\mathcal{D}^2)^{-1/2}$, $\eps>0$, 
then the operators $[F_\eps,a]$ and $a(1-F_\eps^2)$ are $\tau$-compact for every $a\in A$.
Hence $(\H,F_\eps)$ is a pre-Fredholm module for $A$.
\end{lemma}

\subsection{The numerical index pairing}
\label{num-ind-pair}

We will now  make particular Kasparov products explicit by choosing 
specific representatives of the classes.
We will focus on the 
condition $F^2=1$ for Kasparov modules.
Imposing  this condition  simplifies the description
of the Kasparov product with $K$-theory. In the context of Lemma \ref{lem:spec-trip-to-Kas-mod},
this will be the case if and only if $\eps=0$, that is, if and only if $\D$ is invertible. 
We will shortly show how to
modify the pair $(\H,\D)$ in the data given by a  semifinite spectral triple 
$(\A,\H,\D)$, in order that $\D$ is always 
invertible. Before doing that, we need some more Kasparov theory for nonunital $C^*$-algebras.

Suppose that we have two $C^*$-algebras $A,\,B$ and a 
graded Kasparov module $(X= _A\!\!X_B,F,\gamma)$. Assume also that $A$ is nonunital. 
Let $e$ and $f$ be projections in a (matrix algebra over a)
unitization of $A$, which we can take to be the minimal 
unitization $A^\sim=A\oplus\C$ (see \cite{RW}), by excision in $K$-theory, and 
suppose also that we have a class
$[e]-[f]\in K_0(A)$. That is, $[e]-[f]\in \ker(\pi_*:K_0(A^\sim)\to K_0(\C))$ where $\pi:A^\sim\to\C$
is the quotient map.
Then 
the Kasparov product over $A$ of $[e]-[f]$
with $[(X,F,\gamma)]$ gives us a class in $K_0(B)$. We now show that if $F^2={\rm Id}_X$,  
we  can represent this
Kasparov product as a difference of projections over $B$ (in the unital case) or $B^\sim$ (in
the nonunital case).

Here and in the following, we always represent elements $a+\lambda\, {\rm Id}_{A^\sim}\in A^\sim$ on 
$X$ as $a+\lambda\, {\rm Id}_X$, $\lambda\in\C$. 
Set $X_\pm:=\frac{1\pm\gamma}{2}X$ and, ignoring the  matrices to simplify the discussion, let $e\in A^\sim$. 
To show that
$eF_\pm e: eX_\pm\to eX_\mp$ 
is Fredholm  (which in this context means invertible modulo ${\rm End}_B^0(X_\pm,X_\mp)$),
we must display a parametrix. 
Taking $eF_\mp e$ yields
$$
eF_\mp eF_\pm e=eF_\mp[e,F_\pm]e+e(F_\mp F_\pm-{\rm Id}_{X_\pm})e+{\rm Id}_{eX_\pm}.
$$
We are left with showing that $e(F_\mp F_\pm-{\rm Id}_{X_+})e$ and 
$eF_\mp[e,F_\pm]e$ are   ($B$-linear) compact endomorphisms of the $C^*$-module $X_\pm$.
The compactness of $eF_\mp[e,F_\pm]e$ follows since $e$ is represented as
$a+\lambda{\rm  Id}_X$ for some $a\in A$ and $\lambda\in\C$, and thus  $[e,F_\pm]=[a,F_\pm]$ which is compact
by definition of a Kasparov module. 

However $e(F_\mp F_\pm-{\rm Id}_{X_\pm})e$ is generally
not compact, because we are only guaranteed that 
$a(F_\mp F_\pm-{\rm Id}_{X_\pm})$ is compact for
$a\in A$, not $a\in A^\sim$! Nevertheless, if the Kasparov 
module is normalized, i.e.  if $F^2={\rm Id}_X$, we have $F_\mp F_\pm-{\rm Id}_{X_\pm}=0$,
and so we have a parametrix, showing that
$eF_\pm e$ is Fredholm. In this case, the Kasparov product
$([e]-[f])\otimes_A[(X,F)]$ is given by  
$$
\big[\Index(eF_\pm e)\big]-\big[\Index(fF_\pm f)\big]\in K_0(B).
$$
Here the index is defined as the difference 
$[\ker \widetilde{eF_\pm e}]-[\mbox{coker}\,\widetilde{eF_\pm e}]$,
where $ \widetilde{eF_\pm e}$ is any regular amplification of of $eF_\pm e$, 
see \cite[Lemma 4.10]{GVF}. This index is
independent of the amplification chosen, the kernel and cokernel projections
can be chosen finite rank over $B$, or $B^\sim$ if $B$ is nonunital, 
and the index lies in $K_0(B)$ by \cite[Proposition 4.11]{GVF}.

Similarly, in the odd case we would like to have (see \cite[Appendix]{KNR} and \cite[Appendix]{PRe}),
$$
[u]\otimes_A\big[(X,F)\big]=\big[\Index\big(\tfrac{1}{4}(1+F)u(1+F) - \tfrac{1}{2}(1-F)\big)\big]
\in K_0(B),
$$
where $[u]\in K_1(A)$.
As in the even case above, to show that the operator
$\frac{1}{4}(1+F)u(1+F) - \frac{1}{2}(1-F)$ is Fredholm in the nonunital case, it is easier to
assume that  $F^2=1$,
and in this case, writing $(1+F)/2=P$ for the positive
spectral projection of $F$, we have
$$
[u]\otimes_A\big[(X,F)\big]=\big[\Index(PuP)\big]=[\ker \widetilde{PuP}]-[\mbox{coker}\,\widetilde{PuP}]
\in K_0(B),
$$
there being no contribution to the index from $P^\perp=(1-F)/2$.
As in the even case above, $\widetilde{PuP}$ is a regular amplification of $PuP$, and
the projections onto $\ker \widetilde{PuP}$ and ${\rm coker} \,\widetilde{PuP}$ 
are finite rank over $B$ or $B^\sim$. We show in subsection \ref{subsec:odd-index} 
an alternative method to avoid 
the simplifying assumption $F^2=1$ in the odd case.

Given a pre-Fredholm module $(\H,F)$ relative to $(\cn,\tau)$  for a separable 
$*$-algebra $\A$, we 
obtain a Kasparov module
$({}_AC_C,F)$, just as we did for a  spectral
triple in Corollary \ref{cor:spec-trip-to-Kas-mod}. Here  $A$ is the norm completion of $\A$ and 
$C\subset \K_\cn$ is given by the norm closure of the
algebra defined in Definition \ref{defn:not-kn}, using the operator $F$ for the commutators, and polynomials in 
$1-F^2$ in place of $\varphi(\D)$, $\varphi\in C_0(\R)$. Also,
given $(\A,\H,\D)$ relative to $(\cn,\tau)$, the following diagram commutes
$$
\xymatrix{(\A,\H,\D)  \ar[r] \ar[d] & ({}_AC_C,F_\eps)\\
(\H,F_\eps)  \ar[ur] & }\ .
$$
Thus we have a single well-defined Kasparov class arising from either the spectral triple
or the associated pre-Fredholm module. Now we show how to obtain a representative of
this class with $F^2=1$, so simplifying the index pairing.
This reduces to showing that if our spectral triple $(\A,\HH,\D)$
is such that $\D$ is not invertible, we can replace it by a new spectral
triple  for which the unbounded operator is invertible and has the
same $KK$-class. We learned this trick from \cite[page 68]{Co1}.

\begin{definition} 
\label{def:double}
Let $(\A,\HH,\D)$ be a  semifinite spectral triple relative to $(\cn,\tau)$. For any
$\mu>0$,  define the `double' of
$(\A,\HH,\D)$ to be the  semifinite
spectral triple $(\A,\HH^2,\D_\mu)$ relative to $(M_2(\cn),\tau\otimes\tr_2)$, 
with $\HH^2:=\HH\oplus\HH$ and
the action of $\A$ and $\D_\mu$ given by
\ben
\D_\mu:=\bma \D & \mu\\ \mu & -\D\ema,\ \ \ \ a\mapsto \hat a:=\bma a & 0\\ 0 & 0\ema,
\ \ {\rm for\ all}\ a\in\A.
\een
If $(\A,\H,\D)$ is even and graded by $\gamma$ then the double is even
and graded by $\hat\gamma:=\gamma\oplus-\gamma$.
\end{definition}
\index{double construction}

{\bf Remark.} Whether $\D$ is invertible or not, $\D_\mu$ always is invertible,
and $F_\mu=\D_\mu|\D_\mu|^{-1}$
has square 1. This is the chief reason for introducing this construction.

We also need to extend the action of $M_n(\A^\sim)$ on $(\H\oplus\H)\otimes\C^n$,  
in a compatible way with the extended action of $\A$ on $\H\oplus\H$. So, for a generic 
element $b\in M_n(\A^\sim)$,  we let
\begin{equation}
\label{ext-A-sim}
\hat b:= \bma b & 0\\ 0 & {\bf 1}_b\ema\in M_{2n}(\cn),
\end{equation}
with ${\bf 1}_b:=\pi^n(b)\otimes{\rm Id}_{ \cn}$, where $\pi^n:M_n(\A^\sim)\to M_n(\C)$ is the quotient map.

It is known (see for instance \cite[Proposition 12, p. 443]{Co1}), that up to an addition of a 
degenerate module, any Kasparov module is operator homotopic to a normalised 
Kasparov module, i.e. one with $F^2=1$. The following
makes it explicit.

\begin{lemma}
\label{lem:noninv}
When $\A$ is separable,
the $KK$-classes associated with $(\A,\HH,\D)$ and
$(\A,\HH^2,\D_\mu)$ coincide. A representative of this class is
$({}_A(C\oplus C)_C,F_\mu)$
with $F_\mu=\D_\mu|\D_\mu|^{-1}$  and $C$ the norm closure
of the $*$-subalgebra of $\K(\cn,\tau)$ given in Definition \ref{defn:not-kn}.
\end{lemma}

\begin{proof} The $KK$-class of $(\A,\HH,\D)$ is represented 
(via Corollary \ref{cor:spec-trip-to-Kas-mod}) by
$({}_AC_C,F_\eps)$ with $F_\eps=\D(\eps+\D^2)^{-1/2}$, $\eps>0$,  while the
class of $(\A,\HH^2,\D_\mu)$ is
represented
by the
Kasparov module $({}_AM_2(C)_{M_2(C)},F_{\mu,\eps})$ 
with operator defined
by $F_{\mu,\eps}=\D_\mu(\eps+\D_\mu^2)^{-1/2}$. By Morita equivalence,
this module has the same class as the module $({}_A(C\oplus C)_{C},F_{\mu,\eps})$, since
${}_{M_2(C)}(C\oplus C)_C$ is a Morita equivalence bimodule.
The one-parameter
family $({}_A(C\oplus C)_C,F_{m,\eps})_{0\leq m\leq \mu}$ is a continuous operator homotopy,
\cite{KasparovTech},  from
$({}_A(C\oplus C)_C,F_{\mu,\eps})$ to the
direct sum of two Kasparov modules
\ben
({}_AC_C,F_{\eps})\oplus({}_AC_C,-F_\eps).
\een
In the odd case the second Kasparov module is operator homotopic to
$({}_AC_C,{\rm Id}_\cn)$ by the
straight line path
since
$\A$ is represented by zero on this module. In the even case
we
find the second Kasparov module is homotopic to
\ben
\left({}_AC_C,\bma 0 & 1\\ 1 & 0\ema\right),
\een
the matrix decomposition being with respect to the $\mathbb{Z}_2$-grading of $\HH$
which provides a $\mathbb{Z}_2$-grading of $C\subset \K_\cn$.
Thus in both the even and odd cases the second module is degenerate, i.e.
$F^2=1$, $F=F^*$ and $[F,a]=0$ for all $a\in\A$, and so
the $KK$-class of
$({}_A(C\oplus C)_C,F_{\mu,\eps})$, written $[({}_A(C\oplus C)_C,F_{\mu,\eps})]$,
is the $KK$-class of
$({}_AC_C,F_{\eps})$.
In addition, the Kasparov module $({}_A(C\oplus C)_C,F_\mu)$ with $F_\mu=\D_\mu|\D_\mu|^{-1}$ is
operator
homotopic to $({}_A(C\oplus C)_C,F_{\mu,\eps})$ via
\ben t\mapsto \D_\mu(t\eps +\D_\mu^2)^{-1/2},\ \ \ 0\leq t\leq 1.\een
This provides the desired representative.
\end{proof}

The next result records what is effectively a tautology, given our definitions. Namely
we define the $K_0(C)$-valued  index pairing  of $(\A,\H,\D)$ with $K_*(\A)$
in terms of the associated Kasparov module. Similarly, the associated pre-Fredholm module
has  an index pairing defined in terms of the associated Kasparov module. 

\begin{corollary} 
\label{cor:double-index}
Let $(\A,\HH,\D)$ be a   spectral triple relative to $(\cn,\tau)$ with $\A$ separable.
Let
$
(\A, \HH^2, \D_\mu)$ relative to
$(M_2(\cn),\tau\otimes \tr_2)
$ be the double and
$({}_A(C\oplus C)_C,F_\mu)$ the associated
Fredholm module. Then the $K_0(C)$-valued
index pairings defined by the two spectral triples
and the
semifinite Fredholm module all agree: for $x\in K_*(\A)$ of the appropriate parity
and $\mu>0$
\begin{align*}
x\otimes_A[(\A,\HH,\D)]&=x\otimes_A[({}_AC_C,F_\eps)]=
x\otimes_A\left[\left(\A, \HH^2,
\D_\mu\right)\right]=x\otimes_A[({}_A(C\oplus C)_C,F_\mu)]\in K_0(C).
\end{align*}
\end{corollary}

As noted after Corollary \ref{cor:spec-trip-to-Kas-mod},
the trace $\tau$ induces a homomorphism $\tau_*:K_0(C)\to\R$.

An important feature of the double construction 
is that  it allows us to make pairings in the nonunital
case explicit. To be precise, if $e\in M_n(\A^\sim)$ 
is a projection and $\pi^n:M_n(\A^\sim)\to M_n(\C)$
is the quotient map (by $M_n(\A)$), we set as in \eqref{ext-A-sim}
\begin{equation}
{\bf 1}_e:=\pi^n(e)\in M_n(\C).
\label{eq:wots-1-e} 
\end{equation}
Then in the double $e$
is represented on $\HH\otimes\C^n\oplus \HH\otimes \C^n$ (this is the spectral triple picture, but
similar comments hold for Kasparov modules) via
$$
e\mapsto\hat e:=\begin{pmatrix} e & 0\\ 0 & {\bf 1}_e \end{pmatrix}.
$$
Thus $\hat e(\D_\mu\otimes {\rm Id}_n)\hat e$ is $\tau\otimes{\rm tr}_{2n}$-Fredholm 
in $M_{2n}(\cn)$, with the understanding that 
the matrix units $e_{ij}\in M_{2n}(\C)$ sit in $M_{2n}(\cn)$ as $e_{ij}\,{\rm Id}_{\cn}$.

{\bf Example.} Let $p_B\in M_2(C_0(\C)^\sim)$ be the Bott projector, given explicitly by \cite[pp 76-77]{GVF}
\begin{equation}
\label{eq:bott}
p_B(z)=\frac{1}{1+|z|^2}
\begin{pmatrix} 1 & \bar{z}\\ z & |z|^2\end{pmatrix},\quad\mbox{then}\quad {\bf 1}_{p_B}
=\begin{pmatrix} 0 & 0 \\ 0 & 1\end{pmatrix}.
\end{equation}
\index{Bott projector}

We are now ready to define  the numerical index paring for  semifinite
spectral triples.

\begin{definition}
\label{index-paring}
Let $(\A,\H,\D)$ be a  semifinite spectral triple
relative to $(\cn,\tau)$ of parity $\bullet\in \{0,1\}$, $\bullet=0$ for an even triple,
$\bullet=1$ for an odd triple  and with $\A$ separable. We define the numerical index pairing of  $(\A,\HH,\D)$
with $K_\bullet(\A)$ as follows:

1. Take the Kasparov product with the $KK$-class defined by the doubled up  spectral triple
$$
\cdot\otimes_A[({}_A(C\oplus C)_C,F_\mu)]\,:\,K_\bullet(A)\to K_0(C),
$$

2. Apply the homomorphism $\tau_*:K_0(C)\to\R$ to the resulting class.

We will denote this pairing by
$$
\langle [e]-[{\bf 1}_e],(\A,\H,\D)\rangle\in \R,\ \mbox{even case},\ \ \
\langle [u],(\A,\H,\D)\rangle\in \R,\ \mbox{odd case}.
$$
If, in the even case, $[e]-[f]\in K_0(\A)$ then $[{\bf 1}_e]=[{\bf 1}_f]\in K_0(\C)$ and we may define 
$$
\langle [e]-[f],(\A,\H,\D)\rangle:=\langle [e]-[{\bf 1}_e],(\A,\H,\D)\rangle-\langle [f]-[{\bf 1}_f],(\A,\H,\D)\rangle\in \R.
$$
\end{definition}
\index{index pairing!numerical }
From Corollary \ref{cor:double-index} we may deduce the following important result, which justifies the name
`numerical index pairing' for the map given in the previous Definition, as well as our notations.

\begin{prop}
\label{index-explicit}
Let $(\A,\HH,\D)$ be a  semifinite spectral triple relative to $(\cn,\tau)$, of parity $\bullet\in\{0,1\}$
 and with $\A$ separable. 
Let $e$ be a projector in $ M_n(\A^\sim)$ which represents 
$[e]\in K_0(\A)$, for $\bullet=0$ (resp.  $u$  a unitary in $ M_n(\A^\sim)$ which represents 
$[u]\in K_1(\A)$, for $\bullet=1$). Then with $F_\mu:=\D_\mu/|\D_\mu|$ and $P_\mu:=(1+F_\mu)/2$,
we have
\begin{align*}
\langle [e]-[{\bf 1}_e],(\A,\H,\D)\rangle
&=\Index_{\tau\otimes{\tr_{2n}}}\big(\hat e({F_\mu}_+\otimes{\rm Id}_n)\hat e\big),\quad
\mbox{even case},\\
\langle [u],(\A,\H,\D)\rangle
&=\Index_{\tau\otimes{\tr_{2n}}}
\big((P_\mu\otimes{\rm Id}_n)\hat u(P_\mu\otimes{\rm Id}_n)\big),\quad \mbox{odd case}.
\end{align*}
\end{prop}

\subsection{Smoothness and summability for  spectral triples}
\label{subsec:SS}

In this subsection we discuss the notions of  finitely summable spectral triple, 
$QC^\infty$ spectral triple  and most importantly {\em smoothly summable spectral triples}
for nonunital $*$-algebras. We then examine how these notions fit with our discussion of
summability and the pseudodifferential calculus introduced in the previous section.
One of the main technical difficulties that we have to overcome in the nonunital case
is the issue of finding the appropriate definition of a smooth algebra stable under the holomorphic
functional calculus. 

We begin by considering possible notions of summability for 
spectral triples. There are two basic tasks that we need some summability for:

1) To obtain a well-defined Chern character for the associated Fredholm module, and

2) To obtain a local index formula.

Even in the case where $\A$ is unital, point 2) requires extra smoothness assumptions, discussed below, in
addition to the necessary summability. Thus we expect point 2) to require more assumptions on the
spectral triple than point 1). For point 1) we have the following answer.

\begin{prop}
\label{pr:can-do}
Let $(\A,\HH,\D)$ be a  semifinite spectral triple relative to $(\cn,\tau)$.
Suppose further that there exists $p\geq 1$ such that $a(1+\D^2)^{-s/2}\in\L^{1}(\cn,\tau)$ 
for all $s>p$ and all
$a\in\A$. Then $(\H,F_\eps=\D(\eps+\D^2)^{-1/2})$ 
defines a $\lfloor p\rfloor+1$-summable pre-Fredholm module for $\A^2$
whose $KK$-class is independent of $\eps>0$ (or even $\eps\geq 0$ if $\D$ is invertible). 
If in addition
we have $[\D,a](1+\D^2)^{-s/2}\in \L^1(\cn,\tau)$ for all $s>p$ and all
$a\in\A$, then $(\H,F_\eps=\D(\eps+\D^2)^{-1/2})$ defines a 
$\lfloor p\rfloor+1$-summable pre-Fredholm module for $\A$ 
whose $KK$-class is independent of $\eps>0$ (or even $\eps\geq 0$ if $\D$ is invertible).
\end{prop}

{\bf Remark.} Here $\A^2$ means the algebra given by the finite linear span of products $ab$, $a,\,b\in\A$.

\begin{proof}
First we employ Lemma \ref{interpolation} to deduce that for all $\delta>0$ we have
$$
a(1-F_\eps^2)=\eps\,a(\eps+\D^2)^{-1}\in\L^{p/2+\delta}(\cn,\tau).
$$
The same lemma tells us that for all $a\in\A$ and $\delta>0$
$$
a(\eps+\D^2)^{-\frac{\lfloor p\rfloor+\delta}{2(\lfloor p\rfloor+1)}}\in \L^{\lfloor p\rfloor+1}(\cn,\tau).
$$

We again use the integral formula for fractional powers and \cite[Lemma 3.3]{CP1} to obtain
\begin{align*}
[F_\eps,a]&=\frac{-1}{\pi}\int_0^\infty 
\lambda^{-1/2}\D(\eps+\lambda+\D^2)^{-1}[\D,a]\D(\eps+\lambda+\D^2)^{-1}d\lambda\\
&\quad-\frac{1}{\pi}\int_0^\infty 
\lambda^{-1/2}(\eps+\lambda+\D^2)^{-1}[\D,a]\D^2(\eps+\lambda+\D^2)^{-1}d\lambda
+(\eps+\D^2)^{-1/2}[\D,a].
\end{align*}
Now we multiply on the left by $b\in\A$, and estimate the $\lfloor p\rfloor+1$-norm. Since
$$
(\eps+\D^2+\lambda)^{-1}=(\eps+\D^2+\lambda)^{-\frac{\lfloor p\rfloor+\delta}{2(\lfloor p\rfloor
+1)}}(\eps+\D^2+\lambda)^{-\frac{1}{2}-\frac{(1-\delta)}{2(\lfloor p\rfloor+1)}},
$$
and
$$
\Vert \D(\eps+\D^2+\lambda)^{-\frac{1}{2}-\frac{(1-\delta)}{2(\lfloor p\rfloor+1)}}
\Vert_\infty\leq (\eps+\lambda)^{-\frac{(1-\delta)}{2(\lfloor p\rfloor+1)}},
$$
by spectral theory, we find that for $1>\delta>0$
\begin{align*}
&\Vert b[F_\eps,a]\Vert_{\lfloor p\rfloor+1}
\leq 2\Vert [\D,a]\Vert\,\Vert b(\eps+\D^2)^{-\frac{\lfloor p\rfloor+\delta}{2(\lfloor p\rfloor+1)}}
\Vert_{\lfloor p\rfloor+1}\int_0^\infty
\lambda^{-1/2} (\eps+\lambda)^{-\frac{1}{2}-\frac{(1-\delta)}{2(\lfloor p\rfloor+1)}}\,d\lambda<\infty.
\end{align*}
Hence  $b[F_\eps,a]\in\L^{\lfloor p\rfloor+1}(\cn,\tau)$, and taking adjoints 
shows that $[F_\eps,a]b\in\L^{\lfloor p\rfloor+1}(\cn,\tau)$ 
for all $a,\,b\in\A$ also. Now we observe that 
$[F_\eps,ab]=a[F_\eps,b]+[F_\eps,a]b$, and so $[F_\eps,ab]\in \L^{\lfloor p\rfloor+1}(\cn,\tau)$ for
all $ab\in\A^2$. 
This completes the proof of the first part. The second claim follows from a similar estimate without 
the need to multiply by
$b\in\A$.
The independence of the class on $\eps>0$ is as in Lemma \ref{lem:spec-trip-to-Kas-mod}.
\end{proof}

The previous proposition shows that we have  sufficient conditions on a spectral triple in order to
obtain a  finitely summable pre-Fredholm module for $\A^2$ or $\A$. These two conditions are not equivalent.
Here is a counterexample for $p=1$.

Consider
the function $f:\,x\mapsto\sin(x^3)/(1+x^2)$ on
the real line, and the operator $\D=-i(d/dx)$ on $L^2(\R)$. 
Then the operator $f(1+\D^2)^{-s/2}$ is trace class for $\Re(s)>1$, by \cite[Theorem 4.5]{Simon},
while $[\D,f](1+\D^2)^{-s/2}$ is not trace class for any $\Re(s)>1$, by \cite[Proposition 4.7]{Simon}. 
To see the latter, it
suffices to show that with $g(x)=x^2/(1+x^2)$, 
we  have $g(1+\D^2)^{-s/2}$ not trace class. 
However this follows
from $g(1+\D^2)^{-s/2}=(1+\D^2)^{-s/2}-h(1+\D^2)^{-s/2}$ with 
$h=\frac{1}{1+x^2}$. The second operator is trace class,
however $(1+\D^2)^{-s/2}$ is well-known to be non-compact, and so not trace class.

We investigate the weaker of these two summability conditions first,
relating it to our integration theory from Section \ref{sec:psido-and-sum}. 
Indeed the  following two propositions show  that finite summability, in the sense of the next definition,  almost 
uniquely determines where $\A$ must sit inside $\cn$, and justifies
the 
introduction of the  Fr\'echet algebras 
$\B_1^k(\D,p)$.

\begin{definition}\label{def:spec-dim}
A  semifinite spectral triple  $(\A,\cH,\D)$, is said to be finitely summable if
there exists $s>0$ such that for all $a\in\A$, $a(1+\D^2)^{-s/2}\in\cl^1(\cn,\tau)$.
In such a case, we  let
$$
p:=\inf\big\{s>0\,:\,\mbox{for all } a\in\A, \,\, \tau\big(|a|(1+\D^2)^{-s/2}\big)<\infty
\big\},
$$
and call $p$ the spectral dimension of $(\A,\HH,\D)$.
\end{definition}
\index{semifinite spectral triple!finitely summable}\index{semifinite spectral triple!spectral dimension}
\index{spectral dimension}

{\bf Remark.} For the definition of the spectral dimension above to be  meaningful, one needs  two facts.
First, if $\A$ is the algebra of a finitely summable spectral triple,
we have $|a|(1+\D^2)^{-s/2}\in\L^1(\cn,\tau)$ for all $a\in\A$, which follows by 
using the polar decomposition $a=v|a|$ and writing 
$$
|a|(1+\D^2)^{-s/2}=v^*a(1+\D^2)^{-s/2}.
$$
Observe that we are {\em not} asserting that $|a|\in\A$, which is typically {\em not} true in examples.

The second fact we require is that
$ \tau\big(a(1+\D^2)^{-s/2}\big)\geq0$ for $a\geq 0$, which follows from
\cite[Theorem 3]{Bik}, quoted here as Proposition \ref{prop:bikky}.

In contrast to the unital case, checking the finite summability
condition for a nonunital spectral triple can be difficult. This is because our definition
relies on control of the trace norm 
of the non-self-adjoint
operators  $a(1+\D^2)^{-s/2}$, $a\in\A$. The next two results show that for a spectral triple $(\A,\H,\D)$
to be finitely summable with spectral dimension $p$, it is necessary that $\A\subset \B_1(\D,p)$
and this condition is {\em almost} sufficient as well.

\begin{prop}\label{one-way} Let $(\mathcal{A},\mathcal{H},\mathcal{D})$ 
be a  semifinite spectral triple.
If for some $p\geq 1$ we have 
$\mathcal{A}\subset\B_1^\infty(\D,p)$, 
then $(\mathcal{A},\mathcal{H},\mathcal{D})$
is finitely summable with spectral dimension given by the infimum of such $p$'s.
More generally, if for some $p\geq 1$ we have 
$\mathcal{A}\subset\B_2(\D,p)\mathcal{B}_2^{\lfloor p\rfloor+1}(\D,p)\subset \B_1(\D,p)$, 
then $(\mathcal{A},\mathcal{H},\mathcal{D})$
is finitely summable with spectral dimension given by the infimum of such $p$'s.
\end{prop}

\begin{proof}
The first statement is an immediate consequence of Corollary \ref{consistency}.
For the second statement, let $a\in\mathcal{A}$. We need to prove that $a(1+\D^2)^{-s/2}$ is trace class
for $a=bc$ with 
$b\in\mathcal{B}_2(\D,p)$ and $c\in\mathcal{B}_2^{\lfloor p\rfloor+1}(\D,p)$.
Thus, for all $k\leq \lfloor p\rfloor+1$ and all $s>p$ we have 
$$
b(1+\mathcal{D}^2)^{-s/4},\,\,(1+\mathcal{D}^2)^{-s/4}\delta^k(c)\in\mathcal{L}^2(\mathcal{N},\tau).
$$

We start from the  identity
$$
(-1)^k\frac{\Gamma(s+k)}{\Gamma(s)\Gamma(k+1)}(1+|\mathcal{D}|)^{-s-k}
=\frac{1}{2\pi i}\int_{\Re(\lambda)=1/2}\lambda^{-s}(\lambda-1-|\mathcal{D}|)^{-k-1}d\lambda,
$$
and then by induction we have
\begin{align*}
[(\lambda-1-|\mathcal{D}|)^{-1},c]=&
\sum_{k=1}^{\lfloor p\rfloor}(-1)^{k+1}(\lambda-1-|\mathcal{D}|)^{-k-1}\delta^k(c)\\
&+(-1)^{\lfloor p\rfloor}(\lambda-1-|\mathcal{D}|)^{-\lfloor p\rfloor-1}
\delta^{\lfloor p\rfloor+1}(c)(\lambda-1-|\mathcal{D}|)^{-1}.
\end{align*}

It follows that
\begin{align*}
[(1+|\mathcal{D}|)^{-s},c]
&=\frac1{2\pi i}\int_{\Re\lambda=1/2}\lambda^{-s}[(\lambda-1-|\mathcal{D}|)^{-1},c]\,d\lambda\\
&=
-\sum_{k=1}^{\lfloor p\rfloor}\frac{\Gamma(s+k)}{\Gamma(s)\Gamma(k+1)}
(1+|\mathcal{D}|)^{-s-k}\delta^k(c)\\
&\qquad+\frac{(-1)^{\lfloor p\rfloor}}{2\pi i}\int_{\Re(\lambda)=1/2}\lambda^{-s}
(\lambda-1-|\mathcal{D}|)^{-\lfloor p\rfloor-1}\delta^{\lfloor p\rfloor+1}(c)(\lambda-1-|\mathcal{D}|)^{-1}d\lambda.
\end{align*} 
Since $\big|\lambda-1-|\mathcal{D}|\big|\geq |\lambda|$ and since the 
$\|\cdot\|_2-$norms of the operators
$$
b(\lambda-1-|\mathcal{D}|)^{-(\lfloor p\rfloor+1)/2},
\quad (\lambda-1-|\mathcal{D}|)^{-(\lfloor p\rfloor+1)/2}\delta^{\lfloor p\rfloor+1}(c),
$$
are bounded uniformly over $\lambda$, we obtain
$$
\left\|b\frac{(-1)^{\lfloor p\rfloor}}{2\pi i}
\int_{\Re\lambda=1/2}\lambda^{-s}(\lambda-1-|\mathcal{D}|)^{-\lfloor p\rfloor-1}
\delta^{\lfloor p\rfloor+1}(c)(\lambda-1-|\mathcal{D}|)^{-1}d\lambda\right\|_1\leq  C(b,c)
\int_{\Re\lambda=1/2}\frac{|d\lambda|}{|\lambda|^{1+s}},
$$
which is finite. 
Hence we have
$b[(1+|\mathcal{D}|)^{-s},c]\in\mathcal{L}^1(\mathcal{N},\tau)$
and since
$$
b(1+|\mathcal{D}|)^{-s}c=(b(1+|\mathcal{D}|)^{-s/2})\cdot((1+|\mathcal{D}|)^{-s/2}c)
\in\mathcal{L}^1(\mathcal{N},\tau),
$$
we conclude that $a(1+|\mathcal{D}|)^{-s}\in\mathcal{L}^1(\mathcal{N},\tau)$, and so 
$a(1+\mathcal{D}^2)^{-s/2}\in\mathcal{L}^1(\mathcal{N},\tau)$. The claim about the spectral dimension follows immediately.
\end{proof}

\begin{prop}
\label{lem:necessary}
Let $(\A,\H,\D)$ be a finitely summable  semifinite spectral triple of spectral
dimension $p$. Then  $\A$ is a subalgebra of $\B_1(\D,p)$.
\end{prop}

\begin{proof}
Since $\A$ is a $*$-algebra, it suffices to consider self-adjoint elements.
For $a=a^*\in\A$, we have by assumption that
$a (1+\D^2)^{-s/2}\in\cl^1(\cn,\tau)$, for all $s>p$. Now let $a=v|a|=|a|v^*$ be the polar decomposition.
Observe that neither $v$ nor $|a|$ need be in $\A$. However
$$
|a| (1+\D^2)^{-s/2}=v^*a (1+\D^2)^{-s/2}\in\cl^1(\cn,\tau)\ \ \mbox{for all}\ \ s>p.
$$
Now \cite[Theorem 3]{Bik}, quoted here as Proposition \ref{prop:bikky},
implies that  $|a|^{1/2} (1+\D^2)^{-s/4}\in\cl^2(\cn,\tau)$, for all $s>p$,
and so $|a|^{1/2}\in\B_2(\D,p)$.
In addition $v|a|^{1/2}\in \B_2(\D,p)$, since $v|a|^{1/2}=|a|^{1/2}v^*$ by the functional calculus, and
$$
v|a|v^*=|a|^{1/2}v^*v|a|^{1/2}
=|a|,
$$
and $(1+\D^2)^{-s/4}|a|^{1/2}v^*v|a|^{1/2}(1+\D^2)^{-s/4}=(1+\D^2)^{-s/4}|a|(1+\D^2)^{-s/4}$.
From this we can conclude that $a=v|a|^{1/2}\cdot |a|^{1/2}\in (\B_2(\D,p))^2\subset \B_1(\D,p)$.
\end{proof}

{\bf Remark.} The previous two results tell us that a finitely summable spectral triple must have 
$\A\subset \B_1(\D,p)$. However the last result does {\em not} imply that for a finitely summable
spectral triple $(\A,\HH,\D)$ and $a=a^*\in \A$ we have $a_+,\,a_-,\,|a|$ in $\A$. On the other hand, the previous
proof shows that $|a|$ does  belong 
to $\B_1(\D,p)$, and so for a finitely summable spectral triple, we can improve
on the result of  Proposition \ref{cor:polar-decomp}, at least for elements of $\A$.

In addition to the summability of a  spectral triple $(\A,\H,\D)$ relative to $(\cn,\tau)$, 
we need to consider smoothness, and the two notions are
much more tightly related in the nonunital case. One reason for smoothness is that we need 
to be able to control commutators with $\D^2$ to obtain the local index formula. Another reason is that
we need to be able to show that we have a spectral triple for a 
(possibly) larger algebra $\B\supset \A$ where
$\B$ is  Fr\'{e}chet and stable under the holomorphic functional calculus, 
and has the same norm closure as $\A$: $A=\overline{\A}=\overline{\B}$.

The next definition recalls how the problem of finding suitable $\B\supset\A$ is solved in the
unital case.

\begin{definition}
Let $(\A,\cH,\D)$ be a  semifinite spectral triple,
relative to $(\cn,\tau)$. With  $\delta=[|\D|,\cdot]$ as before,
we say that $(\A,\HH,\D)$
is $QC^k$ if for all $b\in\A\cup[\D,\A]$ we have $\delta^j(b)\in\cn$ for all $0\leq j\leq k$. 
We say that $(\A,\H,\D)$ is $QC^\infty$ if it is
$QC^k$ for all  $k\in\N_0$.
\end{definition}
\index{semifinite spectral triple!$QC^k$}
\index{semifinite spectral triple!$QC^\infty$}

{\bf Remark.}
For a 
$QC^\infty$ spectral triple $(\A,\cH,\D)$ with $T_0,\dots,T_m\in\A\cup[\D,\A]$,
we see by iteration of the relation $T^{(1)}=\delta^2(T)+2\delta(T)|\D|$, 
that $T_0^{(k_0)}\cdots T_m^{(k_m)}(1+\D^2)^{-|k|/2}\in \cn$, 
where $|k|:=k_0+\cdots+k_m$ and $T^{(n)}$ is given in Definition \ref{parup}.

For $(\A,\cH,\D)$ a  $QC^\infty$  spectral triple, unital or not,
we may endow the algebra $\A$ with the
topology determined by the family of norms
\begin{align}
\A\ni a\mapsto \|\delta^k(a)\|+\|\delta^k([\D,a])\|,\quad  k\in\N_0.
\label{delta-top}
\end{align}

We call this topology the {\it  $\delta$-topology} and observe that by  \cite[Lemma 16]{Re},
$(\A_\delta,\H,\D)$ is also a $QC^\infty$ \index{$\delta$-topology}\index{semifinite spectral triple!$\delta$-topology}
spectral triple, where $\A_\delta$ is the completion of $\A$ in the
$\delta$-topology. Thus
we may,
without loss of generality, suppose that $\A$ is complete in the  $\delta$-topology by
completing
if necessary. This completion is Fr\'echet and stable under the holomorphic functional
calculus. So,  with $A$ the $C^*$-completion of $\A$, $K_*(\A)\simeq K_*(A)$ via inclusion.

However, and this is crucial in the remaining text, 
in the nonunital case the completion $\A_\delta$ may not satisfy the same summability
conditions as $\A$ (as classical examples show). Thus
we will define and use a finer 
topology which takes into account the summability of the spectral triple, to which we now return.

Keeping in mind Propositions \ref{pr:can-do}, \ref{one-way}, \ref{lem:necessary}, and 
incorporating smoothness in the picture, we see that the natural condition 
for a smooth and finitely summable spectral triple is to require
that $\A\cup[\D,\A]\subset \B_1^\infty(\D,p)$. The extra benefit   is that our algebra $\A$ sits inside a Fr\'echet
algebra  which is stable under the holomorphic functional calculus.

\begin{definition}
\label{delta-phi}
Let $(\A,\HH,\D)$ be a  semifinite spectral triple relative to $(\cn,\tau)$. Then we say that $(\A,\HH,\D)$ is
$QC^k$ summable if $(\A,\HH,\D)$ is finitely summable with spectral dimension $p$ and 
$$
\A\cup[\D,\A]\subset \B_1^k(\D,p).
$$
We say that $(\A,\H,\D)$ is smoothly summable if it is $QC^k$ summable for all $ k\in\N_0$ or, equivalently, 
if 
$$
\A\cup[\D,\A]\subset \B_1^\infty(\D,p).
$$
If $(\A,\HH,\D)$ is smoothly summable with spectral dimension $p$,
the $\delta$-$\varphi$-topology on $\A$ is determined by the family of norms 
$$
\A\ni a\mapsto \PP_{n,k}(a)+\PP_{n,k}([\D,a]), \,\,n\in\N,\,\,k\in\N_0,
$$ 
where the norms $\PP_{n,k}$ are those of Definition \ref{def-Bp},
$$
\cn\ni T\mapsto \PP_{n,k}(T)
:=\sum_{j=0}^k\PP_n(\delta^j(T)).
$$
\end{definition}
\index{semifinite spectral triple!$QC^k$ summable}
\index{semifinite spectral triple!smoothly summable}
\index{$\delta$-$\varphi$-topology}
\index{semifinite spectral triple!$\delta$-$\varphi$-topology}

{\bf Remark.} 
The $\delta$-$\vf$-topology  generalises the $\delta$-topology. 
Indeed, if $(1+\D^2)^{-s/2}$ belongs to $\mathcal{L}^1(\mathcal{N},\tau)$ for $s>p,$ then the 
norm $\PP_{n,k}$ is equivalent to the norm  defined in Equation \eqref{delta-top}. \\

The following result shows that given a smoothly summable  spectral triple $(\A,\HH,\D)$, 
we may without loss of generality assume that  the algebra $\A$ is complete with respect to 
the $\delta$-$\vf$-topology, by completing if necessary. Moreover the completion
of $\A$ in the $\delta$-$\vf$-topology is  stable under the holomorphic
functional calculus.      

\begin{prop}
\label{prop:delta-phi}
Let $(\mathcal{A},\mathcal{H},\mathcal{D})$ 
be a smoothly summable
 semifinite spectral triple with spectral dimension $p$, 
and let $\A_{\delta,\vf}$ denote the completion of $\A$ for the $\delta$-$\vf$ topology.
Then $(\A_{\delta,\vf},\HH,\D)$ is also a smoothly summable 
 semifinite spectral triple with spectral dimension $p$, and moreover $\A_{\delta,\vf}$
is stable under the holomorphic functional calculus.
\end{prop}

\begin{proof}
First observe that a sequence $(a_i)_{i\geq 1}\subset\A$ converges in the $\delta$-$\vf$ topology
if and only if both $(a_i)_{i\geq 1}$ and $([\D,a_i])_{i\geq 1}$ converge in $\B_1^\infty(\D,p)$.
As $\B^\infty_1(\D,p)$ is a Fr\'echet space, both $\A_{\delta,\vf}$ and $[\D,\A_{\delta,\vf}]$
are contained in $\B_1^\infty(\D,p)$.

Next, let us show that 
$(\A_{\delta,\vf},\HH,\D)$ is finitely summable with spectral dimension still given by $p$.
Let $a\in \A_{\delta,\vf}$ and $s>p$. By definition of tame pseudodifferential operators and Corollary 
\ref{consistency}, we have
$$
a(1+\D^2)^{-s/2}\in {\rm OP}_0^{-s}\subset \cl^1(\cn,\tau),
$$
as needed. Since
$\A\subset \A_{\delta,\vf}$, $p$ is the smallest number for which this property holds.

Last, it remains to show that $\A_{\delta,\vf}$ is stable under the holomorphic functional calculus 
inside its (operator) norm completion. 
We complete $\A$ in the norm 
$\Vert\cdot\Vert_{N,k}:=\sum_{n=1}^N\sum_{j=0}^k\PP_{n,j}(\cdot)+\PP_{n,j}([\D,\cdot])$ 
to obtain
a Banach algebra $\A_{N,k}$. 

Then we claim that
$\A_{\delta,\vf}=\bigcap_{N\geq 1,k\geq 0}\A_{N,k}$. 
The inclusion $\A_{\delta,\vf}\subset \bigcap_{N\geq 1,k\geq 0}\A_{N,k}$ is straightforward.
For the inclusion $\A_{\delta,\vf}\supset \bigcap_{N\geq 1,k\geq 0}\A_{N,k}$, suppose that $a$ is
an element of the intersection. Then for each $N,\,k$ there is a sequence $(a^{N,k}_i)_{i\geq 1}$
contained in $\A$ which converges to $a$ in the norm $\Vert\cdot\Vert_{N,k}$.

Now we make the observation that if $N'\leq N$ and $k'\leq k$ then $(a^{N,k}_i)_{i\geq 1}$ converges
in $\A_{N',k'}$ to the same limit. Thus, in this situation, for all $\eps>0$ there is $l\in \N$ such that
$i>l$ implies that $\Vert a^{N,k}_i-a\Vert_{N',k'}<\eps$. Thus for such an $\eps>0$ and $l$ we have
$\Vert a^{N,N}_N-a\Vert_{N',k'}<\eps$ whenever $N>{\rm max}\{N',k',l\}$. Hence the sequence
$(a^{N,N}_N)_{N\geq 1}$ converges in all of the norms $\Vert\cdot\Vert_{N',k'}$ and hence the limit 
$a$ lies in $\A_{\delta,\vf}$. Hence an element of $\A_{\delta,\vf}$ is an 
element of $A$ which lies in each $\A_{N,k}$.

Moreover the norm
completions of $\A$, $\A_{\delta,\vf}$ and $\A_{N,k}$, 
for each $N,\,k$, are all the same since the $\delta$-$\vf$ and 
$\Vert\cdot\Vert_{N,k}$ topologies are finer than the norm topology. We denote the latter by $A$.

Now let $a\in \A_{\delta,\vf}$ and $\lambda\in\C$ be such that $a+\lambda$ is 
invertible in $A^\sim$. Then with 
$b=(a+\lambda)^{-1}-\lambda^{-1}$ we have
\begin{equation}
(a+\lambda)(b+\lambda^{-1})=1=1+ab+\lambda b+\lambda^{-1}a\ \ \Rightarrow\ \ b=-\lambda^{-1}ab-\lambda^{-2}a.
\label{eq:pseudo-inverse}
\end{equation}
Rearranging Equation \eqref{eq:pseudo-inverse} shows that $b=-\lambda^{-1}(\lambda+a)^{-1}a$. 
Now as $\B_1^\infty(\D,p)$ is stable under the holomorphic functional calculus, 
$b\in\B_1^\infty(\D,p)\oplus\C$, but
this formula shows that in fact $b\in \B_1^\infty(\D,p)$.

Now we would like to apply $[\D,\cdot]$ to Equation \eqref{eq:pseudo-inverse}.  Since $b\in \B_1^\infty(\D,p)$, 
$b$ preserves ${\rm dom}\,\D={\rm dom}\,|\D|\subset\H$, and so it makes sense to apply $[\D,\cdot]$ to $b$. Then
$$
[\D,b]=-\lambda^{-1}[\D,a]b-\lambda^{-1}a[\D,b]-\lambda^{-2}[\D,a]\ \ \Rightarrow \ \ [\D,b]=-(\lambda+a)^{-1}[\D,a](\lambda+a)^{-1}.
$$
Thus we see that $[\D,b]\in \B_1^\infty(\D,p)$ since $(\lambda+a)^{-1}\in \B_1^\infty(\D,p)\oplus\C$ and
$[\D,a]\in \B_1^\infty(\D,p)$. Hence $b\in \A_{N,k}$ for all $N\geq 1$ and $k\geq 0$ and so $b\in \A_{\delta,\vf}$. 
\end{proof}

We close this section by giving a sufficient condition for a finitely summable spectral triple to be smoothly summable.
We stress that this  condition is easy to check, as shown  in all of our examples.

\begin{prop}
\label{smooth-sum-sufficient}
Let $(\A,\HH,\D)$ be a finitely summable spectral triple of spectral dimension $p$ relative to $(\cn,\tau)$.
If for all $T\in \A\cup [\D,\A]$, $k\in\N_0$ and all $s>p$ we have 
\begin{equation}
(1+\D^2)^{-s/4}L^k(T)(1+\D^2)^{-s/4}\in\L^1(\cn,\tau),
\label{eq:displayed}
\end{equation}
then $(\A,\HH,\D)$ is smoothly summable. Here $L(T)=(1+\D^2)^{-1/2}[\D^2,T]$.
\end{prop}

\begin{proof}
We need to prove  that the condition \eqref{eq:displayed} guarantees that $\A\cup[\D,\A]\subset \B_1^\infty(\D,p)$, 
that is, for all $a\in\A$, the operators $\delta^k(a)$ and $\delta^k([\D,a])$, $k\in\N_0$, 
all belong to $\B_1(\D,p)$. From $\delta^k(a)^*=(-1)^k\delta^k(a^*)$
(resp. $\delta^k([\D,a])^*=(-1)^{k+1}\delta^k([\D,a^*])$) and since the norms 
$\PP_m$, $m\in\N$, are $*$-invariant, we see that $\delta^k(a)\in \B_1(\D,p)$ 
(resp. $\delta^k([\D,a])\in \B_1(\D,p)$) if and only if
$\delta^k(\Re(a))$ and $\delta^k(\Im(a))$ (resp. $\delta^k([\D,\Re(a)])$ and 
$\delta^k([\D,\Im(a)]$) belong to $ \B_1(\D,p)$. Thus, we may assume that $a=a^*$. 

Let us treat first the case of $\delta^k(a)$ and for $a=a^*$.
Consider the polar decomposition $\delta^k(a)=u_k|\delta^k(a)|$. Depending on the parity of $k$, the partial isometry $u_k$ 
is self-adjoint or skew-adjoint, and in both cases it commutes with $|\delta^k(a)|$. This implies that
$$
\delta^k(a)=|\delta^k(a)|^{1/2}u_k|\delta^k(a)|^{1/2}.
$$
Thus, the condition
$$
 \delta^k(a)\in\B_1(\D,p),\,\,{\rm for\ all}\ k\in\N_0,
 $$
will follow if 
\begin{equation}
| \delta^k(a)|^{1/2},\,\, u_k|\delta^k(a)|^{1/2}\in\B_2(\D,p),\,\,{\rm for\ all}\ k\in\N_0.
\label{eq:no1}
\end{equation}
Since $u_k$ commutes with $| \delta^k(a)|^{1/2}$, and using the  definition of the space $\B_2(\D,p)$, 
the  condition \eqref{eq:no1} is equivalent to
\begin{equation}
| \delta^k(a)|^{1/2}(1+\D^2)^{-s/4},\,\, u_k|\delta^k(a)|^{1/2}(1+\D^2)^{-s/4}\in\L^2(\cn,\tau),\,\,{\rm for\ all}\ k\in\N_0,\,\,{\rm for\ all}\ s>p.
\label{eq:no2}
\end{equation}
 The conditions in \eqref{eq:no2} are equivalent to a single condition
$$
| \delta^k(a)|^{1/2}(1+\D^2)^{-s/4}\in\L^2(\cn,\tau),\,\,{\rm for\ all}\ k\in\N_0,\,\,{\rm for\ all}\ s>p,
$$
which is equivalent to
\begin{equation}
(1+\D^2)^{-s/4}|\delta^k(a)|(1+\D^2)^{-s/4}\in\L^1(\cn,\tau),\,\,{\rm for\ all}\ k\in\N_0,\,\,{\rm for\ all}\ s>p.
\label{eq:no3}
\end{equation}
Now, by  \cite[Theorem 3]{Bik}, see Proposition \ref{prop:bikky}, the condition \eqref{eq:no3} is satisfied if
$$
|\delta^k(a)|(1+\D^2)^{-s/2}\in\L^1(\cn,\tau),\,\,{\rm for\ all}\ k\in\N_0,\,\,{\rm for\ all}\ s>p,
$$
which in turn  is equivalent to 
\begin{equation}
\delta^k(a)(1+\D^2)^{-s/2}\in\L^1(\cn,\tau),\,\,{\rm for\ all}\ k\in\N_0,\,\,{\rm for\ all}\ s>p.
\label{eq:no4}
\end{equation}
Next, since 
$$
\delta^k(a)(1+\D^2)^{-s/2}=(1+\D^2)^{-s/4}\delta^k(\sigma^{s/4}(a))(1+\D^2)^{-s/4},
$$
 by an application of the same ideas leading to Lemmas \ref{root commutator estimate} and \ref{cont-grp}, 
 we see then that  condition \eqref{eq:no4} is equivalent to
\begin{equation}
(1+\D^2)^{-s/4}\delta^k(a)(1+\D^2)^{-s/4}\in\L^1(\cn,\tau),\,\,{\rm for\ all}\ k\in\N_0,\,\,{\rm for\ all}\ s>p.
\label{eq:no5}
\end{equation}
Finally, using  $L=(1+\sigma^{-1})\circ\delta$,  given in Lemma \ref{delta-LR}, we see that condition \eqref{eq:no5} is equivalent to
$$
(1+\D^2)^{-s/4}L^k(a)(1+\D^2)^{-s/4}\in\L^1(\cn,\tau),\,\,{\rm for\ all}\ k\in\N_0,\,\,{\rm for\ all}\ s>p.
$$
In an entirely similar way, we see that $\delta^k([\D,a])\in\B_1(\D,p)$ if
$$
 (1+\D^2)^{-s/4}L^k([\D,a])(1+\D^2)^{-s/4}\in\L^1(\cn,\tau),\,\,{\rm for\ all}\ k\in\N_0,\,\,{\rm for\ all}\ s>p.
$$
This completes the proof.
\end{proof}

\subsection{Some cyclic theory}
\label{subsec:cyclic}
\index{cyclic cohomology}

In the following discussion we recall sufficient cyclic theory for the purposes of this memoir. 
More information  about the complexes and bicomplexes underlying our definitions
is 
contained in \cite{CPRS2,CPRS4}, and much more can be found in \cite{Co4,Lo}.
When we discuss tensor products of algebras we always use the projective tensor product.

Let $\A$ be a {\em unital} Fr\'{e}chet algebra.
A  cyclic $m$-cochain on $\A$ is a multilinear functional $\psi$
such that
\ben 
\psi(a_0,\dots,a_m)=(-1)^{m}\psi(a_m,a_0,\dots,a_{m-1}).
\een
The set of all  cyclic cochains  is denoted $C^m_{\lambda}$.
We say that $\psi$ is a  cyclic cocycle if for all $a_0,\dots,a_{m+1}\in\A$
we have
$(b\psi)(a_0,\dots,a_{m+1})=0$ where $b$ is the Hochschild coboundary in Equation \eqref{eq:hochs-b} below.
The cyclic cochain is  normalised if $\psi(a_0,a_1,\dots,a_m)=0$ whenever any of 
$a_1,\dots,a_m$ is the unit of $\A$.
\index{cyclic cohomology!cyclic cocycle}
\index{cyclic cohomology!normalised cocycle}

A  $(b,B)$-cochain $\phi$ for $\A$ is a finite collection of
multilinear functionals,
$$
\phi=(\phi_m)_{m=0,1,\dots,M},\quad \phi_m:\A^{\otimes m+1}\to\C.
$$
An  odd cochain has $\phi_m=0$ for even $m$, while an
even cochain has $\phi_m=0$ for odd $m$.
Thought of as functionals on the projective tensor product $\A^{\otimes m+1}$,
a  normalised cochain will satisfy
$\phi(a_0,a_1,\ldots,a_n)=0$ whenever for $k\geq 1$, any $a_k=1_\A$.
A normalised cochain is a  $(b,B)$-cocycle if, for all $m$,
$b\phi_m+B\phi_{m+2}=0$ where
$b$
is the  Hochschild coboundary  operator  given by
\begin{align}
(b\phi_m)(a_0,a_1,\ldots,a_{m+1})&=
\sum_{k=0}^{m}(-1)^k\phi_m(a_0,a_1,\ldots,a_ka_{k+1},\ldots,a_{m+1})\nonumber\\
&\qquad\qquad\qquad\qquad+(-1)^{m+1}\phi_m(a_{m+1}a_0,a_1,\ldots,a_{m}),
\label{eq:hochs-b}
\end{align}
and $B$ is Connes' coboundary operator
\begin{align}
(B\phi_m)(a_0,a_1,\ldots,a_{m-1})
&=\sum_{k=0}^{m-1}
(-1)^{(m-1)j}\phi_m(1_{\A},a_k,a_{k+1},\ldots,a_{m-1},a_0,\ldots,a_{k-1}).
\label{eq:connes-B}
\end{align}
 \index{cyclic cohomology!$(b,B)$-cocycle}

We write $(b+B)\phi=0$ for brevity, and observe that this formula for $B$ is 
only valid on the normalised
complex, \cite{Lo}. As we will only consider normalised cochains, this will be sufficient for
our purposes.

 For a {\em nonunital} Fr\'echet algebra $\A$,
a  reduced $(b,B)$-cochain
$(\phi_n)_{n=\bullet,\bullet+2,\dots,M}$ for $\A^\sim$ and of parity $\bullet\in\{0,1\}$,
is a normalised $(b,B)$-cochain  such that 
if $\bullet=0$ we have $\phi_0(1_{\A^\sim})=0$. The formulae for the operators $b,\,B$ are the same.
By \cite[Proposition 2.2.16]{Lo}, the reduced cochains come from a suitable bicomplex called the
reduced $(b,B)$-bicomplex, and gives a cohomology theory for $\A$.
\index{cyclic cohomology!reduced $(b,B)$-cocycle}

Thus far, our discussion has been algebraic. We now remind the reader that when working
with a Fr\'{e}chet algebra, we complete the algebraic tensor product in the projective tensor
product topology.
Given a  spectral triple
$(\A,\H,\D)$, we may without loss of generality 
complete $\A$ in the $\delta$-$\vf$-topology using Proposition \ref{prop:delta-phi}.
Then the
algebraic discussion above carries through. This follows because the operators
$b$ and $B$ are defined using multiplication, which is continuous, and insertion of $1_{\A^\sim}$ in the first
slot. This latter is also continuous, and one just needs to check that $B:C^1(\A)\to C^0(\A)$ maps
{\em normalised} cochains to cochains vanishing on the unit 
$1_{\A^\sim}\in \A^\sim$. This follows  from
the definitions.

Finally, an $(n+1)$-linear functional on an algebra $\A$ is cyclic if and only if 
it is the character of a cycle,
\cite[Chapter III]{Co4}, \cite[Proposition 8.12]{GVF}, and so the 
Chern character of a
Fredholm module over $\A$,
defined in the next section,
will always define a reduced cyclic cocyle for $\A^\sim$.

\subsection{Compatibility of the Kasparov product, numerical index and Chern character}
\label{subsec:put-it-together}

First we discuss the Chern character of semifinite Fredholm modules
and then relate the Chern character
to our analytic index
pairing  and the Kasparov product.

\begin{definition} \label{conditional} Let $(\H,F)$ be a Fredholm module relative to $(\cn,\tau)$.
We define the `conditional trace'  $\tau'$ by
\ben \tau'(T)=\tfrac{1}{2}\tau\big(F(FT+TF)\big),\een
provided $FT+TF\in\LL^1(\cn)$ (as it will be in our case, see
\cite[p. 293]{Co4} and \eqref{Lavender-hoo-hoo} below).
Note that if $T\in\LL^1(\cn)$, using the trace property and $F^2=1$, we find
$ \tau'(T)=\tau(T)$.
\end{definition}
\index{conditional trace}

The {\bf Chern character}, $[{\rm Ch}_F]$, of a $(p+1)$-summable ($p\geq 1$) semifinite Fredholm module
$(\HH,F)$ relative to $(\cn,\tau)$ is
the class in periodic cyclic cohomology of
the single  normalized and reduced cyclic cocycle
\ben 
\lambda_m\tau'\big(\gamma a_0[F,a_1]\cdots[F,a_m]\big), \ \ \ a_0,\dots,a_m
\in\A,\ \  m\geq\lfloor p\rfloor,
\een
where $m$ is even if and only if $(\HH,F)$ is even.
Here $\lambda_m$ are constants ensuring that this collection of cocycles
yields a well-defined periodic class, and they are given by
\ben 
\lambda_m=\left\{\begin{array}{ll} (-1)^{m(m-1)/2}\Gamma(\frac{m}{2}+1)
& \ \ m\ \ \ {\rm even}\\ \sqrt{2i}(-1)^{m(m-1)/2}\Gamma(\frac{m}{2}+1) &
\ \ m\ \ \ {\rm odd}\end{array}\right..
\een
For $p=n\in\N$,
the Chern character of an $(n+1)$-summable Fredholm module 
of the same parity than $n$,
is
represented by
the cyclic cocycle in dimension $n,$ ${\rm Ch}_F\in C^{n}_\lambda(\A)$, given by
\ben 
{\rm Ch}_F(a_0,\dots,a_{n})=\lambda_{n}\tau'(\gamma a_0[F,a_1]
\cdots[F,a_n]), \ \ \ \ \ a_0,\dots,a_n\in\A.
\een
The latter makes good sense since
\begin{equation}
\label{Lavender-hoo-hoo}
F\gamma a_0[F,a_1]\cdots[F,a_n]+\gamma a_0[F,a_1]
\cdots[F,a_n]F= (-1)^n\gamma[F ,a_0][F,a_1]\cdots[F,a_n],
\end{equation}
belongs to $\cl^1(\cn,\tau)$ by the $(p+1)$-summability assumption. 
We will always take the cyclic cochain ${\rm Ch}_F$ (or its $(b,B)$ analogue; see
below) as representative of $[{\rm Ch}_F]$, and will often refer to ${\rm Ch}_F$ as the
Chern character.

Since the Chern character is a cyclic cochain, it lies in the image of the
operator $B$, \cite[Corollary 20, III.1.$\beta$]{Co4}, and as $B^2=0$ we have $B\,{\rm Ch}_F=0$. 
Since $b\,{\rm Ch}_F=0$, we may regard the Chern character as a one term
element of the $(b,B)$-bicomplex. However,
the correct normalisation is (taking the Chern character to be in degree $n$)
\ben 
C^n_\lambda\ni {\rm Ch}_F\mapsto \frac{(-1)^{\lfloor n/2\rfloor}}{n!}{\rm Ch}_F\in C^n.
\een
Thus instead of $\lambda_n$ defined above, we use $\mu_n:=\frac{(-1)^{\lfloor n/2\rfloor}}{n!}\lambda_n$.
The difference in normalisation between periodic and $(b,B)$ is due  to
the way the index pairing is defined in the two cases, \cite{Co4}, and
compatibility with the periodicity operator. From now on we will use the $(b,B)$-normalisation,
and so make the following definition.

\begin{definition}
\index{Chern character}
\index{Chern character!Fredholm module}
\index{Chern character!projection}
\index{Chern character!unitary}
Let $(\H,F)$ be a semifinite $(n+1)$-summable, $n\in\N$,
Fredholm module for a nonunital algebra $\A$,  relative to $(\cn,\tau)$, and suppose the parity
of the Fredholm module is the same as the parity of $n$. Then we define the
Chern character $[{\rm Ch}_F]$ to be the cyclic cohomology class of the single term $(b,B)$-cocycle
defined by
$$
{\rm Ch}_F^n(a_0,a_1,\dots,a_n):=\left\{\begin{array}{ll}
\frac{\Gamma(\frac{n}{2}+1)}{n!} \tau'(\gamma a_0[F,a_1]\cdots[F,a_n]), & \ \ n\ \ \ {\rm even}\\
& \\ \sqrt{2i}\frac{\Gamma(\frac{n}{2}+1)}{n!}\tau'(a_0[F,a_1]\cdots[F,a_n]),  &
\ \ n\ \ \ {\rm odd}\end{array}\right.,\ \ \ a_0,\dots,a_n
\in\A.
$$
If $e\in \A^\sim$ is a  projection we define ${\rm Ch}_0(e)=e\in\A^\sim$ and for $k\geq 1$
$$
{\rm Ch}_{2k}(e)
=(-1)^k\frac{(2k)!}{k!}
(e-1/2)\otimes e\otimes\cdots\otimes e\in (\A^\sim)^{\otimes 2k+1}.
$$
If $u\in \A^\sim$ is a unitary then we define for $k\geq 0$
$$
{\rm Ch}_{2k+1}(u)
=(-1)^k\,k!\,
u^*\otimes u\otimes\cdots
\otimes u^*\otimes u\in (\A^\sim)^{\otimes 2k+2}.
$$
\end{definition}

In order to prove  the equality of our numerical index with the Chern character pairing, we need the 
cyclicity of the trace on a semifinite von Neumann algebra
from  \cite[Theorem 17]{BrK}, quoted here as Proposition \ref{prop:everybody-knows}.

\begin{prop} Let $(\A,\HH,\D)$ be a  semifinite spectral triple,   with $\A$ separable, 
which is smoothly summable with spectral dimension $p\geq1$, and such that $\lfloor p\rfloor$ has
the same parity as the spectral triple. 
Then for a class $[e]\in K_0(\A)$, with $e$ a projection in $M_n(\A^\sim)$ (resp.
 for a class $[u]\in K_1(\A)$, with $u$ a unitary in $M_n(\A^\sim)$) we have for any $\mu>0$
\begin{align*}
&\langle [e]-[{\bf 1}_e], (\A,\H,\D)\rangle
={\rm Ch}_{F_\mu\otimes{\rm Id}_n}^{\lfloor p\rfloor}\big({\rm Ch}_{\lfloor p\rfloor}(\hat e)\big),\quad \mbox{even case},\\
&\langle [u], (\A,\H,\D)\rangle
=-(2i\pi)^{-1/2} \,{\rm Ch}_{F_\mu\otimes{\rm Id}_n}^{\lfloor p\rfloor}\big({\rm Ch}_{\lfloor p\rfloor}(\hat u)\big),\quad
\mbox{odd case}.
\end{align*}
\end{prop}

\begin{proof}
The first thing to prove is that $[F_\mu,\hat a]\in\L^{\lfloor p\rfloor+1}(\cn,\tau)$ for all $a\in\A$. This  
will follow if we have  $[F_\eps,a]\in\L^{\lfloor p\rfloor+1}(\cn,\tau)$ for all $a\in\A$.
By the smooth summability assumption, we have 
$a,\,[\D,a]\in\B^\infty_1(\D,p)={\rm Op}_0^0$ for all $a\in\A$. Thus the
Schatten class property we need follows from Proposition \ref{pr:can-do}.

For the even case the remainder of the proof is just as in
\cite[Proposition 4, IV.1.$\gamma$]{Co4}.
The
strategy in the odd case is the same. However,
we present the proof in the odd case in order to clarify some sign conventions.
To simplify the notation, we let $u$ be a unitary in $\A^\sim$ and suppress the matrices $M_n(\A^\sim)$.

In this case  the
operator $P_\mu \hat uP_\mu:P_\mu(\HH\oplus \HH)\to P_\mu(\HH\oplus \HH)$, is 
$\tau\otimes \tr_2$-Fredholm with parametrix
$P_\mu \hat u^*P_\mu$, where $u\in\A^\sim$ unitary and $P_\mu=(F_\mu+1)/2\in M_2(\cn)$.
To
obtain our result, we need \cite[Lemma 3.5]{Ph1}
which    shows that with $Q_\mu:=\hat uP_\mu{\hat u}^*$ we have
$$
|(1-Q_\mu)P_\mu|^{2n} =[P_\mu(1-Q_\mu)(1-Q_\mu)P_\mu]^n
=[P_\mu-P_\mu Q_\mu P_\mu]^n= (P_\mu-P_\mu\hat uP_\mu\hat u^*P_\mu)^n .
$$
One ingredient in the proof that connects this to odd summability is the
identity 
$$
(Q_\mu-P_\mu)^{2n+1} = |(1-P_\mu)Q_\mu|^{2n} - |(1-Q_\mu)P_\mu|^{2n},
$$
proved by induction in \cite[Lemma 3.4]{Ph1}.
It is then shown in \cite[Theorem 3.1]{CP2} that if $f$ is any odd function with
$f(1)\neq 0$ and $f(Q_\mu-P_\mu)$ trace-class, we have 
$$
\mbox{Index}_{\tau\otimes\tr_2}(P_\mu Q_\mu)=\frac{1}{f(1)}\tau\otimes\tr_2\big(f(Q_\mu-P_\mu)\big).
$$ 
Putting these ingredients together we have
\begin{align*}
\Index_{\tau\otimes\tr_2}(P_\mu\hat uP_\mu)&=\Index_{\tau\otimes\tr_2}(P_\mu\hat uP_\mu \hat u^*)=\Index_{\tau\otimes\tr_2}(P_\mu Q_\mu)\\
&=\tau\otimes\tr_2((P_\mu-P_\mu \hat u^*P_\mu \hat uP_\mu )^n)-\tau\otimes\tr_2((P_\mu-P_\mu \hat uP_\mu \hat u^*P_\mu)^n),
\end{align*}
where $n=(\lfloor p\rfloor+1)/2$ is an  integer, since $\lfloor p\rfloor$ is assumed 
odd. First we observe that
$P_\mu -P_\mu\hat u^*P_\mu \hat uP_\mu=-P_\mu[\hat u^*,P_\mu]\hat uP_\mu$, 
and by replacing $P_\mu$ by $(1+F_\mu)/2$ we have
$$
P_\mu[\hat u^*,P_\mu]\hat uP_\mu=[F_\mu,\hat u^*]\,[F_\mu,\hat u]({1+F_\mu})/{8}.
$$
Since $F_\mu[F_\mu,\hat a]=-[F_\mu,\hat a]F_\mu$ for all $a\in\A$, 
cycling a single $[F_\mu,\hat u^*]$ around using 
Proposition \ref{prop:everybody-knows} yields
\begin{align*}
&\Index_{\tau\otimes\tr_2}(P_\mu \hat uP_\mu)
=\tau\otimes\tr_2\big((P_\mu-P_\mu \hat u^*P_\mu\hat uP_\mu)^n\big)
-\tau\otimes\tr_2\big((P_\mu-P_\mu\hat uP_\mu \hat u^*P_\mu)^n\big)\\
&\quad=\tau\otimes\tr_2\Big(\Big(-\frac{1}{4}[F_\mu,\hat u^*]\,[F_\mu,\hat u]\,\frac{1+F_\mu}{2}\Big)^n\Big)-
\tau\otimes\tr_2\Big(\Big(-\frac{1}{4}[F_\mu,\hat u]\,[F_\mu,\hat u^*]\,\frac{1+F_\mu}{2}\Big)^n\Big)\\
&\quad=(-1)^n\frac{1}{4^n}\tau\otimes\tr_2\Big(\frac{1+F_\mu}{2}([F_\mu,\hat u^*][F_\mu,\hat u])^n\\
&\qquad\qquad\qquad\qquad-
[F_\mu,\hat u^*][F_\mu,\hat u][F_\mu,\hat u^*]\frac{1+F_\mu}{2}[F_\mu,\hat u][F_\mu,\hat u^*]
\cdots \frac{1+F_\mu}{2}[F_\mu,\hat u]\frac{1-F_\mu}{2}\Big).
\end{align*}
Thus
\begin{align*}
\Index_{\tau\otimes\tr_2}(P_\mu\hat uP_\mu)
&=(-1)^n\frac{1}{4^n}\tau\otimes\tr_2\Big(\Big(\frac{1+F_\mu}{2}-\frac{1-F_\mu}{2}\Big)\big([F_\mu,\hat u^*][F_\mu,\hat u]\big)^n\Big)\\
&=(-1)^n\frac{1}{4^n}\tau\otimes\tr_2\big(F_\mu([F_\mu,\hat u^*][F_\mu,\hat u]\big)^n)\\
&=(-1)^n\frac{1}{2^{2n-1}}(\tau\otimes\tr_2)'\big(\hat u^*[F_\mu,\hat u]\cdots[F_\mu,\hat u^*][F_\mu,\hat u]\big),
\end{align*}
where in the last line there are $2n-1=\lfloor p\rfloor$ commutators. Comparing the normalisation
of the  formulae above with the Chern characters using the duplication formula for
the Gamma function, we find
$$
\Index_{\tau\otimes\tr_2}(P_\mu \hat uP_\mu )
=\frac{-1}{\sqrt{2\pi i}}  {\rm Ch}_{F_\mu}^{\lfloor p\rfloor}({\rm Ch}_{\lfloor p\rfloor}(\hat u)),
$$
as needed.
\end{proof}

{\bf Remark.} When the parity of $\lfloor p\rfloor$ does not agree with the parity of the spectral triple, 
we apply the same proof to $\lfloor p\rfloor+1$, and so use 
${\rm Ch}_{F_\mu\otimes{\rm Id}_n}^{\lfloor p\rfloor+1}$ to represent the class of the Chern character.

{\bf Remark.}
An independent check of the sign can be made on the circle, using the unitary $u=e^{i\theta}$ and
the Dirac operator $\frac{1}{i}\frac{d}{d\theta}$. In this case $\Index(PuP)=-1$.
To arrive at this sign we have retained the usual definition of the Chern character
and introduced an additional minus sign in the normalisation.
In \cite{CPRS2} the signs used are all correct, however in \cite{CPRS4} we
introduced an additional minus sign (in error) in the formula for spectral flow. This
disguised the fact that we were not taking a homotopy to the Chern character
(as defined above) but rather to minus the Chern character.
This is of some relevance, as our strategy for proving the local index formula in
the nonunital case is based on the homotopy arguments of \cite{CPRS4}.

\subsection{Digression on the odd index pairing for nonunital algebras}
\label{subsec:odd-index}
To emphasise that the introduction of the double is only a technical device to 
enable us to work with invertible operators,
we explain a different approach to 
handling the problem of constructing an involutive Fredholm module in the odd case.

Assume that we have an odd Fredholm module $(\HH,F)$ over a nonunital 
$C^*$-algebra $A$, with $F^2=1$. 
Then, as mentioned previously,  it is straightforward 
to check that with $P=(1+F)/2$ and $u\in A^\sim$ a unitary,
the operator $PuP$ is Fredholm with parametrix $Pu^*P$ (as operators on $P\HH$).

Now we have constructed a doubled up version of a  spectral triple $(\A,\HH^2,\D_\mu)$, and so
obtained a Fredholm module $(\HH^2,F_\mu)$ with $F_\mu^2=1$. By Lemma \ref{lem:noninv}, this
Fredholm module represents the class of our spectral triple. In this brief digression we show that
the odd index pairing can be defined in terms of the original data with no doubling.

So assume that we have a  spectral triple $(\A,\HH,\D)$.
First we can decompose $P:=\chi_{[0,\infty)}(\D)$ as the kernel projection $P_0$ plus the positive spectral projection $P_+$.
We will use $P_-$ for the negative spectral projection so that $P_-+P_0+P_+$ is the identity of $\cn$.
We let $F=2P-1$ and we want to prove that $F$ can be used to construct
a Fredholm module for $\A$ that is in the same Kasparov class as that given by 
$F_\eps:=\D(\eps+\D^2)^{-1/2}$.

If we can show that $[F,a]$ is compact for all $a\in \A$ then we are done because
the straight-line path $F_t=tF+(1-t)F_\eps$ provides a homotopy of Kasparov modules.
To prove compactness of the commutators  we use the method of \cite{CH}.

\begin{prop} Let $(\A,\HH,\D)$ be a  semifinite spectral triple relative to $(\cn,\tau)$
 with $\A$ separable.
With $F=2\chi_{[0,\infty)}(\D)-1$, the pair $(\HH, F)$ is a Fredholm module for 
$\A$ and $(F,C_C)$ (with $C$ the $C^*$-completion of the subalgebra of $\K(\cn,\tau)$ 
given in Definition \ref{defn:not-kn}) provides a bounded representative 
for the Kasparov class of the  spectral triple $(\A, \HH, \D)$.
\end{prop}

\begin{proof}
Our proof uses the doubled spectral triple $(\A,\HH^2,\D_\mu)$. 
Let $P_\mu = (1+F_\mu)/2$  and use the notation $Q$ for
the operator obtained by taking the strong limit $\lim_{\mu\to 0}P_\mu$ as $\mu\to 0$. 
We note that
$$
Q=\bma P_+ +\frac12P_0 & \frac12P_0 \\ \frac12P_0 & P_- +\frac12P_0\ema
\quad
\mbox{and}\quad
P_\mu=\bma A & A^{1/2}(1-A)^{1/2}\\ A^{1/2}(1-A)^{1/2} & 1-A\ema,
$$
where $A=\frac12\big((\mu^2 +\D^2)^{1/2}+\D\big)(\mu^2+\D^2)^{-1/2}$.
Next a short calculation shows that
$$
2Q-1=\bma F & 0\\ 0 & -F\ema+\bma -P_0 & P_0\\ P_0 & -P_0\ema.
$$
Recall that in the double spectral triple
\ben
 a\mapsto \hat a=\bma a & 0\\ 0 & 0\ema,
\ \ \mbox{for all } a\in\A.
\een
Thus to show that $[F,a]$ is compact for all $a\in\A$, it suffices to show that $[Q,\hat a]$ is compact,
since for any $s>0$ we have  $P_0a=P_0(1+\D^2)^{-s}a$ and so both 
$P_0a$ and $aP_0$ are compact for all $a\in \A$. This  follows since $a(1+\D^2)^{-1/2}$
is compact.
Consider 
$$
[P_\mu,\hat a] - [Q ,\hat a]= [P_\mu-Q,\hat  a],
$$
and the individual matrix elements in $(P_\mu-Q)\hat a$
for example.
We have two terms to deal with: 
the diagonal one
$$
\tfrac12\big((\mu^2 +\D^2)^{1/2}+\D- 2(P_++\tfrac12P_0)(\mu^2 +\D^2)^{1/2}\big)(\mu^2+\D^2)^{-1/2}a,
$$ 
and the off-diagonal one
$$
\tfrac12\mu(\mu^2+\D^2)^{-1/2}a-\tfrac12P_0a.
$$
We have already observed that since we have a  spectral triple, the off-diagonal terms
are compact.
For the diagonal terms, we first observe that
$$
(\mu^2 +\D^2)^{1/2}+\D- 2(P_++\tfrac12P_0)(\mu^2 +\D^2)^{1/2}=\D-(2P-1)(\mu^2+\D^2)^{1/2}-P_0\mu,
$$ 
is a bounded operator. This follows from the functional calculus applied to the function
$f(x)=x-\mbox{sign}(x)(\mu^2 + x^2)^{1/2}$, where $\mbox{sign}(0)$ is defined to be $1$. 
This can be checked for all
 $\mu$ in $[0,1]$.
This boundedness, together with the compactness of $(\mu^2+\D^2)^{-1/2}a$, shows that
$$
\tfrac12\big((\mu^2 +\D^2)^{1/2}+\D- 2(P_++\tfrac12P_0)(\mu^2 +\D^2)^{1/2}\big)(\mu^2+\D^2)^{-1/2}a,
$$  
is compact for all $\mu\in[0,1]$.  This establishes
that 
$[Q, \hat a]$ is compact for all $a\in\A$.  

The second statement now follows immediately.
\end{proof}

Combining this with Proposition \ref{index-explicit} proves the following result.

\begin{corollary}
Let $(\A,\HH,\D)$ be an odd  semifinite smoothly summable spectral triple relative to $(\cn,\tau)$
with spectral dimension $p\geq 1$  and with $\A$ separable. Let $u$ be a unitary in $ M_n(\A^\sim)$ representing a class 
$[u]$ in $ K_1(\A)$ and $P=\chi_{[0,\infty)}(\D)$. Then
\begin{align*}
&\langle [u],(\A,\H,\D)\rangle
=\Index_{\tau\otimes{\tr_n}}\big((P\otimes{\rm Id}_n) u(P\otimes{\rm Id}_n)\big).
\end{align*}
\end{corollary}




\section{The local index formula for semifinite   spectral
triples}\label{sec:LIT}

We have now come to the proof of the local index formula in noncommutative geometry
for  semifinite  smoothly summable spectral triples. This proof is modelled on that in \cite{CPRS4} in
the unital case, which in turn was inspired by Higson's proof in \cite{Hi}.

We have opted to present the proof `almost in full', 
though sometimes just sketching the
algebraic parts of the argument, referring to \cite{CPRS4} for more details. This means 
we have some repetition of material from \cite{CPRS4} in order that the proof be comprehensible.
Due to the nonunital subtleties, we include detailed proofs of the 
analytic statements, deferring the 
lengthier proofs to the Appendix so as not to distract from the main argument.

In the unital case we constructed two $(b,B)$-cocycles, the resolvent and residue cocycles. The proof
in \cite{CPRS4} shows that the residue cocycle is cohomologous to the Chern character, while the 
resolvent cocycle is `almost' cohomologous to the Chern character, in a sense we make precise later.
The aim now is to show that for    smoothly  summable semifinite spectral triples:

1) the resolvent and residue cocycles are still defined as elements of the reduced $(b,B)$-complex
in the nonunital setting;

2) the homotopies from the Chern character to the resolvent and residue
cocycles are still well-defined and continuous in the nonunital setting. In
particular, various intermediate cocycles must be shown to be well-defined and
continuous.

\subsection{The resolvent and residue cocycles and other cochains}
\label{subsec:cocycles_resi}

In order to deal with the even and odd cases simultaneously, we need to introduce some 
further notation to handle the differences in the formulae between the two cases.

In the following, we fix $(\A,\HH,\D)$, a  semifinite, smoothly summable,
spectral triple, with spectral dimension $p\geq 1$ and  parity $\bullet\in\{0,1\}$
($\bullet=0$ for an even spectral triple and $\bullet=1$ for odd triples).
We will use the notation
$da:=[\D,a]$ for commutators in order to save space.\index{$\bullet$, parity notation}
We further require that 
$A$, the norm closure of $\A$, be separable in
order that we can apply the Kasparov product to define the numerical index pairings given in 
Definition \ref{index-paring}. 
Finally, we have seen in   Proposition 
\ref{prop:delta-phi} that we may assume,
without loss of generality, that $\A$ is complete in the $\delta$-$\vf$-topology.

We  define a (partial) $\mathbb Z_2$-grading on ${\rm OP}^*$, 
by declaring that $|\D|$ and the elements of $\A$ have degree zero, 
while $\D$ has degree one. When the triple is even, this  
coincides with the degree defined by the grading  $\gamma$. 
When defined, we denote 
the grading degree of an element $T\in{\rm OP}^*$ by ${\rm deg}(T)$. \index{${\rm deg}$, grading degree for operators}
We also let $M:=2\lfloor( p+\bullet+1)/2\rfloor-\bullet$, the greatest \index{$M$, summation limits in local index formula}
integer of parity $\bullet$ in $[0,p+1]$. In particular, $M=p$ when $p$ is an 
integer of parity $\bullet$  and $M=p+1$ if $p$ is an integer of parity $1-\bullet$. The grading degree allows us to define the graded commutator 
of $S,\,T\in {\rm OP}^*$ of definite grading degree, by
$$
[S,T]_\pm:=ST-(-1)^{{\rm deg}(S)\,{\rm deg}(T)}TS.
$$

\index{$[\cdot,\cdot]_\pm$, graded commutator}
We will begin by defining the various cocycles and cochains we need on $\A^{\otimes (m+1)}$
for appropriate $m$. In order to work in the reduced $(b,B)$-bicomplex for $\A^\sim$, we will
need to extend the definitions of all these cochains to $\A^\sim\otimes\A^{\otimes m}$. We will
carry out this extension in the next subsection.

\subsubsection{The residue cocycle}

In order to define the residue cocycle, 
we need a condition on the singularities of certain zeta functions 
constructed from $\D$ and $\A$.

\begin{definition}
\label{dim-spec}
Let  $(\A,\cH,\D)$ be a  smoothly summable  
spectral triple   of spectral dimension $p$. 
We say that the spectral dimension is isolated, if for any element \index{isolated spectral dimension}
\index{semifinite spectral triple!isolated spectral dimension}
$b\in\cn$, of the form\footnote{Recall $T^{(n)}=[\D^2,T^{(n-1)}]$; see equation \eqref{eq:iterated-comms}.}
$$
b=a_0 \,da_1^{(k_1)}\cdots da_m^{(k_m)}(1+\D^2)^{-|k|-m/2},\quad a_0,\dots,a_m\in\A,
$$
with $k\in\N_0^m$ a multi-index and $|k|=k_1+\cdots+k_m$, the zeta function\index{zeta function}
\index{$\zeta_b(z)$, zeta functions in index formula}
$
\zeta_b(z):=\tau\big(b(1+\D^2)^{-z}),
$
has an analytic continuation to a deleted neighbourhood of $z=0$.
In this case, we define the numbers
\begin{align}
\label{numbers}
\tau_l(b):=\res_{z=0}\,\,z^l\,\zeta_b(z),\quad  l=-1,0,1,2,\dots.
\end{align}
\end{definition}
\index{$\tau_l$, residue functionals}

{\bf Remark.}  The isolated spectral dimension condition is implied by the much
stronger notion of {\it discrete dimension spectrum}, \cite{CM}. We say that a
smoothly summable spectral triple  $(\A,\cH,\D)$, has \index{discrete dimension spectrum}
\index{semifinite spectral triple!discrete dimension spectrum}
discrete dimension spectrum ${\rm Sd} \subset \mathbb C$, if 
${\rm Sd}$ is a discrete set and for all $b$ in the polynomial algebra
 generated by $\delta^k(a)$ and $\delta^k(da)$, with $a\in\A$ and $k\in\N_0$,
the function $ \zeta_b(z)$ is
defined and holomorphic for $\Re(z)$ large, and analytically
continues to $\mathbb C\setminus{\rm Sd}$.

For  a multi-index $k\in\N_0^m$, we define 
\begin{align}
\label{alpha(k)}
\alpha(k)^{-1}:=k_1!\cdots k_m!(k_1+1)(k_1+k_2+2)\cdots(|k|+m),
\end{align}
\index{$\alpha(k)$, combinatorial factors in index formula}and we let $\sigma_{n,l}$ be the non-negative rational numbers defined by the identities
\begin{align}
\label{sigmanl}
\prod_{l=0}^{n-1}(z+l+\tfrac12)=\sum_{l=0}^n z^l\,\sigma_{n,l},\ \ \mbox{when}\ \  \bullet=1,\quad
\prod_{l=0}^{n-1}(z+l)=\sum_{l=1}^{n}z^l\s_{n,l},\ \ \mbox{when}\ \ \bullet=0.
\end{align}
\index{$\s_{n,l}$, combinatorial factors in index formula}

\begin{definition}
\label{def:res-cocycle}
Assume that $(\A,\HH,\D)$ is a  semifinite smoothly summable spectral triple
with isolated spectral dimension $p\geq 1$.
For $m=\bullet,\,\bullet+2,\dots,M$,
with $\tau_l$ defined in Definition \ref{dim-spec},  and for a multi-index $k$ setting
$h=|k|+(m-\bullet)/2$, the $m$-th component of the
{\bf residue cocycle} $\phi_m:\A\otimes {\A}^{\otimes m}\to{\C}$ is defined by
\begin{align*}
\phi_0(a_0)&=\tau_{-1}(a_0),\\
\phi_m(a_0,\dots,a_m)&=(\sqrt{2i\pi})^\bullet\!\sum_{|k|=0}^{M-m}\!
(-1)^{|k|}\alpha(k)\!\!\!\sum_{l=1-\bullet}^{h}\!\!\sigma_{h,l}\,\tau_{l-1+\bullet}\Big(\gamma
a_0\,da_1^{(k_1)}\cdots da_m^{(k_m)}(1+\D^2)^{-|k|-m/2}\Big),
\end{align*}
for $m=1,\dots, M$.
\end{definition}
\index{residue cocycle}
\index{$\phi_m$, residue cocycle}

\subsubsection{The resolvent cocycle and variations}
In this subsection, we do not assume that our spectral triple $(\A,\H,\D)$ 
has isolated spectral dimension, however several of the cochains defined
here require invertibility of $\D$. The issue of invertibility will be discussed in the
next subsection, and we will show in subsection \ref{invert} how this assumption is removed.

For the invertibility we assume that there exists $\mu>0$ such that $\D^2\geq \mu^2$.
For such an invertible $\D$, we may define   
$$
\D_u:=\D\dd^{-u}\ \ \mbox{for}\ \  u\in[0,1],\ \ \mbox{and for}\ a\in\A,\ \ d_u(a):=[\D_u,a].\index{$\D_u=\D\dd^{-u}$, the deformation from $\D$ to $F=\D\dd^{-1}$}
\index{$d_u=[\D_u,\cdot]$}
$$

Thus $\D_0=\D$ and $\D_1=F$.
Note that  $d_u$ maps $\A$ to $\rm OP^0_0$. 
This follows from   the estimates given in the proof of
Lemma \ref{truc} with $|\D|$ instead of $(1+\D^2)^{1/2}$ when $\D$
is invertible. Note also that the family of derivations $\{d_u,\,u\in[0,1]\}$,
interpolates between the two natural notions of differential  in  
quantised calculus, that is  $d_0a=da=[\D,a]$ and $d_1a=[F,a]$.
We also set 
$$
\dot\D_u:=-\D_u\log\dd,
$$ 
the formal derivative of $\D_u$ 
with respect to the parameter $u\in[0,1]$.
We  define  the shorthand notations
\begin{align}
\label{resol}
& \qquad\qquad\qquad\quad R_{s,t,u}(\lambda)
:=(\lambda-(t+s^2+\D^2_u))^{-1},\\  &R_{s,t}(\lambda):= R_{s,t,0}(\lambda),
\qquad  R_{s,u}(\lambda):= R_{s,0,u}(\lambda),\qquad R_{s}(\lambda)
:=R_{s,1,0}(\lambda).\nonumber\index{resolvent function}\index{$R_{s,t,u}(\lambda)$, see resolvent functions}\index{resolvent function!$R_s(\lambda)=(\lambda-(1+s^2+\D^2))^{-1}$}\index{resolvent function!$R_{s,t}(\lambda)=(\lambda-(t+s^2+\D^2))^{-1}$}\index{resolvent function!$R_{s,t,u}(\lambda)=(\lambda-(t+s^2+\D^2_u))^{-1}$}
\end{align}
The range of the parameters is
$\lambda\in\mathbb C$, with
$0<\Re(\lambda)<\mu^2/2$, $s\in[0,\infty)$, and  $t,u\in[0,1]$.
Recall that for a multi-index $k\in\mathbb N^{m}$, we set 
$|k|:=k_1+\cdots+k_m$.
 
{\bf The parameters $s,\,\lambda$  constitute an essential part of the definition of our cocycles, 
while the parameters $t,\,u$ will be the parameters of homotopies which will eventually
take us from the resolvent cocycle to the Chern character.}

Next we have the analogue of \cite[Lemma 7.2]{CPRS2}. This is the lemma which will 
permit us to demonstrate that the resolvent cococyle introduced below is well defined. We 
refer to the Appendix, subsection \ref{lem:crucial-app},  for the proof of this important but technical result.

\begin{lemma}\label{lem:crucial}
Let $\ell$ be the vertical line $\{a+iv:v\in\R\}$ for some $a\in(0,\mu^2/2)$. Also let
$A_l\in {\rm OP}^{k_l}$, $l=1,\dots,m$ and 
$A_0\in{\rm OP}^{k_0}_0$. For $s>0$, $r\in\C$ and $t\in[0,1]$, 
the operator-valued function\footnote{we define $\lambda^{-r}$ using the principal branch of $\log$.}
$$
B_{r,t}(s)=
\frac{1}{2\pi i}\int_\ell  \lambda^{-p/2-r}A_0\,R_{s,t}(\lambda)\,A_1\,
R_{s,t}(\lambda)\cdots R_{s,t}(\lambda)\,A_m\,R_{s,t}(\lambda)\, d\lambda,
$$
is trace class valued 
for $\Re(r)>-m+|k|/2>0$. Moreover, the function
$[s\mapsto s^\alpha\,\Vert B_{r,t}(s)\Vert_1]$, $\alpha>0$, is integrable on $[0,\infty)$ when 
$\Re(r)>-m+(|k|+\alpha+1)/2$.
\end{lemma}

{\bf Remark.} In Corollary \ref{permut}, 
we will generalize this result to the case where any one of the $A_l$'s belongs to 
${\rm OP}^{k_l}_0$.
From  Lemma \ref{lem:crucial} and Corollary \ref{permut}, 
it follows that the expectations and cochains introduced below
are well-defined, for $\Re(r)$ sufficiently large, 
whenever one of its entries  belongs to   
${\rm OP}^{k_l}_0$.

\begin{definition}\label{expectation} For $a\in(0,\mu^2/2)$, let $\ell$ be the vertical line
$\ell=\{a+iv:v\in\R\}$.
Given $m\in\mathbb N$, $s\in\mathbb R^+$, $r\in\mathbb C$ 
and operators $A_0,\dots,A_m\in {\rm OP}^{k_i}$ with
$A_0\in{\rm OP}^{k_0}_0$, such that   $|k|-2m<2\Re(r)$, we  define
\begin{align}
\label{expect1}
 \langle A_0,\dots,A_m\rangle_{m,r,s,t}&:=\frac{1}{2\pi i}\,\tau\Big(\gamma\int_\ell
\lambda^{-p/2-r}
A_0\,R_{s,t}(\lambda)\cdots A_m\,R_{s,t}(\lambda)\,d\lambda\Big),
\end{align}
\index{$\langle\cdots\rangle_{m,r,s,t}$, expectation in the resolvent cocycle}

Here $\gamma$ is the ${\mathbb Z}_2$-grading in the even case and the identity operator
in the odd case.
When $|k|-2m-1<2\Re(r)$  and when the operators $A_l$ have definite grading degree,
 we use the fact that $\D\in {\rm OP}^1$ to allow us to define
\begin{align}
\label{expect2}
\langle\langle A_0,\dots,A_m\rangle\rangle_{m,r,s,t}
&:=\sum_{l=0}^m(-1)^{{\rm deg}(A_l)}\langle A_0,\dots,A_l,\D,A_{l+1},\dots,A_m\rangle_{m+1,r,s,t}.
\end{align}
\end{definition}
\index{$\langle\langle\cdots\rangle\rangle_{m,r,s,t}$, expectation in the transgression cochain}

We now state the definition of the resolvent cocycle in terms of the
expectations  $\langle\cdot,\dots,\cdot\rangle_{m,r,s,t}$. 

\begin{definition}
\label{resolvent}
For   $m= \bullet,\,\bullet+2,\dots,M$, we introduce the constants  $\eta_m$ by
\index{$\eta_m$, constant in the definition of the resolvent cocycle}
\begin{align*}
 \eta_m&=\left(-\sqrt{2i}\right)^\bullet2^{m+1}\frac{\Gamma(m/2+1)}{\Gamma(m+1)}.
  \end{align*}
Then for $t\in[0,1]$ and $\Re(r)>(1-m)/2$, the $m$-th component of the
{\bf resolvent cocycles} $\phi_m^r,\phi_{m,t}^r:\A\otimes{\A}^{\otimes m}\to{\C}$ are defined by
$ \phi_m^r:=\phi_{m,1}^r$ and
\begin{align}
\phi_{m,t}^r(a_0,\dots,a_m)&:=\eta_m\int_0^\infty s^m\langle
a_0,da_1,\dots,da_m\rangle_{m,r,s,t}\,ds,
\label{eq:resolvent-cocycle}
\end{align}
\end{definition}
\index{$\phi^r_m$, $\phi^r_{m,t}$, see resolvent cocycle}
\index{resolvent cocycle}\index{resolvent cocycle!$\phi^r_m$ defined using resolvent function $R_s(\lambda)$}\index{resolvent cocycle!$\phi^r_{m,t}$ defined using resolvent function $R_{s,t}(\lambda)$}

{\bf Remark}. It is important to note that the resolvent cocycle $\phi^r_m$ is well 
defined even when $\D$ is not invertible. 

Our proof of the local index formula involves constructing cohomologies and homotopies
in the reduced $(b,B)$-bicomplex. This involves the use of `transgression' cochains, as well
as some other auxiliary cochains.

The transgression cochains $\Phi^r_{m,t}$ and auxiliary cochains 
$B\Phi^r_{M+1,0,u},\,\Psi^r_{M,u}$ (see below) are defined  similarly to the resolvent cochains. 
However, the cochains
$\Phi^r_{m,t}$ are of the {\em opposite parity} to $\phi^r_m$. Thus, if we have an even
spectral triple, we will only have $\Phi^r_{m,t}$ with $m$ odd.

\begin{definition}
\label{transgression}
For $t\in[0,1]$, $r\in\C$ with $\Re(r)>(1-m)/2$ and with $\D$ invertible, 
the $m$-th component, $m=1-\bullet,1-\bullet+2,\dots,M+1$, of the
{\bf transgression cochains} 
$\Phi_{m,t}^r:\A\otimes{\A}^{\otimes m}\to{\C}$ are defined by
 \begin{align}
\Phi_{m,t}^r(a_0,\dots,a_m)&:=\eta_{m+1}\int_0^\infty s^{m+1}\langle\langle
a_0,da_1,\dots,da_m\rangle\rangle_{m,r,s,t}\,ds.
\label{eq:transgression-cochain}
\end{align}
By specialising the parameter $t$ to $t=1$, we define
$\Phi_m^r:=\Phi_{m,1}^r$.
\end{definition}
\index{transgression cochain}\index{$\Phi_{m,t}^r$, see transgression cochain}\index{transgression cochain!$\Phi^r_{m}$ defined using resolvent function $R_s(\lambda)$}\index{transgression cochain!$\Phi^r_{m,t}$ defined using resolvent function $R_{s,t}(\lambda)$}\index{transgression cochain!$B\Phi^r_{M+1,0,u}$ defined using resolvent function $R_{s,0,u}(\lambda)$}

Finally we need to consider $B\Phi^r_{M+1,0,u}$ and another auxiliary cochain $\Psi^r_{M,u}$
for $u\neq 0$.
We define $\Psi^r_{M,u}$ below, and the definition of $B\Phi^r_{M+1,0,u}$ is the same 
as $B\Phi^r_{M+1,0}$ with every appearance of $\D$ replaced by $\D_u:=\D|\D|^{-u}$, 
including in the resolvents.

To show that  these cochains are well-defined
when $u\neq 0$
requires additional argument beyond  
power counting and Lemma \ref{lem:crucial}.

We outline the argument briefly, beginning with the case $p\geq 2$. We start from the identity,
$$
d_u(a)=[\D_u,a]=[F|\D|^{1-u},a]=F[|\D|^{1-u},a]+\big(da-F\delta(a)\big)|\D|^{-u},
$$
and we note  that $da-F\delta(a)\in{\rm OP}^0_0$. Applying  the second part of
Lemma \ref{truc} and Lemma \ref{interpolation} 
now shows that $d_u(a)\in\cl^q(\cn,\tau)$ for all $q>p/u$. Next,  we find that
$$
R_{s,u}(\lambda)
=({\lambda-s^2-\D_u^2})^{-1}=|\D|^{-2(1-u)}{\D_u^2}{(\lambda-s^2-\D_u^2)^{-1}}=:
|\D|^{-2(1-u)}B(u),
$$
where $B(u)$ is uniformly bounded.  Then Lemma \ref{interpolation}  
and the H\"older inequality show that 
$d_u(a_i)\,R_{s,u}(\lambda)\in\cl^q(\cn,\tau)$ for all $q$ with $(2-u)q>p\geq2$ and
$i=0,\dots,l,l+2,\dots,M$, while  
$R_{s,u}(\lambda)^{1/2}\,d_u(a_{l+1})\,R_{s,u}(\lambda)\in\cl^q(\cn,\tau)$ for all $q$ with
$(3-2u)q>p\geq2$. An application of  the H\"{o}lder inequality now shows that 
$B\Phi^r_{M+1,0,u}$ is well-defined. 
To see that $\Psi_{M,u}^r$ is well-defined requires the arguments above, 
as well as Lemma \ref{truc} to deal with the extra $\log(|\D|)$ factor appearing in $\dot{\D}_u$. 
More details can be found in the proof of
Lemma \ref{diff1} in subsection \ref{lem:diff1-app}. For $2>p\geq 1$ the algebra is a little
more complicated, and we again refer to the proof of
Lemma \ref{diff1} in subsection \ref{lem:diff1-app} for more details.

\begin{definition}
\label{tauxiliaries}
For $t\in[0,1]$, $r\in\C$ with $\Re(r)>(1-M)/2$ and  with $\D$ invertible, the
{\bf auxiliary cochain} 
$\Psi_{M,u}^r:\A\otimes{\A}^{\otimes M}\to{\C}$ 
is defined by  
\begin{align}
\Psi^r_{M,u}(a_0,\dots,a_M)&:=-\frac{\eta_M}{2}\int_0^\infty s^M\langle\langle
a_0\dot{\D}_u,d_u(a_1),\dots,d_u(a_M)\rangle\rangle_{M,r,s,0} \,ds,
\label{eq:auxiliary-cochain}
\end{align}
where the expectation uses the resolvent $R_{s,t,u}(\lambda)$ for $\D_u$.
\end{definition}
\index{auxiliary cochain}
\index{$\Psi_{M,u}^r$, auxiliary cochain}

These are all the cochains that will appear in our homotopy arguments
connecting the resolvent and residue cocycles to the Chern character.
However, we still need to ensure that we can extend all these 
cochains to $\A^\sim\otimes\A^{\otimes m}$, in such a way that we obtain reduced cochains. 
This extension must also allow us
to remove the invertibility assumption on $\D$ when we reach the end of the argument.
We deal with these two related issues next.




\subsection{The double construction, invertibility and reduced cochains}
\label{subsec:reduced}

The  cochains  $\phi_{m,t}^r$, $B\Phi^r_{m,t,u}$ and $\Psi^r_{M,u}$
require the invertibility of $\D$ for $u\neq 0$  and $t=0$. 
Thus we will need to assume the invertibility of $\D$ for
the main part of our proof, and show how to remove the assumption at the end. 

More importantly, we need to know that all our cochains and cocycles lie in the reduced
$(b,B)$-bicomplex. The good news is that the same mechanism we employ to deal with 
invertibility also ensures that our homotopy to the Chern character takes place within the
reduced bicomplex.

The mechanism we employ is the double spectral triple $(\A,\H^2,\D_\mu,\hat\gamma)$
(see Definition \ref{def:double}), 
with invertible operator $\D_\mu$.  We know that this spectral triple defines the same
index pairing with $K_*(\A)$ as $(\A,\HH,\D,\gamma)$. 
Now we show how the various cochains associated to the double spectral triple extend 
naturally to $\A^\sim\otimes\A^{\otimes m}$. Recall that this is really only an issue when $m=0$,
and in particular does not affect any odd cochains.

To distinguish
the residue and resolvent cocycles associated with the double spectral triple $(\A,\H^2,\D_\mu,\hat\gamma)$, we use for them the notations $\phi_{\mu,m}$, $\phi_{\mu,m}^r$,
and similarly for the other cochains.

Let $\overline{\rm OP^0_0}$ be the $C^*$-closure of $\rm OP^0_0$ 
({\em defined using the 
operator $\D_\mu$}!), and let 
$\{\psi_\lambda\}_{\lambda\in \Lambda}\subset \rm OP^0_0$ 
be a  net forming an approximate unit for $\overline{\rm OP^0_0}$.
Such an approximate unit always  exists by the density of $\rm OP^0_0$. 
In terms of the two-by-two matrix picture of our doubled spectral triple, 
we can suppose that there is an approximate unit $\{\widetilde{\psi}_\lambda\}_{\lambda\in\Lambda}$ 
for the $\rm OP^0_0$
algebra defined by $\D$ (rather than $\D_\mu$) such that
$\psi_\lambda=\tilde\psi_\lambda\otimes{\rm Id}_2$.
Then we define  for $m>0$ and $c_0,c_1,\dots,c_m\in\C$
\begin{equation}
\label{eq:extends}
\phi_{\mu,m}(a_0+c_0{\rm Id}_{\A^\sim},a_1+c_1{\rm Id}_{\A^\sim},\dots,a_m+c_m{\rm Id}_{\A^\sim})
:=\phi_{\mu,m}(a_0+c_0,a_1,\dots,a_m).
\end{equation}
This makes sense as the residue cocycle is already normalised.

For $m>0$ this is well-defined since 
$[\D_\mu,\hat a_1]^{(k_1)}\cdots[\D_\mu,\hat a_m]^{(k_m)}(1+\D_\mu^2)^{-|k|/2}\in 
{\rm OP}^0_0$, by Lemma \ref{lem:derivs}.
Then by definition of isolated spectral dimension, we see that for $m>0$ the 
components of the residue cocycle take finite values on 
$\A^\sim\otimes{\A}^{\otimes m}$. 

For $m=\bullet=0$, we define 
\begin{align*}
&\phi_{\mu,0}(1_{\A^\sim}):=
\lim_{\lambda\to\infty} \mbox{res}_{z=0}\,\frac{1}{z}\,\tau\otimes \tr_2
\begin{pmatrix}\gamma \tilde\psi_\lambda(1+\mu^2+\D^2)^{-z} & 0 \\
0 & -\gamma \tilde\psi_\lambda(1+\mu^2+\D^2)^{-z}\end{pmatrix}=0.
\end{align*}
Thus this extension of the residue cocycle for $\D_\mu$ defines a reduced cochain for $\A$.

The resolvent cochains $\phi^r_{\mu,m}$, $m=\bullet,\bullet+2,\dots$, are normalised cochains
by definition. We extend all of these cochains to $\A^\sim\otimes\A^{\otimes m}$ just as 
we did for the residue cocycle in Equation \eqref{eq:extends}. The resulting cochains are then
reduced cochains. For 
$\Psi^r_{\mu,M,u}$ and $B\Phi^r_{\mu,M+1,0,u}$ there is no issue since
$M\geq 1$ in all cases. 

For $\Phi^r_{\mu,m,t}$ the
situation is different as we will employ an even version of 
$\Phi$ when $\bullet=1$, and so there is no grading.
However, when $m=0$ we can perform the Cauchy 
integral in the definition of $\Phi_{\mu,0,t}^r$, and
so we obtain for $\Re(r)>1/2$ a constant $C$ such that
\begin{align*}
&\Phi_{\mu,0,t}^r(1_{\A^\sim}):=
\lim_{\lambda\to\infty}C\int_0^\infty  s\ 
\tau\otimes \tr_2
\left(\begin{pmatrix} \tilde{\psi}_\lambda & 0\\ 0 & \tilde{\psi}_\lambda\end{pmatrix}
\begin{pmatrix} \D & \mu\, \\
\mu\, & - \D\end{pmatrix}\right)(t+\mu^2+s^2+\D^2)^{-p/2-r}\,ds=0.
\end{align*}

These arguments prove the following:

\begin{lemma}
\label{chien}
Let $t\in[0,1]$ and $r\in\C$. Provided $\Re(r)>(1-m)/2$, the components of the residue $(\phi_{\mu,m})_{m=\bullet,\bullet+2,\dots,M}$,
the resolvent cochain $(\phi_{\mu,m,t}^r)_{m=\bullet,\bullet+2,\dots,M}$, the  transgression cochain
$(\Phi_{\mu,m,t}^r)_{m=1-\bullet,1-\bullet+2,\dots,M+1}$ and the auxiliary cochains
 $\Psi^r_{\mu,M,u}$ and
$B\Phi^r_{\mu,M+1,0,u}$ 
 are finite on  $\A^\sim\otimes\A^{\otimes m}$, and moreover define cochains
in the reduced $(b,B)$-bicomplex for $\A^\sim$.
\end{lemma}

Thus all the relevant cochains defined using the double live in the reduced
bicomplex for $\A^\sim$, and $\D_\mu$ is invertible. For the central part of 
our proof, from subsection \ref{subsec:properties} until the beginning of subsection \ref{invert},
we shall simply assume that our smoothly summable spectral triple $(\A,\HH,\D)$ has $\D$ invertible
with $\D^2\geq \mu^2>0$. In subsection \ref{invert} we will complete the proof by 
relating cocycles for the double, for which our arguments are valid, to cocycles
for our original spectral triple.

\subsection{Algebraic properties of the expectations}
\label{subsec:properties}

Here we develop some of the properties of the expectations given in
Definition \ref{expectation}. These properties are the same as those stated in \cite{CPRS4}, 
but some of the proofs require extra care in the nonunital setting. 

We refer to the following two lemmas as the 
$s$-trick and the $\lambda$-trick, respectively.  Their proofs are given in the 
Appendix, subsections \ref{lem:s-trick-app} and \ref{lem:lambda-trick-app} respectively. 
Both the $s$-trick and the $\lambda$-trick provide a way of integrating by parts. 
Unfortunately, justifying these tricks is somewhat technical.

Formally, the $s$-trick follows by integrating $\frac{d}{ds}\left(s^\alpha\langle\cdot,\dots,\cdot\rangle_{m,r,s,t}\right)$
and using the fundamental Theorem of calculus.
\begin{lemma}
\label{s-trick}
Let $m\in\mathbb N$, $\alpha>0$, $t\in[0,1]$ and $r\in\mathbb C$ 
such that $2\Re(r)>1+\alpha+|k|-2m$.
Also let
$A_l\in {\rm OP}^{k_l}$, $l=1,\dots,m$ and  $A_0\in{\rm OP}^{k_0}_0$.
 Then
$$
\alpha\int_0^\infty s^{\alpha-1}\langle A_0,\dots,A_m\rangle_{m,r,s,t}\,ds=-2\sum_{l=0}^m
\int_0^\infty s^{\alpha+1}\langle A_0,\dots,A_l,{\rm Id}_\cn,A_{l+1},\dots,A_m\rangle_{m+1,r,s,t}\,ds,
$$
and if $2\Re(r)>\alpha+|k|-2m$ then
$$
\alpha\!\!\int_0^\infty \!\!s^{\alpha-1}\langle\langle A_0,\dots,A_m\rangle\rangle_{m,r,s,t}\,ds\!
=\!-2\sum_{l=0}^m
\int_0^\infty \!\!s^{\alpha+1}
\langle\langle A_0,\dots,A_l,{\rm Id}_\cn,A_{l+1},\dots,A_m\rangle\rangle_{m+1,r,s,t}\,ds.
$$
\end{lemma}
\index{$s$-trick}

Differentiating the $\lambda$-parameter under the Cauchy integral, we obtain in a similar manner:

\begin{lemma}
\label{lambda-trick}
Let  $m\in\mathbb N$, $\alpha>0$, $t\in[0,1]$, $s>0$ and $r\in\mathbb C$ such that $2\Re(r)>|k|-2m$. 
Let also
$A_l\in {\rm OP}^{k_l}$, $l=1,\dots,m$ and  $A_0\in{\rm OP}^{k_0}_0$.
Then
$$
-(p/2+r)\langle A_0,\dots,A_m\rangle_{m,r+1,s,t}=\sum_{l=0}^m
\langle A_0,\dots,A_l,{\rm Id}_\cn,A_{l+1},\dots,A_m\rangle_{m+1,r,s,t},
$$
and if $2\Re(r)>|k|-2m-1$ then
$$
-(p/2+r)\langle\langle A_0,\dots,A_m\rangle\rangle_{m,r+1,s,t}=\sum_{l=0}^m
\langle\langle A_0,\dots,A_l,{\rm Id}_\cn,A_{l+1},\dots,A_m\rangle\rangle_{m+1,r,s,t}.
$$
\end{lemma}
\index{$\lambda$-trick}

\begin{corollary}
\label{permut}
Let $A_l\in{\rm OP}^{k_l}$ have definite grading degree, and suppose that there exists 
$l_0\in\{0,\dots,m\}$ with $A_{l_0}\in{\rm OP}^{k_{l_0}}_0$. 
Then, for $\Re(r)$ sufficiently large and with $1-\bullet$ the anti-parity,
 the signed  expectations
$$
(-1)^{(1-\bullet)\,\sum_{k=l}^m{\rm deg}(A_k)}
\langle A_l,A_{l+1},\dots,A_0,\dots,A_m,\dots ,A_{l-1}\rangle_{m,r,s,t},
\quad l=0,\dots,m,
$$
are all finite and coincide, and similarly for the expectations \eqref{expect2}.
In particular, Lemmas \ref{lem:crucial}, \ref{s-trick} and \ref{lambda-trick} remain 
valid if one assumes instead that $A_{l}\in{\rm OP}^{k_{l}}_0$, 
for any $l\in\{0,\dots,m\}$.
\end{corollary}

\begin{proof}
Formally, the proof is to integrate by parts until the integrand is trace-class, and then
apply cyclicity of the trace. To make such a formal argument rigorous, we employ the
$\lambda$-trick.

We assume first $A_{0}\in{\rm OP}^{k_{0}}_0$. From the 
same reasoning as at the beginning of the proof of  Lemma \ref{lem:crucial}, one can 
further assume that $A_m\in{\rm OP}^0$, at the price that 
$A_{m-1}$ will be in ${\rm OP}^{k_{m-1}+k_m}$.
Then, we repeat the $\lambda$-trick (Lemma \ref{lambda-trick}) until the integrand of
$$
\langle A_0,1,\dots,1,A_2,1,\dots,1,A_m,1,\dots,1\rangle_{M+1,r,s,t},
$$
is trace class. We then move the bounded (by \cite[Lemma 6.10]{CPRS2}, see the Appendix 
Lemma \ref{normtrick})
operator $R^{-k}A_mR^{k}$ 
($k$ is the number of resolvents on the right of $A_m$) to the front, using the trace property.
This gives after recombination
$$
\langle A_0,\dots,A_m\rangle_{m,r,s,t}
=(-1)^{(1-\bullet)\,{\rm deg}(A_m)}\langle A_m, A_0,\dots,A_{m-1}\rangle_{m,r,s,t}.
$$
The sign comes from the relation $A_m\gamma=(-1)^{(1-\bullet)\,{\rm deg}(A_m)}\gamma A_m$. 
One concludes iteratively. The proof for the expectations \eqref{expect2} is entirely similar.
\end{proof}

We quote several results from \cite{CPRS4} which carry over to our setting with no
substantial change in their proofs. 

\begin{lemma}\label{basicidentities} 
Let $m\geq 0$, $A_0,\dots,A_m$, 
$A_i\in {\rm OP}^{k_i}$, with definite grading degree and  with $|k|-2m-1<2\Re(r)$, 
and suppose there exists $l\in\{0,\dots,m\}$ with $A_{l}\in{\rm OP}^{k_{l}}_0$. 
Then for $1\leq j<m$ we have
\begin{align*} 
&-\langle A_0,\dots,[\D^2,A_j],\dots,A_m\rangle_{m,r,s,t}\\
&\qquad\qquad=\langle A_0,\dots,A_{j-1}A_j,\dots,A_m\rangle_{m-1,r,s,t}-\langle 
A_0,\dots,A_jA_{j+1},\dots,A_m\rangle_{m-1,r,s,t},
\end{align*}
while for $j=m$ we have
\begin{align*}
&-\langle A_0,\dots,A_{m-1},[\D^2,A_m]\rangle_{m,r,s,t}\\
&\qquad\qquad=\langle A_0,\dots,A_{m-1}A_m\rangle_{m-1,r,s,t}
-(-1)^{(1-\bullet)\,{\rm deg}(A_m)}\langle A_mA_0,\dots,A_{m-1}\rangle_{m-1,r,s,t}.
\end{align*}
For $k\geq 1$ we have
\begin{equation}
 \int_0^{\infty}s^k\langle\D A_0,A_1,\dots,A_m\rangle_{m,r,s,t}ds=
(-1)^{1-\bullet} \int_0^{\infty}s^k\langle A_0,A_1,\dots,A_m\D\rangle_{m,r,s,t}ds.
\label{traceofcommutator}
\end{equation}
If furthermore $\sum_{i=0}^m {\rm deg}(A_i)\equiv 1-\bullet\,(\bmod\ 2)$, we define 
$$
{\rm deg}_{-1}=0\;\;and\;\;{\rm deg}_k={\rm deg}(A_0)+{\rm deg}(A_1)+\cdots +{\rm deg}(A_k),
$$ 
then
\begin{equation} 
\sum_{j=0}^m(-1)^{{\rm deg}_{j-1}}\int_0^{\infty}s^k\langle A_0,\dots,[\D,A_j]_\pm,\dots,
A_m\rangle_{m,r,s,t}ds=0.
\label{addstozero}
\end{equation}
\end{lemma}

\begin{lemma} 
\label{lala-la-identity}
Let $m\geq 0$, $A_0,\dots,A_m$, $A_i\in {\rm OP}^{k_i}$, 
with definite grading degree and
with $|k|-2m-2<2\Re(r)$, and suppose there exists 
$l\in\{0,\dots,m\}$ with $A_{l}\in{\rm OP}^{k_{l}}_0$. Then 
for $1\leq j < m$ we have the identity
\begin{align} &-\langle\langle A_0,\dots,[\D^2,A_j],\dots,A_m\rangle\rangle_{m,r,s,t}
-(-1)^{{\rm deg}_{j-1}}\langle 
A_0,\dots,[\D,A_{j}]_\pm,\dots,A_m\rangle_{m,r,s,t}\nonumber\\
 &=\langle\langle A_0,\dots,A_{j-1}A_j,\dots,A_m\rangle\rangle_{m-1,r,s,t}-\langle\langle 
A_0,\dots,A_jA_{j+1},\dots,A_m\rangle\rangle_{m-1,r,s,t}.
\label{useforb}
\end{align}
For $j=m$ we also have 
\begin{align*} 
&-\langle\langle A_0,\dots,A_{m-1},[\D^2,A_m]\rangle\rangle_{m,r,s,t}
-(-1)^{{\rm deg}_{m-1}}\langle 
A_0,\dots,[\D,A_{m}]_\pm\rangle_{m,r,s,t}\\
 &=\langle\langle A_0,\dots,A_{m-1}A_m\rangle\rangle_{m-1,r,s,t}-(-1)^{\bullet {\rm deg}(A_m)}
 \langle\langle 
A_mA_0,\dots,A_{m-1}\rangle\rangle_{m-1,r,s,t}.\nonumber\\
\end{align*}
If $\sum_{i=0}^m{\rm deg}(A_i)\equiv \bullet\,(\bmod\ 2)$ and $\alpha\geq 1$, then 
we also have 
\begin{align} 
&\sum_{k=0}^m(-1)^{{\rm deg}_{k-1}}\!\int_0^{\infty}\!\!s^\alpha\langle\langle 
A_0,\dots,[\D,A_k]_\pm,\dots,A_m\rangle\rangle_{m,r,s,t}ds\nonumber\\
&\qquad\qquad\qquad=\sum_{i=0}^m2\!\int_0^{\infty}\!\!s^\alpha\langle 
A_0,\dots,A_i,\D^2,\dots,A_m\rangle_{m+1,r,s,t}ds.
\label{nothertrick}
\end{align}
On the other hand, if $\sum_{i=0}^m{\rm deg}(A_i)\equiv 1-\bullet\,(\bmod\ 2)$ and $\alpha\geq 1$
then $\langle\langle\cdots\rangle\rangle$ satisfies the cyclic property
$$
\int_0^\infty s^\alpha\langle\langle A_0,\dots,A_m\rangle\rangle_{m,r,s,t}ds
=(-1)^{\bullet\, {\rm deg}(A_m)}
\int_0^\infty s^\alpha\langle\langle A_m,A_0,\dots,A_{m-1}\rangle\rangle_{m,r,s,t}ds.
$$ 
\end{lemma}

From these various algebraic identities and $\D^2R_{s,t}(\lambda)=-1+(\lambda-(t+s^2))R_{s,t}(\lambda)$
we deduce the following important relationship between powers of $\D$ and 
the values of our parameters.

\begin{lemma}
\label{differentfort} 
Let $m,\alpha\geq 0$, $A_i\in 
{\rm OP}^{k_i}$, with definite grading degree, $r\in\C$ be such that $2\Re(r)>1+\alpha-2m+|k|$, and suppose there exists 
$l\in\{0,\dots,m\}$ with $A_{l}\in{\rm OP}^{k_{l}}_0$. Then
\begin{align*} 
&\sum_{j=0}^m\int_0^\infty s^\alpha\langle 
A_0,\dots,A_j,\D^2,A_{j+1},\dots,A_m\rangle_{m+1,r,s,t}ds\\
&=-(m+1)\int_0^\infty s^\alpha\langle 
A_0,\dots,A_m\rangle_{m,r,s,t}ds+(1-p/2-r)\int_0^\infty s^\alpha\langle 
A_0,\dots,A_m\rangle_{m,r,s,t}ds\\
&+\frac{(\alpha+1)}{2}\int_0^\infty s^\alpha\langle A_0,\dots,A_m\rangle_{m,r,s,t}ds
-t\,\sum_{j=0}^m\int_0^\infty s^\alpha\langle 
A_0,\dots,A_j,1,A_{j+1},\dots,A_m\rangle_{m+1,r,s,t}ds.
\end{align*}
\end{lemma}

\subsection{Continuity of the resolvent, transgression and auxiliary cochains} 
\label{subsec:reg}
In this subsection, we demonstrate  the continuity, differentiability and holomorphy 
properties, allowing us to prove that the resolvent cocycle represents the Chern character.

\begin{definition}
We let $\mathcal O_m$ be the set of holomorphic functions on the open
half-plane $\{z\in\mathbb C:\Re(z)>(1-m)/2\}$. We endow $\mathcal O_m$
with the topology of uniform convergence on compacta.
\end{definition}
\index{$\mathcal{O}_m$, functions holomorphic for $\Re(z)>(1-m)/2$}

\begin{lemma}
\label{hoo}
Let $m=\bullet,\bullet+2,\dots,M$ and $t\in[0,1]$.
For $A_0,\dots,A_m\in{\rm OP}^0$ such that there exists $l\in\{0,\dots,m\}$ with 
$A_{l}\in{\rm OP}_0^0$, we have
$$
\left[r\mapsto \int_0^\infty s^m\langle A_0,\dots,A_m\rangle_{m,r,s,t}\,ds\right],
\left[r\mapsto \int_0^\infty s^{m+1}\langle\langle A_0,\dots,A_m\rangle\rangle_{m,r,s,t}\,ds\right]
\in {\mathcal O}_m.
$$
\end{lemma}

\begin{proof}
We prove a stronger result, namely that the operator-valued function
$$
B_{r,t}(s,\eps)=
\frac{1}{2\pi i}\int_\ell \lambda^{-p/2-r}
\Big(\eps^{-1}({\lambda^{-\eps}-1})+\log\lambda\Big)
A_0\,R_{s,t}(\lambda)\,A_1\,R_{s,t}(\lambda)\cdots R_{s,t}(\lambda)\,A_m\,R_{s,t}(\lambda)\, 
d\lambda,
$$
satisfies $\lim_{\eps\to 0}\int_0^\infty s^m\|B_{r,t}(s,\eps)\|_1ds=0$, whenever $\Re(r)>(1-m)/2$. 
(Here $\ell$ is the vertical line $\ell=\{a+iv:v\in\R\}$ with $0<a<\mu^2/2$.) 

By Corollary \ref{permut}, we can assume that $A_{0}\in{\rm OP}_0^0$. 
The proof then follows by a minor modification of the arguments of the proof of 
Lemma \ref{lem:crucial}  (see the Appendix Section  \ref{lem:crucial-app}), 
so that we only outline it. (We use the shorthand 
notation $R:=R_{s,t}(\lambda)$.)

We start by writing for any $L\in\N$, using Lemma \ref{firstexpan} (see \cite[Lemma 6.11]{CPRS2})
$$
A_0\,R\,A_1\,R\cdots R\,A_m\,R
=\sum_{|n|=0}^LC(n) A_0\,A_1^{(n_1)}\cdots A_m^{(n_m)}\,R^{m+|n|+1}+A_0\,P_{L,m},
$$
with $P_{L,m}\in{\rm OP}^{-2m-L-3}$.
The conclusion for the remainder term follows then from the estimate
$$
\left| \lambda^{-p/2-r}\Big(\eps^{-1}({\lambda^{-\eps}-1})
+\log(\lambda)\Big)\right| \leq C\,|\eps|\,|\lambda|^{-p/2-\Re(r)},
$$
together with the same techniques as those used in  the proof of 
Lemma \ref{lem:crucial}. A more detailed
account can be found in \cite[Lemma 7.4]{CPRS2}.

For the non-remainder terms, we perform the Cauchy integrals
\begin{align*}
&\hspace{-0,3cm}\frac{1}{2\pi i}\int_\ell \lambda^{-p/2-r} 
\Big(\eps^{-1}({\lambda^{-\eps}-1})+\log\lambda\Big)A_0A_1^{(n_1)}\cdots A_m^{(n_m)}R^{m+1+|n|}
d\lambda\\
&=
(-1)^{m+|n|}\frac{\Gamma(p/2+r+m+|n|)}{\Gamma(p/2+r)}
A_0A_1^{(n_1)}\cdots A_m^{(n_m)}(t+s^2+\D^2)^{-p/2-r-m-|n|}\\
&\qquad\qquad\qquad\qquad\qquad\qquad
\times\left(\eps^{-1}({(t+s^2+\D^2)^{-\eps}-1})+\log\big(t+s^2+\D^2)\right)\\
&\quad+\sum_{k=0}^{m+|n| -1}\binom{m+|n|}{k}(-1)^{m+|n|}\frac{\Gamma(p/2+r+k)}{\Gamma(p/2+r)}
A_0A_1^{(n_1)}\cdots A_m^{(n_m)}(t+s^2+\D^2)^{-p/2-r-m-|n|}\\
&\qquad\qquad\qquad\qquad\qquad\qquad\times
\Big(\frac{\Gamma(\eps+m+|n|-k)}{\Gamma(\eps+1)}(t+s^2+\D^2)^{-\eps}
-\Gamma(m+|n|-k)\Big).
\end{align*}
Let $\rho>0$ such that $\Re(z)>(1-m)/2+\rho$. Call  $T_k(s)$ the terms with no logarithm.
Using the estimates of Lemma \ref{lem:crucial} and  
$$
(t+s^2+\D^2)^{-\rho}\big((t+s^2+\D^2)^{-\eps}- 1\big)\to 0\quad \mbox{as}\quad \eps\to 0,
$$ 
in norm,
we see that 
$\lim_{\eps\to 0}\int_0^\infty s^m \,\Vert T_k(s)\Vert_1\,ds=0$.
For the  first term (with a logarithm), one concludes using the fact that for any $\rho>0$
$$
\left\|(t+s^2+\D^2)^{-\rho}\left(\frac{{(t+s^2+\D^2)^{-\eps}-1}}{\eps}+\log\big(t+s^2+\D^2)\right)\right\|\leq C \,\eps,
$$
where the constant $C$ is independent of $s$ (and of $t$).
\end{proof}

We finally arrive at the main result of this subsection.

\begin{prop}
\label{hii}
Let $m=\bullet,\bullet+2,\dots,M$ for the resolvent cocycle,  
$m=1-\bullet,1-\bullet+2,\dots,M+1$ for the transgression cochain, and $t\in[0,1]$.
The maps
 $$
 a_0\otimes\dots\otimes a_m\mapsto\big[r\mapsto\phi_{m,t}^r(a_0,\dots,a_m)\big],\quad
 a_0\otimes\dots\otimes a_m\mapsto\big[r\mapsto\Phi_{m,t}^r(a_0,\dots,a_m)\big],
 $$
 are continuous multilinear maps from $ \A\otimes\A^{\otimes m}$ to $\mathcal O_m$.
\end{prop}
\begin{proof}
We only give the proof for the resolvent cocycle, 
the case of the transgression cochain being similar. So
let us first fix  $r\in \C$ with $\Re(r)>(1-m)/2$. 
Since Lemma \ref{chien} ensures that our functionals are
finite for these values of $r$, all that we need to do 
is to improve the estimates of Lemma \ref{lem:crucial} 
to prove continuity. We do this using the $s$- and $\lambda$-tricks. 
We recall that we have defined $M=2\lfloor(p+\bullet+1)/2\rfloor-\bullet$ 
(which is the biggest integer of parity $\bullet$ less than or equal to  $p+1$). 
By applying successively the $s$- and $\lambda$-tricks (which commute) $(M-m)/2$ times, 
we obtain 
\begin{align}
\phi_{m,t}^r(a_0,\dots,a_m)&=
2^{(M-m)/2}(M-n)!\prod_{l_1=1}^{(M-m)/2}\frac1{p/2+r-l_1}\prod_{l_2=1}^{(M-m)/2}\frac1{m+l_2}\nonumber\\
&\hspace{1.8cm}\times\sum_{|k|=M-m}
\int_0^\infty s^M
\langle a_0,1^{k_0},da_1,1^{k_1},\dots,da_m,1^{k_m}\rangle_{M,r-(M-m)/2,s,t}ds,
\label{haa}
\end{align}
where $1^{k_i}=1,1,\dots,1$ with $k_i$ entries. Since $M\leq p+1$, the poles 
associated to the prefactors are outside the region $\{z\in\mathbb C:\Re(z)>(1-m)/2\}$. 
Ignoring the prefactors, setting $n_i=k_i+1$ and $R:=R_{s,t}(\lambda)$, we need to deal 
with the integrals
$$
\int_0^\infty s^M\tau\Big( \gamma\int_\ell 
\lambda^{-p/2-r-(M-m)/2}a_0R^{n_0}da_1R^{n_1}\cdots da_mR^{n_m}d\lambda\Big)ds,
\qquad |n|=M+1,
$$
where $\ell$ is the vertical line $\ell=\{a+iv:v\in\R\}$ with $a\in(0,\mu^2/2)$.
Let $p_{l}:=(M+1)/n_l$, so that $\sum_{l=0}^m p_l^{-1}=1$. The H\"older inequality gives
$$
\|a_0R^{n_0}da_1R^{n_1}\cdots da_mR^{n_m}\|_1\leq
\|a_0R^{n_0}\|_{p_0}\|da_1R^{n_1}\|_{p_1}\cdots\|da_mR^{n_m}\|_{p_m}.
$$
By Lemma \ref{lem:was-5.3}, we obtain for $\eps>0$, and 
with $A_0=a_0$, $A_l=da_l$, $l=1,\dots,m$,
$$
\|A_lR^{n_l}\|_{p_l}\leq 
\Vert A_l(\D^2-\mu^2/2)^{-(p/p_l+\eps/(m+1))/2}\Vert_{p_l}
((s^2+a)^2+v^2)^{-n_l/2+(p/p_l+\eps/(m+1))/4}.
$$
Since $\sum_{l=0}^m n_l=M+1$, this gives
\begin{equation}
\|a_0R^{n_0}da_1R^{n_1}\cdots da_mR^{n_m}\|_1
\leq C(a_0,\dots,a_m)\,\,((s^2+a)^2+v^2)^{-(M+1)/2+(p+\eps)/4},
\label{eq:cts-est}
\end{equation}
which is enough to show the absolute convergence of the iterated integrals 
(see \cite[Lemma 5.4]{CPRS2}). 
Now observe that the constant in Equation \eqref{eq:cts-est} is equal to 
$$
C(a_0,\dots,a_m)=\Vert a_0(\D^2-\mu^2/2)^{-(p/p_0+\eps/(m+1))/2}\Vert_{p_0}\cdots
\Vert da_m(\D^2-\mu^2/2)^{-(p/p_m+\eps/(m+1))/2}\Vert_{p_m}.
$$
Note also that the explicit interpolation inequality of Lemma \ref{interpolation} reads
$$
\|A(\D^2-\mu^2/2)^{-\alpha/2}\|_q\leq
\|A(\D^2-\mu^2/2)^{-\alpha q/2}\|_1^{1/q}\,\|A\|^{1-1/q},\quad 
A\in{\rm OP}_0^0,\quad q>p/\alpha,
$$
and the latter is bounded by $\PP_{n,k}(A)$ for $n=\lfloor(\alpha q-p)^{-1}\rfloor$ and 
$k=3\lfloor \alpha q/4\rfloor+1$, by a simultaneous application of Lemma \ref{cont-grp}  and Corollary 
\ref{consistency}.
Thus, with the same notations as above, we find for $l\neq 0$ and some constant $C>0$
\begin{align*}
\Vert da_l(\D^2-\mu^2/2)^{-(p/p_l+\eps/(m+1))/4}\Vert_{p_l}&\leq
\Vert da_l(\D^2-\mu^2/2)^{-(p+p_l\eps/(m+1))/4}\Vert_2^{1/p_l}\|da_l\|^{1-1/p_l}\\
&\leq C\,\PP_{n,k}(da_l),
\end{align*}
for suitable $n,k\in\mathbb N$.
For $l=0$ we have a similar but easier calculation.
This proves the joint continuity of the resolvent cocycle for  the $\delta$-$\vf$-topology.

The proof that the map $r\mapsto \phi_{m,t}^r(a_0,\dots,a_m)$ is holomorphic in the region 
$\Re(r)>(1-m)/2$ follows from Lemma \ref{hoo}.
\end{proof}

\begin{prop}\label{generalt} For each $m=\bullet,\bullet+2,\dots,M$, the map
$$[0,1]\ni t\mapsto\big[r\mapsto \phi^r_{m,t}\big]\in {\rm Hom}(\A^{\otimes m+1},\mathcal O_m),$$
is continuously differentiable and
$$\frac{d}{dt}\Big[t\mapsto\big[r\mapsto\phi^r_{m,t}\big]\Big]
=\Big[t\mapsto\big[r\mapsto-(q/2+r)\,\phi^{r+1}_{m,t}\big]\Big].$$
\end{prop}

\begin{proof} We do the case $m<M$ where we must use some initial trickery
to reduce to a computable situation. For $m= M$ such tricks are not needed.
We proceed as in the proof of Proposition \ref{hii}, applying the $s$- and $\lambda$- tricks
to obtain \eqref{haa}. 
Keeping the same notations as in the cited proposition, in particular $p_i=(M+1)/n_i$, 
and ignoring the prefactors, 
we are left with the integrals
$$
\int_0^\infty s^M\tau\Big( \gamma\int_\ell \lambda^{-p/2-r-(M-m)/2}
a_0\,R_{s,t}^{n_0}\,da_1\,R_{s,t}^{n_1}\cdots da_m\,R_{s,t}^{n_m}\,d\lambda\Big)ds.
$$
(Here $\ell$ is the vertical line $\ell=\{a+iv:v\in\R\}$ with $0<a<\mu^2/2$.)
Now each integrand is not only
trace class, but  also $t$-differentiable in trace norm. This is a consequence of the product rule, 
H\"{o}lder's inequality
and the following argument showing the Schatten norm differentiability of 
$AR^{n}_{s,t}$ for $A\in {\rm OP}^0_0$.
By adding and substracting suitable terms, the resolvent identity gives
\begin{align*}
A\Big(\varepsilon^{-1}(R_{s,t+\varepsilon}^n-R_{s,t}^n)+nR_{s,t}^{n+1}\Big)
&=nAR_{s,t}^n\Big(R_{s,t}-\frac{1}{n}\sum_{k=1}^nR_{s,t}^{-k+1}R_{s,t+\varepsilon}^k\Big).
\end{align*}
The term in brackets on the right hand side converges to zero in operator norm since 
$R_{s,t}^{-k+1}R_{s,t+\varepsilon}^{k-1}$ is uniformly bounded. Thus 
$$
\Vert A\left(\varepsilon^{-1}(R_{s,t+\varepsilon}^n-R_{s,t}^n)+nR_{s,t}^{n+1}\right)\Vert_{p}
\leq \Vert nAR_{s,t}^n\Vert_p\ 
\Big\Vert R_{s,t}-\frac{1}{n}\sum_{k=1}^nR_{s,t}^{-k+1}R_{s,t+\varepsilon}^k\Big\Vert\to 0,\quad\eps\to0.
$$
Choosing $A=a_0$ or $A=da_i$ and $p=p_0$ or $p=p_i$ respectively proves the differentiability
of each term $AR_{s,t}^n$ in the integrand in the appropriate $p$-norm, and so an 
application of H\"{o}lder's inequality completes the proof of trace norm differentiability.

The existence of the integrals can now be deduced from the formula for the derivative of
the integrand and Lemma \ref{lem:crucial}.

This proves differentiability, and so
the $t$-derivative of  $\phi_{m,t}^r(a_0,\dots,a_m)$ exists and (reinstating the prefactors) equals
\begin{align*}
&\eta_m 2^{\frac{M-m}{2}}(M-m)!\prod_{b=1}^{\frac{(M-m)}{2}}\frac{1}{p/2+r-b}
\prod_{j=1}^{\frac{(M-m)}{2}}\frac{1}{m+j}\\
&\qquad\times \sum_{|k|=M-m}\sum_{i=0}^m\int_0^\infty s^M (k_i+1)\langle a_0,1^{k_0},\dots,
da_i,1^{k_i+1},\dots,da_m,1^{k_m}\rangle_{M+1,r-(M-m)/2,s,t} ds.
\end{align*}
Now undoing our applications of the $s$-trick and the $\lambda$-trick gives
$$ 
\frac{d}{dt}\phi^r_{m,t}(a_0,\dots,a_m)=
\eta_m\sum_{j=0}^m\int_0^\infty s^m\langle a_0,\dots,da_j,1
,da_{j+1},\dots,da_m \rangle_{m+1,r,s,t}ds,
$$
and a final application of the $\lambda$-trick yields our final formula,
$$ 
\frac{d}{dt}\phi^r_{m,t}(a_0,\dots,a_m)=
-(p/2+r)\phi^{r+1}_{m,t}(a_0,\dots,a_m).
$$ 
We note that by our estimates the convergence is uniform in $r$, for $r$ in
a compact subset of a suitable right half-plane.
\end{proof}

\subsection{Cocyclicity and relationships between the resolvent and residue cocycles}
\label{subsec:cocycles_reso}
We start by explaining why the resolvent cochain is termed the resolvent cocycle.

\begin{prop}
\label{hjj}
Provided $\Re(r)>1/2$, there exists $\delta\in(0,1)$ such that 
the resolvent cochain $(\phi_{m,t}^r)_{m=\bullet}^M$ is a reduced
$(b,B)$-cocycle of parity  $\bullet\in\{0,1\}$ for $\A$, modulo functions  holomorphic 
in the half plane $\Re(r)>(1-p)/2-\delta$.
\end{prop}

\begin{proof}
Since  
$(\phi_{m,t}^r)_{m=\bullet}^M$ is  a reduced cochain, 
the proof of the first claim will follow from the same algebraic 
arguments as in \cite[Proposition 7.10]{CPRS2} (odd case) 
and \cite[Proposition 6.2]{CPRS3} (even case). We reproduce the main elements of 
the proof for the odd case here.

We start with the computation of the coboundaries of the $\phi_{m,t}^r$. The
definition 
of the operator $B$ and $\phi^r_{m+2,t}$ gives
\begin{align*} 
(B\phi^r_{m+2,t})(a_0,\dots,a_{m+1})&=\sum_{j=0}^{m+1}
\phi^r_{m+2,t}(1,a_j,\dots,a_{m+1},a_0,\dots,a_{j-1})\\
&=\sum_{j=0}^{m+1}\eta_{m+2}\int_0^\infty s^{m+2}
\langle 1,[\D,a_j],\dots,[\D,a_{j-1}]\rangle_{m+2,r,s,t}ds.
\end{align*}
Using Lemma \ref{permut} and Lemma \ref{s-trick}, this is equal to
\begin{align*} &\sum_{j=0}^{m+1}\eta_{m+2}\int_0^\infty s^{m+2}
\langle [\D,a_0],\dots,[\D,a_{j-1}],1,[\D,a_j],\dots,[\D,a_{m+1}]\rangle_{m+2,r,s,t}ds\\
&\qquad\qquad\qquad\qquad\qquad\qquad=-\eta_{m+2}\frac{(m+1)}{2}
\int_0^\infty s^{m}\langle [\D,a_0],\dots,[\D,a_{m+1}]\rangle_{m+1,r,s,t}ds.
\end{align*}

We observe at this point that $\eta_{m+2}(m+1)/2=\eta_m$, 
using the functional equation for the Gamma function.

Next we write $[\D,a_0]=\D a_0-a_0\D$ and anticommute the second $\D$ through 
the remaining 
$[\D,a_j]$ using $\D[\D,a_j]+[\D,a_j]\D=[\D^2,a_j]$. This gives, after some algebra
and an application of Equation \eqref{traceofcommutator} from Lemma \ref{basicidentities}, 
\begin{align} 
&(B\phi^r_{m+2,t})(a_0,\dots,a_{m+1})\nonumber\\
&\qquad\quad=-\eta_m\int_0^\infty s^m \sum_{j=1}^{m+1}(-1)^{j}
\langle a_0,[\D,a_1],\dots,[\D^2,a_j],\dots,[\D,a_{m+1}]\rangle_{m+1,s,r,t}ds.
\label{bigbee}
\end{align}
Observe that for $\phi_{1,t}^r$ we have
$$ 
(B\phi_{1,t}^r)(a_0)=\frac{\eta_1}{2\pi i}\int_0^\infty s
\tau\left( \int_\ell\lambda^{-p/2-r}R_{s,t}(\lambda)[\D,a_0]R_{s,t}(\lambda)
d\lambda\right)ds=0,
$$
by a variant of  Lemma \ref{basicidentities}.
We now compute the Hochschild coboundary of $\phi_{m,t}^r$. From the 
definitions we have
\begin{align*} 
(b\phi_{m,t}^r)(a_0,\dots,a_{m+1})=\phi_{m,t}^r(a_0a_1,a_2,\dots,a_{m+1})
&+\sum_{i=1}^m(-1)^i\phi^r_{m,t}(a_0,\dots,a_ia_{i+1},\dots,a_{m+1})\\
&+\phi^r_{m,t}(a_{m+1}a_0,a_1,\dots,a_m),
\end{align*}
but this is equal to
\begin{align*}
& \eta_m\int_0^\infty s^m\Big(\langle a_0a_1,[\D,a_2],\dots,[\D,a_{m+1}]
\rangle_{m,r,s,t}+\langle a_{m+1}a_0,[\D,a_1],\dots,[\D,a_m]\rangle_{m,r,s,t}\\
&\qquad\qquad+\sum_{i=1}^m(-1)^i\langle a_0,[\D,a_1],\dots,a_i[\D,a_{i+1}]+[\D,a_i]a_{i+1},
\dots,[\D,a_{m+1}]
\rangle_{m,r,s,t}\Big)ds.
\end{align*}
We now reorganise the terms so that we can employ the first identity of Lemma 
\ref{basicidentities}. So 
\begin{align}
\label{littlebee} 
&(b\phi_{m,t}^r)(a_0,\dots,a_{m+1})\nonumber\\
&\qquad\quad=\sum_{j=1}^{m+1}(-1)^{j}
\eta_m\int_0^\infty s^m\langle a_0,[\D,a_1],\dots,[\D^2,a_j],\dots,[\D,a_{m+1}]
\rangle_{m+1,r,s,t}ds.
\end{align}

For $m=1,3,5,\dots,M+\bullet-3$ comparing Equations \eqref{littlebee} and \eqref{bigbee} 
now shows that 
\begin{equation*} 
(B\phi_{m+2,t}^r+b\phi_{m,t}^r)(a_0,\dots,a_{m+1})=0.
\end{equation*}
So we just need to check the claim that $b\phi_{M+\bullet-1}^r$  is holomorphic for 
$\Re(r)>-p/2+\delta$ for some suitable $\delta$. 
From the computation given  above, we have (up to a constant)
$$
b\phi^r_{M,t}(a_0,\dots,a_{M+1})=C(M)\sum_{l=1}^{M+1}(-1)^l\int_0^\infty 
s^M\langle a_0,da_1,\dots,[\D^2,a_l],\dots,da_{M+1}\rangle_{M+1,r,s,t}\,ds,
$$
Now, since  the total order $|k|$ of the pseudodifferential operator entries of the expectation  is 
equal to one, 
we obtain by Lemma \ref{lem:crucial} that 
$b\phi^r_{M,t}(a_0,\dots,a_{M+1})$ is finite for ($\eps>0$ is arbitrary)
$$
\Re(r)>-M-1+({1+M+1})/2+\eps=({1-p})/2+({p-M-1+2\eps})/2.
$$
Since $p-M-1<0$, one can always find $\eps>0$ such that 
$-\delta:=p-M-1+2\eps\in(-1,0)$. The holomorphy follows from Lemma  \ref{hoo}.
\end{proof}

We can now relate the resolvent and residue cocycles.

\begin{prop}
\label{hee}
Assume that our smoothly summable  spectral triple $(\A,\H,\D)$ has isolated spectral dimension. 
Then for $m=\bullet,\bullet+2,\dots,M$,  $a_0,a_1\dots,a_m\in\A$, the map 
$\big[r\mapsto\phi_m^r(a_0,\dots,a_m)]\in \mathcal O_m$, analytically continues to a 
deleted neighbourhood  of the critical point $r=(1-p)/2$. Keeping the 
same notation for this continuation, we have
$$
{\rm res}_{r=(1-p)/2}\,\,\phi_m^r(a_0,\dots,a_m)=\phi_m(a_0,\dots,a_m),
\quad m=\bullet,\bullet+2,\dots,M.
$$
\end{prop}

\begin{proof}
For the even case and $m=0$, we can explicitly compute 
$$
\phi^r_0(a_0)=\frac{1}{r-(1-p)/2}\tau(\gamma a_0(1+\D^2)^{-(r-(1-p)/2)}),
$$
modulo a function of $r$ holomorphic at $r=(1-p)/2$. So we need only consider the case $m\geq 1$.

We start with the expansion, described in detail in the Appendix, Lemma \ref{firstexpan}, 
with $L=M-m$ and $R:=R_s(\lambda)$
$$
a_0\,R\,da_1\,R\cdots R\,da_m\,R
=\sum_{|n|=0}^{M-m}C(n) a_0\,da_1^{(n_1)}\cdots da_m^{(n_m)}\,R^{m+|n|+1}+a_0\,P_{M-m,m}.
$$
Ignoring for a moment the remainder term $P_{M-m,m}$, performing the Cauchy integrals gives
\begin{align*}
\phi_m^r(a_0,\dots,a_m)
=\sum_{|n|=0}^{M-m}C'(n,m,r)\int_0^\infty s^{m} 
\tau\Big(\gamma a_0\,da_1^{(n_1)}\cdots da_m^{(n_m)}\,(1+s^2+\D^2)^{-m-|n|-p/2-r}\Big)\,ds.
\end{align*}
Setting
$h=|n|+(m-\bullet)/2$,
and for $\Re(r)>(1-m)/2$, one can  perform the $s$-integral to 
obtain (after some manipulation of the constants as in \cite[Theorem 6.4]{CPRS3}) for $m>0$ 
\begin{align}
\label{cdots-ha-ha}
\phi_m^r(a_0,\dots,a_m)=(\sqrt{2i\pi})^\bullet\sum_{|n|=0}^{M-m}
(-1)^{|n|}&\alpha(n)
\sum_{l=1-\bullet}^{h}\!\sigma_{h,l}\,
\big(r-(1-p)/2\big)^{l-1+\bullet}\nonumber\\
 &\quad \times\tau\Big(\gamma a_0\,da_1^{(n_1)}\cdots da_m^{(n_m)}
(1+\D^2)^{-|n|-m/2-r+1/2-p/2}\Big).
\end{align}
From this the result will be clear  if  the remainder term is holomorphic for $\Re(r)>(1-p)/2$,
since under the isolated spectral dimension
assumption the residues of the right hand side of the previous expression 
are individually well defined.  
This can be shown using the estimate of the remainder term given in the proof of 
Lemma \ref{lem:crucial} presented in \ref{lem:crucial-app}.
\end{proof}

\subsection{The homotopy to the Chern character}
\label{subsec:nasty-homotopy}

We explain here the sequence of results that leads to the fact that the 
Chern character in degree $M$
is cohomologous to the residue cocycle.

\begin{lemma}
\label{hyy}
Let $t\in[0,1]$, $\Re(r)>1/2$ and $m\equiv \bullet\,\,{\rm mod}\,2$. Then we have
$$
B\Phi_{m+1,t}^r+b\Phi_{m-1,t}^r
=\Big(\frac{p-1}2+r\Big)\phi^r_{m,t}-t\,\frac{p+2r}2\phi^{r+1}_{m,t}.
$$
\end{lemma}

\begin{proof}
By Proposition \ref{hii}, we see that both sides are well 
defined as continuous multi-linear maps from $\A^{\otimes (m+1)}$ to 
the set of holomorphic functions on the half plane $\Re(r)>(m-1)/2$. 
We include the following argument from \cite[Proposition 5.14]{CPRS4} for completeness.

First,
using  the cyclic property of $\langle\langle\cdots\rangle\rangle$ of Lemma 
\ref{lala-la-identity} and the fact 
that $m\equiv \bullet\,(\bmod\ 2),$ we have 
\begin{align} 
B\Phi^r_{m+1,t}(a_0,\dots,a_m)&=\frac{\eta_{m+2}}{2}\sum_{j=0}^m\int_0^\infty 
s^{m+2}(-1)^{mj}\langle\langle 1,da_j,\dots,da_{j-1}\rangle\rangle_{m+1,r,s,t} ds\nonumber\\
&=\frac{\eta_{m+2}}{2}\sum_{j=0}^m\int_0^\infty s^{m+2}\langle\langle 
da_0,\dots,da_{j-1},1,da_j,\dots,da_m\rangle\rangle_{m+1,r,s,t} ds\nonumber\\
&=-\frac{\eta_{m+2}(m+1)}{4}\int_0^\infty s^{m}\langle\langle 
da_0,\dots,da_m\rangle\rangle_{m,r,s,t} ds\nonumber\\
&=-\frac{\eta_{m}}{2}\int_0^\infty s^{m}\langle\langle 
da_0,\dots,da_m\rangle\rangle_{m,r,s,t} ds,\label{formulaforBPhi}
\end{align}
using the $s$-trick (Lemma \ref{s-trick}) in the second last line.
The computation for $b\Phi_{m-1,t}^r$ is the same as for $b\phi^r_{m-1,t}$ in Equation \eqref{littlebee}, 
except we need to take account of the extra term in Equation (\ref{useforb}). This 
gives
\begin{align*} 
b\Phi_{m-1,t}^r(a_0,\dots,a_m)&=\frac{\eta_{m}}{2}\sum_{j=1}^m(-1)^j\int_0^\infty 
s^m\langle\langle a_0,da_1,\dots,[\D^2,a_j],\dots,da_m\rangle\rangle_{m,s,r,t} ds\\
&\qquad\qquad-\frac{\eta_{m}}{2}\sum_{j=1}^m\int_0^\infty s^m\langle 
a_0,da_1,\dots,da_j,\dots,da_m\rangle_{m,s,r,t} ds\\
&=\frac{\eta_{m}}{2}\sum_{j=1}^m(-1)^j\int_0^\infty s^m\langle\langle 
a_0,da_1\dots,[\D^2,a_j],\dots,da_m\rangle\rangle_{m,s,r,t} ds\\
&\qquad\qquad-\frac{\eta_{m}m}{2}\int_0^\infty s^m\langle 
a_0,da_1,\dots,da_m\rangle_{m,s,r,t} ds.
\end{align*}
Now put them together. First, using $\eta_{m+2}(m+1)/2=\eta_{m}$ we have
\begin{align*} 
&(B\Phi^r_{m+1,t}+b\Phi_{m-1,t}^r)(a_0,\dots,a_m)=-\frac{\eta_m}{2}\int_0^\infty s^{m}\langle\langle 
da_0,\dots,da_m\rangle\rangle_{m,s,r,t} ds\\
&\qquad\qquad+\frac{\eta_{m}}{2}\sum_{j=1}^m(-1)^j\int_0^\infty s^m\langle\langle 
a_0,da_1,\dots,[\D^2,a_j],\dots,da_m\rangle\rangle_{m,s,r,t} ds\\
&\qquad\qquad\qquad\qquad-\frac{\eta_{m}m}{2}\int_0^\infty s^m\langle 
a_0,da_1,\dots,da_m\rangle_{m,s,r,t} ds,
\end{align*}
and then applying $[\D^2,a_j]=[\D,[\D,a_j]]_{\pm}$ yields
\begin{align*}
&-\frac{\eta_m}{2}(-1)^{{\rm deg}(a_0)}\int_0^\infty s^{m}\langle\langle 
[\D,a_0]_\pm,da_1,\dots,da_m\rangle\rangle_{m,s,r,t} ds\\
&+\frac{-\eta_{m}}{2}\sum_{j=1}^m(-1)^{{\rm deg}(a_0)+{\rm deg}(da_1)+\cdots+{\rm deg}(da_{j-1})}
\int_0^\infty s^m\langle\langle a_0,da_1\dots,[\D,da_j]_\pm,\dots,da_m\rangle\rangle_{m,s,r,t} 
ds\\
&-\frac{\eta_{m}m}{2}\int_0^\infty s^m\langle 
a_0,da_1,\dots,da_m\rangle_{m,s,r,t} ds.
\end{align*}
Then identity (\ref{nothertrick}) of Lemma \ref{lala-la-identity} shows that this is equal to
\begin{align*}
\frac{-2\eta_{m}}{2}\int_0^\infty s^m\Big(\sum_{j=0}^m\langle 
a_0,\dots,da_j,\D^2,da_{j+1},\dots,da_m\rangle_{m+1,s,r,t} +\frac{m}{2}\langle 
a_0,da_1,\dots,da_m\rangle_{m,s,r,t}\Big) ds,
\end{align*}
then, applying Lemma \ref{differentfort} gives us finally
\begin{align} 
(B\Phi^r_{m+1,t}+b\Phi_{m-1,t}^r)(a_0,\dots,a_m)
&=\eta_{m}\frac{p+2r-1}{2}\int_0^\infty s^m\langle 
a_0,da_1,\dots,da_m\rangle_{m,s,r,t} ds\nonumber\\
&+t\,\eta_{m}\sum_{j=0}^m\int_0^\infty s^m\langle 
a_0,\dots,da_j,1,da_{j+1},\dots,da_m\rangle_{m+1,s,r,t}ds\nonumber\\
&=\frac{p+2r-1}{2}\phi^r_{m,t}(a_0,\dots,a_m)-t\frac{p+2r}{2}\phi_{m,t}^{r+1}(a_0,\dots,a_m),
\label{aaagh}
\end{align}
where we used the $\lambda$-trick (Lemma \ref{lambda-trick}) in the last line. 
\end{proof}

\begin{prop}
\label{rrrr}
Viewed as a cochain with non-trivial components for $m=M$ only,
$$
(r-(1-p)/2)^{-1}B\Phi^r_{M+1,0},
$$
is a $(b,B)$-cocycle modulo cochains with values in functions holomorphic at  
$r=(1-p)/2$ and is cohomologous to the resolvent cocycle $(\phi_{m,0}^r)_{m=\bullet}^M$.
\end{prop}

\begin{proof}
 By Proposition \ref{hyy}, applying $(B,b)$ to the finitely supported cochain 
\begin{equation*} 
\Big(\frac{1}{(r-(1-p)/2)}\Phi_{1-\bullet,0}^r,\dots,\frac{1}{(r-(1-p)/2)}
\Phi^r_{M-1,0},
0,0,\dots\Big),
\end{equation*}
yields
\begin{align*} 
&\Big(\phi^r_{\bullet,0},\phi^r_{\bullet+2,0},\dots ,\phi^r_{M,0}-\frac{B\Phi^r_{M+1,0}}
{(r-(1-p)/2)},0,0,\dots\Big)=\Big((\phi^r_{m,0})_{m=\bullet}^{M}-
\frac{B\Phi^r_{M+1,0}}{(r-(1-p)/2)}\Big).
\end{align*}
That is, $(\phi^r_{m,0})_{m=\bullet}^M$ is cohomologous to 
$(r-(1-p)/2)^{-1}B\Phi^r_{M+1,0}$. Observe that 
 because it is in the image of $B$, 
$(r-(1-p)/2)^{-1}B\Phi^{r}_{M+1,0}$ is cyclic. It is also a $b$-cyclic cocycle 
modulo cochains with values in the functions holomorphic at $r=(1-p)/2$. This 
follows from 
\begin{equation*}
b\Phi^r_{M-1,0}+B\Phi^r_{M+1,0}=(r-(1-p)/2)\phi^r_{M,0},
\end{equation*}
by applying $b$ and recalling that $b\phi^r_{M,0}$ is holomorphic at 
$r=(1-p)/2$.
\end{proof}

Taking residues at $r=(1-p)/2$ and applying 
Proposition \ref{hee}, together with the two preceding results, 
leads directly to

\begin{corollary}
\label{hgg}
If  the spectral triple $(\A,\H,\D)$ has isolated dimension spectrum, then the residue cocycle 
$(\phi_{m,0})_{m=\bullet}^M$ is cohomologous to $B\Phi^{(1-p)/2}_{M+1,0}$ 
(viewed as a single term  cochain).
\end{corollary}

\begin{prop}
\label{prop:t1-t2}
Let $R,\,T\in[0,1]$. Then,
modulo coboundaries and cochains yielding holomorphic functions at the 
critical point $r=(1-p)/2$,
we have
$(\phi_{m,R}^r)_{m=\bullet}^M=(\phi_{m,T}^r)_{m=\bullet}^M$.
\end{prop}

\begin{proof}
Replacing  $r$ by $r+k$ in Proposition \ref{hyy} yields the formula
\begin{equation}
\phi^{r+k}_{m,t}
=\frac{1}{r+k+(p-1)/2}\left(B\Phi^{r+k}_{m+1,t} + b\Phi^{r+k}_{m-1,t} 
+\left(\frac{p}{2} +r +k\right) t\phi^{r+k+1}_{m,t}\right).
\label{eq:iterate}
\end{equation}
Recall from Proposition \ref{generalt} that for $\D$ invertible,
$\phi^r_{m,t}$ is defined and holomorphic for $\Re(r)>(1-m)/2$ for all
$t\in[0,1]$. As $[0,1]$ is compact, the integral
$$
\int_0^1\phi^r_{m,t}(a_0,\dots,a_m)dt,
$$
is holomorphic for $\Re(r)>(1-m)/2$ and any $a^0,\dots,a^m\in\A$. 
Now we make some simple observations, omitting the
variables $a_0,\dots,a_m$ to lighten the notation. For $T,\,R\in[0,1]$ we
have
\begin{equation}
\phi^r_{m,T}-\phi^r_{m,R}=\int_R^T \frac{d}{dt}\phi^r_{m,t}dt
=-(p/2+r)\int_R^T\phi^{r+1}_{m,t}dt.
\label{dopeytrick}
\end{equation}
Now apply the formula of Equation \eqref{eq:iterate} iteratively. At the first step we have
$$
\phi^r_{m,T}-\phi^r_{m,R}=\frac{-(p/2+r)}{r+1+(p-1)/2}
\int_R^T\left(B\Phi^{r+1}_{m+1,t} + b\Phi^{r+1}_{m-1,t} 
+\left(\frac{p}{2} +r +1\right) t\phi^{r+2}_{m,t}\right)dt.
$$
Observe that the numerical factors are holomorphic at
$r=(1-p)/2$. 
Iterating this procedure $L$ times gives us
\begin{align*}\phi^r_{m,T}-\phi^r_{m,R}&=\frac{-(p/2+r)\cdots(p/2+r+L)}
{(r+1+(p-1)/2)\cdots(r+L+(p-1)/2)}
\int_R^T t^L\phi^{r+L+1}_{m,t}dt\\
&+\sum_{j=1}^L\frac{-(p/2+r)\cdots(p/2+r+j-1)}{(r+1+(p-1)/2)\cdots
(r+j+(p-1)/2)}\int_R^T\left(B\Phi^{r+j}_{m+1,t} + b\Phi^{r+j}_{m-1,t} 
\right)t^{j-1}dt.
\end{align*}
In fact the smallest $L$ guaranteeing that $\phi^{r+L+1}_{m,t}$ is holomorphic at
$r=(1-p)/2$ for all $m$ is $(M-\bullet)/2$. See \cite[Lemma 5.20]{CPRS4} for a proof.
With this choice of $L=(M-\bullet)/2$, we have modulo cochains yielding functions 
holomorphic in a
half plane containing $(1-p)/2$,
$$
\phi^r_{m,T}-\phi^r_{m,R}=\sum_{j=1}^L\frac{-(p/2+r)\cdots(p/2+r+j-1)}
{(r+1+(p-1)/2)\cdots(r+j+(p-1)/2)}\int_R^T\left(B\Phi^{r+j}_{m+1,t}
  + b\Phi^{r+j}_{m-1,t} \right)t^{j-1}dt.$$
Thus a simple rearrangement yields the cohomology, valid for $\Re(r)>(1-\bullet)/2$,
\begin{align*}
&(\phi^r_{m,T}-\phi^r_{m,R})_{m=\bullet}^M-B\sum_{j=1}^L\frac{-(p/2+r)\cdots
(p/2+r+j-1)}{(r+1+(p-1)/2)\cdots(r+j+(p-1)/2)}
\int_R^T\Phi^{r+j}_{M+1,t}t^{j-1}dt\\
&=(B+b)\left(\sum_{j=1}^L\frac{-(p/2+r)\cdots(p/2+r+j-1)}{(r+1+(p-1)/2)
\cdots(r+j+(p-1)/2)}\int_R^T\Phi^{r+j}_{m,t}t^{j-1}dt\right)_{m=1-\bullet}^{M-1}.
\end{align*}
Hence modulo coboundaries and cochains yielding functions holomorphic at
$r=(1-p)/2$, we have the equality
$$
(\phi^r_{m,T}-\phi^r_{m,R})_{m=\bullet}^M=B\sum_{j=1}^L\frac{-(p/2+r)
\cdots(p/2+r+j-1)}{(r+1+(p-1)/2)\cdots(r+j+(p-1)/2)}\int_R^T
\Phi^{r+j}_{M+1,t}t^{j-1}dt.
$$
However, an application of Lemma \ref{lem:crucial}
now shows that the right hand side is holomorphic at $r=(1-p)/2$,
since $j\geq 1$ in all cases. Hence,
modulo coboundaries and cochains yielding functions holomorphic at
$r=(1-p)/2$, we have 
$$
(\phi^r_{m,T})_{m=\bullet}^M=(\phi^r_{m,R})_{m=\bullet}^M,
$$
which is the equality we were looking for.
\end{proof}

\begin{corollary}
\label{wotwewanted} 
Modulo coboundaries and cochains 
yielding functions holomorphic in a half plane containing
$r=(1-p)/2$, we have the equality
$$
(\phi^r_m)_{m=\bullet}^M:=(\phi^r_{m,1})_{m=\bullet}^M=B\Phi^r_{M+1,0}.
$$
\end{corollary}

Thus at this point we have shown that the resolvent cocycle is $(b,B)$-cohomologous to the cocycle
$(r-(1-p)/2)^{-1}B\Phi^r_{M+1,0}$ (modulo functions holomorphic at $r=(1-p)/2$),
while the residue cocycle is $(b,B)$-cohomologous to
$B\Phi^{(1-p)/2}_{M+1,0}$. We remark that $B\Phi^{(1-p)/2}_{M+1,0}$ is well-defined (i.e. finite)
by an application of Lemma \ref{lem:crucial}.

Our aim now is to use the map $[0,1]\ni u\to \D|\D|^{-u}$ to obtain a homotopy from 
$B\Phi^{(1-p)/2}_{M+1,0}$ to the Chern character. This is the most technically difficult 
part of the proof, and we defer the proof of the next lemma to the 
Appendix, Lemma \ref{lem:diff1-app}.
This lemma proves a trace class differentiability result.

\begin{lemma}
\label{diff1}
For  $a_0,\dots,a_M\in\A$ and $l=0,\dots,M$, we let
$$
 T_{s,\lambda,l}(u):=d_u(a_0)\,R_{s,u}(\lambda)\cdots d_u(a_l)\,R_{s,u}(\lambda)\,\D_u\,
R_{s,u}(\lambda)\,d_u(a_{l+1})\,R_{s,u}(\lambda)\cdots d_u(a_M)\,R_{s,u}(\lambda).
$$
Then  the map
$\big[u\mapsto T_{s,\lambda,l}(u)\big]$ is continuously differentiable for the 
trace norm topology.
Moreover, with  $R_u:=R_{s,u}(\lambda)$ 
and $\dot{\D_u}=-\D_u\log|\D|$, we obtain
\begin{align}
\label{norm-derivative}
\frac{dT_{s,\lambda,l}}{du}(u)\nonumber
&=\sum_{k=0}^M d_u(a_0)\,R_u\cdots R_u\,d_u(a_k)
\,(2R_u \,\D_u\,\dot{\D_u}\,R_u)\,d_u(a_{k+1})\,R_u\cdots d_u(a_M)\,R_u\\\nonumber
&\quad+ d_u(a_0)\,R_u\cdots R_u\,d_u(a_l)\,R_u\,\D_u\,(2R_u \,\D_u\,\dot{\D_u}\,R_u)\,d_u(a_{l+1})
\cdots R_u\,d_u(a_M)\,R_u\\\nonumber
&\quad+ \sum_{k=0}^M d_u(a_0)\,R_u\,d_u(a_1)\,R_u\cdots R_u
[\dot{\D_u},a_k]\,R_u\cdots R_u\,d_u(a_M)\,R_u\\
&\quad+ d_u(a_0)\,R_u\,d_u(a_1)\,R_u\cdots R_u\,
d_u(a_l)\,R_u\,\dot{\D_u}\,R_u\,d_u(a_{l+1})\cdots R_u\,d_u(a_M)\,R_u.
\end{align}
\end{lemma}

\begin{lemma}
\label{aagghhh} 
For  $a_0,\dots,a_M\in\A$  and for $r>(1-M)/2$, we have
\begin{align*}&(bB\Psi^r_{M,u})(a_0,\dots,a_M)
=\frac{d}{du}(B\Phi^r_{M+1,0,u})(a_0,\dots,a_M)\\
&\qquad-\eta_M(r+(p-1)/2)\sum_{i=0}^M(-1)^i\int_0^\infty s^M\langle
[\D_u,a_0],\dots,[\D_u,a_i],\dot{\D}_u,\dots,[\D_u,a_M]\rangle_{M+1,r,s,0} \,ds,
\end{align*}
where the expectation uses the resolvent for $\D_u$, that is $R_{s,0,u}(\lambda)$.
Moreover, 
$$
r\mapsto -\eta_M\sum_{i=0}^M(-1)^i\int_0^\infty s^M\langle
[\D_u,a_0],\dots,[\D_u,a_i],\dot{\D}_u,\dots,[\D_u,a_M]\rangle_{M+1,r,s,0} \,ds,
$$
is a holomorphic function 
of $r$ in a right half plane containing the critical point $r=(1-p)/2$.
\end{lemma}

\begin{proof}
Lemma \ref{diff1}, and together with arguments
of a similar nature, show that $\Psi^r_{M,u}$
and $\frac{d}{du}\Phi^r_{M+1,0,u}$ are well-defined and are
continuous. The proof of Lemma \ref{diff1} also shows that 
the formal differentiations given below are in fact justified.

First of all, using the $\D_u$ version of
Equation \ref{formulaforBPhi} of Lemma \ref{hyy} and the $R_u$ version of 
Definition \ref{expectation} to expand $(B\Phi^r_{M+1,0,u})(a_0,\dots,a_M)$,
we see that it is the sum of the $T_{s,\lambda,j}(u)$ and so its derivative is the
sum over $j$ of the derivatives in Lemma \ref{diff1}. Using the $R_u$
version of
Definition \ref{expectation} again to rewrite this in terms of 
$\langle\langle\cdots\rangle\rangle$ where possible, shows that
\begin{align*} &\frac{d}{du}(B\Phi^r_{M+1,0,u})(a_0,\dots,a_M)\\
=&-\frac{\eta_M}{2}\int_0^\infty s^M\sum_{i=0}^M\left(\langle\langle 
[\D_u,a_0],\dots,[\D_u,a_i],2\D_u\dot{\D}_u,\dots,[\D_u,a_M]\rangle\rangle_{M+1,s,r,0}\right.\\
&\qquad\qquad\qquad+\left.\langle\langle 
[\D_u,a_0],\dots,[\dot{\D}_u,a_i],\dots,[\D_u,a_M]\rangle\rangle_{M,s,r,0}\right) ds\\
&-\frac{\eta_M}{2}\int_0^\infty s^M\sum_{i=0}^M(-1)^i\langle 
[\D_u,a_0],\dots,[\D_u,a_i],\dot{\D}_u,\dots,[\D_u,a_M]\rangle_{M+1,s,r,0} ds.
\end{align*}
For the next step we compute $Bb\Psi^r_{M,u}$, and then use $bB=-Bb$. First we
apply $b$
\begin{align*} &(b\Psi^r_{M,u})(a_0,\dots,a_{M+1})=-\frac{\eta_M}{2}\int_0^\infty s^M\langle\langle 
a_0a_1\dot{\D}_u,[\D_u,a_2],\dots,[\D_u,a_{M+1}]\rangle\rangle_{M,s,r,0} ds\\
&-\frac{\eta_M}{2}\sum_{j=1}^M(-1)^j\int_0^\infty s^M\langle\langle 
a_0\dot{\D}_u,\dots,[\D_u,a_ja_{j+1}],\dots,[\D_u,a_{M+1}]\rangle\rangle_{M,s,r,0} ds\\
&-(-1)^{M+1}\frac{\eta_M}{2}\int_0^\infty s^M\langle\langle 
a_{M+1}a_0\dot{\D}_u,[\D_u,a_1],\dots,[\D_u,a_M]\rangle\rangle_{M,s,r,0} ds\\
&=-\frac{\eta_M}{2}\int_0^\infty s^M\sum_{j=1}^{M+1}(-1)^j\langle\langle 
a_0\dot{\D}_u,[\D_u,a_1],\dots,[\D^2_u,a_j],\dots,[\D_u,a_{M+1}]\rangle\rangle_{M+1,s,r,0} ds\\
&-\frac{\eta_M}{2}\int_0^\infty 
s^M\sum_{j=1}^{M+1}(-1)^j(-1)^{{\rm deg}(a_0\dot{\D}_u)+\cdots+{\rm deg}([\D_u,a_{j-1}])}\langle 
a_0\dot{\D}_u,[\D_u,a_1],\dots,[\D_u,a_{M+1}]\rangle_{M+1,s,r,0} ds\\
&+\frac{\eta_M}{2}\int_0^\infty s^M\langle\langle 
a_0[\dot{\D}_u,a_1],\dots,[\D_u,a_{M+1}]\rangle\rangle_{M,s,r,0} ds.
\end{align*}
The last equality follows from the 
$R_u$ version of Lemma \ref{lala-la-identity}. In the above, we note that
${\rm deg}(a_0\dot{\D_u})=1= {\rm deg}([\D_u,a_k])$ for all $k$ so that 
${\rm deg}(a_0\dot{\D_u})+\cdots+{\rm deg}([\D_u,a_{j-1}])=j$ and 
${\rm deg}(a_0\dot{\D_u})+\cdots+{\rm deg}([\D_u,a_{M+1}])=M+2\equiv \bullet(mod\ 2).$
We also note the commutator identity $[\D_u^2,a_j]=\{\D_u,[\D_u,a_j]\}=[\D_u,[\D_u,a_j]]_{\pm}$
so in order to apply the $\D_u$ version of Equation \eqref{nothertrick} of 
Lemma \ref{lala-la-identity} we first add and substract
\begin{equation*} 
-\frac{\eta_M}{2}\int_0^\infty s^M \langle\langle 
\{\D_u,a_0\dot{\D}_u\},[\D_u,a_1],\dots,[\D_u,a_{M+1}]\rangle\rangle_{M+1,s,r,0} ds,
\end{equation*}
and then an application of  Equation \eqref{nothertrick} yields
\begin{align*} 
&-2\frac{\eta_M}{2}\int_0^\infty s^M \sum_{j=0}^{M+1}\langle 
a_0\dot{\D}_u,\dots,[\D_u,a_j],\D^2_u,\dots,[\D_u,a_{M+1}]\rangle_{M+2,s,r,0} ds\\
&+\frac{\eta_M}{2}\int_0^\infty s^M\langle\langle 
a_0\{\D_u,\dot{\D}_u\}+[\D_u,a_0]\dot{\D}_u,[\D_u,a_1],\dots,[\D_u,a_{M+1}]\rangle\rangle_{M+1,s,r,0} ds\\
&-\frac{\eta_M}{2}(M+1)\int_0^\infty s^M \langle 
a_0\dot{\D}_u,[\D_u,a_1],\dots,[\D_u,a_{M+1}]\rangle_{M+1,s,r,0} ds\\
&+\frac{\eta_M}{2}\int_0^\infty s^M \langle\langle 
a_0[\dot{\D}_u,a_1],\dots,[\D_u,a_{M+1}]\rangle\rangle_{M,s,r,0} ds.
\end{align*}
Then we apply the $\D_u$ version of Lemma \ref{differentfort} to obtain
\begin{align*}
&\frac{\eta_M}{2}(p +2r)\int_0^\infty s^M\langle 
a_0\dot{\D}_u,[\D_u,a_1],\dots,[\D_u,a_{M+1}]\rangle_{M+1,s,r,0} ds\\
&+\frac{\eta_M}{2}\int_0^\infty s^M\langle\langle 
a_0\{\D_u,\dot{\D}_u\}+[\D_u,a_0]\dot{\D}_u,[\D_u,a_1],\dots,[\D_u,a_{M+1}]\rangle\rangle_{M+1,s,r,0} ds\\
&+\frac{\eta_M}{2}\int_0^\infty s^M \langle\langle 
a_0[\dot{\D}_u,a_1],\dots,[\D_u,a_{M+1}]\rangle\rangle_{M,s,r,0} ds.
\end{align*}
The next step is to apply $B$ to these three terms, producing
\begin{align*} 
&(Bb\Psi^r_{M,u})(a_0,\dots,a_M)\\
&=(p+2r)\frac{\eta_M}{2}\sum_{j=0}^M(-1)^{(M+1)j}\int_0^\infty s^M\langle 
\dot{\D}_u,[\D_u,a_j],\dots,[\D_u,a_{j-1}]\rangle_{M+1,s,r,0} ds\\
&+\frac{\eta_M}{2}\sum_{j=0}^M(-1)^{(M+1)j}\int_0^\infty s^M \langle\langle 
\{\D_u,\dot{\D}_u\},[\D_u,a_j],\dots,[\D_u,a_{j-1}]\rangle\rangle_{M+1,s,r,0} ds\\
&+\frac{\eta_M}{2}\sum_{j=0}^M(-1)^{(M+1)j}\int_0^\infty s^M\langle\langle 
[\dot{\D}_u,a_j],\dots,[\D_u,a_{j-1}]\rangle\rangle_{M,s,r,0} ds,
\end{align*}
which is identical to
\begin{align*}
&\frac{(p+2r)\eta_M}{2}\sum_{j=0}^M(-1)^{(M+1)j+(1-\bullet)j}\!\!\int_0^\infty s^M\langle 
[\D_u,a_0],\dots,[\D_u,a_{j-1}],\dot{\D}_u,\dots,[\D_u,a_M]\rangle_{M+1,s,r,0}ds\\
&+\frac{\eta_M}{2}\sum_{j=0}^M(-1)^{(M+1)j+(2-\bullet)j}\!\!\int_0^\infty s^M 
\langle\langle[\D_u,a_0],\dots, 
\{\D_u,\dot{\D}_u\},[\D_u,a_j],\dots,[\D_u,a_M]\rangle\rangle_{M+1,s,r,0} ds\\
&+\frac{\eta_M}{2}\sum_{j=0}^M(-1)^{(M+1)j+(2-\bullet)j}\!\!\int_0^\infty 
s^M\langle\langle[\D_u,a_0],\dots,[\D_u,a_{j-1}], 
[\dot{\D}_u,a_j],\dots,[\D_u,a_M]\rangle\rangle_{M,s,r,0} ds.
\end{align*}
This last expression equals 
\begin{align*}
&(p+2r)\frac{\eta_M}{2}\sum_{j=0}^M(-1)^j\int_0^\infty s^M\langle 
[\D_u,a_0],\dots,[\D_u,a_{j-1}],\dot{\D}_u,\dots,[\D_u,a_M]\rangle_{M+1,s,r,0} ds\\
&+\frac{\eta_M}{2}\sum_{j=0}^M\int_0^\infty s^M \langle\langle[\D_u,a_0],\dots, 
2\D_u\dot{\D}_u,[\D_u,a_j],\dots,[\D_u,a_M]\rangle\rangle_{M+1,s,r,0} ds\\
&+\frac{\eta_M}{2}\sum_{j=0}^M\int_0^\infty 
s^M\langle\langle[\D_u,a_0],\dots,[\D_u,a_{j-1}], 
[\dot{\D}_u,a_j],\dots,[\D_u,a_M]\rangle\rangle_{M,s,r,0} ds.
\end{align*}
Using $bB=-Bb$, and our formula for 
$\frac{d}{du}(B\Phi_{M+1,0,u}^r)(a_0,\dots,a_M)$ gives
\begin{align*} 
&(bB\Psi^r_{M,u})(a_0,\dots,a_M)\\
&=-(p+2r)\frac{\eta_M}{2}\sum_{j=0}^M(-1)^j\int_0^\infty s^M\langle 
[\D_u,a_0],\dots,[\D_u,a_{j-1}],\dot{\D}_u,\dots,[\D_u,a_M]\rangle_{M+1,s,r,0} ds\\
&+\frac{\eta_M}{2}\sum_{i=0}^M(-1)^i\int_0^\infty s^M \langle 
[\D_u,a_0],\dots,[\D_u,a_i],\dot{\D}_u,\dots,[\D_u,a_M]\rangle_{M+1,s,r,0} ds\\
&+\frac{d}{du}(B\Phi_{M+1,0,u}^r)(a_0,\dots,a_M).
\end{align*}
This proves the result.
\end{proof}
Thus we have proven the following key statement.

\begin{corollary}
\label{hff}
We have
\ben
\frac{1}{(r+(p-1)/2)}(bB\Psi^r_{u,M})(a_0,\dots,a_M)=\frac{1}{(r+(p-1)/2)}\frac{d}
{du}(B\Phi^{r}_{M+1,0,u})(a_0,\dots,a_M)+holo(r),
\een
where $holo$ is analytic for $\Re(r)>-M/2$, and by taking residues
$$
(bB\Psi^{(1-p)/2}_{M,u})(a_0,\dots,a_M)=\frac{d}{du}(B\Phi^{(1-p)/2}_{M+1,0,u})(a_0,\dots,a_M).
$$
\end{corollary}

We now have the promised cohomologies.
\begin{theorem}
\label{thm:final-cohoms}
Let $(\A,\HH,\D)$ be a smoothly summable 
 spectral triple relative to $(\cn,\tau)$ and of spectral dimension
$p\geq 1$, parity $\bullet\in\{0,1\}$,  with $\D$ invertible and $ \A$ separable. Then\\
(1) In the $(b,B)$-bicomplex with coefficients in the set of holomorphic functions on  the
right half plane $\Re(r)>1/2$, the resolvent cocycle $(\phi_m^r)_{m=\bullet}^M$  is 
cohomologous to the  single term cocycle 
$$(r-(1-p)/2)^{-1}{\rm Ch}^M_{F},
$$
modulo cochains with values in 
the set of holomorphic functions on a right half plane containing the critical point 
$r=(1-p)/2$. Here $F=\D\,|\D|^{-1}$.\\
(2) If moreover, the spectral triple $(\A,\H,\D)$ has isolated spectral dimension, 
then the residue cocycle $(\phi_m)_{m=\bullet}^M$  is cohomologous to the 
Chern character ${\rm Ch}^M_{F}$.
\end{theorem}

\begin{proof}
Up to cochains holomorphic at the critical point (the integral on a 
compact domain doesn't modify the holomorphy property), Lemma \ref{aagghhh} gives
$$
\frac{1}{r-(1-p)/2}\int_0^1(bB\Psi^r_{M,u})(a_0,\dots,a_M)\,du
=\frac{1}{r-(1-p)/2}\int_0^1\frac{d}{du}(B\Phi^{r}_{M+1,0,u})(a_0,\dots,a_M)\,du.
$$
Since $\frac{1}{r-(1-p)/2}\int_0^1bB\Psi^r_{M,u}$ is a coboundary, we obtain
the following equality in cyclic cohomology (up to coboundaries and a  
cochain holomorphic at the critical point)
$$
\frac{1}{r-(1-p)/2}(B\Phi^{r}_{M+1,0,1})=\frac{1}{r-(1-p)/2}(B\Phi^{r}_{M+1,0,0}).
$$
One can now compute directly to see that the left hand side is 
$(r-(1-p)/2)^{-1}{\rm Ch}^M_{F}$ as follows.

Recalling that $F^2=1$ and 
using our previous formula for $B\Phi_{M+1,0,u}^r$ (the $\D_u$ version of
Proposition \ref{hyy} with $u=1$) we have
\begin{align*} 
& (B\Phi^{r}_{M+1,0,u})(a_0,\dots,a_M)|_{u=1}\\
&=-\frac{\eta_{M}}{2}\sum_{j=0}^{M}(-1)^{j+1}\int_0^\infty s^{M}\langle 
[F,a_0],\dots,[F,a_j],F,[F,a_{j+1}],\dots,[F,a_M]\rangle_{M+1,s,r,0} ds\\
&=-\frac{\eta_{M}}{2}\sum_{j=0}^M\int_0^\infty s^{M}\frac{1}{2\pi i}
\tau\left(\gamma\int_\ell\lambda^{-p/2-r}F[F,a_0]\cdots[F,a_M](\lambda-(s^2+1))^{-M-2}d\lambda\right)ds\\
&=\frac{\eta_{M}}{2}\frac{(-1)^{M}}{M!}\frac{\Gamma(M+1+p/2+r)}{\Gamma(p/2+r)}\
\int_0^\infty s^{M}\tau\big(\gamma F[F,a_0]\cdots[F,a_M](s^2+1)^{-M-1-p/2-r}\big)ds.
\end{align*}
In the second equality we anticommuted $F$ past the commutators, and pulled all 
the resolvents to the right (they commute with everything, since they involve 
only scalars). In the last equality we used the Cauchy integral formula to do 
the contour integral, and performed the sum. 

Now we pull out $(s^2+1)^{-M-1-p/2-r}$ from the trace, leaving the identity 
behind. The $s$-integral is given by
\begin{align*} 
&\int_0^\infty s^M(s^2+1)^{-M-1-p/2-r}ds
=\frac{\Gamma((M+1)/2)\Gamma(p/2+r+M/2+1/2)}{2\Gamma(M+1+p/2+r)}.
\end{align*}
Putting the pieces together gives
\begin{align*} &(B\Phi^{r}_{M+1,0,u})(a_0,\dots,a_M)|_{u=1}\\
&\qquad=\frac{\eta_M}{2}(-1)^M\frac{\Gamma((M+1)/2)}{\Gamma(p/2+r)}
\frac{\Gamma(((p-1)/2+r)+M/2+1)}{2M!}\tau(\gamma F[F,a_0]\cdots[F,a_M]).
\end{align*}
Now  $\eta_M=\sqrt{2i}^\bullet(-1)^M2^{M+1}\Gamma(M/2+1)/\Gamma(M+1)$, and the 
duplication formula for the Gamma function tells us that 
$\Gamma((M+1)/2)\Gamma(M/2+1)2^M=\sqrt{\pi}\Gamma(M+1)$. Hence 
\begin{align*} 
&(B\Phi^{r}_{M+1,0,u})(a_0,\dots,a_M)|_{u=1}\\
&\qquad\qquad\qquad\qquad=\frac{\sqrt{\pi}\sqrt{2i}^\bullet
\Gamma(((p-1)/2+r)+M/2+1)}{\Gamma(p/2+r)2\cdot M!}\tau(\gamma 
F[F,a_0][F,a_1]\cdots[F,a_M]).
\end{align*}
Now we use the functional equation for the Gamma function
\begin{align*} 
&\Gamma(((p-1)/2+r)+M/2+1)=\Gamma((p-1)/2+r)\prod_{j=0}^{(M-\bullet)/2} ((p-1)/2+r+j+\bullet/2),
\end{align*}
to write this as
\begin{align*} 
&(B\Phi^{r}_{M+1,0,u})(a_0,\dots,a_M)|_{u=1}\\
&\qquad\qquad\qquad=\frac{C_{p/2+r}\sqrt{2i}^\bullet}{2\cdot M!}
\sum_{j=1-\bullet}^{(M-\bullet)/2+1}(r+(p-1)/2)^j\s_{(M-\bullet)/2,j}\tau(\gamma F[F,a_0][F,a_1]\cdots[F,a_M]),
\end{align*}
where the $\s_{(M-\bullet)/2,j}$ are elementary symmetric functions of the integers 
$1,2,\dots,M/2$ (even case) or of the half integers $1/2,3/2,\dots,M/2$ 
(odd case). The `constant' 
$$
C_{p/2+r}:=\frac{\sqrt{\pi}\Gamma((p-1)/2+r)}{\Gamma(p/2+r)},
$$ 
has a simple pole at 
$r=(1-p)/2$ with residue equal to $1$, and $\s_{M/2,1-\bullet}=\Gamma(M/2+1)$ in 
both even and odd cases, and recalling Definition \ref{conditional} of
$\tau^\prime$ we see that
\begin{align*} 
&\frac{1}{(r-(1-p)/2)}(B\Phi^{r}_{M+1,0,u})(a_0,\dots,a_M)|_{u=1}
=\frac{1}{(r-(1-p)/2)}{\rm Ch}_F(a_0,a_1,\dots,a_M)+holo(r),
\end{align*}
where $holo$ is a function holomorphic at $r=(1-p)/2$, and on the right 
hand side the Chern character appears with its $(b,B)$ normalisation.

As the left hand side is cohomologous to 
the resolvent cocycle by Proposition \ref{rrrr}, the first part is proven. The proof of 
the second part is now a consequence of Proposition \ref{hee}.
\end{proof}

\subsection{Removing the invertibility of $\D$}\label{invert}

We can now apply Theorem \ref{thm:final-cohoms} to the double of a 
smoothly summable spectral triple of spectral dimension $p\geq 1$. In this case, 
the resolvent and residue cocycles extend to the reduced $(b,B)$-bicomplex for $\A^\sim$, and it is 
simple to check that they are still cocycles there. Moreover, as noted in Lemma \ref{chien},
all of our cohomologies can be considered to take place in the reduced complex for $\A^\sim$.

Thus under the isolated spectral dimension assumption,
the residue cocycle for $(\A,\HH\oplus\HH,\D_\mu,\hat\gamma)$ 
is cohomologous to the Chern character
${\rm Ch}_{F_\mu}^M$, and similarly for the resolvent cocycle. We now show how to obtain a 
residue and resolvent formula for the index in terms of the original spectral triple.

In the following we  write
$\{\phi^r_{\mu,m}\}_{m=\bullet,\bullet+2,\dots,M}$ for the resolvent cocycle
for $\A$ defined  using the double spectral triple and 
$\{\phi^r_{m}\}_{m=\bullet,\bullet+2,\dots,M}$ for the resolvent cocycle
for $\A$ defined by using original spectral triple, according to the notations introduced in subsection
\ref{subsec:reduced}.

The formula for ${\rm Ch}_{F_\mu}^M$ is scale invariant, \index{scale invariance}
in that it remains unchanged if we replace
$\D_\mu$ by $\lambda\,\D_\mu$ for any $\lambda>0$. This 
scale invariance is the main tool we employ.

In the double up procedure we will start with $0< \mu< 1$.
We are interested in the relationship between 
$(1+\D^2)\otimes{\rm {\rm Id}_2}$  and $1+\D_\mu^2$, given 
by
\ben 
1+\D_\mu^2=\bma 1+\mu^2+\D^2 & 0\\ 0 & 1+\mu^2+\D^2\ema.
\een

If we perform the scaling  $\D_\mu\mapsto (1-\mu^2)^{-1/2}\D_\mu$ then 
\ben 
(1+\D_\mu^2)^{-s}\mapsto (1-\mu^2)^{s}(1+\D^2)^{-s}\otimes{\rm Id_2}.
\een

This algebraic simplification is not yet enough. We need to scale every appearance of 
$\D$ in the formula for the resolvent cocycle. Now
Proposition \ref{hee} provides the following formula for the resolvent cocycle in terms of
zeta functions, modulo functions holomorphic at $r=(1-p)/2$:

\begin{align}
\phi_{\mu,m}^r(a_0,\dots,a_m)&=(\sqrt{2i\pi})^\bullet\sum_{|k|=0}^{M-m}
(-1)^{|n|}\alpha(n)
\sum_{l=1-\bullet}^{(M-\bullet)/2+|k|}\!\sigma_{h,l}\,
\big(r-(1-p)/2\big)^{l-1+\bullet}\nonumber\\
 &\qquad \times\tau\otimes{\rm tr}_2\Big(\gamma a_0\,[\D_\mu,a_1]^{(k_1)}\cdots [\D_\mu,a_m]^{(k_m)}
(1+\D_\mu^2)^{-|k|-m/2-r+1/2-p/2}\Big).
\label{eq:resolvent-zeta}
\end{align}

So we require the scaling properties of the 
coefficient operators
$$
\omega_{\mu,m,k}=[\D_\mu,a_1]^{(k_1)}\cdots[\D_\mu,a_m]^{(k_m)},
$$ 
that appear in this
zeta function representation of the resolvent cocycle. 
In order to study these coefficient operators,
it is useful to introduce the following operations (arising from the
periodicity operator in cyclic cohomology, see \cite{Co4,CPRS1}).

We 
define $\hat{S}:\A^{\otimes m}\to {\rm OP_0^0}$, for any $m\geq 0$ by
\ben 
\hat{S}(a_1)=0,\ \ \hat{S}(a_1,\dots,a_m)=
\sum_{i=1}^{m-1}da_1\cdots (da_{i-1})a_ia_{i+1} 
da_{i+2}\cdots da_m,
\een
and extend it by linearity to the tensor product $\A^{\otimes m}$.
As usual,  we write $da=[\D,a]$.
To define `powers' of $\hat{S}$, we 
recursively set
\begin{align*} 
\hat{S}^k(a_1,\dots,a_m)&=\sum_{l=0}^{k-1}\binom{k-1}{l}\sum_{i=1}^{m-1}\hat{S}^l(a_1,\dots,a_{i-1})
\hat{S}^{k-l-1}(a_ia_{i+1},\dots,a_m).
\end{align*}

The following lemma is proven in \cite[Appendix]{CPRS1}.

\begin{lemma} 
\label{lem:recursion}
The maps $\hat{S}^l$ satisfy the following relations:
\be 
\hat{S}(a_1,\dots,a_{m-1})da_m=
\hat{S}(a_1,\dots,a_m)-da_1\cdots (da_{m-2})
a_{m-1}a_m,\label{one}
\ee
and for $l>1$
\begin{align*}
\hat{S}^l(a_1,\dots,a_{m-1})da_m&=\hat{S}^l(a_1,\dots,a_m)-l\,\hat{S}^{l-1}
(a_1,\dots,a_{m-2})a_{m-1}a_m,\\
l\,\hat{S}^{l-1}(a_1,\dots,a_{2l-2})a_{2l-1}a_{2l}&=
\hat{S}^l(a_1,\dots,a_{2l}),\qquad
\hat{S}^l(a_1,\dots,a_{2l-1})=0.
\end{align*}
\end{lemma}

As a last generalisation, we note that if $k$ is now a multi-index then we can define analogues 
of the operations $\hat{S}^l$ by 
\ben 
\hat{S_k}(a_1):=0,\ \ \hat{S_k}(a_1,\dots,a_m):=
\sum_{l=1}^{n-1}(da_1)^{(k_1)}\cdots( da_{l-1})^{(k_{l-1})}a_l^{(k_l)}a_{l+1}^{(k_{l+1})} 
(da_{l+2})^{(k_{l+2})}\cdots (da_m)^{(k_m)}.
\een

With these operations in hand we can state the result.

\begin{lemma}
\label{lem:psycho-induction} With $\D$ and $\D_\mu$ as above, and 
for $m>1$, the operator $[\D_\mu,a_1]^{(k_1)}\cdots[\D_\mu,a_m]^{(k_m)}$ is given by
$$
\begin{pmatrix} 
\begin{array}{c}\omega_{m,k} +\sum_{i=1}^{\lfloor m/2\rfloor}
c_i\,\hat{S}^i(a_1,\dots,a_m)\end{array} &
\begin{array}{c}-\mu\omega_{m-1,k}a_m^{(k_m)}\\
-\mu\sum_{i=1}^{\lfloor(m-1)/2\rfloor}c_i\,\hat{S}^i(a_1,\dots,a_{m-1})a_m^{(k_m)}\end{array}\\
 & \\
\begin{array}{c}\mu a_1^{(k_1)}\tilde{\omega}_{m-1,k}\\
+\mu\sum_{m=1}^{\lfloor(m-1)/2\rfloor}c_i\,a_1^{(k_1)}\hat{S}^i(a_2,\dots,a_m)\end{array} &
\begin{array}{c}-\mu^2a_1^{(k_1)}\,\widehat{\omega}_{m-2,k}a_m^{(k_m)}\\
-\mu^2\sum_{i=1}^{\lfloor m/2\rfloor-1}c_i\,a_1^{(k_1)}\hat{S}^i(a_2,\dots,a_{m-1})a_m^{(k_m)}\end{array}
\end{pmatrix}.
$$
In this expression 
$$
\omega_{m,k}=(da_1)^{(k_1)}\cdots(da_m)^{(k_m)},\qquad 
\omega_{m-1,k}=(da_1)^{(k_1)}\cdots(da_{m-1})^{(k_{m-1})},
$$
$$
\tilde{\omega}_{m-1,k}=(da_2)^{(k_2)}\cdots(da_m)^{(k_m)}, \qquad
\widehat{\omega}_{m-2,k}=(da_2)^{(k_2)}\cdots(da_{m-1})^{(k_{m-1})},
$$ 
the superscript $(k_l)$'s
refer to commutators with $\D^2$ (Definition \ref{parup}), and
$c_i=(-1)^i\mu^{2i}/i!$. 
\end{lemma}

\begin{proof} This is proved by induction using 
$$
[\D_\mu,a_{n+1}]^{(k_{n+1})}=
[\D_\mu,a_{n+1}^{(k_{n+1})}]=\begin{pmatrix} da_{n+1}^{(k_{n+1})} & -\mu \,a_{n+1}^{(k_{n+1})}\\ 
\mu\, a_{n+1}^{(k_{n+1})} & 0\end{pmatrix}.
$$
It is important to note that the formulae for the $\hat{S}$ operation are unaffected by the commutators
with $\D_\mu^2$, since $\D_\mu^2$ is diagonal.
A similar calculation in \cite[Appendix]{CPRS1}, where there is a sign error corrected here, indicates
how the proof proceeds.
\end{proof}

Multiplying the operator in Lemma \ref{lem:psycho-induction} by 
$\hat a_0=\begin{pmatrix} a_0 & 0\\ 0 & 0\end{pmatrix}$ gives us $a_0\omega_{m,\mu,k}$.
Having identified the $\mu$ dependence of $\omega_{m,\mu,k}(1+\D_\mu^2)^{-|k|-m/2-r-(p-1)/2}$
arising from the 
coefficient operators $\omega_{m,\mu,k}$, we
now identify the remaining $\mu$ dependence in 
$a_0\omega_{m,\mu,k}(1+\D_\mu^2)^{-|k|-m/2-r-(p-1)/2}$
coming from $(1+\D_\mu^2)^{-|k|-m/2-r-(p-1)/2}$.
So replacing $\D_\mu$ by $(1-\mu^2)^{-1/2}\D_\mu$,
our calculations give for $m>0$

\begin{align*}
&a_0\omega_{m,\mu,k}(1+\D_\mu^2)^{-|k|-m/2-r-(p-1)/2}\longmapsto\\ 
&\qquad\qquad\qquad\qquad(1-\mu^2)^{-r-(p-1)/2}a_0\omega_{m,k}(1+\D^2)^{-|k|-m/2-r-(p-1)/2}\otimes\begin{pmatrix} 1 & 0\\ 0 & 0\end{pmatrix} +O(\mu),
\end{align*}

where the $O(\mu)$ terms, are those arising from Lemma \ref{lem:psycho-induction}. 
Of course at $r=(1-p)/2$ the 
numerical factor $(1-\mu^2)^{-r-(p-1)/2}$ is equal to one, and contributes nothing when we take residues.
For $m=0$ there are no additional $O(\mu)$ terms.

Ignoring the factor of $(1-\mu^2)^{-r-(p-1)/2}$, we collect all terms 
in $\{\phi^r_{\mu,m}\}_{m=\bullet,\bullet+2,\dots,M}$
with the same power of $\mu$, arising from
the expansion of $a_0\omega_{m,k,\mu}$.  This gives us  a finite family of 
$(b,B)$-cochains of different lengths but the same parity, 
one for each power of $\mu$ in the expansion of 
$a_0\omega_{m,k,\mu}$. Denote these new cochains by $\psi_i^r=(\psi_{i,m}^r)_{m=\bullet,\bullet+2,\dots}$, 
where $\psi_i^r$ is assembled as the coefficient cochain for $\mu^i$. To simplify the
notation, we will consider the cochains $\psi^r_i$ as functionals on 
suitable elements in ${\rm OP}^*$. With these conventions, and
modulo functions holomorphic at $r=(1-p)/2$, we have
$$
\phi^r_{\mu,m}(a_0,\dots,a_m)
=(1-\mu^2)^{-r+(1-p)/2}\Big(\sum_{i=0}^{2\lfloor\frac{m}{2}\rfloor+1}\psi_{i,m}^r(a_0\omega_{m,k,i})\,\mu^i\Big),
$$
where $\omega_{m,k}^i$ are some 
coefficient operators depending on $a_1,\dots,a_m$, but not on $\mu$, and 
$\omega_{m,k,0}=\omega_{m,k}$, as defined in Lemma \ref{lem:psycho-induction}.

Let $\alpha=(\alpha_m)_{m=\bullet,\bullet+2,\dots}$ be a 
$(b,B)$-boundary in the reduced complex for $\A^\sim$. 
Then as ${\rm Ch}_{F_\mu}^M$ is a $(b,B)$-cocycle, we find by 
performing the pseudodifferential expansion that there are reduced $(b,B)$-cochains 
$C_0,\dots,C_{2\lfloor M/2\rfloor+\bullet}$ such  that 
\begin{align*}
0={\rm Ch}_{F_\mu}^M(\alpha_M)
&={\rm res}_{r=(1-p)/2}\sum_{m=\bullet}^M\phi^r_{\mu,m}(\alpha_m)
=C_0(\alpha)+C_1(\alpha)\,\mu+\cdots+C_{2\lfloor M/2\rfloor+\bullet}(\alpha)\,\mu^{2\lfloor M/2\rfloor+\bullet}.
\end{align*}
The class of ${\rm Ch}_{F_\mu}^M$ is independent of $\mu>0$, and as we can vary $\mu\in (0,1)$, 
we see that each of the coefficients $C_i(\alpha)=0$. As the $C_i(\alpha)$
arise as the result of pairing a $(b,B)$-cochain with the $(b,B)$-boundary $\alpha$, and $\alpha$
is an arbitrary boundary, we see that all the $\psi_i^r$ are (reduced) cocycles modulo functions
holomorphic at $r=(1-p)/2$.

Now let $\beta$ be a $(b,B)$-cycle. Then by performing the pseudodifferential expansion
we find that 
\begin{align*}
{\rm Ch}_{F_\mu}^M(\beta_M)&={\rm res}_{r=(1-p)/2}\sum_{m=\bullet}^M\phi^r_{\mu,m}(\beta_m)
=C_0(\beta)+C_1(\beta)\,\mu+\cdots+C_{2\lfloor M/2\rfloor+\bullet}(\beta)\,\mu^{2\lfloor M/2\rfloor+\bullet}.
\end{align*}
The left hand side is independent of $\mu$, and so taking the derivative with respect to $\mu$ yields
$$
0=C_1(\beta)+\cdots+(2\lfloor M/2\rfloor+\bullet)C_{2\lfloor M/2\rfloor+\bullet}(\beta)\,\mu^{2\lfloor M/2\rfloor+\bullet-1}.
$$
Again, by varying $\mu$ we see that each coefficient $C_i(\beta)$, $i>0$, must vanish. As $\beta$ is an 
arbitrary $(b,B)$-cycle, for $i\neq 0$,
$\psi_i^r$ is a coboundary modulo functions holomorphic at $r=(1-p)/2$. 
The conclusion is that 
${\rm res}\,\psi^r_0$ represents the Chern character. We now turn to
making this representative explicit.

The cocycle $\psi^r_0$ is given, in terms of the original spectral triple $(\A,\HH,\D)$, 
in all degrees except zero, by $\{\phi^r_{m}\}_{m=\bullet,\bullet+2,\dots,M}$, that is the formula for the 
resolvent cocycle presented in 
Definition \ref{resolvent} with $\D$ in place of $\D_\mu$. In degree zero we need some care, and after a computation we find that
for $b\in\A^\sim$ and $\mu\in(0,1)$, $\phi^r_{\mu,0}(b)$ is given by

\begin{align*}
&\phi^r_{\mu,0}(b)=\lim_{\lambda\to\infty}\frac{\Gamma(r-(1-p)/2)\sqrt{\pi}(1-\mu^2)^{-(r-(1-p)/2)}}{\Gamma(p/2+r)}\\
&\times\tau\otimes {\tr}_2\begin{pmatrix} \gamma (b-{\bf 1}_b)(1+\D^2)^{-z}+
\gamma \tilde{\psi}_\lambda {\bf 1}_b(1 +\D^2)^{-(r-(1-p)/2)} & 0\\ 
0 & -\gamma \tilde{\psi}_\lambda {\bf 1}_b(1+\D^2)^{-(r-(1-p)/2)} \end{pmatrix},
\end{align*}
 where ${\bf 1}_b$ is defined after Equation \eqref{ext-A-sim}. 
Canceling the ${\bf 1}_b$ terms  and taking the limit shows that $\phi^r_{\mu,0}(b)$ is given by
\begin{align*}
{\Gamma(r-(1-p)/2)\sqrt{\pi}(1-\mu^2)^{-(r-(1-p)/2)}}({\Gamma(p/2+r)})^{-1}
\,\tau\left(\gamma (b-{\bf 1}_b) (1+\D^2)^{-(r-(1-p)/2)}\right).
\end{align*}
The function of $r$ outside the trace has a simple 
pole at $r=(1-p)/2$ with residue equal to $1$,
and can be replaced by any other such function, 
such as $({r-(1-p)/2})^{-1}$. Thus modulo  functions holomorphic at
the critical point, we have
$$
\phi^r_{\mu,0}(b)=\phi^r_{0}(b-{\bf 1}_b).
$$
Thus we have proved the following proposition.

\begin{prop}
\label{mu-gives-cobound}
Let $(\A,\H,\D)$ be a  smoothly summable spectral triple of spectral dimension $p\geq 1$
and of parity $\bullet\in\{0,1\}$.
Let also $a_0\otimes a_1\otimes\dots\otimes a_m\in\A^\sim\otimes \A^{\otimes m}$. 
Let   $\{\phi^r_{\mu,m}\}_{m=\bullet,\bullet+2,\dots,M}$ and 
$\{\phi^r_{m}\}_{m=\bullet,\bullet+2,\dots,M}$ be the resolvent cocycles 
defined respectively  by the double and the original spectral triple. Then 
$\{\phi^r_{m}-\phi^r_{\mu,m}\}_{m=\bullet,\bullet+2,\dots,M}$ is a reduced $(b,B)$-coboundary 
modulo functions holomorphic at $r=(1-p)/2$.

If moreover the spectral dimension of $(\A,\H,\D)$ is isolated, for each $m>0$ we have
$$
{\rm res}_{r=(1-p)/2}\,\phi^r_{\mu,m}(a_0,\dots,a_m)={\rm res}_{r=(1-p)/2}\,\phi^r_{m}(a_0,\dots,a_m),
$$
and for $m=0$
$$
{\rm res}_{r=(1-p)/2}\,\phi^r_{\mu,0}(a_0)={\rm res}_{r=(1-p)/2}\,\phi^r_{0}(a_0-{\bf 1}_{a_0}).
$$
\end{prop}

\subsection{The local index formula}
Let $u\in M_n(\A^\sim)$ be a unitary and let $e\in M_n(\A^\sim)$ be a projection.
Set ${\bf 1}_e=\pi^n(e)\in M_n(\C)$ as in Equation \eqref{eq:wots-1-e}. We also observe that inflating
a smoothly summable spectral triple $(\A,\HH,\D)$ to 
$(M_n(\A),\H\otimes \C^n,\D\otimes{\rm Id}_n)$ yields a smoothly summable spectral triple
for $M_n(\A)$, with the same spectral dimension. Then
we can summarise the results of Sections 3 and 4 as follows.

\begin{theorem}
\label{localindex}\index{local index formula}
Let $(\A,\HH,\D)$ be a  semifinite spectral triple of 
parity $\bullet\in\{0,1\}$,
which is smoothly summable with spectral dimension $p\geq1$ and with $ \A$ separable. 
Let also  $M=2\lfloor(p+\bullet+1)/2\rfloor-\bullet$
be the largest integer of parity  $\bullet$ less than or equal to $p+1$.
Let $\D_{\mu,n}$ denote the operator coming from the 
double of the inflation $(M_n(\A),\H\otimes\C^n,\D\otimes {\rm Id}_n)$
of $(\A,\HH,\D)$, with phase $F_\mu\otimes{\rm Id}_n$ and $\D_{n}$ be the operator coming from the inflation
of $(\A,\HH,\D)$.
Then with the notations introduced above:\\
(1) The Chern character in cyclic homology computes the numerical index pairing, so
\begin{align*}
\langle [u], [(\A,\HH,\D)]\rangle 
&=\frac{-1}{\sqrt{2\pi i}}{\rm Ch}_{F_\mu\otimes{\rm Id}_n}^M\big({\rm Ch}^M(\hat u)\big),
\quad\mbox{(odd case)},\\
\langle [e]-[{\bf 1}_e], [(\A,\HH,\D)]\rangle 
&={\rm Ch}_{F_\mu\otimes{\rm Id}_n}^M\big({\rm Ch}^M(\hat e)\big),
\quad\mbox{(even case)}.
\end{align*}

(2) The numerical index pairing can also be computed with the resolvent cocycle of $\D_{n}$ via
\begin{align*}
\langle [u], [(\A,\HH,\D)]\rangle &=\frac{-1}{\sqrt{2\pi i}}{\rm res}_{r=(1-p)/2}
\sum_{m=1,\,odd}^M\phi^r_{m}\big({\rm Ch}^m(u)\big),
\quad\mbox{(odd case)},\\
\langle [e]-[{\bf 1}_e], [(\A,\HH,\D)]\rangle &={\rm res}_{r=(1-p)/2}
\sum_{m=0,\,even}^M\phi^r_{m}\big({\rm Ch}^m(e)-{\rm Ch}^m({\bf 1}_e)\big),
\quad\mbox{(even case)},
\end{align*}
and in particular for $x=u$ or $x=e$, depending on the parity,
$\sum_{m=\bullet}^M\phi_{m}^r({\rm Ch}_m(x))$ analytically 
continues to a deleted neighborhood of the critical point $r=(1-p)/2$ with at 
worst a simple pole at that point.\\
(3) If moreover the triple $(\A,\H,\D)$ has isolated spectral dimension, then the 
numerical index can also be computed with the residue cocycle for $\D_n$, via
\begin{align*}
\langle [u], [(\A,\HH,\D)]\rangle 
&=\frac{-1}{\sqrt{2\pi i}}\sum_{m=1,\,odd}^M\phi_{m}\big({\rm Ch}^m(u)\big),
\quad\mbox{(odd case)},\\
\langle [e]-[{\bf 1}_e], [(\A,\HH,\D)]\rangle 
&=\sum_{m=0,\,even}^M\phi_{m}\big({\rm Ch}^m(e)-{\rm Ch}^m({\bf 1}_e)\big),
\quad\mbox{(even case)}.
\end{align*}
\end{theorem}

\subsection{A nonunital McKean-Singer formula}
To illustrate this theorem, we prove a nonunital version of the McKean-Singer formula. 
To the best knowledge of the authors, there is no other version of McKean-Singer which
is valid without  the assumption that $f(\D^2)$ is trace class for some function $f$. 
Our assumptions are quite different from the usual 
McKean-Singer formula.

Let $(\A,\HH,\D)$ be an even semifinite smoothly summable spectral triple relative to 
$(\cn,\tau)$ with spectral dimension $p\geq 1$. Also, let $e\in M_n(\A^\sim)$ be a projection
with $\pi^n(e)={\bf 1}_e\in M_n(\C)\subset M_n(\cn)$. Then using the well known homotopy (with
$\D_n=\D\otimes{\rm Id}_n$)
\begin{align}
\label{eq:std-homotopy}
&\D_n=e\D_n e+(1-e)\D_n(1-e) 
+ t\big(e\D_n(1-e)+(1-e)\D_n e\big)\\
&\qquad =e\D_n e+(1-e)\D_n(1-e)
+t\big((1-e)[\D_n,e]-e[\D_n,e]\big)=:\D_{e}-t(2e-1)[\D_n,e],\nonumber
\end{align}
we see that we have an equality of the $KK$-classes associated to the spectral triples 
$$
[(M_n(\A),\HH\otimes\C^n,\D_n)]=[(M_n(\A),\HH\otimes\C^{ n},\D_e)]\in KK^0(\A,C),
$$
where $C$ is the (separable) $C^*$-algebra generated by the $\tau$-compact
operators listed in Definition \ref{defn:not-kn}.
However the property of smooth summability may not be preserved by this homotopy. 
The next lemma shows that the summability part is preserved.

\begin{lemma} 
\label{lem:stability}
Let $(\A,\H,\D)$ be a smoothly summable  semifinite spectral triple  relative
to $(\cn,\tau)$
with spectral dimension $p\geq 1$. Let $A\in {\rm OP}^0_0$
be a self-adjoint element. Then 
$$
\B_2(\D+A,p)= \B_2(\D,p)\quad{\rm and}\quad\B_1(\D+A,p)= \B_1(\D,p).
$$
\end{lemma}

\begin{proof}
For $K\in\N$ arbitrary, Cauchy's formula and the resolvent expansion gives 
\begin{align*}
(1+(\D+A)^2)^{-s/2}-(1+\D^2)^{-s/2}
&=\sum_{m=1}^K\frac{1}{2\pi i}\int_\ell\lambda^{-s/2}
\left(R(\lambda)(\{\D,A\}+A^2)\right)^mR(\lambda)d\lambda\\
&+\frac{1}{2\pi i}\int_\ell\lambda^{-s/2}\left(R(\lambda)(\{\D,A\}+A^2)\right)^{K+1}R_A(\lambda)d\lambda,
\end{align*}
where $R(\lambda)=(\lambda-(1+\D^2))^{-1}$,
$R_A(\lambda)=(\lambda-(1+(\D+A)^2))^{-1}$ and $\{\cdot,\cdot\}$ denotes 
the anticommutator.
Now since $\{\D,A\}+A^2$ is in ${\rm OP}^1_0$, Lemma \ref{lem:crucial} can be 
applied to all terms except the last, to see that each is  trace-class
for $s>p-m$. 
Using Lemma \ref{lem:was-5.3}, the H\"older inequality and estimating $R_A(\lambda)$ in norm,
we see that the integrand of the remainder term has trace norm
$$
\Vert \big(R(\lambda)(\{\D,A\}+A^2)\big)^{K+1}R_A(\lambda)\Vert_1\leq C_\eps(a^2+v^2)^{-(K+1)/4+(K+1)p/4q +(K+1)\eps-1/2},
$$
where $q> p$ and $\eps>0$. Choosing $q=p+\delta$ for some $\delta>0$, we may choose $K$ large enough so that
the integral over $v=\Re(\lambda)$ converges absolutely whenever 
$s>p-1$. Hence we can suppose that
the remainder term is trace-class for $s>p-1$.

Now let $T\in \B_2(\D,p)$ and use the tracial property to see that
\begin{align*}
\tau((1+(\D+A)^2)^{-s/4}T^*T(1+(\D+A)^2)^{-s/4})
&=\tau(|T|(1+(\D+A)^2)^{-s/2}|T|)\\
&=\tau(|T|(1+\D^2)^{-s/2}|T|)+C_s\\
&=\tau((1+\D^2)^{-s/4}T^*T(1+\D^2)^{-s/4})+C_s,
\end{align*}
where $C_s=\tau(|T|\left((1+(\D+A)^2)^{-s/2}-(1+\D^2)^{-s/2}\right)|T|)$ is finite for $s>p-1$
by the previous considerations.
By repeating the argument for $T^*$ we have $T\in \B_2(\D+A,p)$. As $\D=(\D+A)-A$, the argument is
symmetric, and we see that $\B_2(\D,p)=\B_2(\D+A,p)$. Now by definition $\B_1(\D,p)=\B_1(\D+A,p)$.
\end{proof}

Unfortunately, there is no reason to suppose that the smoothness properties of the spectral triple
$(M_n(\A),\H^n,\D_n)$ are preserved by the homotopy from $\D_n$ to $\D_e$. 
Instead, consider $(\A_e,\H^{ n},\D_e)$, where $\A_e$ is 
the algebra of polynomials in $e-{\bf 1}_e\in M_n(\A)$.
Then by Lemma \ref{lem:stability} and $[\D_e,e-{\bf 1}_e]=[\D_e,e]=0$
(which implies since $\D_e$ is self-adjoint that $[|\D_e|,e-{\bf 1}_e]=[|\D_e|,e]=0$ too) and
we easily check that $(\A_e,\H^n,\D_e)$
is a smoothly summable  spectral triple. 
Now employing the resolvent cocycle of $(\A_e,\H^{ n},\D_e)$ yields
\begin{align*}
\mbox{Index}_{\tau\otimes \tr_{2n}}\big(\hat e(F_{\mu,+}\otimes {\rm Id}_n)\hat e\big)&=
\mbox{res}_{r=(1-p)/2}\Big(\sum_{m=2,even}^M\phi^r_{\mu,m}
\big({\rm Ch}_m(\hat e)\big)\\
&\qquad\quad+\ \frac{1}{(r-(1-p)/2)}\tau\otimes\tr_n
\big(\gamma(e-{\bf 1}_e)(1+\D_{e}^2)^{-(r-(1-p)/2)}\big)\Big).
\end{align*} 
This equality follows from Proposition \ref{mu-gives-cobound} and the explicit computation of
the zero degree term. Now since $[\D_{e},e]=0$, $\phi^r_m({\rm Ch}_m(e))=0$ for all
$m\geq 2$. This proves the following nonunital McKean-Singer formula.

\begin{theorem}
\label{prop:McK}\index{McKean-Singer}
Let $(\A,\HH,\D)$ be an even semifinite smoothly summable  spectral triple relative to 
$(\cn,\tau)$ with spectral dimension $p\geq 1$  and with $\A$ separable.
Also, let $e\in M_n(\A^\sim)$ be a projection. Then
\begin{align*}
\langle [e]-[{\bf 1}_e], [(\A,\HH,\D)]\rangle &=
\langle [e]-[{\bf 1}_e], [(\A_e,\HH,\D)]\rangle\\
&={\rm res}_{r=(1-p)/2}\frac{1}{(r-(1-p)/2)}\tau\otimes\tr_n
\big(\gamma(e-{\bf 1}_e)(1+\D_{e}^2)^{-(r-(1-p)/2)}\big).
\end{align*}
\end{theorem}

This gives a nonunital analogue of the McKean-Singer
formula. Observe that the formula has $\D_{e}$ {\em not} 
$\D_n$.

{\bf Remark.} We have also proved a nonunital version of the Carey-Phillips spectral flow
formula for paths $(\D_t)_{t\in[0,1]}$ with unitarily equivalent   endpoints   and with $\dot{\D_t}$ satisfying 
suitable summability constraints. The
proof is quite lengthy, and so we will present this elsewhere.

\subsection{A classical example with weaker integrability properties}
\label{bott}

Perhaps surprisingly, given the difficulty of the nonunital case, we  gain
a little more freedom in choosing representatives of $K$-theory classes than we might have expected.
We do not formulate a general statement, but instead illustrate with an example.
This example involves a projection which does not live in 
a matrix algebra over (the unitisation of) 
our `integrable algebra' $\B_1(\D,p)$, but we may still use the local index formula 
to compute index pairings.

We will employ the uniform Sobolev algebra $W^{\infty,1}(\R^2)$, i.e. the Fr\'echet 
completion of $C^\infty_c(\R^2)$ for 
the seminorms 
$\mathfrak q_{n}(f):=\max_{n_1+n_2\leq n}\|\partial_1^{n_1}\partial_2^{n_2} f\|_1$. 
By the Sobolev Lemma, $W^{\infty,1}(\R^2)$ is continuously embedded
in $L^\infty(\R^2)$, and is  separable for the uniform topology as 
it contains $C^\infty_c(\R^2)$
as a dense subalgebra, and $C^\infty_c(\R^2)$ is 
separable for the uniform norm topology.

The spin Dirac operator on $\R^2\simeq\C$ is 
$
\dslash:=\begin{pmatrix} 0  & \partial_1+i\partial_2\\ -\partial_1+i\partial_2 & 0\end{pmatrix}
$,
with grading $\gamma:=\begin{pmatrix}1&0\\0&-1\end{pmatrix}$.  
Identifying a function with the operator of pointwise 
multiplication by it, an element $f\in W^{\infty,1}(\R^2)$ 
is represented as $f\otimes{\rm Id}_2$ on $L^2(\R^2,\C^2)$.

Anticipating the results of the next Section, we know by Proposition \ref{main-stuff} that
the triple
$\big(W^{\infty,1}(\R^2),L^2(\R^2,\C^2),\dslash\big)$ is 
smoothly summable,  relative to 
$\big(\B(L^2(\R^2,\C^2)),\mbox{Tr}\big)$ whose spectral dimension is $2$ and is isolated.
Thus, we can employ the residue cocycle to compute indices. 

Let $p_B\in M_2(C_0(\C)^\sim)$ be the Bott projector\index{Bott projector}
\begin{equation}
p_B(z):=\frac{1}{1+|z|^2}\begin{pmatrix} 1 & \bar{z}\\ z & |z|^2\end{pmatrix},\quad 
{\bf 1}_{p_B}=\begin{pmatrix} 0 & 0 \\ 0 & 1\end{pmatrix}.
\end{equation}
It is important to observe that
$p_B-{\bf 1}_{p_B}$ is {\em not} in $\B_1(\dslash,2)$  
since the off-diagonal terms are
not even  $L^2$-functions.

Since the fibre trace of $p_B-{\bf 1}_{p_B}$ is identically zero,
the zero degree term of the local index formula does not contribute to the index pairing. This
observation holds in general for commutative algebras since elements of $K_0$ then correspond 
to virtual bundles of virtual rank zero.

Thus there is only one term to consider in the local index formula, in degree $2$. More
generally, for even 
dimensional manifolds we will only ever need to consider the terms in the local index formula with
$m\geq 2$. 

This means that all we really require is that 
$[\dslash\otimes {\rm Id}_2,p_B][\dslash\otimes{\rm Id}_2,p_B]$ lies in $M_2(W^{\infty,1}(\R^2))$,
and this is straightforward to check.
Indeed, 
the routine computation
$$
(p_B-1/2)[\dslash\otimes {{\rm Id}_2},p_B][\dslash\otimes {{\rm Id}_2},p_B]
=\frac{-4}{(1+|z|^2)^3}\begin{pmatrix} 1/2 & \bar{z}/2 & 0 & 0\\
z/2 & |z|^2/2 & 0 & 0\\ 0 & 0 & -|z|^2/2 & \bar{z}/2\\
0& 0 & z/2 & -1/2\end{pmatrix},
$$
 shows that $(p_B-1/2)[\dslash\otimes   {\rm Id}_2,p_B][\dslash\otimes   {\rm Id}_2,p_B]$ is
a matrix over $W^{\infty,1}(\R^2)$. The fibrewise trace gives 
$$
\tr_2\big((p_B-1/2)[\dslash\otimes {\rm Id}_2,p_B][\dslash\otimes {\rm Id}_2,p_B]\big)
=\frac{-2}{(1+|z|^2)^2}\begin{pmatrix} 1 & 0\\ 0 &-1\end{pmatrix}.
$$
Applying \cite[Corollary 14]{Re2}, we find (the prefactor of $1/2$ comes from the coefficients in the local
index formula)
\begin{align*}
\frac{1}{2}\mbox{Tr}\otimes\tr_2
\big(\gamma(p_B-1/2)[\dslash\otimes {\rm Id}_2,p_B][\dslash\otimes {\rm Id}_2,p_B](1+\D^2)^{-1-\xi}\big)
&=-\frac{\Gamma(\xi)}{\Gamma(1+\xi)}\int_0^\infty \frac{r}{(1+r^2)^2}dr\\
&=-\frac{1}{2\xi}.
\end{align*}

Recalling that  the second component of the Chern character of $p_B$ introduces a factor of $-2$, 
we arrive at the numerical index
$$
\langle [p_B]-[{\bf 1}_{p_B}], \big[\big(W^{\infty,1}(\R^2),L^2(\R^2,\C^2),\dslash\big)\big]\rangle=1,
$$
as expected. This indicates that the resolvent cocycle extends by continuity to a larger 
complex, defined using norms of iterated projective tensor product type  associated to the norms $\PP_n$.
We leave a more thorough discussion of this to another place.




\section{Applications to index theorems on open manifolds}
\label{Mfd}

This section contains a discussion of some of what the 
noncommutative residue formula implies
for the classical situation of a noncompact manifold. The main contribution 
of the noncommutative approach
that we have endeavoured to explain here, is the extent to which compact 
support assumptions such as those in 
\cite{GL} may be avoided. However we do not exhaust all 
of the applications of the residue formula
in the classical case in this memoir. 

Our aim is to write an account of our results in a relatively complete fashion. 
We recall the basic definitions of spin geometry, \cite{LM}, and heat kernel estimates for manifolds
of bounded geometry. Using this data we construct a smoothly summable spectral triple for
manifolds of bounded geometry. Having done this, we use results of Ponge and Greiner to 
obtain an Atiyah-Singer formula for the index pairing on manifolds of bounded geometry.
Then we utilise the semifinite framework to obtain an $L^2$-index theorem for covers of manifolds
of bounded geometry.

\subsection{A smoothly summable spectral triple for manifolds of bounded geometry}

\subsubsection{Dirac-type operators and Dirac bundles}
\label{preliminaries}

Let $(M,g)$ be a (finite dimensional, paracompact, second countable) geodesically complete
Riemannian manifold. We let $n\in\N$ be the dimension of $M$ and $\mu_g$ be the Riemannian 
volume form. Unless otherwise specified, the measure involved in the definition of the Lebesgue function
spaces  $L^q(M)$, $1\leq q\leq\infty$, is the one associated with $\mu_g$.

We let $\D_S$ be a {\it Dirac-type operator} in the sense of \cite{GL, LM}. \index{Dirac-type operator}
Such operators are of the following form. Let $S\to M$, be a vector bundle, 
complex for simplicity, of rank $m\in\N$
and $(\cdot|\cdot)$,  a   fiber-wise Hermitian form. We suppose that $S$ is a bundle of left modules
over the Clifford bundle algebra ${\rm Cliff}(M):={\rm Cliff}(T^*M,g)$ \index{Clifford bundle}
which is such that for each unit vector $e_x$ of $T_x^*M$, 
the Clifford module multiplication $c(e_x):S_x\to S_x$ is a (smoothly varying) isometry. 
It is further  equipped with a metric compatible connection $\nabla^S$,  
such that for any smooth sections 
$\sigma \in\Gamma^\infty(S)$ and $\vf\in\Gamma^\infty({\rm Cliff}(M))$, it satisfies
\begin{equation}
\label{Liebniz-type}
\nabla^S(c(\vf)\sigma)=c(\nabla \vf)\sigma+c(\vf)\nabla^S( \sigma).
\end{equation} 
Here, $\nabla$ is the Levi-Civita connection naturally extended to
a (metric compatible) connection on ${\rm Cliff}(M)$ which satisfies, for $\vf,\,\psi\in{\rm Cliff}(M)$,
$\nabla(\vf\cdot\psi)=\nabla(\vf)\cdot\psi+\vf\cdot\nabla(\psi)$ 
(the dot here is the Clifford multiplication). 
We call such a bundle a {\it  Dirac bundle}, \cite[Definition 5.2]{LM}. \index{Dirac bundle}
Then, $\D_S$ is defined as the composition
$$
\Gamma^\infty(S)\to\Gamma^\infty(T^*M\otimes S)\to\Gamma^\infty(S),
$$
where the first arrow is given by $\nabla^S$ and the second by the Clifford multiplication.

For any orthonormal basis $\{e^\mu\}_{\mu=1,\dots,n}$ of $T_x^*M$, 
at each point $x\in M$ and   $\{e_\mu\}_{\mu=1,\dots,n}$ 
the dual basis of $T_xM$, with Einstein summation convention understood, we therefore have
$$
\D_S= c(e^\mu)\,\nabla^S_{e_\mu}.
$$
Let $\langle \sigma_1,\sigma_2\rangle_S=\int_M(\sigma_1|\sigma_2)(x)\,\mu_g(x)$ 
be the $L^2$-inner product on $\Gamma^\infty_c(S)$,
with $(\cdot | \cdot)$ the Hermitian form on $S$. 
As usual $L^2(M,S)$ is the associated  Hilbert space  completion
of $\Gamma^\infty_c(S)$. Recall that under the assumption of geodesic completeness, 
$\D_S$ is essentially self-adjoint and $\Gamma^\infty_c(S)$ is a core for $\D_S$, \cite[Corollary 10.2.6]{HR}
and \cite[Theorem 1.17]{GL}.
Moreover, if the Dirac bundle $S\to M$ is a $\mathbb Z_2$-graded ${\rm Cliff}(M)$-module, 
then $\D_S$ is odd, and in the usual 
matrix decomposition, it reads
$$
\D_S=\begin{pmatrix}0&\D_S^+\\ \D_S^-&0\end{pmatrix},\quad\mbox{with}\quad  (\D_S^\pm)^*=\D_S^\mp.
$$
We identify $L^\infty(M)$ with a subalgebra of the
bounded Borel sections of ${\rm Cliff}(M)$ in the usual way. 
We thus have a left action  $L^\infty(M)\times L^2(M,S)\to L^2(M,S)$ given by 
$(f,\sigma)\mapsto c(f)\sigma$. In a local trivialization of $S$, 
this action is given by the diagonal point-wise multiplication. It moreover satisfies $\|c(f)\|=\|f\|_\infty$. 

We recall now the important Bochner-Weitzenb\"ock-Lichnerowicz 
formula for the square of a Dirac-type operator:
\begin{equation}
\label{lich}
\D^2_S=\Delta_S+\tfrac12 \mathcal R,\qquad\mathcal R:=c(e^\mu)\,c( e^\nu)\,\,
F(e_\mu,e_\nu),
\end{equation}
where $\Delta_S:=(\nabla^S)^*\,\nabla^S$ is the Laplacian on $S$ and
 $ F:\Lambda^2T^*M\to {\rm End}(S)$ is the curvature tensor of $\nabla^S$. 

{\bf Remark.} Using the formula \eqref{lich}, 
Gromov and Lawson \cite[Theorem 3.2]{GL} have proven that if there exists 
a compact set $K\subset M$ such that
$$
\inf_{x\in M\setminus K}\,\sup\{\kappa\in\R:\mathcal R(x)\geq\kappa\,{\rm Id}_{S_x}\}> 0,
$$
then $\D_S$ (and thus $\D_S^\pm$ in the graded case) is Fredholm in the ordinary sense.

Note that the Leibniz-type relation \eqref{Liebniz-type} shows that  
for any $f\in C_c^\infty(M)$, the commutator $[\D_S,c(f)]$ extends to a bounded operator 
since an explicit computation gives
\begin{align}
\label{clascom}
[\D_S,c(f)]=c(df).
\end{align}

\subsubsection{The case of a manifold with bounded geometry}

Recall that  the injectivity radius  $r_{\rm inj}\in[0,\infty)$, is defined as\index{injectivity radius}
$$
r_{\rm inj}:=\inf_{x\in M}\sup\{r_x>0\},
$$
where $r_x\in(0,\infty)$ is such that the exponential map  $\exp_x$ is a 
diffeomorphism from  $B(0,r_x)\subset T_xM$ to
$U_{r,x}$, an open neighborhood of $x\in M$. 
We call {\it canonical coordinates} the coordinates given by 
 $\exp_x^{-1}: U_{r,x}\to B(0,r)\subset T_xM\simeq\R^n$.
Note that  $r_{\rm inj}>0$ implies that $(M,g)$ 
is geodesically complete. 

With these preliminaries, we recall the  definition of bounded geometry.

\begin{definition}
\label{BG}
 A  Riemannian manifold  $(M,g)$ is said to  
have bounded geometry  if it has strictly positive injectivity radius and all  the covariant derivatives  
of the curvature tensor are bounded on $M$. 
A Dirac bundle on $M$ is said to have bounded geometry if in addition all the  covariant derivatives  
of $F$, the curvature tensor of the connection $\nabla^S$, are bounded on $M$. 
For brevity, we simply say that $(M,g,S)$ has bounded geometry.
\end{definition}
\index{bounded geometry}

We summarise some facts about manifolds of bounded geometry.
Bounded geometry allows the 
construction of  canonical coordinates which are such that the 
transition functions have bounded derivatives of all orders, uniformly 
on $M$, \cite[Proposition 2.10]{Ro}. Moreover, for all $\eps\in(0,r_{\rm inj}/3)$, there exist
countably many points $x_i\in M$, such that $M=\cup B(x_i,\eps)$ and 
such that the covering of $M$ by the balls $B(x_i,2\eps)$ has finite order. (Recall that the order
of a covering of a topological space, is the least integer $k$, such that such the intersection of any
$k+1$ open sets of this covering, is empty.)
Subordinate to the covering by the balls $B(x_i,2\eps)$, there exists
a partition of unity, $\sum_i \vf_i=1$, with ${\rm supp} \,\vf_i\in B(x_i,2\eps)$ 
and such that their  derivatives of all orders and in normal coordinates, are  
bounded, uniformly in the covering index $i$.  See \cite[Lemmas 1.2, 1.3, Appendix 1]{Shu4} for details and proofs
of all these assertions. Also, a differential operator is said to have 
{\it uniform $C^\infty$-bounded coefficients}, if for any atlas consisting of charts of 
normal coordinates, the derivatives of all order of the coefficients are bounded on the 
chart domain and the bounds are uniform on the atlas.

The next proposition follows from results of Kordyukov \cite{Kordyukov} and Greiner \cite{Gr}, and records 
everything that we need
to know about the heat semi-group with generator $\D_S^2$.

\begin{prop} Let $(M,g)$ be a Riemannian manifold of dimension $n$ with bounded geometry.
Let $\D_S$ be a Dirac type operator acting on the sections of a Dirac bundle $S$
of bounded geometry and $P$ a differential operator on $\Gamma^\infty_c(S)$ of order 
$\alpha\in\mathbb N$, with uniform $C^\infty$-bounded coefficients. Let then  
$K_{t,P}^S(x,y)\in{\rm Hom}(S_x,S_y)$ be the operator kernel of $P\,e^{-t\D_s^2}$.
Then:\\
i) We have the global  off-diagonal gaussian upper bound
$$
\big|K_{t,P}^S(x,y)\big|_\infty\leq C\,t^{-(n+\alpha)/2}\, \exp\Big(-\frac{d_g^2(x,y)}{4(1+c)t}\Big),\quad t>0,
$$
where $|\cdot|_\infty$ denotes the operator norm on ${\rm Hom}(S_x,S_y)$ 
and $d_g$ the geodesic distance function.\\
ii) We have the  short-time asymptotic expansion
\begin{equation*}
\tr \big(K_{t,P}^S(x,x)\big) \sim_{t\to 0^+} t^{-\lfloor \alpha/2\rfloor-n/2}\sum_{i\geq 0} t^ib_{P,i}(x),\quad\mbox{for all }\,x\in M,
\end{equation*}
where the functions $b_{P,i}(x)$ are determined by a finite number of jets of the 
principal symbol of $P(\partial_t +\D_S^2)^{-1}$.\\
iii) Moreover, this local asymptotic expansion carries through to give a global one: 
For any $f\in L^1(M)$, we have
$$
\int_Mf(x)\, \tr \big(K_{t,P}^S(x,x)\big)\,d\mu_g(x) 
\sim_{t\to 0^+} t^{-\lfloor \alpha/2\rfloor-n/2}\sum_{i\geq 0} t^i \int_Mf(x)\, b_{P,i}(x)\,d\mu_g(x) .
$$
\label{rrr}
\end{prop}
\index{bounded geometry!heat kernel expansion}

\begin{proof} 
When $M$ is compact, the first two results can be found in  \cite[Chapter I]{Gr}. 
When $M$ is noncompact but has bounded geometry, Kordyukov
 has proven in  \cite[Section 5.2]{Kordyukov} that all the relevant gaussian bounds used in \cite{Gr} 
 to construct a fundamental solution, via the Levi method, of a parabolic equation associated 
 with an elliptic differential operator, remains valid for any uniformly elliptic differential operator with
 $C^\infty$-bounded coefficients, which is the case for $\D_S^2$. 
 The only restriction for us is that Kordyukov treats the scalar
 case only. However, a careful inspection of his arguments shows that the same bounds
 still hold for a uniformly elliptic differential operator acting on the smooth sections of a vector bundle of
 bounded geometry, as far as the operators under consideration have $C^\infty$-bounded coefficients. 
 With these gaussian bounds at hand (for the approximating solution and for the remainder term), one 
 can then repeat word for word the arguments of Greiner to conclude for i) and ii). 
 For  iii) one uses Kordyukov's bounds extended to the vector 
 bundle case, \cite[Proposition 5.4]{Kordyukov}, to see that for all $k\in\N_0$, one has
 $$
 \Big|\tr \big(K_{t,P}^S(x,x)\big) -t^{-\lfloor \alpha/2\rfloor-n/2}\sum_{i= 0}^k t^i b_{P,i}(x)\Big|
 \leq C\,t^{-\lfloor \alpha/2\rfloor-n/2+k+1},
 $$
for a constant $C>0$, independent of $x\in M$.
This is enough to conclude.
\end{proof}

 Given  $\omega$, 
a weight function (positive and nowhere vanishing) on $M$, 
we denote by $W^{k,l}(M,\omega)$, $1\leq k\leq\infty$, $0\leq l<\infty$,  
the weighted uniform Sobolev space. That is to say, the  completion of $C^\infty_c(M)$ for
the topology associated to the norm
$$
\|f\|_{k,l,\omega}:=\Big(\int_M|\Delta^{l/2}f|^k\,\omega\,d\mu_g\Big)^{1/k},
$$
where,  $\Delta$ denotes the scalar Laplacian on $M$. 
For $\omega=1$ we simply denote this space by  $W^{k,l}(M)$ 
and the associated norm by $\|\cdot\|_{k,l}$. We also write
$W^{k,\infty}(M,\omega):=\bigcap_{l\geq 0}W^{k,l}(M,\omega)$ 
endowed with the projective limit topology.

 When $M$ has strictly positive injectivity radius (thus in particular 
 for manifolds of bounded geometry), the standard Sobolev 
 embedding  
 $$
 W^{k,l}(M)\subset L^\infty(M),
 $$
holds  for any $1\leq k\leq\infty$ 
 and $l>n/k$ (see \cite[Chapter 2]{Aubin82}). 
 In particular, if $\eps>0$ then $W^{k,n/k+\eps}(M)$ 
 is not only a Fr\'echet space but a Fr\'echet algebra.
 Moreover, $W^{k,l}(M)\subset C_0(M)$ for $1\leq k\leq\infty$
 and $0\leq l\leq\infty$, so that it is separable for the uniform topology
 as $M$ is metrisable.
  The next lemma gives equivalent
 norms for the weighted Sobolev spaces $W^{k,l}(M,\omega)$.
 
 \begin{lemma}
 \label{sobo-local}
 Let $\sum\vf_i=1$ be a partition of unity subordinate to a covering of 
 $M$ by balls of radius $\eps\in(0,r_{\rm inj}/3)$.
 Then the norm $\|\cdot\|_{k,l,\omega}$ on $W^{k,l}(M,\omega)$, 
 $1\leq k\leq\infty$, $l\in\N_0$, is equivalent to
 $$
f\mapsto \sum_{i=1}^\infty \|\vf_i f\|_{k,l,\omega}.
$$
 \end{lemma}
 
 \begin{proof}
 This is the weighted version of the discussion which follows   \cite[Lemma 1.3, Appendix 1]{Shu4}, which is a 
 consequence of the fact that the normal derivatives of $\vf_i$ are bounded uniformly in 
 the covering index and because this covering has finite order.
 \end{proof}

In the following lemma, we examine  first  the question of (ordinary) smoothness before 
turning to smooth summability.

\begin{lemma}
\label{dfg}
Let $(M,g,S)$ have bounded geometry. For $T$ an operator on $L^2(M,S)$ 
preserving the domain of $\D_S$, define
$\delta(T)=[|\D_S|,T]$. Then for any  $f\in W^{\infty,\infty}(M)$, 
the  operators $c(f)$ and $c(df)$ on $L^2(M,S)$   belong to $\bigcap_{l=0}^\infty{\rm dom}\,\delta^l$.
\end{lemma}

\begin{proof}
By the discussion  following Definition \ref{parup}, it suffices to show that for  
$f\in W^{\infty,\infty}(M)$, $c(f)$ belongs to $\bigcap_{l=0}^\infty{\rm dom}\,R^l$, 
with $R(T)=[\D_S^2,T](1+\D_S^2)^{-1/2}$.
Next observe that since $[c(f),\mathcal R]=0$, with $\mathcal R$ the zero-th order operator
appearing in \eqref{lich}, we have
\begin{align*}
R^k\big(c(f)\big)&=[\D_S^2,[\dots,[\D^2_S,[\D_S^2,c(f)]]\dots]](1+\D_S^2)^{-k/2}\nonumber\\
&=
[\Delta_S+\tfrac12\mathcal R,[\dots,[\Delta_S+\tfrac12\mathcal R,[\Delta_S,c(f)]]\dots]](1+\D_S^2)^{-k/2},
\label{end}
\end{align*}
with $k$ commutators. Define
$$
B_k:=[\Delta_S+\tfrac12\mathcal R,[\dots,[\Delta_S+\tfrac12\mathcal R,[\Delta_S,c(f)]]\dots]],
$$
so that $R^k\big(c(f)\big)=B_k\,(1+\D_S^2)^{-k/2}$.
Since the principal symbol of $\Delta_S$ is $|\xi|^2{\rm Id}_{S_x}$,
a local computation shows that  $B_k$ is a differential operator of order $k$. 
With the bounded geometry assumption, we see moreover that $B_k$ has 
uniform $C^\infty$-bounded coefficients. (This follows because the 
covariant derivatives of $\mathcal R$ will appear in the expression of the 
coefficients of $B_k$ and since 
$\mathcal R(x)=c(e^\mu_x)\,c(e^\nu_x)\, F(e_{\mu,x},e_{\nu,x})\in{\rm End}(S_x)$.) 
In particular, $B_k$ is a properly supported pseudodifferential operator with bounded 
symbol (in the sense of \cite[Definition 2.1]{Kordyukov}) of order $k$. 
While $(1+\D_S^2)^{-k/2}$ is not a properly supported pseudodifferential operator, 
it can be written as the sum of a properly supported pseudodifferential operator of order $-k$ and
an infinitely smoothing operator; see \cite[Theorem 3.3]{Kordyukov} for more information.
Hence by 
\cite[Proposition 2.7]{Kordyukov}, $R^k\big(c(f)\big)$  is  properly supported  with bounded symbol of zeroth order.
Then one concludes using \cite[Proposition 2.9]{Kordyukov}, where one needs \cite[Theorem 3.6, Appendix]{Shu4}
instead of \cite[Lemma 2.2]{Kordyukov} used in that proof, to extend the result to the case of a 
vector bundle of bounded geometry. The proof for $c(df)$ is entirely similar.
\end{proof}

As before, we let $K_t^S$, $t>0$, be the Schwartz kernel of the 
heat semigroup with generator $\D_S^2$. When it exists, we let 
$k_s$, $s>0$, be the restriction to the 
diagonal of the fibre-wise trace of the distributional kernel of 
$(1+\D_S^2)^{-s/2}$. That is for $s>0$ and $x\in M$, we set
$$
k_s(x)=\tr\big([(1+\D_S^2)^{-s/2}]_{x,x}\big),
$$
where the trace $\tr$ is the matrix trace on ${\rm End}(S_x)$ 
and for $A$ a bounded operator on $L^2(M,S)$ we denote
by $[A]_{x,y}$ its distributional kernel.

Now assuming the geodesic completeness of $M$, the heat kernel 
$K_t^S$, $t>0$, is a 
smooth section of the endomorphism bundle of $S$. Combining this with
the Laplace transform representation
$$
k_s(x)=\frac1{\Gamma(s/2)}\int_0^\infty t^{s/2-1}\,e^{-t}\,\tr\big(K_t^S(x,x)\big)\,dt,\quad \mbox{for all } \,x\in M,
$$
we see that  the question 
of existence of $k_s$   is uniquely determined by the integrability 
of the on-diagonal fibre-wise trace of the Dirac heat kernel with 
respect to the parameter $t$. More precisely, Proposition
\ref{rrr} i) gives

\begin{lemma}
Let $\D_S$ be a Dirac type operator operating on the sections of a Dirac bundle $S$
of bounded geometry. Then, for $s>n$, the function $k_s$ is uniformly bounded on $M$.
\label{reff}
\end{lemma}

As a corollary of the lemma above, we see that $W^{r,t}(M)\subset W^{r,t}(M,k_{s})$ with
$\|\cdot\|_{r,t,k_{s}}\leq C(s)\|\cdot\|_{r,t}\,$, for some  constant $C(s)$  independent of $r\in[1,\infty]$
and of $t\in\R$. 
\begin{lemma}
\label{HS}
Let $\D_S$ be a Dirac type operator operating on the sections of a Dirac bundle $S$
of bounded geometry. Then provided $f\in W^{2,0}(M,k_{s})$ and $s>n$, 
the operator $c(f) (1+\D_S^2)^{-s/4}$ is Hilbert-Schmidt on $L^2(M,S)$, with
$$
\|c(f) (1+\D_S^2)^{-s/4}\|_2=\Big(\int_M |f|^2(x)\,k_{s}(x)\,d\mu_g(x)\Big)^{1/2}=\|f\|_{2,0,k_{s}}.
$$
\end{lemma}

\begin{proof}
From Lemma \ref{reff}, the function $k_{s}$ is well defined and uniformly bounded on $M$.
Now let $A$ be a bounded operator acting on $L^2(M,S)$, with 
distributional kernel $[A]_{x,y}$. Then for $f\in L^\infty(M)$, 
a calculation shows that $A\,c(f) $ has distributional kernel $f(y)[A]_{x,y}$. 
We then have  the following expression for the Hilbert-Schmidt norm of $A\,c(f)$:
\begin{align*}
\|Ac(f)\|_2^2&=\int_{M\times M}\tr\big(|[A\,c(f)]_{x,y}|^2\big)\,d\mu_g(x)\,d\mu_g(y)
=\int_{M\times M}|f(y)|^2\tr\big(|[A]_{x,y}|^2\big)\,d\mu_g(x)\,d\mu_g(y)\\
&=\int_{M\times M}|f(y)|^2\tr\big([A^*]_{y,x}[A]_{x,y}\big)\,d\mu_g(x)\,d\mu_g(y)
=\int_{M}|f(y)|^2\tr\big([A^*A]_{y,y}\big)\,d\mu_g(y),
\end{align*} 
where in the last equality we used the operator-kernel product rule. Then, the proof follows by setting 
$A=(1+\D_S^2)^{-s/4}$.
\end{proof}
As explained above, we identify the von Neumann algebra 
generated by $\{c(f), \,\,f\in C^\infty_c(M)\}$ acting on $L^2(M,S)$ with $L^\infty(M)$. 
Then, from the previous Hilbert-Schmidt norm computation, we can determine 
the weights $\vf_s$ of Definition \ref{def:d-does-int}, constructed  with $\D_S$.

\begin{corollary} 
Let $\D_S$ be a Dirac type operator operating on the sections of a Dirac bundle $S$
of bounded geometry. 
For $s>n$, let $\vf_s$ be the faithful normal semifinite weight of Definition \ref{def:d-does-int},  
on the type I von Neumann
algebra  $\B(L^2(M,S))$ with operator trace.  When  restricted to $L^\infty(M)$, $\vf_s$ 
coincides with the integral on $M$ with respect to the
Borel measure $k_s\,d\mu_g$.
\end{corollary}

We turn now to the question of which functions on the manifold
are in  $\B_1^\infty(\D_S,n)$. Combining Proposition \ref{tracial-case} with Lemma \ref{HS}   
allows us to determine the norms $\PP_m$ restricted to $L^\infty(M)$.

\begin{corollary}
\label{aB1}
Let $\D_S$ be a Dirac type operator operating on the sections of a Dirac bundle $S$
of bounded geometry. Then 
$$
\B_1(\D_S,n)\bigcap L^\infty(M)= L^\infty(M)\bigcap_{m\in\N}L^1(M,k_{s+1/m}d\mu_g).
$$
Moreover we have the    equality 
\begin{align*}
\PP_m\big(c(f)\big)= \|f\|_\infty+2\|f\|_{1,k_{n+1/m}}\,,\quad m\in\N.
\end{align*}
\end{corollary}

By Lemma \ref{reff}, we see that $\bigcap_{m\in\N}L^1(M,k_{s+1/m}d\mu_g)$ contains $L^1(M)$. 
Note also that if a uniform  on-diagonal lower bound for the Dirac heat kernel of the form
$$
\big|K_t^S(x,x)\big|_\infty\geq c t^{-n/2},
$$ 
holds (with $|\cdot|_\infty$ the operator norm on ${\rm End }(S_x)$), then  
$\bigcap_{m\in\N}L^1(M,k_{s+1/m}d\mu_g)=L^1(M)$. Such an estimate holds 
for the spin Dirac operator on Euclidean spaces, for example, and for the 
scalar heat kernel for any manifold of bounded geometry.

We now arrive at the main statement of this Section.
\begin{prop}
\label{main-stuff}
Let $\D_S$ be a Dirac type operator operating on the sections of a Dirac bundle $S$
of bounded geometry on a manifold of bounded $M$ of dimension $n$. 
Relative to the $I_\infty$ factor  $\B(L^2(M,S))$ with operator trace, the spectral  triple $\big(W^{\infty,1}(M),L^2(M,S),\D_S\big)$  is smoothly summable and 
of spectral dimension $n$. Moreover, the spectral dimension is isolated in the sense of Definition \ref{dim-spec}.
\end{prop}

\begin{proof}
We  first show that for any $f\in W^{\infty,1}(M)$, the operators $\delta^k(c(f))$ and $\delta^k(c(df))$, $k\in\N_0$, 
all belong to $\B_1(\D_S,n)$. That $c(f)\in \B_1(\D_S,n)$ for $f\in W^{\infty,1}(M)$ has already been proven in 
Corollary \ref{aB1} since $\bigcap_{m}W^{\infty,1}(M,k_{n+1/m})\supset W^{\infty,1}(M)$. 
For the rest, we know by Proposition \ref{smooth-sum-sufficient} that it is sufficient to prove that
$$
(1+\D_S^2)^{-s/4}R^k(c(f))(1+\D_S^2)^{-s/4}\in\L^1\big(L^2(M,S)\big),\,\,\mbox{for all } k\in\N_0,\,\,\mbox{for all } s>n,
$$
and similarly for $c(df)$.

From the proof of Lemma \ref{dfg}, we also know that for 
$f\in W^{\infty,1}(M)\subset W^{\infty,\infty}(M)$, the operators
$R^k(c(f))$ and $R^k(c(df))$ are of the form $B_k(1+\D_S^2)^{-k/2}$, 
where $B_k$ is a differential operator of order
$k$, with uniform $W^{\infty,1}(M)$-coefficients. This means  
that for any covering of $M=\cup B(x_i,\eps)$ of balls 
of radius $\eps\in(0,r_{\rm inj}/3)$ and partition of unity $\sum\vf_i=1$ subordinate to the covering, 
there exist  elements $f_\alpha\in{\rm End}(S_x)$ with
$B_k|_{B(x_i,\eps)}=\sum_{|\alpha|\leq k}f_{\alpha}\partial^\alpha$ in normal coordinates.  
Moreover, $\sum_{i=0}^\infty\|\vf_i |f_{\alpha}|_\infty\|_1<\infty$, where $|\cdot|_\infty$ 
is the operator norm on ${\rm End}(S_x)$, each $\vf_i$ has 
bounded derivatives of all order, uniformly in the covering index $i$. 
Now take $\sum\psi_i=1$ a second partition of unity subordinate to the covering 
 $M=\cup B(x_i,2\eps)$ (recall that the latter has finite order), with $\psi_i(x)=1$ in a 
 neighbourhood of ${\rm supp}(\vf_i)$. We then have
 $$
 B_k=\sum_{i=0}^\infty \psi_i B_k \vf_i
 =\sum_{i=0}^\infty  \sum_{|\alpha|\leq k} \psi_if_{\alpha}\partial^\alpha \vf_i
 =\sum_{i=0}^\infty  \sum_{|\alpha|,|\beta|\leq k} \psi_if_{\alpha}\partial^\beta(\vf_i)\partial^\alpha .
 $$
Let $\psi_if_{\alpha}\partial^\beta(\vf_i)=u_{i,\alpha,\beta}|\psi_if_{\alpha}\partial^\beta(\vf_i)|$ 
be the polar decomposition. Define 
 $$
 C_{i,\alpha,\beta}:=u_{i,\alpha,\beta}|\psi_if_{\alpha}\partial^\beta(\vf_i)|^{1/2},\quad
 D_{i,\alpha,\beta}:=|\psi_if_{\alpha}\partial^\beta(\vf_i)|^{1/2}\partial^\alpha,
 $$
 so that 
 $$
  (1+\D_S^2)^{-s/4}B_k(1+\D_S^2)^{-s/4}=\sum_{i=0}^\infty  \sum_{|\alpha|,|\beta|\leq k} 
   (1+\D_S^2)^{-s/4}C_{i,\alpha,\beta}\, D_{i,\alpha,\beta}(1+\D_S^2)^{-(s+2k)/4}.
 $$
 The fibre-wise trace of the on-diagonal operator  kernel of  
 $ C_{i,\alpha,\beta}^*(1+\D_S^2)^{-s/2}C_{i,\alpha,\beta}$ being given by
 $|\psi_i(x)f_{\alpha}(x)\partial^\beta(\vf_i)(x)|_1k_s(x)$ (with $|\cdot|_1$  
 the trace-norm on ${\rm End}(S_x)$), we have
  for $s>n$
 $$
{\rm Tr}\big(C_{i,\alpha,\beta}^*(1+\D_S^2)^{-s/2}C_{i,\alpha,\beta}\big)
= \int_{  B(x_i,2\eps)}|\psi_i(x)f_{\alpha}(x)\partial^\beta(\vf_i)(x)|_1k_s(x)d\mu_g(x),
$$ 
 so that 
 $$
 \| (1+\D_S^2)^{-s/4}C_{i,\alpha,\beta}\|_2
 =\|\psi_i|f_{\alpha}|_1\partial^\beta(\vf_i)\|_{1,0,k_s}^{1/2}
 \leq C_{\alpha,\beta}\|\psi_i|f_{\alpha}|_\infty\|_1^{1/2}.
 $$
 For $D_{i,\alpha,\beta}$, note that the off-diagonal  kernel of 
 $D_{i,\alpha,\beta}(1+\D_S^2)^{-(s+2k)/2}D_{i,\alpha,\beta}^*$
 reads up to a $\Gamma$-function factor
 $$
i^{|\alpha|} |\psi_if_{\alpha}\partial^\beta(\vf_i)|(x)^{1/2}\,\int_0^\infty t^{(s+2k)/2-1}\,e^{-t}\,\partial^\alpha_x\partial^\alpha_yK_t^S(x,y)\,dt \,\,|\psi_if_{\alpha}\partial^\beta(\vf_i)|(y)^{1/2}.
 $$
 But Proposition \ref{rrr} i) gives
 $$
 |\partial^\alpha_x\partial^\beta_yK_t^S(x,y)|_\infty\leq C'(\alpha,\beta)
 t^{-(n+|\alpha|+|\beta|)/2} \exp\Big(-\frac{d_g^2(x,y)}{4(1+c)t}\Big),\quad t>0.
 $$
 Since $|\alpha|,\,|\beta|\leq k$, we finally obtain the inequality
 \begin{align*}
 \|D_{i,\alpha,\beta}(1+\D_S^2)^{-(s+2k)/4}\|_2^2
 &\leq C'(\alpha)\int_{  B(x_i,2\eps)} |\psi_if_{\alpha}\partial^\beta(\vf_i)|_\infty(x)d\mu_g(x)\\
 &\leq
 C''(\alpha,\beta)\int_{  B(x_i,2\eps)} |\psi_i|\,|f_{\alpha}|_\infty(x)d\mu_g(x)
 = C''(\alpha,\beta)\|\psi_i|f_{\alpha}|_\infty\|_1.
 \end{align*}
 Thus, 
  \begin{align*}
  \|(1+\D_S^2)^{-s/4}B_k(1+\D_S^2)^{-s/4}\|_1&\leq\sum_{i=0}^\infty  \sum_{|\alpha|,|\beta|\leq k}\| 
   (1+\D_S^2)^{-s/4}C_{i,\alpha,\beta}\|_2\,\| D_{i,\alpha,\beta}(1+\D_S^2)^{-(s+2k)/4}\|_2\\
   &\leq C\sum_{i=0}^\infty  \sum_{|\alpha|\leq k}\|\psi_i|f_{\alpha}|_\infty\|_1,
   \end{align*}
 which is finite by Lemma \ref{sobo-local}.
    This proves that for all $k\in\N_0$,  $\delta^k(c(f))$ and  $\delta^k(c(df))$ are in $\B_1(\D_S,n)$.
   We also have proven that the triple $\big(W^{\infty,1}(M),L^2(M,S),\D_S\big)$ is finitely summable.
   
   That $n$ is the smallest number such that $c(f)(1+\D_S^2)^{-s/2}$ is trace class for all $s>n$ follows from Proposition \ref{rrr} iii), 
   since 
   $$
  {\rm Tr}\big( c(f)(1+\D_S^2)^{-s/2}\big)=
  \frac1{\Gamma(s/2)}\int_0^\infty t^{s/2-1}\,e^{-t}\,\int_M f(x)\,\tr\big(K_t^S(x,x)\big)\,d\mu_g(x)\,dt  ,
  $$ 
  and
  $$
  \tr\big(K_t^S(x,x)\big)\sim_{t\to 0}t^{-n/2}\sum_{i\geq 0} t^i\,b_i(x).
  $$
  Thus, the spectral dimension is  $n$.
  
  Last, that the spectral dimension is isolated follows from the fact that it has discrete dimension spectrum, which follows from Proposition \ref{rrr} iii) 
  and the trace computation above, since for any
  $f_0,f_1,\dots,f_m\in W^{\infty,1}(M)$, the operator 
  $$
  c(f_0)c(df_1)^{(k_1)}\cdots c(df_m)^{(k_m)},
  $$
  is a differential operator of order $|k|=k_1+\dots+k_m$ with uniform $C^\infty$-bounded coefficients.
 \end{proof}

\subsection{An index formula for manifolds of bounded geometry}

\subsubsection{Extension of the Ponge approach.}
We still consider  $(M,g)$,  a complete Riemannian manifold   of 
dimension $n$, but now suppose that $(M,g)$ is spin. We fix
 $S$  to be the spinor bundle endowed with a 
connection $\nabla^S$ which is the usual lift of the Levi-Civita connection. We
 let $\D_S$ be the associated Dirac operator. We still assume that $(M,g,S)$ has bounded 
 geometry, in the sense of Definition \ref{BG}. 

Now we need to explain how to use the asympotic expansions of Proposition \ref{rrr} iii), to deduce
the Atiyah-Singer local index formula from the residue cocycle formula for 
the index. (Recall that by Proposition \ref{main-stuff}, 
the spectral triple $\big(W^{\infty,1}(M),L^2(M,S),\D_S\big)$
has isolated spectral dimension, so that we can use the last version of Theorem \ref{localindex} to compute the index.)
 The key tool is Ponge's adaptation of Getzler's arguments, \cite{Ponge}.

As Ponge and Roe explain, \cite{Ponge,Ro},
the arguments that Gilkey uses to prove that the coefficients in the 
asymptotic expansion of the Dirac Laplacian are universal polynomials 
carries over to the noncompact situation and produces universal polynomials 
identical to those of the compact case. 
Moreover Ponge's argument is purely 
local; that is, it proceeds by choosing a single point in $M$ and checking what the 
asymptotic expansion gives for the terms in the residue cocycle formula at that point. 
As such there is no change needed in Ponge's argument
to handle complete manifolds of bounded geometry.

Thus both the following results are proven just as in Ponge, and the only work is in checking
that the constants are consistent with our conventions.

\subsubsection{The odd case} We treat the odd case first, which is not affected by our `doubling up' construction.
\begin{theorem}
Let $(M,g,S)$ be a Riemannian spin  manifold with bounded geometry  
and of odd dimension $n=1,3,5,\dots$.
Let $\big(W^{\infty,1}(M),L^2(M,S),\D_S\big)$ be the smoothly summable 
spectral triple of spectral dimension $n$ described in the last section. 
The components of the odd residue cocycle are given by
$$
\phi_{2m+1}(f^0, f^1,\ldots, f^{2m+1})=
 \frac{(-1)^m\sqrt{2\pi i}}{(2\pi i)^{ \frac{n+1}{2} }(2m+1)!\,m!}
\int_M f^0df^1\wedge\cdots\wedge df^{2m+1}\wedge\hat{A}(R)^{(n-2m-1)},$$
for $f^0,f^1,\ldots ,f^{2m+1}\in W^{\infty,1}(M)$,  $m\geq 0$, $R$ being the  curvature tensor of $M$.
\end{theorem}

{\bf Remark.} The $A$-roof genus, $\hat{A}(R)$, is computed here with no normalisation
of the Pontryagin classes by factors of $2\pi i$. To obtain the index formula in the next result, one
should use the $(b,B)$-Chern character of a unitary $u\in M_N\big(W^{\infty,1}(M)^\sim\big)$, antisymmetrising after taking
the matrix trace.

\begin{corollary}
For any unitary $u\in M_N\big(W^{\infty,1}(M)^\sim\big)$ 
 and with $2P_\mu-1$ 
being the phase of $\D_{S,\mu}\otimes {\rm Id}_N$ and $P=\chi_{[0,\infty)}( {\D_S})\otimes {\rm Id}_N$, we have
the odd index pairing given by
$$
{\rm Ind}(PuP)={\rm Ind}(P_\mu\hat uP_\mu)= -\frac{1}{(2\pi i)^{\frac{n+1}{2}}}\sum_{m=0}^{\frac{n-1}2}
\frac{(-1)^m}{(2m+1)!\,m!}\int_M  {\rm Ch}_{2m+1}(u)\wedge \hat A(R)^{(n-2m-1)}.
$$
\end{corollary}
\index{index theorem!Atiyah-Singer formula}

\subsubsection{The even case}

Now  as the rank of a projection $f\in M_N\big(W^{\infty,1}(M)^\sim\big)$ is constant 
on connected components and equal to the rank of ${\bf 1}_f$,
the contribution of the zeroth term to the local index formula is zero.
It remains therefore to compute $\phi_{2m}$ for $m\geq 1$ 
evaluated on the Chern character of a projection $f$.

\begin{theorem}
Let $(M,g,S)$ be a Riemannian spin manifold with bounded geometry  
and of even dimension $n=2,4,6,\dots$.
Let $\big(W^{\infty,1}(M),L^2(M,S),\D_S\big)$ be the smoothly summable 
spectral triple of spectral dimension $n$ described in the last section. 
The non-zero components of the even residue cocycle  are given by
$$
\phi_{2m}(f^0, f^1,\ldots, f^{2m})= \frac{(-1)^m}{(2\pi i)^{n/2}(2m)!}
\int_M f^0df^1\wedge  \cdots\wedge df^{2m} \wedge \hat{A}(R)^{(n-2m)},\quad m\geq 1,
$$
for $f^0,f^1,\ldots, f^{2m}\in W^{\infty,1}(M)$, $R$ being the  curvature tensor of $M$.
\end{theorem}

Again the $A$-roof genus is defined without $2\pi i$ normalisations, 
and in the following result one uses the $(b,B)$-Chern character of $f\in M_N\big(W^{\infty,1}(M)^\sim\big)$, 
antisymmetrising after taking the trace.

\begin{corollary}
For any projector $f\in M_N\big(W^{\infty,1}(M)^\sim\big)$
 and with $F_\mu$ 
being the phase of $\D_{S,\mu}\otimes {\rm Id}_N$, we have
$$
{\rm Ind}(\hat f\, F_{\mu,+}\,\hat f\big) =(2\pi i)^{-n/2}\sum_{m=1}^{\frac n2}\frac{(-1)^m}{(2m)!}
\int_M {\rm Ch}_{2m} (f)\wedge\hat{A}(R)^{(n-2m)}  .
$$
\end{corollary}
\index{index theorem!Atiyah-Singer formula}




\subsection{An $L^2$-index theorem  for coverings of manifolds of bounded geometry}
\index{index theorem!$L^2$-index theorem}
\index{covering space}

In this section we show how a version of the relative $L^2$-index (see  
\cite{Vaillant} for another version)
which generalises that in \cite{A}, can be obtained from our residue formula.

As above, we fix $(\tM,\tg)$, a  Riemannian 
manifold of dimension $n$ and of bounded geometry. 
Let also $G$ be a countable discrete group
acting freely and properly on $\tM$ by (smooth) isometries. 
Note that {\it we do not assume } $\tM$ to be $G$-compact and
we let $M:=G\setminus \tM$ be the possibly noncompact manifold 
(by properness) of right cosets. It is then natural to think of $\tM$ as the total 
space of a principal $G$-bundle with noncompact base $M$. 
We denote by $q:\tM\to M$ the projection map.
Note that the metric $\tg$ on $\tM$ then naturally yields a metric $g$ on 
$M$ given by $g_{x}(v_1,v_2)=\tg_{\tx}(\tilde v_1,\tilde v_2)$, if 
$x=q(\tx) \in M$ and $v_i=q(\tilde v_i) \in T_x M$  where we have 
identified $T_x M\simeq G.(T_{\tx}\tM)$, since 
the action of $G$ naturally extends to $T\tM$. In particular, $(M,g)$ 
also has bounded geometry.

An important class of examples is given by universal coverings. 
In this case, $G$ is the fundamental group of a manifold of bounded 
geometry $M$ and $\tM$ is its  universal cover. Also, in this case 
 $q:\tM\to M$ is the covering map and $\tg$ is
the lifted metric   on $\tM$   by $\tg_{\tx}=g_{q(\tx)}$. 

Let now $\D_S$ be a Dirac type operator acting on the sections of a 
Dirac bundle $S$ of bounded geometry on $M$.
To simplify the notations, we denote by 
$(\A,\H,\D_S):=\big(W^{\infty,1}(M) ,L^2(M,S),\D_S\big)$ the smoothly summable spectral triple 
constructed  in Section \ref{preliminaries}.
If the triple is either even or odd,
then we have various formulae for 
$$
\mbox{Index}(\hat eF_{\mu,+}\hat e)\quad \mbox{even case},\quad 
\mbox{Index}(P_\mu \hat u P_\mu)\quad \mbox{odd case},
$$
where $F_\mu$ is the phase of $\D_{S,\mu}$, is the double of $\D_S$ (see Definition \ref{def:double}), 
and $P_\mu=(F_\mu+1)/2$. 

We
lift the bundle $S$ to a bundle $\tS$ on $\tM$ (pullback by $q$) and we  also lift the operator
$\D_S$ to an equivariant operator $\tD_S$ on sections of $\tS$. This requires that the 
action of $G$ on $\tM$ lifts to an action on $\tS$, and we assume that this is the case.
We also denote by $\tilde c$ the Clifford action of ${\rm Cliff}(\tilde M)$ on $\tS$.
We let $\tH=L^2(\tM,\tS)$, and
observe that $\A$ acts on $\tH$ by 
$(\tilde c(f)\xi)(\tx)=c(f(x))\xi(\tx)$, for $f\in\A$, $\xi\in\tH$, and $\tx\in\tM$ with $x=q(\tx)$.

We now briefly review the setting for $L^2$-index theory referring for 
example to the review \cite{Ros} for some details and
references to the original literature.
Since the action of $G$ on $\tM$ is free and proper, 
we have an isometric identification $L^2(\tM,\tS)\cong L^2(M,S)\otimes \ell^2(G)$.
This allows us to define the
von Neumann algebra $\cn_G=G'\cong \B(\H)\otimes R(G)''$, 
where $R(G)$ is the group algebra 
consisting of the span of the unitaries giving the 
right action of $G$ on $\ell^2(G)$.  There is a canonical semifinite faithful normal  trace $\tau_G$    
defined on elementary tensors  
$T\otimes U\in\B(\H)\otimes R(G)''$ by
$$
\tau_G(T\otimes U)={\rm Tr}_\H(T)\,\tau_e(U),
$$
where ${\rm Tr}_\H$ is the operator trace on $\H$ and $\tau_e$ is the usual 
finite faithful normal  trace on $R(G)''$ given by evaluation at the neutral element.
Let now $\tilde{T}$ be a  pseudo-differential operator on $\tH$ with smooth 
kernel $[\tilde T]\in \Gamma^\infty(\tS\boxtimes\tS)$.  
Then, $\tilde T$ is $G$-equivariant if and only if 
$$
[\tilde T](h\cdot\tilde x,h\cdot\tilde y)=e_{ \tilde x}(h)\,[T](\tilde x,\tilde y)\,
e_{\tilde y}(h)^{-1},\quad \mbox{for all }\, (h,\tilde x,\tilde y)\in G\times\tM^2,
$$ 
 where $ e_{\tilde x}:G\to {\rm Aut} (\tS_{\tilde x})$ 
is the fibre-wise lift of the action of $G$ to $\tilde S$. 
For such $G$-equivariant pseudo-differential operators on $\tH$
which belongs to $\cl^1(\cn_G,\tau_G)$, we have
\begin{equation}
\label{traceG}
\tau_G(\tilde{T})=\int_{F}\tr\big([T](\tx,\tx)\big)\,d\mu_{\tg}(\tx),
\end{equation}
where $F$ is a fundamental domain in $\tM$ and $\tr$ is the fibre-wise trace on ${\rm End}(
\tS_{\tilde x})$. 
This latter formulation is the natural one, and was  initially defined by Atiyah \cite{A}.
It is clear from its definition that $\tau_G$ is faithful so that the algebra $\cn_G$
 is semi-finite. It need not be a factor because (as is well known) 
 the algebra $R(G)''$ has a non-trivial centre precisely when the group 
 $G$ has finite conjugacy classes
\cite{Ros}.

We note that when $T$ is a pseudo-differential 
operator  of trace class on $L^2(M,S)$ with Schwartz kernel 
$[T]$ (and thus order less than $-n$ and with $L^1$-coefficients), and $U\in R(G)''$, we have,
using the identification above,
$$
\tau_G(T\otimes U):=\int_{M} \tr\big([T](x,x)\big)\,\mu_g(x)\,\times\,\tau_e(U).
$$

When the original triple $(\A,\H,\D_S)$ on $M$ is even with grading $\gamma$, 
we denote by $\tilde \gamma:=\gamma\otimes{\rm Id}_{\ell^2(G)}$ the grading lifted to $\tH$.

{\bf Remark.} The ideal of $\tau_G$-compact operators $\K_{\cn_G}=\K(\cn_G,\tau_G)$ is 
given by the norm closure of the 
$G$-equivariant $\Psi DO$'s of strictly negative order and with integral kernel 
vanishing at infinity inside a fundamental domain.

\begin{lemma}
Let $(\tM,\tg)$ be a  Riemannian manifold of bounded 
geometry endowed with a free and proper
action of a countable group $G$. Let also $P$ be a differential 
operator of order $\alpha\in\N_0$ and of uniform 
$C^\infty$-bounded coefficients, acting on the sections of 
$S$ and let $\tilde P$ be its lift as a $G$-equivariant operator 
on $\tS$ (which has also uniform $C^\infty$-bounded coefficients). Assume further that
$$
\kappa:=\inf\big\{d_{\tg}(\tx,h\cdot\tx)\,:\,\tx\in\tM,\,\,h\in G\setminus\{e\}\big\}>0.
$$ 
Then there exist two constants $C>$ and $c>0$, such that for any
 $(\tx,x)\in \tM\times M$, with $x=q(\tx)$ we have
 $$
 \big|[\tilde Pe^{-t\tD_S^2}](\tx,\tx)-[P e^{-t\D_S^2}](x,x)\big|_\infty\leq C\,t^{ -(n+\alpha)/2}e^{-c/t},
 $$
 where $|\cdot|_\infty$ is the operator norm on ${\rm End}(\tS_x)$.
 \label{ccc}
\end{lemma}

\begin{proof}
Note first that for any  $(\tx,x),(\ty,y)\in \tM\times M$, with $x=q(\tx)$, $y=q(\ty)$, we have
$$
[Pe^{-t\D_S^2}](x,y)=\sum_{h\in G}\, [\tilde Pe^{-t\tD_S^2}](\tx,h\cdot \ty),
$$ 
which is proven using the uniqueness  of solutions of the heat equation on $\tM$ and on $M$.
Thus
$$
[P e^{-t\D_S^2}](x,x)-[\tilde Pe^{-t\tD_S^2}](\tx,\tx)
=\sum_{h\in G,\,h\ne e}\, [\tilde Pe^{-t\tD_S^2}](\tx,h\cdot\tx).
$$
From Proposition \ref{rrr}, we immediately deduce
$$
 \big|[\tilde Pe^{-t\tD_S^2}](\tx,\tx)-[Pe^{-t\D_S^2}](x,x)\big|_\infty
 \leq C\,t^{-(n+\alpha)/2}\, \sum_{h\in G,\,h\ne e}\,e^{-{d_g^2(\tx,h\cdot\tx)}/{4(1+c)t}}.
$$
Since $(\tM,\tg)$ has bounded geometry, the sectional curvature is bounded below,  by say 
$-K^2$ with $K>0$.
From \cite{M}, we have for any $\rho>0$ the existence of a uniform (over $\tM$)  constant $C'>0$ such that
$$
N_{\tx}(\rho):={\rm Card}\big\{h\in G\,:\, d_{\tg}(\tx,h\cdot\tx)\leq\rho\big\}\leq C' e^{(n-1)K\rho}.
$$
Then the assumption that 
$\kappa:=\inf\big\{d_{\tg}(\tx,h\cdot\tx)\,:\,\tx\in\tM,h\in G\setminus\{e\}\big\}>0$,
yields the inequality
$$
 \big|[\tilde Pe^{-t\tD_S^2}](\tx,\tx)-[Pe^{-t\D_S^2}](x,x)\big|_\infty\leq C''\,t^{-(n+\alpha)/2}\,\int_\kappa^\infty\,e^{-{\rho^2}/{4(1+c)t}}dN_{\tx}(\rho),
$$
which after an integration by parts, gives the proof.
\end{proof}

\begin{lemma} 
\label{lem:holo-laplace}
Under the hypotheses of Lemma \ref{ccc} and
for $f\in \A$ and $P$ a differential operator on $S$ with 
uniform $C^\infty$-bounded coefficients 
(and $\tilde P$ its lift on $\tilde S$ as a $G$-equivariant operator), the functions
$$
\C\ni z\mapsto \tau_G\Big(\tilde c(f) \tilde P\int_1^\infty t^{z}e^{-t(1+\tD_S^2)}
 dt\Big),\quad
\C\ni z\mapsto {\rm Tr}\Big(c(f) P\int_1^\infty t^{z}e^{-t(1+\D_S^2)} dt\Big),
$$
are entire. 
\end{lemma}

\begin{proof} From Proposition \ref{rrr} and Equation \eqref{traceG}, 
we see that the integral is absolutely convergent.
We thus may differentiate under the integral sign with respect to $z$ 
and since the resulting integral is again 
absolutely convergent,  we are done. 
\end{proof}

\begin{prop}
Under the hypotheses of Lemma \ref{ccc},
for $f\in \A$, $P$ a differential operator of uniform $C^\infty$-bounded
coefficients and $\Re(z)>n$, there is an equality
$$
\tau_G\big(\tilde\gamma\tilde c(f)\tilde P(1+\tD_S^2)^{-z/2}\big)
=
{\rm Tr}\big(\gamma c(f) P(1+\D_S^2)^{-z/2}\big),
$$
modulo an entire function of $z$.
\label{error}
\end{prop}

\begin{proof}
This is a combinations of Lemmas \ref{ccc}  and \ref{lem:holo-laplace} together
with the usual   Laplace transform representation for the operators concerned. 
\end{proof}

The following result, whose proof follows from the previous 
discussion and the same arguments as in Section \ref{preliminaries},  is key.

\begin{corollary}
The triple $(\A,\tH,\tD)$ is a smoothly summable   semifinite
spectral triple with respect to $(\cn_G,\tau_G)$, of isolated spectral dimension
$n$.
\end{corollary}

\begin{proof}
This follows from Proposition \ref{error} combined  with Proposition \ref{main-stuff} together with similar arguments
as those of Proposition \ref{main-stuff} to prove that the operators $\delta^k(\tilde c(f))$ and $\delta^k(\tilde c(df))$, $k\in\N_0$, all belong to $\B_1(\D_S,n)$ for $f\in\A$.
\end{proof}

We arrive at the main result of this section.

\begin{theorem} The numerical pairing of $(\A,\HH,\D)$ with $K_*(\A)$ coincides with the 
numerical pairing of $(\A,\tH,\tD)$ with $K_*(\A)$ (which is thus  integer-valued).
\end{theorem}

\begin{proof}
Since both spectral triples  $(\A,\HH,\D)$ and $(\A,\tH,\tD)$  have isolated spectral dimension,
one can use the last version of Theorem \ref{localindex} to compute the index pairing, i.e.
we can use the residue cocycle. Then the result follows from Proposition \ref{error}.
\end{proof}
\index{index theorem!$L^2$-index theorem}




\section{Noncommutative examples}
In this section, we apply our results to purely noncommutative examples. 
The first source of examples
comes from torus actions on $C^*$-algebras and the construction follows \cite{PRe} and \cite{PReS}
where
explicit special cases for graph and $k$-graph algebras were studied. The second describes the
Moyal plane and uses the results of \cite{GGISV}. 

\subsection{Torus actions on $C^*$-algebras}
\newcommand{\ZZZ}{\mathbb{Z}}
\index{torus actions}

We are interested here in spectral triples arising from an action 
of a compact abelian Lie group 
$\T^p=(\R/2\pi\R)^p$ on 
a separable 
$C^*$-algebra $A$, which we denote by $\sigma_\cdot:\T^p\to {\rm Aut}(A)$. 
We suppose that $A$ possesses a $\T^p$-invariant norm lower-semicontinuous
faithful semifinite trace, $\tau$. Recall that $\tau$ is norm lower-semicontinuous if whenever
we have a norm convergent sequence of positive elements, $A\ni a_j\to a\in A$, then
$\tau(a)\leq \lim\inf\tau(a_j)$, and the tracial property says that $\tau(a^*a)=\tau(aa^*)$ for all $a\in A$.

We  show  that with this data
we obtain a smoothly summable  spectral triple,
even if we dispense with the assumption that the algebra has local units employed in
\cite{PRe,PReS,Wa}. 

We begin by setting $\HH_1=L^2(A,\tau)$, the GNS space for $A$ constructed using $\tau$. 
The action of  $\T^p$ on   our algebra $A$ gives  a $\ZZZ^p$-grading on $A$ by the spectral
subspaces 
$$
A=\overline{\bigoplus_{m\in\ZZZ^p}A_m},\quad A_m=\{a\in A:\,\sigma_z(a)=z^ma=z_1^{m_1}\cdots z_p^{m_p}a\}.
$$
So for all $a\in A$ we can write $a$ as a sum of elements $a_m$ homogenous for 
the action of $\T^p$
$$
a=\sum_{m\in\ZZZ^p}a_m, \qquad t\cdot a_m=e^{i\langle m,t\rangle}a_m,\quad t=(t_1,\dots,t_p)\in\T^p.
$$

The invariance of the trace $\tau$ implies that 
the $\T^p$ action extends to a unitary action $U$ on  $\HH_1$ which implements the action on $A$.
As a consequence 
there exist pairwise orthogonal projections $\Phi_m\in \B(\HH_1)$, $m\in\ZZZ^p$, 
such that $\sum_{m\in\ZZZ^p}\Phi_m={\rm Id}_{\HH_1}$ (strongly) 
and $a_m\Phi_k=\Phi_{m+k}a_m$ for a 
homogenous algebra 
element $a_m\in A_m$. 
Moreover, we say that $A$ has full spectral subspaces if for all $m\in \mathbb{Z}^p$ 
we have $\overline{A_mA_m^*}=A_0$.  Observe that $A_0$ coincides with  $A^{\T^p}$, 
 the fixed point algebra of $A$ for the action of $\T^p$.

Let
$\HH:=\HH_1\otimes_\C\H_{f}$, where $\H_{f}:=\C^{2^{\lfloor p/2\rfloor}}$.
We define our operator $\D$ as the operator affiliated to $\B(\H)$, given by the `push-forward' of the flat 
Dirac operator on $\T^p$ to 
the Hilbert space $\HH$. More precisely we first define the domain ${\rm dom}(\D)$ by
$$
{\rm dom}(\D):=\H_1^\infty\otimes\H_f,\quad \H_1^\infty
:=\big\{\psi\in \H_1\,:\,[t\mapsto t\cdot\psi]\in C^\infty(\T^p,\H_1)\big\}.
$$
Then we define $\D$ on ${\rm dom}(\D)$ by
$$
\D=\sum_{n\in\ZZZ^p}\Phi_n\otimes \gamma(in),
$$
where
$\gamma(in)=i\sum_{j=1}^p\gamma_jn_j$, $n=(n_1,\dots,n_p)$, and the $\gamma_j$ are Clifford matrices 
acting on $\H_f$ with
$$
\gamma_j\gamma_l+\gamma_l\gamma_j=-2\,\delta_{jl}\,{\rm Id}_{\H_f}.
$$ 

In future we will abuse notation by
letting $\Phi_n$ denote the projections acting on $\H_1$, on $A$, 
and also the projections $\Phi_n\otimes {\rm Id}_{\H_{f}}$ acting on $\H$. Similarly we will 
speak of $A$ and $A_0$ acting on $\H$, by tensoring the GNS representation on $\H_1$
by ${\rm Id}_{\H_f}$. To simplify the notations, we just identify $A$ with its 
image in the GNS representation.

We let $\cn\subset \B(\HH)$ be the commutant of the {\em right} 
multiplication action of the fixed point algebra $A_0$ on
$\HH$. Then it can be checked that the left multiplication representation of $A$ is in $\cn$ and
$\D$ is affiliated to $\cn$.   

To obtain a faithful normal semifinite trace, which we call ${\rm Tr}_\tau$, 
on $\cn$, we have two possible routes, which both lead
to the same trace, and which yield different and complementary information about the trace.

The first approach is to 
let ${\rm Tr}_\tau$ be the dual trace on $\cn=(A_0)'$.
The dual trace is defined using spatial derivatives, and is a faithful normal semifinite
trace on $\cn$. A detailed discussion of this construction, and its equivalence with our next
construction, is to be found in \cite[pp 471-478]{LN}. The discussion referred to in \cite{LN} is
in the context of KMS weights, but by specialising to the case of invariant traces, the particular case
of $\beta$-KMS weights with $\beta=0$, 
we obtain the description we want. (Alternatively, the reader may examine 
\cite[Theorem 1.1]{LN} for a trace specific description of our next construction).

In fact, the article \cite{LN} is, in part, concerned with inducing traces from the coefficient
algebra of a $C^*$-module to traces on the algebra of compact endomorphisms on that module.
To make contact with \cite{LN}, 
we make $A\otimes\H_f$ a right inner product module over $A_0$
via the inner product 
$$
(a\otimes\xi|b\otimes\eta):=\Phi_0(a^*b)\langle\xi,\eta\rangle_{\H_f},\quad a,\,b\in A,\ \ \xi,\,\eta\in \H_f.
$$
Calling the 
completed right $A_0$-$C^*$-module $X$, it can be shown, see \cite{LN}, that ${\rm End}_{A_0}(X)$ acts
on $\H$ and that $\cn={\rm End}_{A_0}(X)''$. We introduce this additional structure because
we can compute ${\rm Tr}_\tau$ on all rank one endomorphisms on $X$. Given $x,\,y,\,z\in X$, 
the rank one endomorphism $\Theta_{x,y}$ acts on 
$z$ by $\Theta_{x,y}z:=x(y|z)$. 

Then by \cite[Lemma 3.1 \& Theorem 3.2]{LN} specialised to invariant traces, 
see also \cite[Theorem 1.1]{LN}, we have
\begin{equation}
\label{eq:tr-tau}
{\rm Tr}_\tau(\Theta_{x,y})=\tau((y|x)):=\sum_{i=1}^{2^{\lfloor p/2\rfloor }}\tau((y_i|x_i)),
\end{equation}
where $x=\sum_i x_i\otimes e_i$, the $e_i$ are the standard basis vectors of $\H_f$,
and similarly $y=\sum_i y_i\otimes e_i$. 
Moreover, ${\rm Tr}_\tau$ 
restricted to the compact endomorphisms
of $X$ is an ${\rm Ad}\,U(\T^P)$-invariant norm lower-semicontinuous trace, \cite[Theorem 3.2]{LN},
where $U$ is the action of $\T^p$ on $\H$.

\begin{lemma}
\label{mon-beau}
Let $0\leq a\in {\rm dom}\,\tau\subset A\subset\cn$. 
Then for $m\in\ZZZ^p$ we have 
\begin{equation}
0\leq {\rm Tr}_\tau\big(a\,\Phi_m\big)\leq 2^{\lfloor p/2\rfloor}\,\tau(a).
\label{eq:spec-subs}
\end{equation}
Moreover, we have equality in Equation \eqref{eq:spec-subs} if $A$ has full spectral subspaces
and
$$
{\rm Tr}_\tau\big(a\,\Phi_0\big)= 2^{\lfloor p/2\rfloor}\,\tau(a),
$$
in all cases.
\end{lemma}

\begin{proof}
We prove the statement for $a\in A_0$, and then proceed to general elements of $A$.

We begin with the case of full spectral subspaces.
Consider first $a=bb^*$ for $b\in A_k\cap {\rm dom}^{1/2}\,\tau$ homogenous of degree $k$,
so that $a\in A_0\cap {\rm dom}\,\tau$ (since $\tau$ is a trace). 
Then a short calculation shows that
$\Phi_k a\Phi_k=a\Phi_k=\sum_{i=1}^{2^{\lfloor p/2\rfloor}}\Theta_{b\otimes e_i,b\otimes e_i}$ 
where the $e_i$ are the standard basis vectors in $\H_{f}$.
Hence
$$
{\rm Tr}_\tau(a\Phi_k)=\sum_{i=1}^{2^{\lfloor p/2\rfloor}}\tau(b^*b)
=\sum_{i=1}^{2^{\lfloor p/2\rfloor}}\tau(bb^*)=2^{\lfloor p/2\rfloor}\tau(a).
$$
Therefore ${\rm Tr}_\tau(a\Phi_k)=2^{\lfloor p/2\rfloor}\tau(a)$ if $a$ is a finite sum
of elements of the form $bb^*$, $b\in A_k$. 
Thus if
$\overline{A_kA_k^*}=A_0$ for all $k\in\ZZZ^p$ we get equality for all 
${\rm dom}\,\tau\cap  A_0^+\ni a$ and $k\in \ZZZ^p$.
In particular, we always have ${\rm Tr}_\tau(a\Phi_0)=2^{\lfloor p/2\rfloor}\tau(a)$.

In the more general situation consider the closed ideal
$\overline{A_kA_k^*}$ in $A_0$, which is $\sigma$-unital
by  the separability of $A$, and of $\overline{A_kA_k^*}$
. Choose a  positive approximate unit
$\{\psi_n\}_{n\geq 1}\subset A_kA_k^*$ for $\overline{A_kA_k^*}$. 
Since $A_kA_k^*A_k$ is dense in
$A_k$, we have $\psi_n x\to x$ for any $x\in X_k=A_k\otimes\H_f$. Hence
$\psi_n a\psi_n\in A_kA_k^*$ converges strongly to the action of~$a$ on
$X_k$ for any $a\in A_0$. Since ${\rm Tr}_\tau$ is strictly lower
semicontinuous, \cite[Theorem 3.2]{LN}, for $A_0\cap {\rm dom}\,\tau \ni a\ge0$ we therefore get
\begin{align*}
{\rm Tr}_\tau(a\Phi_k)\le\liminf_n {\rm Tr}_\phi(\psi_n a\psi_n\Phi_k)
&=\liminf_n 2^{\lfloor p/2\rfloor}\tau(\psi_n a\psi_n)\\&=
\liminf_n 2^{\lfloor p/2\rfloor}\tau(a^{1/2}\psi_n^2 a^{1/2})
\le 2^{\lfloor p/2\rfloor}\tau(a).
\end{align*}
This proves the Lemma for $a\in A_0\cap {\rm dom}\,\tau$.

Now for general $0\leq a\in {\rm dom}\,\tau$, we may use the ${\rm Ad}U$-invariance of ${\rm Tr}_\tau$
to see that
$$
{\rm Tr}_\tau(a\Phi_k)={\rm Tr}_\tau(\Phi_0(a)\Phi_k)\leq 2^{\lfloor p/2\rfloor}\tau(\Phi_0(a)),
$$
with equality for $k=0$ or for all $k\in\ZZZ^p$ 
if $A$ has full spectral subspaces. Thus if we write $a=\sum_{m\in\ZZZ^p}a_m$
as a sum of homogenous components, 
$$
{\rm Tr}_\tau(a\Phi_k)={\rm Tr}_\tau(a_0\Phi_k)\leq 2^{\lfloor p/2\rfloor}\tau(a_0)=2^{\lfloor p/2\rfloor}\tau(a),
$$
with equality if $k=0$ or for all $k\in\ZZZ^p$  if $A$ has full spectral subspaces.
\end{proof}

\begin{corollary}
\label{la-belle}
Let $A,\H,\D,\cn,{\rm Tr}_\tau$ be as above. 
Use $\D$ and ${\rm Tr}_\tau$ to construct the weights $\vf_s$, $s>p$, on $\cn$ via  Definition
\ref{def:d-does-int}.
Consider the restrictions $\psi_s$ of the weights $\vf_s$
 to the domain of $\tau$ in $A$. 
Then
$$
\psi_s(a)\leq 2^{\lfloor p/2\rfloor}\Big(\sum_{m\in\ZZZ^p}(1+|m|^2)^{-s/2}\Big)\,\tau(a)\,,
\quad a\in A_+\cap {\rm dom}\,\tau\,,\;s>p,
$$
with equality if $A$ has full spectral subspaces.
\end{corollary}

\begin{proof}
Note first that
$$
(1+\D^2)^{-s/2}=\sum_{m\in\ZZZ^p}(1+|m|^2)^{-s/2}\Phi_m,
$$
so that for $a\in A_+$ and $s>p$, we have by definition of the weights $\vf_s$ that
$$
\vf_s(a)={\rm Tr}_\tau\big((1+\D^2)^{-s/4}a(1+\D^2)^{-s/4}\big),
$$
which by traciality of ${\rm Tr}_\tau$ implies
\begin{align*}
\vf_s(a)&={\rm Tr}_\tau\big(\sqrt{a}(1+\D^2)^{-s/2}\sqrt{a}\big)
={\rm Tr}_\tau\Big(\sum_{m\in\ZZZ^p}(1+|m|^2)^{-s/2}\sqrt{a}\Phi_m\sqrt{a}\Big).
\end{align*}
The normality of ${\rm Tr}_\tau$ allows us to permute the sum and the trace
\begin{align}
\vf_s(a)&=\sum_{m\in\ZZZ^p}(1+|m|^2)^{-s/2}{\rm Tr}_\tau\big(\sqrt{a}\,\Phi_m\,\sqrt{a}\big)
=\sum_{m\in\ZZZ^p}(1+|m|^2)^{-s/2}{\rm Tr}_\tau\big(\Phi_m\,a\,\Phi_m\big)\nonumber\\
&=\sum_{m\in\ZZZ^p}(1+|m|^2)^{-s/2}{\rm Tr}_\tau\big(\Phi_0(a)\,\Phi_m\big)
\leq 2^{\lfloor p/2\rfloor}\Big(\sum_{m\in\ZZZ^p}(1+|m|^2)^{-s/2}\Big)\,\tau(a),
\label{eq:makes-easy}
\end{align}
the last inequality following from Lemma \ref{mon-beau}, and it is an equality if $A$
has full spectral subspaces.
\end{proof}

Let $\A\subset A$ be the algebra of smooth vectors for the action of $\T^p$
\begin{align*}
\A&:=\big\{a\in A\,:\,[t\mapsto t\cdot a]\in C^\infty(\T^p,A)\big\}\\
&=\Big\{a=\sum_{m\in\ZZZ^p}a_m\in\bigoplus_{m\in\ZZZ^p}A_m\,:
\,\sum_{m\in\ZZZ^p}|m|^k\|a_m\|<\infty\ {\rm for\ all\ }k\in\N_0\Big\}.
\end{align*}
Then, as expected, $\A$ is contained in ${\rm OP}^0$. We let $\delta(T)=[|\D|,T]$ for $T\in\cn$
preserving $\H_\infty$.

\begin{lemma}
\label{BN}
The subalgebra $\A$ of smooth vectors in $A$ for the action of $\T^p$ is contained 
in $\bigcap_k {\rm dom}(\delta^k)$. 
More explicitly,  for $a=\sum_{m\in\ZZZ^p}a_m\in\bigoplus _{m\in\ZZZ^p}A_m$ we have the bound
$$
\|\delta^k(a)\|\leq C_k\sum _{m\in\ZZZ^p} |m|^{2k}\,\|a_m\|.
$$
\end{lemma}

\begin{proof}
By the discussion following Definition \ref{parup}, the claim is equivalent to $\A\subset \cap_k {\rm dom}(R^k)$,
where $R(T)=[\D^2,T](1+\D^2)^{-1/2}$. Recall that for $a\in \A$ and $k=\in\N$, we have
$$
R^k(a)=[\D^2,\dots[\D^2,a]\dots](1+\D^2)^{-k/2}.
$$
For $j=1,\dots,p$, denote by $\partial_j$  the generators of the $\T^p$-action on 
both $\A$ and $\H_1$. For
$\alpha\in\mathbb N ^p$, let $\partial^\alpha:=\partial_1^{\alpha_1}\dots\partial_p^{\alpha_p}$.
Since $\D^2=-(\sum_{j=1}^p\partial_j^2)\otimes{\rm Id}_{\H_{\bf f}}$, 
an elementary computation shows that
$$
R^k(a)=\sum_{|\alpha|\leq 2k,|\beta|\leq k} C_{\alpha,\beta}\,\partial^\alpha(a)\,\partial^\beta\otimes{\rm Id}_{\H_{\bf f}}\, (1+\D^2)^{-k/2}.
$$
This is enough to conclude since $a\in\A$ implies 
that $\|\partial^\alpha(a)\|<\infty$, and elementary spectral theory
of $p$ pairwise commuting operators shows that for $|\beta|\leq k$,
$\partial^\beta\otimes{\rm Id}_{\H_{\bf f}}\, (1+\D^2)^{-k/2}$ is bounded too. 
The bound then follows  from
$$
\partial^\alpha(a_m)=i^{|\alpha|}m^\alpha\, a_m\,,\quad a_m\in A_m,
$$
which delivers the proof.
\end{proof}

Define the algebras $\B,{\mathcal C}\subset\A\subset A$ by

\begin{align*}
\B&=
\Big\{a=\sum_{m\in\ZZZ^p}a_m\in\A\,:\, \sum_{m\in\ZZZ^p}|m|^k\,\tau(a_m^*a_m)
<\infty\ {\rm for\ all\ }k\in\N_0\Big\},
\\
{\mathcal C}&=
\Big\{a=\sum_{m\in\ZZZ^p}a_m\in\A\,:\,\sum_{m\in\ZZZ^p}|m|^k\tau(|a_m|)<\infty\ {\rm for\ all\ }k\in\N_0\Big\}.
\end{align*}

The following is the main result of this subsection.
\begin{prop} Let $\T^p$  be a torus acting on
a $C^*$-algebra $A$
with a norm lower-semicontinuous
faithful $\T^p$-invariant trace $\tau$.
Then $({\mathcal C},\HH,\D)$ defined as above is a semifinite  
spectral triple relative to $(\cn,{\rm Tr}_\tau)$.
Moreover
$({\mathcal C},\HH,\D)$ is smoothly summable with spectral dimension $p$. The square integrable
and integrable elements of $A$ satisfy
$$
\B_2(\D,p)\bigcap A= ({\rm dom}(\tau))^{1/2},\quad \B_1(\D,p)\bigcap A= {\rm dom}(\tau),
$$
The  space of smooth square integrable and the space  of smooth integrable elements
of $A$  contain $\B$ and $\mathcal C$ respectively. More precisely,
$$
\B^\infty_2(\D,p)\supset\B\cup[\D,\B],
\quad \B_1^\infty(\D,p)\supset  {\mathcal C} \cup[\D, {\mathcal C} ].
$$
Furthermore, if  $0\leq a\in{\rm dom}(\tau)$ and $A$ has full spectral subspaces then
$$
{\rm res}_{z=0}{\rm Tr}_\tau(a(1+\D^2)^{-p/2-z})= 2^{\lfloor p/2\rfloor-1}\,{\rm Vol}(S^{p-1})\,\tau(a).
$$
\end{prop}

\begin{proof}
We begin by proving that $\B_2(\D,p)\bigcap A\supset ({\rm dom}(\tau))^{1/2}$. 
Lemma \ref{mon-beau}   shows 
that for all $a\in {\rm dom}(\tau)$ with $a\geq 0$ 
and all $m\in\ZZZ^p$ we have
\begin{equation}
{\rm Tr}_\tau(a\Phi_m)\leq 2^{\lfloor p/2\rfloor}\,\tau(a)\,,
\label{eq:spec-subs-again}
\end{equation}
and equality holds when we have full spectral subspaces or $m=0$. 

Thus for $a\in({\rm dom}(\tau))^{1/2}$ and  $\Re(s)>p$
we see that, using the normality of ${\rm Tr}_\tau$ and the same arguments as in 
Equation \eqref{eq:makes-easy}, 
\begin{align*}
{\rm Tr}_\tau((1+\D^2)^{-s/4}a^*a(1+\D^2)^{-s/4})
&=\sum_{n\in\ZZZ^p}(1+|n|^2)^{-s/2}\,{\rm Tr}_\tau(a^*a\Phi_n)\\
&\leq \tau(a^*a)\,2^{\lfloor p/2\rfloor}\,\sum_{n\in\ZZZ^p}(1+|n|^2)^{-s/2}<\infty.
\end{align*}
Hence
$({\rm dom}(\tau))^{1/2}\subset \B_2(\D,p)$. Conversely, if $a\in A$ lies in $\B_2(\D,p)$ we have 
$a(1+\D^2)^{-s/4}\in\L^2(\cn,{\rm Tr}_\tau)$ for all $s$ with $\Re(s)>p$. Then
$$
a\Phi_0a^*\leq a(1+\D^2)^{-s/2}a^*\in\L^1(\cn,{\rm Tr}_\tau),\quad \Re(s)>p,
$$
and so $a\Phi_0a^*\in \L^1(\cn,{\rm Tr}_\tau)$. Then 
$$
\infty>{\rm Tr}_\tau(a\Phi_0a^*)={\rm Tr}_\tau(\Phi_0a^*a\Phi_0)=\tau(a^*a).
$$
Thus $a^*a\in{\rm dom}(\tau)$, and so $a\in{\rm dom}(\tau)^{1/2}$. Since $\B_2(\D,p)$
is a $*$-algebra, $a^*(1+\D^2)^{-s/4}\in\L^2(\cn,{\rm Tr}_\tau)$ also, and so $a^*\in {\rm dom}(\tau)^{1/2}$
as expected.

Now for $0\leq a\in A$, Lemma \ref{lem:b1-pos} tells us that
$a\in\B_1(\D,p)$ if and only if $a^{1/2}\in \B_2(\D,p)$. So $a\in{\rm dom}(\tau)_+$ if and only if 
$a^{1/2}\in ({\rm dom}(\tau))^{1/2}_+= (\B_2(\D,p)\cap A)_+$, proving that 
${\rm dom}(\tau)_+= \B_1(\D,p)_+\bigcap A_+$. 

Since $\B_1(\D,p)$ 
is the span of its positive cone by Proposition \ref{cor:polar-decomp}, we have 
$$
\B_1(\D,p)\bigcap A={\rm span}(\B_1(\D,p)_+\bigcap A_+)={\rm span}({\rm dom}(\tau)_+)={\rm dom}(\tau).
$$

Now we turn to the smooth subalgebras. The definitions show that for $k\in\ZZZ^p$, 
and a  homogeneous element $a_m\in A_m$, we have
$$
\delta(a_m)\Phi_k=(|m+k|-|k|)a_m\Phi_k.
$$
Since $\delta(a_m)$ is also homogenous of degree $m$, which follows since $|\D|$ is invariant, 
we find that  for all $\alpha\in\N_0$ 
\begin{align*}
\delta^\alpha(a_m)\Phi_k
&=(|m+k|-|k|)^{\alpha}a_m\Phi_k.
\end{align*}
Hence  for $a=\sum_m a_m\in\B$ and $s>p$ we have
\begin{align}
\label{eq:tee-pee-invariant}
&{\rm Tr}_\tau\left((1+\D^2)^{-s/4}|\delta^\alpha(a)|^2(1+\D^2)^{-s/4}\right)
=
\sum_{m,n,k\in\ZZZ^p}(1+|k|^2)^{-s/2}
{\rm Tr}_\tau\left(\Phi_k\delta^\alpha(a_m)^*\delta^\alpha(a_n)\Phi_k\right)\\
&\hspace{2.9cm}=
\sum_{m,n,k\in\ZZZ^p}(|m+k|-|k|)^{\alpha}(|n+k|-|k|)^{\alpha}(1+|k|^2)^{-s/2}
{\rm Tr}_\tau\big(\Phi_ka_m^*a_n\Phi_k\big).\nonumber
\end{align}
Now, using $a_m\Phi_k=\Phi_{m+k}a_m$ for $a_m\in A_m$ we have
$$
\Phi_ka_m^*a_n\Phi_k=a_m^*a_n\Phi_{k-n+m}\Phi_k=\delta_{n,m}a_m^*a_n\Phi_k.
$$
Inserting this equality into the last line of Equation \eqref{eq:tee-pee-invariant} yields

\begin{align*}
&\sum_{m,k\in\ZZZ^p}\big||m+k|-|k|\big|^{2\alpha}(1+|k|^2)^{-s/2}
{\rm Tr}_\tau\left(a_m^*a_m\Phi_k\right)
\\
&\leq
\sum_{k\in\ZZZ^p}(1+|k|^2)^{-s/2}\sum_{m\in\ZZZ^p}|m|^{2\alpha}
{\rm Tr}_\tau\left(a_m^*a_m\Phi_k\right)
\leq 2^{\lfloor p/2\rfloor}
\sum_{k\in\ZZZ^p}(1+|k|^2)^{-s/2}\sum_{m\in\ZZZ^p}|m|^{2\alpha}\tau\left(a_m^*a_m\right),
\end{align*}
where we used Lemma \ref{mon-beau} in the last step and the latter is finite  by definition of $\B$. 
Since
\begin{align*} \Q_n(\delta^\alpha(a))^2=\Vert\delta^\alpha(a)\Vert^2 
&+{\rm Tr}_\tau\left((1+\D^2)^{-p/4-1/n}|\delta^\alpha(a)|^2(1+\D^2)^{-p/4-1/n}\right)\\
&+{\rm Tr}_\tau\left((1+\D^2)^{-p/4-1/n}|\delta^\alpha(a)^*|^2(1+\D^2)^{-p/4-1/n}\right),
\end{align*}
we deduce that $\B\subset \B_2^\infty(\D,p)$.
Finally, for $m\in\ZZZ^p$ and $a_m\in\B$ homogenous of degree $m$, we have
$$
[\D,a_m]=a_m\, {\rm Id}_{\H_1}\otimes \gamma(im).
$$
Then by the same arguments as above, we deduce that $[\D,a_m]\in \B_2(\D,p)$, 
and thus  $[\D,\B]\subset \B_2(\D,p)$.
By combining the estimates for $[\D,a]$ and $\delta^\alpha(a)$,
we see that $\B\cup[\D,\B]\subset \B_2^\infty(\D,p)$.

Now let $a=\sum_ma_m\in {\mathcal C}$, so that in particular $|a_m|,\,|a_m^*|\in {\rm dom}(\tau)$. Then 
$v_m|a_m|^{1/2},\,|a_m|^{1/2}\in ({\rm dom}(\tau))^{1/2}\subset \B_2(\D,p)$ where $a_m=v_m|a_m|$ is the polar decomposition  in $\cn$. 

To deal with smooth summability, we need 
another operator inequality. For $a_m\in A_m$, $k\in\ZZZ^p$ we have the simple computation
\begin{align*} 
\delta^\alpha(a_m)^*\delta^\alpha(a_m)\Phi_k&=(-1)^\alpha \delta^\alpha(a_m^*)\delta^\alpha(a_m)\Phi_k\\
&=(-1)^\alpha(|k|-|m+k|)^\alpha(|m+k|-|k|)^\alpha a_m^*a_m\Phi_k
=(|m+k|-|k|)^{2\alpha} a_m^*a_m\Phi_k.
\end{align*}
Since $0\leq(|m+k|-|k|)^{2\alpha}\leq |m|^{2\alpha}$ for all $k\in\ZZZ^p$, we deduce that
$$
0\leq\delta^\alpha(a_m)^*\delta^\alpha(a_m)\leq |m|^{2\alpha}a_m^*a_m.
$$

With this inequality in hand, and using $a\in {\mathcal C}$, we 
use the polar decomposition as above to see that for all $\alpha\in\N_0$, the decomposition
$$
\delta^\alpha(a)=\sum_m\delta^\alpha(a_m)
=\sum_mv_{\alpha,m}|\delta^\alpha(a_m)|^{1/2}\,|\delta^\alpha(a_m)|^{1/2}\in \B_1(\D,p),
$$
gives a representation of $\delta^{\alpha}(a_m)$ as an element of $\B_1(\D,p)$. To see this we first
check that $|\delta^\alpha(a_m)|^{1/2}\in \B_2(\D,p)$, which follows from

\begin{align}
&{\rm Tr}_\tau\left((1+\D^2)^{-p/4-1/n}|\delta^\alpha(a_m)|(1+\D^2)^{-p/4-1/n}\right)\nonumber\\
&\qquad=\sum_{k\in\ZZZ^p}(1+k^2)^{-p/2-1/2n}\,{\rm Tr}_\tau(\Phi_k\sqrt{\delta^\alpha(a_m)^*\delta^\alpha(a_m)}\Phi_k)\nonumber\\
&\qquad\leq \sum_{k\in\ZZZ^p}(1+k^2)^{-p/2-1/2n}|m|^\alpha\tau(\sqrt{a_m^*a_m})
=|m|^\alpha\,\tau(|a_m|)\,\sum_{k\in\ZZZ^p}(1+k^2)^{-p/2-1/2n}.
\label{eq:earlier}
\end{align}

Since
$$
\big(v_{\alpha,m}|\delta^\alpha(a_m)|^{1/2}\big)^*v_{\alpha,m}|\delta^\alpha(a_m)|^{1/2}=|\delta^\alpha(a_m)|,
$$
the corresponding term is handled in the same way. Finally we have
\begin{align}
&{\rm Tr}_\tau\left((1+\D^2)^{-p/4-1/n}v_{\alpha,m}|\delta^\alpha(a_m)|v_{\alpha,m}^*(1+\D^2)^{-p/4-1/n}\right)\nonumber\\
&\qquad=\sum_{k\in\ZZZ^p}(1+k^2)^{-p/2-1/2n}\,{\rm Tr}_\tau(\Phi_kv_{\alpha,m}|\delta^\alpha(a_m)|v_{\alpha,m}^*\Phi_k)\nonumber\\
&\qquad= \sum_{k\in\ZZZ^p}(1+k^2)^{-p/2-1/2n}
{\rm Tr}_\tau(|\delta^\alpha(a_m)|^{1/2}v_{\alpha,m}^*\Phi_kv_{\alpha,m}|\delta^\alpha(a_m)|^{1/2})\nonumber\\
&\qquad= \sum_{k\in\ZZZ^p}(1+k^2)^{-p/2-1/2n}
{\rm Tr}_\tau(|\delta^\alpha(a_m)|^{1/2}\Phi_{k-m}v_{\alpha,m}^*v_{\alpha,m}|\delta^\alpha(a_m)|^{1/2})\label{star-one} \\
&\qquad= \sum_{k\in\ZZZ^p}(1+k^2)^{-p/2-1/2n}
{\rm Tr}_\tau(|\delta^\alpha(a_m)|^{1/2}\Phi_{k-m}v_{\alpha,m}^*v_{\alpha,m}\Phi_{k-m}|\delta^\alpha(a_m)|^{1/2})\label{star-two}\\
&\qquad\leq \sum_{k\in\ZZZ^p}(1+k^2)^{-p/2-1/2n}
{\rm Tr}_\tau(|\delta^\alpha(a_m)|^{1/2}\Phi_{k-m}|\delta^\alpha(a_m)|^{1/2})\nonumber\\
&\qquad=\sum_{k\in\ZZZ^p}(1+k^2)^{-p/2-1/2n}\,{\rm Tr}_\tau(\Phi_{k-m}|\delta^\alpha(a_m)|\Phi_{k-m})\nonumber\\
&\qquad\leq\sum_{k\in\ZZZ^p}(1+k^2)^{-p/2-1/2n}|m|^\alpha
\,{\rm Tr}_\tau(\Phi_{k-m}|a_m|\Phi_{k-m})\label{star-three}\\
&\qquad\leq |m|^\alpha\,\tau(|a_m|)\,\sum_{k\in\ZZZ^p}(1+k^2)^{-p/2-1/2n}\nonumber.
\end{align}
In line \eqref{star-one} we again used $v^*_{\alpha,m}\Phi_k=\Phi_{k-m}v_{\alpha,m}^*$, which is true since 
$\delta^\alpha(a_m)$ is homogenous of degree $m$ and $|\delta^\alpha(a_m)|$ is homogenous of
degree zero.
In line \eqref{star-two} we used this again for
both $v_{\alpha,m}$ and $v_{\alpha,m}^*$. In \eqref{star-three} we again used this trick, and the fact 
that $|\delta^\alpha(a_m)|$ is homogenous of degree zero.
The last two inequalities follow just as in Equation \eqref{eq:earlier}.
So 
\begin{align*}
\Q_n(|\delta^\alpha(a_m)|^{1/2})
&\leq |m|^{\alpha/2}(\Vert a_m\Vert +\tau(|a_m|)+\tau(|a_m^*|))^{1/2}\Big(\sum_{k\in\ZZZ^p}
(1+k^2)^{-p/2-1/2n}\Big)^{1/2}\\
&=|m|^{\alpha/2}(\Vert a_m\Vert +2\tau(|a_m|))^{1/2}\Big(\sum_{k\in\ZZZ^p}(1+k^2)^{-p/2-1/2n}
\Big)^{1/2},
\end{align*}
and similarly for $v_{\alpha,m}|\delta^\alpha(a_m)|^{1/2}$. Hence
\begin{align*}
\PP_{n,\beta}(a)
&\leq \sum_{\alpha=0}^\beta\sum_m\Q_n(v_{\alpha,m}|\delta^\alpha(a_m)|^{1/2})\,\Q_n(|\delta^\alpha(a_m)|^{1/2})\\
&\leq \sum_{k\in\ZZZ^p}(1+k^2)^{-p/2-1/2n}\sum_{\alpha=0}^\beta\sum_m |m|^\alpha (\Vert a_m\Vert +2\tau(|a_m|)),
\end{align*}
which is enough to show that $\delta^\alpha(a)\in \B_1(\D,p)$. 
Since similar arguments show that $\delta^\alpha([\D,a])\in \B_1(\D,p)$,
we see that ${\mathcal C}\cup [\D,{\mathcal C}]\subset \B_1^\infty(\D,p)$.

The computation of the zeta function is straightforward, using Lemma \ref{mon-beau}, once one realises
that $\sum_{k\in\ZZZ^p}(1+k^2)^{-p/2-z}$ is just $(2\pi)^p$ times the trace of 
the Laplacian on a flat torus.  This precise value of the residue can be deduced from the Dixmier trace 
calculation for the torus in \cite[Example 7.1, p291]{GVF}, and the 
relationship between residues of zeta functions
and Dixmier traces in \cite[Lemma 5.1]{CPS2}. 
This also 
proves that the spectral dimension is $p$. 
\end{proof}

Semifinite spectral triples for more general compact group actions on $C^*$-algebras have been constructed in
\cite{Wa}. These spectral triples are shown to satisfy some summability conditions, but it
is not immediately clear that they satisfy our definition of smooth summability. We leave this 
investigation to another place.

For torus actions we can give a simple description of the index formula. 
First we observe that elementary Clifford algebra considerations, \cite[Appendix]{BCPRSW} and
\cite{PRe,PReS}, reduce the 
resolvent cocycle to a single term in degree $p$. This means that we automatically obtain the 
analytic continuation of the single zeta function which arises, and so the spectral dimension is
isolated, and there is at worst a simple pole at $r=(1-p)/2$. Hence the residue cocycle is given by
the single functional, defined on $a_0,\dots,a_p\in{\mathcal C}$ by

$$
\phi_p(a_0,\dots,a_p)=\left\{\begin{array}{ll}\sqrt{2i\pi}
\frac{1}{p!}\,{\rm res}_{s=0}{\rm Tr}_\tau\Big(
a_0\,[\D,a_1]\cdots [\D,a_p](1+\D^2)^{-p/2-s}\Big) & p\ \mbox{odd},\\
 & \\
\frac{1}{p!}\,{\rm res}_{s=0}{\rm Tr}_\tau\Big(\gamma
a_0\,[\D,a_1]\cdots [\D,a_p](1+\D^2)^{-p/2-s}\Big) & p\ \mbox{even}.
\end{array}\right.
$$

Applications of this formula to graph and $k$-graph algebras appear in \cite{PRe,PReS}. 
Both these papers show that the index  is sensitive to the group action, by presenting an 
algebra with two different actions of the same group which yield different indices.

\subsection{Moyal plane}
\label{MP}
\index{Moyal plane}
\subsubsection{Definition of the Moyal product}
Recall that the Moyal product of a pair of functions (or distributions) $f,\,g$ on $\real^{2d}$, is given  by
\begin{equation}
f \star_\theta g(x) := (\pi\theta)^{-2d} \iint e^{\frac{2i}{\theta}\omega_0(x-y,x-z)}  f(y) g(z)\,
\,dy \,dz.
\label{mop}
\end{equation}
The parameter $\theta$ lies in $\real\setminus\{0\}$  and plays 
the role of the Planck constant.  The quadratic form $\omega_0$ is the 
canonical symplectic form of $\real^{2d}\simeq T^*\R^d$. 
With basic Fourier analysis one  shows that the Schwartz space, 
$\mathcal S(\R^{2d})$, endowed with this product is a 
(separable) Fr\'echet $*$-algebra  with jointly continuous
product (the involution being given by the complex conjugation). 
For instance, 
when $f,g\in\mathcal S(\R^{2d})$, we have the relations
\begin{equation}
\label{ptr}
\int f \star_\theta g(x)\,dx=\int f(x)\, g(x)\,dx,\quad 
\partial_j(f\star_\theta g)=\partial_j(f)\star_\theta g+f\star_\theta \partial_j(g),\quad 
\overline{f\star_\theta g}=\overline g\star_\theta \overline f.
\end{equation}

This noncommutative product is nothing but  the composition law of symbols, 
in the framework of  the Weyl pseudo-differential calculus on $\real^{d}$. Indeed, let\index{Weyl pseudo-differential calculus}
${\rm Op}_W$ be  the 
Weyl quantization map:
\begin{align*}
&{\rm Op}_W:T\in S'(\R^{2d})\mapsto\\
&\qquad\Big[\vf\in S(\R^{d})\mapsto\big[q_0\in\R^d\mapsto(2\pi)^{-d}\int_{\R^{2d}}T\big((q_0+q)/2,p\big)
\vf(q_0)e^{i(q_0-q)p} \,d^dq\,d^dp\big]\in S'(\R^{d})\Big].
\end{align*}
Again, Fourier analysis shows that ${\rm Op}_W$ restricts to
  a unitary operator from the Hilbert space $L^2(\R^{2d})$ 
(the $L^2$-symbols) to the Hilbert space of Hilbert-Schmidt operators acting on 
$L^2(\R^d)$, with
\begin{equation}
\label{rtv}
\| {\rm Op}_W(f)\|_2=(2\pi)^{-d/2}\|f\|_2\,,
\end{equation}
where the first $2$-norm is the Hilbert-Schmidt norm on $L^2(\R^d)$ while the second is the Lebesgue $2$-norm
on $L^2(\R^{2d})$.
Thus, the algebra $(L^2(\R^{2d}),\star_\theta)$ turns out to be a full Hilbert-algebra. It is then 
natural to use the GNS construction (associated with the operator trace on $L^2(\R^d)$ 
in the operator picture, or with the Lebesgue integral in the symbolic picture) to represent 
this algebra. To keep track of the dependence on the deformation parameter $\theta$, the 
left regular  representation is denoted by $L^\theta$.
With this notation we have (see \cite[Lemma 2.12]{GGISV})
\begin{equation}
\label{EQQQ}
L^\theta(f) g:=f\star_\theta g,\quad \|L^\theta(f)\|\leq  (2\pi\theta)^{-d/2}\|f\|_2,\qquad f,g\in L^2(\R^{2d}).
\end{equation}
Note the singular nature of this estimate in the commutative $\theta\to 0$ limit.
Since the operator norm of a bounded operator on a Hilbert space $\H$ coincides 
(via the left regular representation) with
the operator norm of the same bounded operator acting by left multiplication on 
the Hilbert space $\L^2(\B(\H))$ of Hilbert-Schmidt operators, we have
\begin{equation}
\label{BB}
\|L^\theta(f)\|=(2\pi)^{d/2}\| {\rm Op}_W(f)\|,
\end{equation}
where the first norm is the operator norm on $L^2(\R^{2d})$ and the second is the operator norm
on $L^2(\R^{d})$. In particular, the Weyl quantization gives the identification of von Neumann algebras:
\begin{equation}
\label{identification}
\B\big(L^2(\R^{2d})\big)\supset\big\{ L^\theta(f),\,f\in L^2(\R^{2d})\big\}''\simeq \B\big(L^2(\R^{2})\big).
\end{equation}
The following Hilbert-Schmidt norm equality on $L^2(\R^{2d})$,
is proven in \cite[Lemma 4.3]{GGISV} (this is the analogue of Lemma \ref{HS} in this context):
\begin{align}
\label{HS2}
 \|L^\theta(f)g(\nabla)\|_2=(2\pi)^{-d}\|g\|_2\|f\|_2.
\end{align}
Note the independence of $\theta$ on the right hand side.

\subsubsection{ A smoothly summable spectral triple for Moyal plane}

In this paragraph, we generalize the result of \cite{GGISV}. For simplicity, we 
restrict ourself to the simplest $d=2$ case,  despite the fact that our analysis can 
be carried out in any even dimension. Here we let $\H:=L^2(\R^2)\otimes \C^2$ 
the Hilbert space of square integrable sections of the trivial spinor bundle on $\R^2$. 
In Cartesian coordinates,  the flat Dirac operator reads 
$$
\D:=\begin{pmatrix} 0& i\partial_1-\partial_2\\ i\partial_1+\partial_2&0\end{pmatrix}.
$$
Elements of the algebra $(\mathcal S(\R^2),\star_\theta)$ are represented on 
$\H$ via $L^\theta\otimes {\rm Id}_2$, the diagonal left regular representation. 
In \cite{GGISV}, it is proven that $\big((\mathcal S(\R^2),\star_\theta),\H,\D\big)$ 
is an even $QC^\infty$ finitely summable  spectral triple with spectral 
dimension 2 and  with grading
$$
\gamma=\begin{pmatrix} 1& 0\\ 0&-1\end{pmatrix}.
$$
In particular, 
the Leibniz rule in the first display of Equation  \eqref{ptr} gives
\begin{equation}
\label{bound-com}
[\D,L^{\theta}(f)\otimes {\rm Id}_2]=\begin{pmatrix} 0& iL^\theta(\partial_1f)-L^\theta(\partial_2f)\\ 
iL^\theta(\partial_1f)+L^\theta(\partial_2f)&0\end{pmatrix},
\end{equation}
which together with \eqref{EQQQ}, shows that for $f$ a Schwartz function, the commutator 
$[\D,L^{\theta}(f)\otimes {\rm Id}_2]$ extends to a bounded operator.

Then, from the Hilbert-Schmidt norm computation of Equation \eqref{HS2}, we can determine 
the weights $\vf_s$ of Definition \ref{def:d-does-int}, constructed  with the flat Dirac operator on $\R^2$.

\begin{lemma} 
\label{FFF}
For $s>2$, let $\vf_s$ be the faithful normal semifinite 
weight of Definition \ref{def:d-does-int} determined by $\D$  on the type I von Neumann
algebra  $\B(\H)$ with operator trace.  When  restricted to the 
von Neumann subalgebra of $\B(\H)$ generated by $L^\theta(f)\otimes{\rm Id}_2$, $\vf_s$ is a tracial weight and
for $f\in L^2(\R^2)$ we have
\begin{align*}
\vf_s\big(L^\theta(f)^*L^\theta(f)\otimes{\rm Id}_2\big)
= ({\pi(s-2)})^{-1}\int \bar f(x)\star_\theta f(x) dx= 2({s-2})^{-1}\|{\rm Op}_W(f)\|_2^2.
\end{align*}

\end{lemma}
\begin{proof}
Since $\D^2=\Delta\otimes{\rm Id}_2$, with $0\leq \Delta$ the usual Laplacian on $\R^2$, we have  
\begin{align*}
\vf_s\big(L^\theta(f)^*L^\theta(f)\otimes{\rm Id}_2\big)
&= 2 {\rm Tr}_{L^2(\R^2)}\big((1+\Delta)^{-s/4}L^\theta(f)^*L^\theta(f)(1+\Delta)^{-s/4}\big).
\end{align*}
Thus the result follows from Equations \eqref{ptr}, \eqref{rtv} and \eqref{HS2}. 
\end{proof}

We turn now to the question of which  elements of the von Neumann 
algebra  generated by $L^\theta(f)\otimes{\rm Id}_2$ are in  $\B_1^\infty(\D,2)$. 
The next result follows by combining Proposition \ref{tracial-case} with Lemma \ref{FFF}.  

\begin{corollary}
\label{aB2}
Identifying the von Neumann subalgebra of $\B(L^2(\R^2))$   
generated by $L^\theta(f)\otimes{\rm Id}_2$, $f\in L^2(\R^2)$,   with $\B(L^2(\R))$ as in Equation \eqref{identification}
yields the identifications
$$
\B_1(\D,2)\bigcap \B(L^2(\R))\simeq L^2(\R^2)\star_\theta L^2(\R^2)\simeq \cl^1\big(L^2(\R)\big).
$$
Moreover, for all $m\in\N$, the norms on $L^2(\R^2)\star_\theta L^2(\R^2)$
$$
f\mapsto \PP_m\big(L^\theta(f)\otimes{\rm Id}_2\big),
$$ 
are equivalent to the single norm
$$
f\mapsto\|{\rm Op}_W(f)\|_1.
$$
\end{corollary}

\begin{proof}
The identification $L^2(\R^2)\star_\theta L^2(\R^2)\simeq \cl^1\big(L^2(\R)\big)$ follows from the identification
$L^2(\R^2)\simeq \cl^2\big(L^2(\R)\big)$ given by the unitarity of the Weyl quantization map, and the equality
$$
 \cl^2\big(L^2(\R)\big)\cdot \cl^2\big(L^2(\R)\big)= \cl^1\big(L^2(\R)\big).
 $$
By  Proposition \ref{tracial-case} we know that $\B_1(\D,2)\bigcap \B(L^2(\R))$ is identified with 
$$
\bigcap_{n\geq 1}\cl^1\big(  \B(L^2(\R)),\vf_{2+1/n}\big).
$$
Lemma \ref{FFF}  says that restricted to $ \B(L^2(\R))$, all the weights $\vf_{2+1/n}$ are 
proportional to the operator trace of $ \B(L^2(\R))$, giving the final identification. 
Moreover,   Proposition \ref{tracial-case} also gives the equality
 $$
 \PP_n(.)=  2\|\cdot\|_{\tau_n}+\|\cdot\|,
 $$
where $\|\cdot\|_{\tau_n}$  is the trace norm associated to the tracial weight  
$\vf_{2+1/n}$ restricted to $ \B(L^2(\R))$. As the latter is proportional to the operator 
trace on  $ \B(L^2(\R))$, which dominates the operator norm   since we are in the $I_\infty$ factor case, we get
 the equivalence of the norms
 $$
 f\mapsto \PP_n\big(L^\star(f)\otimes{\rm Id}_2\big),\quad\mbox{and}\quad\|{\rm Op_W}(f)\|_1\quad n\in\N,
 $$
 and we are done.
\end{proof}

On the basis of the previous result, we construct a 
Fr\'echet algebra yielding a smoothly summable spectral triple
of spectral dimension $2$, for the Moyal product.

\begin{lemma}
Endowed with the set of seminorms
$$
f\mapsto \|f\|_{1,\alpha}:=\|{\rm Op}_W(\partial^\alpha f)\|_1,\quad \alpha\in\N_0^2,
$$
the set
$$
\A:=\big\{f\in C^\infty(\R^2)\,\,:\,\,\mbox{for all } n\in\N_0^2,
\quad\exists f_1,f_2 \in L^2(\R^2),\quad \partial_1^{n_1}\partial_2^{n_2}f=f_1\star_\theta f_2\big\},
$$
is a  Fr\'echet algebra for  the Moyal product.
\end{lemma}

\begin{proof}
From the Leibniz rule for the Moyal product (see Equation \eqref{ptr} second display) 
and the fact that $L^2(\R^2)\star_\theta L^2(\R^2)\subset L^2(\R^2)$,   
the set $\A$ is an algebra for the Moyal product. 
Since $L^2(\R^2)\star_\theta L^2(\R^2)\simeq \cl^1\big(L^2(\R)\big)$, the seminorms
$\|\cdot\|_{1,\alpha}$, $\alpha\in\N_0^2$, take finite values on $\A$. It remains to show that 
$\A$ is complete for the topology induced by
these seminorms. So let $(f_k)_{k\in\N}$ be a Cauchy sequence on $\A$, i.e. 
Cauchy for each seminorm $\|\cdot\|_{1,\alpha}$. Since $\cl^1(L^2(\R))$ is complete, 
for each $\alpha\in\N_0^2$, $\big({\rm Op}_W(\partial^\alpha f_k)\big)_{k\in\N}$ 
converges to $A_\alpha$, a trace-class operator on $L^2(\R)$. But 
since $\cl^1(L^2(\R))\simeq L^2(\R^2)\star_\theta L^2(\R^2)$, via the 
Weyl map, $A_\alpha={\rm Op}_W(f_\alpha)$ for some element $f_\alpha\in
L^2(\R^2)\star_\theta L^2(\R^2)$. In particular for $\alpha=(0,0)$, the 
sequence $(f_k)_{k\in\N}$ converges to an element 
$f\in L^2(\R^2)\star_\theta L^2(\R^2)$. But we need to show that $f\in\A$, that is, 
we need to show
that $\|{\rm Op}_W(\partial^\alpha f)\|_1<\infty$ for all $\alpha\in\N_0^2$. 
This will be the case if $\partial^\alpha f=f_\alpha$. Note that 
$f\in L^2(\R^2)\star_\theta L^2(\R^2)\subset L^2(\R^2)\subset \cS'(\R^2)$, so that
$\partial^\alpha f\in  \cS'(\R^2)$ too.
With $\langle\cdot|\cdot\rangle$ denoting the duality bracket 
$\cS'(\R^2)\times \cS(\R^2)\to\C$, we have 
for any $k\in\N$ and any $\psi\in\cS(\R^2)$
\begin{align*}
\big|\langle(\partial^\alpha f-f_\alpha)|\psi\rangle\big|
&=\big|\langle(\partial^\alpha f-\partial^\alpha f_k)|\psi\rangle-
\langle(f_\alpha-\partial^\alpha f_k)|\psi\rangle\big|\\
&=\big|(-1)^{|k|}\langle( f- f_k)|\partial^\alpha\psi\rangle-
\langle(f_\alpha-\partial^\alpha f_k)|\psi\rangle\big|\\
&\leq\| f- f_k\|_2\,\|\partial^\alpha\psi\|_2+\|f_\alpha-\partial^\alpha f_k\|_2\,\|\psi\|_2\\
&\quad=(2\pi)^{1/2} \|\partial^\alpha\psi\|_2\,\|{\rm Op}_W( f- f_k)\|_2+(2\pi)^{1/2}\|\psi\|_2\,
\|{\rm Op}_W(f_\alpha-\partial^\alpha f_k)\|_2,
\end{align*}
where we have used  Equation \eqref{HS2}. Now, since the  the trace-norm dominates
the Hilbert-Schmidt norm, we find
$$
\big|\langle(\partial^\alpha f-f_\alpha)|\psi\rangle\big|\leq C(\psi)\big(\| {\rm Op}_W(f)- {\rm Op}_W(f_k)\|_1+
\|{\rm Op}_W(f_\alpha)-{\rm Op}_W(\partial^\alpha f_k)\|_1\big).
$$
But since ${\rm Op}_W(\partial^\alpha f_k)\to {\rm Op}_W(f_\alpha)$ in trace-norm for all $\alpha\in\N_0^2$, 
we see that $\big|\langle(\partial^\alpha f-f_\alpha)|\psi\rangle\big|\leq \eps$ for all $\eps>0$ and thus
$\langle(\partial^\alpha f-f_\alpha)|\psi\rangle=0$ for all $\psi\in\cS(\R^2)$. Hence $\partial^\alpha f=f_\alpha$
in $\cS'(\R^2)$, but since $f_\alpha\in  L^2(\R^2)\star_\theta L^2(\R^2)$, $\partial^\alpha f\in L^2(\R^2)\star_\theta L^2(\R^2)$ too.
This completes the proof.
\end{proof}

{\bf Remark.} Note that the $C^*$-completion of $(\A,\star_\theta)$, is 
isomorphic to the $C^*$-algebra of compact operators acting
on $L^2(\R)$ and that $\A$ contains $\mathcal S(\R^2)$. 

Combining all these preliminary statements, we now improve  the results of \cite{GGISV}.

\begin{prop}
The data $(\A,\H,\D,\gamma)$ defines an even 
smoothly summable  spectral triple with spectral dimension 2. 
\end{prop}

\begin{proof}
We first need to prove that $(\A,\H,\D,\gamma)$ (which is  even) is finitely summable,
that is, we need to show that
$$
\delta^k\big(L^\theta(f)\otimes{\rm Id}_2\big)\,(1+\D^2)^{-s/2}\in\cl^1(\H),\quad\mbox{for all } f\in\A,\quad\mbox{for all } s>2,\quad \mbox{for all } k\in\N_0.
$$
But from the proof of Proposition \ref{smooth-sum-sufficient},
this will follow if 
$$
(1+\D^2)^{-s/4}R^k\big(L^\theta(f)\otimes{\rm Id}_2\big)(1+\D^2)^{-s/4}\in\L^1(\cn,\tau),
$$
for all $ f\in\A$, for all $s>2$ and  for all $ k\in\N_0$.
Now, by the Leibniz rule (Equation \ref{ptr} first display), we have with $\Delta=-\partial_1^2-\partial_2^2$,
$$
[\Delta,L^\theta(f)]=L^\theta(\Delta f)+2L^\theta(\partial_1f)\partial_1+2L^\theta(\partial_2f)\partial_2,
$$
so that since $\D^2=\Delta\otimes{\rm Id}_2$, we have for all $k\in\N_0$
$$
R^k\big(L^\theta(f)\otimes{\rm Id}_2\big)=\sum_{|\alpha|,|\beta|\leq k}
C_{\alpha,\beta}L^\theta(\partial^\alpha f)\partial^\beta (1+\Delta)^{-k/2}  \otimes{\rm Id}_2,
$$
and thus
\begin{align*}
&(1+\D^2)^{-s/4}R^k\big(L^\theta(f)\otimes{\rm Id}_2\big)(1+\D^2)^{-s/4}\\
&\qquad\qquad\qquad\qquad\qquad=
\sum_{|\alpha|,|\beta|\leq k}
C_{\alpha,\beta} (1+\Delta)^{-s/4}L^\theta(\partial^\alpha f)(1+\Delta)^{-s/4}\partial^\beta (1+\Delta)^{-k/2}  \otimes{\rm Id}_2,
\end{align*}
which is trace class because $\partial^\beta (1+\Delta)^{-k/2}$ is bounded and by definition of $\A$,
$\partial^\alpha f=f_1\star_\theta f_2$ with $f_1,f_2\in L^2(\R^2)$, so that this operator appears as the product of two
Hilbert-Schmidt by Equation \eqref{HS2}. Thus, the spectral triple is finitely summable, 
and the spectral dimension is $2$ by\cite[Lemma 4.14]{GGISV}, which gives for any $f\in\A$
$$
{\rm Tr}\big(L^\theta(f)\otimes{\rm Id}_2(1+\D^2)^{-s/2}\big)= \frac{1}{\pi(s-2)}\int_{\R^2} f(x)\,dx.
$$
From Proposition \ref{smooth-sum-sufficient}, we also have verified one of the condition ensuring
that $\A\cup[\D,\A]\subset \B_1^\infty((\D,2)$. The second is to verify that
$$
 (1+\D^2)^{-s/4}R^k\big([\D,L^\theta(f)\otimes{\rm Id}_2]\big)(1+\D^2)^{-s/4}\in\L^1(\cn,\tau),\,\,\mbox{for all } k\in\N_0,\,\,\mbox{for all } s>p.
$$
This can be done as for $R^k\big(L^\theta(f)\otimes{\rm Id}_2\big)$ by noticing that 
\begin{align*}
&R^k\big([\D,L^\theta(f)\otimes{\rm Id}_2]\big)=\sum_{|\alpha|\leq k}\sum_{|\beta_1|,|\beta_2|\leq k+1}
C_{\alpha,\beta_1,\beta_2}\begin{pmatrix} 0& L^\theta(\partial^{\beta_1}f)\\
L^\theta(\partial^{\beta_2}f)&0\end{pmatrix}\partial^\alpha (1+\Delta)^{-k/2}\otimes{\rm Id}_2,
\end{align*}
and the proof is complete.
\end{proof}

\subsubsection{An index formula for the Moyal plane} 
In order to obtain an explicit 
index formula out of the spectral triple
previously constructed, we need to introduce a suitable family of projectors.

Let $H:=\tfrac12(x_1^2+x_2^2)$ be the (classical) Hamiltonian of the one-dimensional 
harmonic oscillator. Let also
$a:=2^{-1/2}(x_1+ix_2)$, $\bar a:=2^{-1/2}(x_1-ix_2)$ be the annihilation and creation functions. 
Define next
$$
f_{m,n}:=
\frac{1}{\sqrt{\theta^{n+m}n!m!}}\,\bar a^{\star_\theta m}\star_\theta f_{0,0}\star_\theta a^{\star_\theta n}
\quad \mbox{where}\quad f_{0,0}:=2 e^{-\tfrac2\theta H},\qquad m,n\in\N_0.
$$
The family $\{f_{m,n}\}_{m,n\in\N_0}$ forms an orthogonal basis of $L^2(\R^2)$, 
consisting of Schwartz functions.
They constitute an important tool in the analysis of \cite{GGISV}, since they allow to construct 
local units. In fact, they 
are the Weyl symbols of the rank
one  operators $\vf\mapsto \langle\vf_m|\vf\rangle\,\vf_n$, with $\{\vf_{n}\}_{n\in\N_0}$ 
the basis of $L^2(\R)$ consisting of eigenvectors
for the one-dimensional quantum harmonic oscillator. 
The proof of the next lemma can be found in \cite[subsection 2.3 and Appendix]{GGISV}.

\begin{lemma}
The following relations hold true.
$$
\overline{f_{m,n}}=f_{n,m},\qquad f_{m,n}\star_\theta f_{k,l}
=\delta_{n,k}\,f_{m,l},\qquad \int f_{m,n}(x) \,dx=2\pi\theta \,\delta_{m,n},
$$
so in particular $\{f_{n,n}\}_{n\in\N_0}$, is a family of pairwise orthogonal projectors. 
Moreover we have:
\begin{align*}
&\big[\D,L^{\theta}(f_{m,n})\otimes {\rm Id}_2\big]=\\
&\qquad-i\sqrt{\frac 2\theta}\begin{pmatrix} 0& \sqrt m\,L^\theta(f_{m-1,n})-\sqrt{n+1}\,L^\theta(f_{m,n+1})\\ 
\sqrt n\, L^\theta(f_{m,n-1})-\sqrt{m+1}\,L^\theta(f_{m+1,n})&0\end{pmatrix},
\end{align*}
\label{someprops}
with the convention
that $f_{m,n}\equiv 0$ whenever $n<0$ or $m<0$.
\end{lemma}

We are in the situation
where the projectors $f_{n,n}$  belong to the algebra (not its unitization, nor a 
matrix algebra over it).
Thus if we set $F=\D(1+\D^2)^{-1/2}$ then  
$L^\theta(f_{n,n})F_\pm L^\theta(f_{n,n})$ is a Fredholm operator 
from $L^2(\R^2)$ to itself, according to the discussion at the beginning of the subsection \ref{num-ind-pair}.
Thus, we don't need the `double picture' here. In particular, 
$[f_{n,n}]\in K_0(\A)$. The next result
computes the numerical  index pairing between $(\A,L^2(\R^2,\C^2),\D)$ and $K_0(\A)$.

\begin{prop}
For $J$ a finite subset of $\N_0$, let $p_J:=\sum_{n\in J} L^\theta(f_{n,n})$.
Setting $F=\D(1+\D^2)^{-1/2}$, we have the integer-valued index paring
$$
{\rm Index}\big(p_J F_+ p_J\big)=\big\langle [p_J], [(\A,L^2(\R^2,\C^2),\D)]\big\rangle={\rm Card}(J).
$$
In particular, the index map gives an explicit isomorphism between 
$K_0\big(\mathcal K(L^2(\R))\big)$ and $\mathbb Z$.
\end{prop}
 
\begin{proof}
Assume first that $J=\{n\}$, $n\in\N_0$.
The degree zero term is not zero in this case as the projection lies in our algebra.
Hence, including all the constants from the local index formula and the Chern character of $f_{n,n}$ gives
\begin{align*}
&{\rm Index}\big(L^\theta(f_{n,n})F_+ L^\theta(f_{n,n})\big)
={\rm res}_{z=0}\frac{1}{z}{\rm Tr}\big(\gamma L^\theta(f_{n,n})(1+\D^2)^{-z}\big)\\
&\qquad-\,{\rm res}_{z=0}
\mbox{Tr}\Big(\gamma\big( L^\theta(f_{n,n})\otimes {\rm Id}_2-1/2\big)[\D, L^\theta(f_{n,n})
\otimes {\rm Id}_2][\D, L^\theta(f_{n,n})\otimes {\rm Id}_2](1+\D^2)^{-1-z}\Big).
\end{align*}
The second term is computed with the help of  Lemma \ref{someprops}. First we have
\begin{align*}
&\gamma\big( L^\theta(f_{n,n})\otimes {\rm Id}_2-1/2\big)[\D, L^\theta(f_{n,n})
\otimes {\rm Id}_2][\D, L^\theta(f_{n,n})\otimes {\rm Id}_2]\\
&\qquad\qquad\qquad=\frac 1{\theta}
\begin{pmatrix}  
 n\,L^\theta(f_{n-1,n-1})-(n+1)\,L^\theta(f_{n,n})&0\\ 
0&-(n+1)\, L^\theta(f_{n+1,n+1})+n\,L^\theta(f_{n,n})
\end{pmatrix}.
\end{align*}
Since $\D^2=\Delta\otimes {\rm Id}_2$, with here $\Delta=-\partial_1^2-\partial_2^2$, we find that
\begin{align*}
&\mbox{Tr}\big(\gamma( L^\theta(f_{n,n})\otimes {\rm Id}_2-1/2)[\D, L^\theta(f_{n,n})]\otimes {\rm Id}_2[\D, L^\theta(f_{n,n})\otimes {\rm Id}_2](1+\D^2)^{-1-z}\big)\\
&\qquad=\frac1{\theta}\mbox{Tr}\Big(\big(-L^\theta(f_{n,n})-(n+1)L^\theta(f_{n+1,n+1})+nL^\theta(f_{n-1,n-1})\big)(1+\Delta)^{-1-z}\Big)\\
&\qquad=\frac1{\theta}\frac1{(2\pi)^2}\,\int \big(-f_{n,n}(x)-(n+1)f_{n+1,n+1}(x)+nf_{n-1,n-1}(x)\big)dx\,
\int (1+|\xi|^2)^{-1-z} \,d\xi\\
&\qquad=\frac1{\theta}\,\frac1{(2\pi)^2}\, \big(-1-(n+1)+n\big)\,(2\pi\theta)\, \frac{2\pi}{2z}=-\frac{1}{z}.
\end{align*}
In the second equality we have used \cite[Lemma 4.14]{GGISV}--the factor $(2\pi)^{-2}$ can also be deduced from
\eqref{HS2}--and we have used Lemma \ref{someprops} to obtain the last 
line--this is where the factor $2\pi\theta$
comes from. Thus the residue from the second term gives us $1$. 
For the first term we compute
\begin{align*}
{\rm res}_{z=0}\,\frac{1}{z}{\rm Tr}\big(\gamma L^\theta(f_{n,n})\otimes {\rm Id}_2\,(1+\D^2)^{-z}\big)=0,
\end{align*}
because  the grading $\gamma$ cancels the traces on each half of the spinor space. 
This gives the result in this elementary case, 
 ${\rm Index}\big(L^\theta(f_{n,n})F_+ L^\theta(f_{n,n})\big)=1$.
 For the general case, note that since for $n\neq m$, $f_{m,m}$ and $f_{n,n}$ are 
 orthogonal projectors, we have $[f_{m,m}+f_{n,n}]=[f_{m,m}]+[f_{n,n}]\in K_0(\A)$ and the final 
 result follows immediately.
\end{proof}




\appendix
\section{Estimates and technical lemmas}
\label{sec:technical}
\subsection{Background material on the pseudodiferential expansion}
To aid the reader, this Appendix recalls five Lemmas from \cite{CPRS2} which are 
used repeatedly in Section \ref{sec:psido-and-sum}
and in Section \ref{sec:LIT}.
All were proved in the unital setting, however all norm estimates remain unchanged, and 
in the pseudodifferential expansion in Lemmas  \ref{higsexpan}, \ref{firstexpan}, 
if the operators $A_i$ lie in 
${\rm OP}^{*}_0$, then so does the remainder, by the invariance of 
${\rm OP}_0^*$ under
the one parameter group $\sigma$ (see Proposition \ref{cont-op}). 
The integral estimate in Lemma \ref{intest} is unaffected
by any changes. 

We begin by giving the algebraic version of the pseudodifferential expansion developed by Higson.
This expansion gives simple formulae, and sharp estimates on remainders. In the statement
$Q=t+s^2+\D^2$, $t\in[0,1]$, $s\in[0,\infty)$. 

\begin{lemma} \label{higsexpan}(see \cite[Lemma 6.9]{CPRS2}) 
Let $m,n,k$ be non-negative integers and 
$T\in {\rm OP}^m_0$ (resp. $T\in {\rm OP}^m$). Then for $\lambda$ in the resolvent set of $Q$
\bean 
(\lambda-Q)^{-n}T= \sum_{l=0}^{k}\binom{n+l-1}{l}T^{(l)}
(\lambda-Q)^{-n-l} + P(\lambda),
\eean
where the remainder $P(\lambda)$ belongs to $ {\rm OP}^{-(2n+k-m+1)}_0$ 
(resp. $ {\rm OP}^{-(2n+k-m+1)}$)  and is given by
\ben 
P(\lambda)=\sum_{l=1}^n\binom{l+k-1}{k}(\lambda-Q)^{l-n-1}T^{(k+1)}(\lambda-Q)^{-l-k}.
\een
\end{lemma}

In the following  lemmas, we let $R_s(\lambda)=(\lambda-(1+\D^2+s^2))^{-1}$.

\begin{lemma}\label{normtrick}(see \cite[Lemma 6.10]{CPRS2}) 
Let $k,n$ be non-negative integers, $s\geq 0$, 
and suppose
$\lambda\in{\C}$, $0<\Re(\lambda)<1/2$. Then for $A \in {\rm OP}^k$, we have
$$ \Vert R_s(\lambda)^{n/2+k/2}AR_s(\lambda)^{-n/2}\Vert\leq C_{n,k}
\;\;and\;\;\Vert R_s(\lambda)^{-n/2}AR_s(\lambda)^{n/2+k/2}\Vert\leq C_{n,k},$$
where $C_{n,k}$ is constant independent of $s$ and $\lambda$ (square roots use 
the principal branch of $\log$.)
\end{lemma}

\begin{lemma}\label{firstexpan}(see \cite[Lemma 6.11]{CPRS2}) 
Let $A_i\in {\rm OP}^{n_i}_0$ (resp. $A_i\in {\rm OP}^{n_i}$) for $i=1,\dots,m$ and let
$0<\Re(\lambda)<1/2$ as above. We consider the operator 
$$
R_s(\lambda)A_1R_s(\lambda)A_2R_s(\lambda)\cdots R_s(\lambda)A_m R_s(\lambda),
$$
Then for all $M\geq 0$
\ben 
R_s(\lambda)A_1R_s(\lambda)A_2\cdots 
A_m R_s(\lambda)=\sum_{|k|=0}^MC(k)A_1^{(k_1)}\cdots 
A_m^{(k_m)}R_s(\lambda)^{m+|k|+1}+P_{M,m},
\een
where $P_{M,m}\in {\rm OP}^{|n|-2m-M-3}_0$ (resp. $P_{M,m}\in {\rm OP}^{|n|-2m-M-3}$), 
and  $k$ and $n$
are multi-indices with $|k|=k_1+\cdots+k_m$ and $|n|= n_1+\cdots+n_m$.
The constant $C(k)$ is given by
\ben 
C(k)=\frac{(|k|+m)!}{k_1!k_2!\cdots k_m!(k_1+1)(k_1+k_2+2)\cdots(|k|+m)}.
\een
\end{lemma}

\begin{lemma}\label{uni}(see \cite[Lemma 6.12]{CPRS2})  
With the assumptions and notation of the last Lemma
including the assumption that $A_i\in {\rm OP}^{n_i}$ for each $i$, 
there is a positive constant $C$ such that
\ben 
\Vert (\lambda-(1+\D^2+s^2))^{m+M/2+3/2-|n|/2}P_{M,m}\Vert\leq C,
\een
independent of $s$ and $\lambda$ (though it depends on $M$ and $m$ and 
the $A_i$). 
\end{lemma}

\begin{lemma}\label{intest}(see \cite[Lemma 5.4]{CPRS2}) 
Let $0<a<1/2$ and $0\leq c\leq\sqrt{2}$ 
and $j=0$ or $1$. Let $J$,$K$, and $M$ be nonegative constants. 
Then the integral
\be 
\int_0^\infty \int_{-\infty}^\infty s^J\sqrt{a^2+v^2}^{-M}
\sqrt{(s^2+1/2-a)^2+v^2}^{-K}\sqrt{(s^2+1-a-sc)^2+v^2}^{-j}dvds,
\label{integral}
\ee
converges provided $J-2K-2j<-1$ and $J-2K-2j+1-2M<-2$.
\end{lemma}

\subsection{Estimates for Section \ref{sec:LIT}}
In this subsection, we collect the proofs of the
key lemmas in our homotopy arguments which are essentially nonunital variations of
proofs  appearing in \cite{CPRS4}.

The first result we prove is the analogue of \cite[Lemma 7.2]{CPRS2}, needed to prove that
the expectations used to define our various cochains are well-defined and holomorphic.

\subsubsection{Proof of Lemma \ref{lem:crucial}}
\label{lem:crucial-app}

Most of the proof relies on the same algebraic arguments  and norm estimates  
as in \cite[Lemma 7.2]{CPRS2}. We just need to adapt the arguments which use some trace 
norm estimates.
To simplify the notations for $0\leq t\leq 1$, we use the shorthand 
$$
R:=R_{s,t}(\lambda)=(\lambda-(t+s^2+\D^2))^{-1},
$$
 as in Equation \eqref{resol}.
We first remark that we can always assume $A_0\in{\rm OP}^0_0$, 
at the price that $A_1$ will be in ${\rm OP}^{k_0+k_1}$, so that the global 
degree $|k|$ remains unchanged. Indeed, we can write
$$
A_0\,R\,A_1\,R\cdots R\,A_m\,R=
A_0(1+\D^2)^{-k_0/2}\,R\,(1+\D^2)^{k_0/2}A_1\,R\cdots R\,A_m\,R,
$$
and this remark follows from the change
$$
A_0\in{\rm OP}^{k_0}_0\mapsto 
A_0(1+\D^2)^{-k_0/2}\in{\rm OP}^{0}_0,\quad A_1\in{\rm OP}^{k_1}\mapsto 
(1+\D^2)^{k_0/2}A_1\in{\rm OP}^{k_0+k_1}.
$$

From  Lemma \ref{firstexpan}, we know 
that for any $L\in\N$, there exists a regular pseudodifferential operator $P_{L,m}$ 
of order (at most) $|k|-2m-L-3$ (i.e. $P_{L,m}\in {\rm OP}^{|k|-2m-L-3}$), such that
\begin{align}
\label{expansion}
A_0\,R\,A_1\,R\cdots R\,A_m\,R
=\sum_{|n|=0}^LC(n) A_0\,A_1^{(n_1)}\cdots A_m^{(n_m)}\,R^{m+|n|+1}+A_0\,P_{L,m}.
\end{align}
Regarding the remainder term $P_{L,m}$, 
  by Lemma \ref{uni} we know that it satisfies the norm inequality
$$
\|R_{s,t}(\lambda)^{-m-L/2-3/2+|k|/2}\, P_{L,m}\| \leq C,
$$
where the constant $C$ is uniform in $s$ and $\lambda$. 
(Here the complex square root function is defined with its principal branch.) 
Using Lemma \ref{lem:was-5.3} and  $A_0\in{\rm OP}^0_0$, we obtain the 
trace norm bound
$$
\|A_0\,P_{L,m}\|_1\leq C\|A_0 R_{s,t}(\lambda)^{m+L/2+3/2-|k|/2}\|_1
\leq C' ((s^2+a)^2+v^2)^{-m/2-L/4-3/4+|k|/4+(p+\eps)/4}.
$$
Thus, the corresponding $s$-integral of the  trace-norm of $B_{r,t}(s)$ is bounded by
\begin{align*}
\int_0^\infty s^\alpha\Big\|\int_\ell \lambda^{-p/2-r}\,A_0\,P_{L,m}d\lambda\Big\|_1ds&
\leq\int_0^\infty s^\alpha \int_\ell |\lambda|^{-p/2-r}\|A_0\,P_{L,m}\|_1|d\lambda|ds\\
&\hspace{-5cm}\leq C
\int_0^\infty s^\alpha\int_{-\infty}^\infty 
(a^2+v^2)^{-p/4-\Re(r)/2} ((s^2+a)^2+v^2)^{-m/2-L/4-3/4+|k|/4+(p+\eps)/4}dvds,
\end{align*}
where $\ell$ is the vertical line $\ell=\{a+iv:v\in\R\}$ with $a\in(0,\mu^2/2)$.
By Lemma \ref{intest}, the latter integral is finite for 
$L>|k|+\alpha+p+\eps-2-2m$, which can always be arranged.
 To perform the Cauchy integrals
$$
\frac{1}{2\pi i}\int_\ell \lambda^{-p/2-r} A_0A_1^{(n_1)}\cdots A_m^{(n_m)}R^{m+1+|n|}
d\lambda,
$$
we refer to  \cite[Lemma 7.2]{CPRS2} for the 
precise justifications. This gives a multiple of
$$
A_0A_1^{(n_1)}\cdots A_m^{(n_m)}(t+s^2+\D^2)^{-p/2-r-m-|n|}.
$$
By Lemmas \ref{usefull} and \ref{lem:derivs}, 
we see that  $A_0A_1^{(n_1)}\cdots A_m^{(n_m)}\in{\rm OP}^{|k|+|n|}_0$,  so that
$$
B:=A_0A_1^{(n_1)}\cdots A_m^{(n_m)}|\D|^{-|n|-|k|}\in {\rm OP}^0_0.
$$
(Remember that in this setting we assume $\D^2\geq\mu^2$).
Thus for $\varepsilon>0$, Equation \eqref{eq:q-trace-est} gives
\begin{align*}
&\big\| A_0A_1^{(n_1)}\cdots A_m^{(n_m)}(t+s^2+\D^2)^{-p/2-r-m-|n|}\big\|_1
=\big\|B|\D|^{|n|+|k|}(t+s^2+\D^2)^{-p/2-r-m-|n|}\big\|_1\\
&\qquad\leq
\big\|B(t+s^2+\D^2)^{-p/2-r-m-|n|/2+|k|/2}\big\|_1 \big\|{|\D|^{|n|+|k|}}{(t+s^2+\D^2)^{-|n|/2-|k|/2}}\big\| \\
&\qquad\leq C (\mu/2+s^2)^{-\Re(r)-m-|n|/2+|k|/2+\eps/2}.
\end{align*}
In particular,  
the constant $C$ is uniform in $s$.
The worst term being  that with $|n|=0$, we obtain that the 
corresponding $s$-integral is convergent for $\Re(r)>-m+(|k|+\alpha+1)/2+\eps$. $\square$

\subsubsection{Proof of Lemma \ref{s-trick}}
\label{lem:s-trick-app}

We give the proof for the expectation $\langle A_0,\dots,A_m\rangle_{m,r,s,t}$.
The proof for  $\langle\langle A_0,\dots,A_m\rangle\rangle_{m,r,s,t}$ is similar 
with suitable modification of the domain of the parameters.
From Lemma \ref{lem:crucial}, we first see that each term of the equality is well defined, 
provided  $2\Re(r)>1+\alpha+|k|-2m$, and since $2m+2>\alpha>0$, Lemma \ref{lem:crucial}
also shows that $\langle\langle A_0,\dots,A_m\rangle\rangle_{m,r,s,t}$ vanishes at $s=0$ and $s=\infty$. 
All we have to do is to show that the map 
$[s\mapsto \langle A_0,\dots,A_m\rangle_{m,r,s,t}]$ is differentiable, 
with derivative given by 
$$
2s\sum_{l=0}^m\langle A_0,\dots,A_l,1,A_{l+1},\dots,A_m\rangle_{m+1,r,s,t},
$$ 
since then the result will follow by integrating  between $0$ and $+\infty$ the following total  derivative
\begin{align*}
&\frac d{ds} s^\alpha \langle A_0,\dots,A_m\rangle_{m,r,s,t}\\
&\hspace{1,5cm}=\alpha \,s^{\alpha-1}\langle A_0,\dots,A_m\rangle_{m,r,s,t}
+2\sum_{l=0}^m
s^{\alpha+1}\langle A_0,\dots,A_l,1,A_{l+1},\dots,A_m\rangle_{m+1,r,s,t}.
\end{align*}
As
$
\frac{1}\eps\big(R_{s+\eps,t}(\lambda)-R_{s,t}(\lambda)\big)=
-R_{s+\eps,t}(\lambda)(2s+\eps)R_{s,t}(\lambda),
$
we see that the resolvent is continuously norm-differentiable in the 
$s$-parameter, with norm derivative given by $2sR_{s,t}(\lambda)^2$. We then write
\begin{align*}
&2\pi i\tfrac1{\eps}
\big(\langle A_0,\dots,A_m\rangle_{m,r,s+\eps,t}-\langle A_0,\dots,A_m\rangle_{m,r,s,t}\big)\\
&=\sum_{l=0}^m\tau\Big(\!\gamma\!\!\int_\ell \lambda^{-p/2-r}A_0\,R_{s+\eps,t}(\lambda)\dots
 A_l\,R_{s+\eps,t}(\lambda)(2s+\eps)
R_{s,t}(\lambda)\,A_{l+1}\dots R_{s,t}(\lambda)\,A_m\,R_{s,t}(\lambda)\,d\lambda\!\Big),
\end{align*}
where $\ell$ is the vertical line $\ell=\{a+iv:v\in\R\}$ with $a\in(0,\mu^2/2)$.
This leads to

\begin{align*}
&\tfrac1{\eps}\big(\langle A_0,\dots,A_m\rangle_{m,r,s+\eps,t}-
\langle A_0,\dots,A_m\rangle_{m,r,s,t,0}\big) -2s\sum_{l=0}^m
\langle A_0,\dots,A_l,1,A_{l+1},\dots,A_m\rangle_{m+1,r,s,t}\\
&=\frac{\eps}{2\pi i}\sum_{l=0}^m
\tau\Big(\gamma\int_\ell \lambda^{-p/2-r}A_0\,R_{s+\eps,t}(\lambda)
\cdots A_l\,R_{s+\eps,t}(\lambda)^2\,A_{l+1}
\cdots R_{s,t}(\lambda)\,A_m\,R_{s,t}(\lambda)\,d\lambda\Big)\nonumber\\
&\quad+
\frac{2s\eps}{2\pi i}\sum_{k\leq l=0}^m
\tau\Big(\gamma\int_\ell \lambda^{-p/2-r}A_0\,R_{s+\eps,t}(\lambda)
\cdots A_k\,R_{s+\eps,t}(\lambda)(2s+\eps,0)R_{s,t}(\lambda)\,A_{l+k}\cdots \nonumber\\
&\hspace{8cm}\times A_l\,R_{s,t}(\lambda)^2A_{l+1}
\cdots R_{s,t}(\lambda)\,A_m\,R_{s,t}(\lambda)\,d\lambda\Big).\nonumber
\end{align*}
We now proceed as in Lemma \ref{lem:crucial}. We write each 
integrand (of the first or second type)  as
\begin{align}
\label{gtr}
A_0\,R\,A_1\,R\cdots R\,A_{m+j}\,R
=\sum_{|n|=0}^MC(k) A_0\,A_1^{(n_1)}\cdots 
A_{m+j}^{(n_{m+j})}\,R^{m+j+|n|+1}+A_0\,P_{M,m+j},
\end{align}
where $j\in\{1,2\}$ depending the type of term we are looking at, the 
$A_l$'s have been redefined and now $R$ stands for $R_{s,t}(\lambda)$ 
or $R_{s+\eps,t}(\lambda)$. To treat the non-remainder terms, before 
applying the Cauchy formula, one needs to perform a resolvent expansion
$$
R_{s+\eps,t}(\lambda)
=\sum_{l=0}^M(-\eps(2s+\eps))^{l-1} R_{s,t}(\lambda)^l
+(-\eps(2s+\eps))^{M} R_{s,t}(\lambda)^MR_{s+\eps,t}(\lambda).
$$
We can always choose  $M$ big enough so that the integrand associated 
with the remainder term in the resolvent expansion 
is integrable in  trace norm, by Lemma \ref{lem:crucial}. 
Provided  $\Re(r)+m-|k|/2>0$, one sees with the same estimates as in 
Lemma \ref{lem:crucial}, that the corresponding term in the 
difference-quotient goes to zero with $\eps$. For the non-remainder 
terms of the resolvent expansion, we can use the Cauchy formula as in 
Lemma \ref{lem:crucial}, and obtain the same conclusion. All that is left is
to treat the remainder term in \eqref{gtr}. The main difference with the 
corresponding term  in Lemma \ref{lem:crucial} is that $P_{M,m+j}$ 
is now $\eps$-dependent.  But the $\eps$-dependence only occurs
in $R_{s+\eps,t}(\lambda)$ and since the estimate of Lemma \ref{normtrick}
is uniform in $s$, we still have
$$
\|R_{s,t}(\lambda)^{-m-M/2-3/2+|k|/2}\, P_{M,m+j}\| \leq C,
$$
where the constant is uniform in $s$, $\lambda$ and $\eps$. 

This is enough (see again the proof of Lemma \ref{lem:crucial}) to 
show that the corresponding term in the  difference-quotient goes to 
zero with $\eps$, provided  $\Re(r)+m-|k|/2>0$.
Thus $\langle A_0,\dots,A_m\rangle_{m,r,s,t}$ is 
differentiable in  $s$, concluding the proof.
\hfill $\square$

\subsubsection{Proof of Lemma \ref{lambda-trick}}
 \label{lem:lambda-trick-app}
According to our assumptions, one first notes  from Lemma \ref{lem:crucial}, 
that all the terms involved in the equalities above are well defined.
From
$$
\tfrac1\eps\big(R_{s,t}(\lambda+\eps)-R_{s,t}(\lambda)\big)+R_{s,t}(\lambda)^2
=\eps R_{s,t}(\lambda+\eps)R_{s,t}(\lambda)^2,
$$
we readily conclude that the map $\lambda\mapsto R_{s,t}(\lambda)$ is 
norm-continuously differentiable, with norm derivatives given by $-R_{s,t}(\lambda)^2$.
We deduce
that for $A_l\in {\rm OP}^{k_l}$, the map $\lambda\mapsto A_lR_{s,t}(\lambda)$ is continuously 
differentiable for the topology of ${\rm OP}^{k_l-2}$, with derivative given by 
$-A_lR_{s,t}(\lambda)^2$. Thus
$A_0R\cdots A_mR$  is continuously differentiable for the topology of 
${\rm OP}^{|k|-2m}_0$, with derivative given by
$$
-\sum_{l=0}^mA_0\,R_{s,t}(\lambda)\cdots A_l\,R_{s,t}(\lambda)^2\,A_{l+1}\cdots A_m\,R_{s,t}(\lambda).
$$
We thus arrive at the identity in ${\rm OP}^{|k|-2m}_0$:
\begin{align*}
\frac{d}{d\lambda}\Big(\lambda^{-q/2-r}A_0\,R_{s,t}(\lambda)\cdots A_m\,R_{s,t}(\lambda)\Big)
&=-(p/2-r)
\lambda^{-q/2-r-1}A_0\,R_{s,t}(\lambda)\cdots A_m\,R_{s,t}(\lambda)\\
&\hspace{-2cm}-
\sum_{l=0}^m\lambda^{-q/2-r}A_0\,R_{s,t}(\lambda)\cdots A_l\,R_{s,t}(\lambda)^2\,A_{l+1}\cdots A_m\,R_{s,t}(\lambda)\,\\
&=-(p/2-r)
\lambda^{-q/2-r-1}A_0\,R_{s,t}(\lambda)\cdots A_m\,R_{s,t}(\lambda)\\
&\hspace{-2cm}-
\sum_{l=0}^m\lambda^{-q/2-r}A_0\,R_{s,t}(\lambda)\cdots 
A_l\,R_{s,t}(\lambda)\,1\,R_{s,t}(\lambda)\,A_{l+1}\cdots A_m\,R_{s,t}(\lambda).
\end{align*}
By Lemma \ref{lem:crucial}, the $\lambda$-integral of the right hand side of the former 
equality is well defined as a trace class operator for $2\Re(r)>|k|-2m$. Performing the 
integration gives the result, since
$\langle\langle A_0,\dots,A_m\rangle\rangle_{m,r+1,s,t}$ vanishes
at the endpoints of the integration domain.
\hfill $\square$

\bigskip

We now present the proof of the trace norm differentiability result, Lemma \ref{diff1}, 
needed to
complete the homotopy 
to the Chern character.

\subsubsection{Proof of Lemma \ref{diff1}}
\label{lem:diff1-app}

Recall that our assumptions are that $a_0,\dots,a_M\in\A^\sim$ so that
$da_i,\,\delta(a_i)\in {\rm OP}^0_0$ for $i=0,\dots,M$. This means we can use
the result of Lemma \ref{truc}.
We first assume  $p\geq 2$. 
We start from the identity,
$$
d_u(a)=[\D_u,a]=[F|\D|^{1-u},a]=F[|\D|^{1-u},a]+\big(da-F\delta(a)\big)|\D|^{-u},
$$
and we note  that $da-F\delta(a)\in{\rm OP}^0_0$. 
Applying  the second part of
Lemma \ref{truc} and Lemma \ref{interpolation} 
now shows that $d_u(a)\in\cl^q(\cn,\tau)$ for all $q>p/u$. Next,  we find that
$$
R_{s,u}(\lambda)
=({\lambda-s^2-\D_u^2})^{-1}=|\D|^{-2(1-u)}{\D_u^2}{(\lambda-s^2-\D_u^2)^{-1}}=:
|\D|^{-2(1-u)}B(u),
$$
where $B(u)$ is uniformly bounded.  Then Lemma \ref{interpolation} 
and the H\"older inequality show that 
$$
d_u(a_i)\,R_{s,u}(\lambda)\in\cl^q(\cn,\tau),\quad\mbox{ for all}\quad q>p/(2-u)\geq p/2\geq 1\quad\mbox{ and}\quad
i=0,\dots,l,l+2,\dots,M,
$$
while  

$$
R_{s,u}(\lambda)^{1/2}\,d_u(a_{l+1})R_{s,u}(\lambda)\in\cl^q(\cn,\tau)\quad\mbox{ for all}\quad q\geq 2\quad\mbox{ with}\quad
(3-2u)q>p.
$$
 The worst case is $u=1$ for which we find $q\geq p\geq 2$, allowing
us to use the first and simplest case of Lemma \ref{interpolation}. 
Since $T_{s,\lambda,l}(u)$ contains $M$ terms
$d_u(a_i)\,R_{s,u}(\lambda)$ and contains  one term 
$R_{s,u}(\lambda)^{1/2}\,d_u(a_{l+1})\,R_{s,u}(\lambda)$ and one  
bounded term $\D_u R_{s,u}(\lambda)^{1/2}$, the H\"{o}lder inequality gives
$$
T_{s,\lambda,l}(u)\in\cl^q(\cn,\tau), \quad\mbox{for all}
\quad  q>p/(M(2-u)+(3-2u))=p/(2M+3-u(M+2)).
$$
Since $u\in[0,1]$ and $M>p-1$, we obtain
$$
p/(2M+3-u(M+2))<p/(M+1)<1,
$$
that is $T_{s,\lambda,l}(u)\in\cl^1(\cn,\tau)$.
The proof then proceeds by showing that
$$
\big[u\mapsto d_u(a_i)R_{s,u}(\lambda)\big]\in C^1\big([0,1],\cl^q(\cn,\tau)\big),
\quad q>p/(2-u),\quad i=0,\dots,l,l+2,\dots,M,
$$
and
$$
\big[u\mapsto \D_u R_{s,u}(\lambda)\,d_u(a_{l+1})R_{s,u}(\lambda)\big]
\in C^1\big([0,1],\cl^q(\cn,\tau)\big),\quad q>p/(3-2u),
$$
with derivatives given respectively by
$$
[\dot{\D_u},a_i]R_{s,u}(\lambda)+2d_u(a_i)R_{s,u}(\lambda)\dot{\D_u}\D_uR_{s,u}(\lambda),
$$
and
\begin{align*}
&\dot{\D_u} R_{s,u}(\lambda)\,d_u(a_{l+1})R_{s,u}(\lambda)
+2\D_u R_{s,u}(\lambda)\dot{\D_u}\D_uR_{s,u}(\lambda)\,d_u(a_{l+1})R_{s,u}(\lambda)\\&
+\D_u R_{s,u}(\lambda)\,[\dot{\D_u},a_{l+1}]R_{s,u}(\lambda)
+2\D_u R_{s,u}(\lambda)\,d_u(a_{l+1})R_{s,u}(\lambda)\dot{\D_u}\D_uR_{s,u}(\lambda).
\end{align*}
This will eventually  imply the statement of the lemma.

We only treat the first term, the arguments for the second term being similar but algebraically 
more involved.
We write,
\begin{equation}
\label{eq:vaca}
\eps^{-1}({d_{u+\eps}(a_i)R_{s,u+\eps}(\lambda)-d_u(a_i)R_{s,u}(\lambda)})-[\dot{\D_u},a_i]R_{s,u}(\lambda)-2d_u(a_i)R_{s,u}(\lambda)\dot{\D_u}\D_uR_{s,u}(\lambda)
\end{equation}
$$=\Big(\eps^{-1}({d_{u+\eps}(a_i)-d_u(a_i)})-[\dot{\D_u},a_i]\Big)R_{s,u}(\lambda)
+\big(d_{u+\eps}(a_i)-d_u(a_i)\big)\eps^{-1}({R_{s,u+\eps}(\lambda)-R_{s,u}(\lambda)})
$$
\begin{align*} 
\hspace{3cm} +d_u(a_i)\Big(\eps^{-1}({R_{s,u+\eps}(\lambda)-R_{s,u}(\lambda)})-2R_{s,u}(\lambda)\dot{\D_u}\D_uR_{s,u}(\lambda)\Big).
\end{align*}
The first term of Equation \eqref{eq:vaca} is the most involved. We start by writing
\begin{align*}
&\eps^{-1}({d_{u+\eps}(a_i)-d_u(a_i)})-[\dot{\D_u},a_i]=\Big[\eps^{-1}({\D_{u+\eps}-\D_u})+\D_u\log|\D|,a_i\Big]\\&=\Big[F|\D|^{1-u}\Big(\eps^{-1}({|\D|^{-\eps}-1})+\log|\D|\Big),a_i\Big]\\
&=F\Big[|\D|^{1-u}\Big(\eps^{-1}({|\D|^{-\eps}-1})+\log|\D|\Big),a_i\Big]+\big(da_i-F\delta(a_i)\big)|\D|^{-u}\Big(\eps^{-1}({|\D|^{-\eps}-1})+\log|\D|\Big).
\end{align*}

We are seeking convergence  for the Schatten norm   $\Vert\cdot\Vert_q$ 
with $q>p/(2-u)$.
So, let $\rho>0$, be such that for 
$A\in{\rm OP}_0^0$, $A|\D|^{-2+u+\rho}\in\cl^q(\cn,\tau)$. 
Thus, the last term of the previous expression, multiplied by $R_{s,u}(\lambda)$ 
can be estimated in $q$-norm by:
\begin{align*}
&\Big\|\big(da_i-F\delta(a_i)\big)|\D|^{-u}\Big(\eps^{-1}({|\D|^{-\eps}-1})
+\log|\D|\Big)R_{s,u}(\lambda)\Big\|_q\\
&\qquad\leq
\big\|\big(da_i-F\delta(a_i)\big)|\D|^{-2+u+\rho}\big\|_q\,\big\||\D|^{-2(1-u)}R_{s,u}(\lambda)\big\|\,
\Big\|\Big(\eps^{-1}({|\D|^{-\eps}-1})+\log|\D|\Big)\D^{-\rho}\Big\|,
\end{align*}
which treats this term since the last operator norm goes to zero with $\eps$.
We now show that
\begin{equation}
\label{qwerty}
\Big[|\D|^{1-u}\Big(\eps^{-1}({|\D|^{-\eps}-1})+\log|\D|\Big),a_i\Big],
\end{equation}
converges to zero in $q$-norm (for the same values of $q$ as before). 
We first remark that we can assume $u>0$. Indeed, when $u=0$, 
we can use (as before) the little room left between $q$ and $p/2$, 
find $\rho>0$ such that  $a|\D|^{-2+\rho}\in\cl^q(\cn,\tau)$ and write
\begin{align*}
&\Big[|\D|\Big(\eps^{-1}({|\D|^{-\eps}-1})+\log|\D|\Big),a_i\Big]|\D|^{-\rho}\\
&\quad=\Big[|\D|^{1-\rho}\Big(\eps^{-1}({|\D|^{-\eps}-1})+\log|\D|\Big),a_i\Big]-|\D|^{1-\rho}\Big(\eps^{-1}{|\D|^{-\eps}-1})+\log|\D|\Big)\big[|\D|^{\rho},a_i\big]|\D|^{-\rho},
\end{align*}
and use an estimate of the previous type plus the content of Lemma \ref{truc}.

To take care of the term \eqref{qwerty} (for $u>0$), we use the integral formula for fractional powers.
After some rearrangements, this gives the following expression for \eqref{qwerty}:
\begin{align*}
&\int_0^\infty \lambda^{u-1}(\pi\eps)^{-1}\Big\{(\sin\pi(1-u-\eps)-\sin\pi(1-u))(\lambda^\eps-1)
+\sin\pi(1-u)(\lambda^\eps-1-\eps\log\lambda)\\
&+\Big((\pi\eps)^{-1}({\sin\pi(1-u-\eps)-
\sin\pi(1-u)})+\cos\pi(1-u)\Big)\Big\}({1+\lambda|\D|})^{-1}\delta(a_i)({1+\lambda|\D|})^{-1}\,d\lambda.
\end{align*}
The last term can be recombined as
$$
\Big((\pi\eps)^{-1}({\sin\pi(1-u-\eps)-\sin\pi(1-u)})+
\cos\pi(1-u)\Big)\pi{(\sin\pi(1-u))^{-1}}\big[|\D|^{1-u},a_i\big],
$$
and one concludes (for this term) using Lemma \ref{truc} 
together with an (ordinary) Taylor expansion for the pre-factor.

Since $\D^2\geq\mu^2>0$, the first term (multiplied by $R_{s,u}(\lambda)$) 
is estimated (up to a constant) in $q$-norm by
$$
\big|\sin\pi(1-u-\eps)-\sin\pi(1-u)\big|\big\|\delta(a_i)R_{s,u}(\lambda)\big\|_q\,
\int_0^\infty \lambda^{u-1}\eps^{-1}({\lambda^\eps-1}){(1+\lambda\,\mu^{1/2})^{-2}}\,d\lambda,
$$
which goes to zero with $\eps$, as seen by a Taylor expansion of the 
prefactor and since $(\lambda^\eps-1)/\eps$ is uniformly bounded in 
$\eps$ for $\lambda\in[0,1]$, while between $1$ in $\infty$, we use
\begin{align*}
\int_1^\infty \lambda^{u-1}\eps^{-1}({\lambda^\eps-1}){(1+\lambda\,\mu^{1/2})^{-2}} \,d\lambda
&\leq(\mu\,\eps)^{-1}\int_1^\infty \big(\lambda^{u-3+\eps}-\lambda^{u-3}\big)  \,d\lambda\\
&\qquad=({\mu(2-u-\eps)})^{-1}\leq  ({\mu(1-u)})^{-1}.
\end{align*}
For the middle term, we obtain instead the bound (up to a constant depending only on $u$)
$$
\big\|\delta(a_i)R_{s,u}(\lambda)\big\|_q\,\int_0^\infty 
\lambda^{u-1}\eps^{-1}({\lambda^\eps-1-\eps\log(\lambda)}) {(1+\lambda\,\mu^{1/2})^{-2}}\,d\lambda,
$$
and one concludes using the same kind of arguments as employed previously.

Similar (and easier) arguments show that the two other terms in \eqref{eq:vaca} 
converge to zero in $q$-norm. That the derivative of $T_{s,\lambda,l}(u)$ is 
continuous for the trace norm topology follows from analogous arguments.  

Now we consider the case $1\leq p<2$. In this case $M=1$ in the odd case and $M=2$ in the
even case. For the odd case we have two terms to consider,
$$
T_{s,\lambda,0}(u)=d_u(a_0)R_{s,u}(\lambda)\D_u R_{s,u}(\lambda) d_u(a_1)R_{s,u}(\lambda),
$$
and 
$$
T_{s,\lambda,1}(u)=d_u(a_0)R_{s,u}(\lambda) d_u(a_1)R_{s,u}(\lambda)\D_u R_{s,u}(\lambda).
$$
We write $T_{s,\lambda,0}(u)$ as
$$
\underbrace{d_u(a_0)|\D|^{-\frac{5}{2}(1-u)}}_A\ 
\underbrace{R_{s,u}(\lambda)\D_uR_{s,u}(\lambda)|\D|^{3(1-u)}}_B\ 
\underbrace{|\D|^{-\frac{1}{2}(1-u)}d_u(a_1)R_{s,u}(\lambda)}_C.
$$
Now the operator $B$ is uniformly bounded  in $u\in[0,1]$, while Lemma \ref{interpolation} shows that both
$A$ and $C$ lie in $\L^q(\cn,\tau)$ for all $q\geq p$. Since $1>p/2$, 
the H\"{o}lder inequality now shows that $T_{s,\lambda,0}(u)$ lies
in $\L^1(\cn,\tau)$ for each $u\in [0,1]$. Now the strict inequality $1>p/2$ allows us to handle the 
difference quotients as in the $p\geq 2$ case above to obtain the trace norm 
differentiability of $T_{s,\lambda,0}(u)$.

For $T_{s,\lambda,1}(u)$ we write
$$
\underbrace{d_u(a_0)R_{s,u}(\lambda)|\D|^{-2(1-u)}}_A\ 
\underbrace{d_u(\sigma^{(1-u)/2}(a_1))|\D|^{-2(1-u)}R_{s,u}(\lambda)\D_uR_{s,u}(\lambda)}_B.
$$
Applying Lemma \ref{interpolation} and the H\"{o}lder inequality again shows that 
$T_{s,\lambda,1}(u)\in\L^1(\cn,\tau)$. The strict inequality $1>p/2$ again allows us to prove
trace norm differentiability.

For the even case where $M=2$ we have more terms to consider, but the pattern is now clear.
We break up $T_{s,\lambda,j}(u)$ into a product of terms whose Schatten norms we can 
control, and obtain a strict inequality allowing us to control the logarithms arising in the formal
derivative. This completes the proof.
\hfill $\square$

\printindex


\end{document}